\def\todaysdate{30\textsuperscript{th} December 2024}
\documentclass[reqno,a4paper]{article}
\usepackage{etex}
\usepackage[T1]{fontenc}
\usepackage[utf8]{inputenc}
\usepackage{lmodern}

\usepackage[style=alphabetic,giveninits=true,backref,backend=biber,isbn=false,url=false,maxbibnames=99,maxalphanames=4,sorting=nyt]{biblatex}

\renewbibmacro{in:}{}
\newbibmacro{string+doi}[1]{\iffieldundef{doi}{\iffieldundef{url}{#1}{\href{\thefield{url}}{#1}}}{\href{http://dx.doi.org/\thefield{doi}}{#1}}}
\DeclareFieldFormat*{title}{\usebibmacro{string+doi}{\emph{#1}}}
\DefineBibliographyStrings{english}{%
  backrefpage = {$\uparrow$},
  backrefpages = {$\uparrow$},
}

\usepackage{amssymb,amsmath,amsthm}
\usepackage{thmtools,xcolor}
\usepackage{units}
\usepackage{graphicx}
\usepackage[all]{xy}
\usepackage{tikz}
\usetikzlibrary{arrows,calc,positioning,decorations.pathreplacing}
\usepackage{tikz-cd}
\usepackage{relsize}
\usepackage{mhsetup}
\usepackage{mathtools}
\usepackage{stmaryrd}
\usepackage{tocloft}
\usepackage{paralist} 
\usepackage[left=3cm,right=3cm,top=3cm,bottom=2.9cm]{geometry} 
\usepackage[bottom]{footmisc}
\usepackage[colorlinks,citecolor=blue,linkcolor=blue,urlcolor=blue,filecolor=blue,breaklinks]{hyperref}
\usepackage{footnotebackref}
\usepackage[format=plain,indention=1em,labelsep=quad,labelfont={up},textfont={footnotesize},margin=2em]{caption}
\usepackage{subcaption}
\usepackage[mathscr]{euscript}
\usepackage{accents}
\usepackage{xparse}

\definecolor{lightblue}{rgb}{0.8,0.8,1}

\numberwithin{equation}{section}
\numberwithin{figure}{section}

\declaretheoremstyle[
  spaceabove=\topsep,
  spacebelow=\topsep,
  headpunct=,
  numbered=no,
  postheadspace=1ex,
  headfont=\normalfont\bfseries,
  bodyfont=\normalfont\itshape,
]{italic}
\declaretheoremstyle[
  spaceabove=\topsep,
  spacebelow=\topsep,
  headpunct=,
  numbered=no,
  postheadspace=1ex,
  headfont=\normalfont\bfseries,
  bodyfont=\normalfont\upshape,
]{upright}
\declaretheorem[style=italic,name=Theorem,numbered=yes,numberwithin=section]{thm}
\declaretheorem[style=italic,name=Lemma,numbered=yes,numberlike=thm]{lem}
\declaretheorem[style=italic,name=Proposition,numbered=yes,numberlike=thm]{prop}
\declaretheorem[style=italic,name=Corollary,numbered=yes,numberlike=thm]{coro}

\declaretheorem[style=italic,name=Theorem,numbered=yes,numberwithin=section]{athm}

\declaretheorem[style=italic,name=Corollary,numbered=yes,numberlike=athm]{acoro}

\declaretheorem[style=upright,name=Definition,numbered=yes,numberlike=thm]{defn}
\declaretheorem[style=upright,name=Remark,numbered=yes,numberlike=thm]{rmk}
\declaretheorem[style=upright,name=Example,numbered=yes,numberlike=thm]{eg}
\declaretheorem[style=upright,name=Examples,numbered=yes,numberlike=thm]{egs}
\declaretheorem[style=upright,name=Notation,numbered=yes,numberlike=thm]{notation}
\declaretheorem[style=upright,name=Convention,numbered=yes,numberlike=thm]{convention}

\declaretheorem[style=upright,name=Construction,numbered=yes,numberlike=thm]{construction}

\makeatletter
\renewcommand*{\@seccntformat}[1]{\upshape\csname the#1\endcsname.\hspace{1ex}}
\renewcommand*{\section}{\@startsection{section}{1}{\z@}%
	{2.5ex \@plus 1ex \@minus 0.2ex}%
	{1.5ex \@plus 0.2ex}%
	{\normalfont\Large\bfseries}}
\renewcommand*{\subsection}{\@startsection{subsection}{2}{\z@}%
	{2.5ex \@plus 1ex \@minus 0.2ex}%
	{1.5ex \@plus 0.2ex}%
	{\normalfont\large\bfseries}}
\renewcommand*{\subsubsection}{\@startsection{subsubsection}{3}{\z@}%
	{2.5ex \@plus 1ex \@minus 0.2ex}%
	{1.5ex \@plus 0.2ex}%
	{\normalfont\normalsize\bfseries}}
\newcommand*{\subsubsubsection}{\@startsection{paragraph}{4}{\z@}%
	{2.5ex \@plus 1ex \@minus 0.2ex}%
	{1.5ex \@plus 0.2ex}%
	{\normalfont\normalsize\bfseries}}
\newcommand*{\subsubsubsubsection}{\@startsection{paragraph}{5}{\z@}%
	{2.5ex \@plus 1ex \@minus 0.2ex}%
	{1.5ex \@plus 0.2ex}%
	{\normalfont\normalsize\bfseries}}
\makeatother

\newcommand{\Diff}{\mathrm{Diff}}
\newcommand{\Aut}{\mathrm{Aut}}

\newcommand{\B}{\mathbf{B}}
\newcommand{\F}{\mathrm{F}}

\newcommand{\Fct}{\mathbf{Fct}}

\newcommand{\MCGo}{\mathbf{\Gamma}}

\newcommand{\MCGno}{\boldsymbol{\mathcal{N}}}
\newcommand{\Sym}{\mathfrak{S}}
\newcommand{\Hom}{\mathrm{Hom}}
\newcommand{\MCG}{\mathrm{MCG}}
\newcommand{\id}{\mathrm{id}}

\newcommand{\bw}{\mathbf{w}}
\newcommand{\bv}{\mathbf{v}}

\newcommand{\Image}{\mathrm{Im}}

\newcommand{\mediumoplus}{\ensuremath{\mathbin{\textstyle\bigoplus}}}
\newcommand{\M}{\cM_{2}}
\newcommand{\rfl}{\diamond}

\newcommand{\incl}[3][right]%
{%
\draw[<-,>=#1 hook] #2 to ($ #2!0.5!#3 $);
\draw[->,>=stealth'] ($ #2!0.5!#3 $) to #3;%
}
\tikzset{
  Isom/.style={
    draw=none,
    every to/.append style={
      edge node={node [sloped, allow upside down, auto=false]{$\cong$}}}
  }
}

\newenvironment{itemizeb}%
{\begin{compactitem}

}%
{\end{compactitem}}

\newcommand{\cA}{\mathcal{A}}

\newcommand{\cC}{\mathcal{C}}
\newcommand{\cD}{\mathcal{D}}

\newcommand{\cG}{\mathcal{G}}

\newcommand{\cL}{\mathcal{L}}
\newcommand{\cM}{\mathcal{M}}

\newcommand{\cO}{\mathcal{O}}
\newcommand{\cP}{\mathcal{P}}
\newcommand{\cQ}{\mathcal{Q}}

\newcommand{\cS}{\mathcal{S}}

\newcommand{\cW}{\mathcal{W}}

\newcommand{\bA}{\mathbb{A}}

\newcommand{\bC}{\mathbb{C}}
\newcommand{\bD}{\mathbb{D}}

\newcommand{\bI}{\mathbb{I}}

\newcommand{\bK}{\mathbb{K}}
\newcommand{\bL}{\mathbb{L}}
\newcommand{\bM}{\mathbb{M}}
\newcommand{\bN}{\mathbb{N}}

\newcommand{\bQ}{\mathbb{Q}}

\newcommand{\bS}{\mathbb{S}}
\newcommand{\bT}{\mathbb{T}}
\newcommand{\bU}{\mathbb{U}}
\newcommand{\bV}{\mathbb{V}}
\newcommand{\bW}{\mathbb{W}}

\newcommand{\bZ}{\mathbb{Z}}

\newcommand{\tn}{\mathtt{n}}

\newcommand{\tm}{\mathtt{m}}
\newcommand{\one}{\mathtt{1}}
\newcommand{\two}{\mathtt{2}}
\newcommand{\zero}{\mathtt{0}}
\newcommand{\sS}{\mathscr{S}}
\newcommand{\Identity}{\mathrm{Id}}
\newcommand{\wdeg}{\mathrm{wdeg}}
\newcommand{\sdeg}{\mathrm{sdeg}}
\newcommand{\rM}{\mathscr{M}}
\newcommand{\rG}{\mathscr{G}}
\newcommand{\tM}{\mathtt{M}}
\newcommand{\rT}{\mathscr{T}}
\newcommand{\Beta}{\boldsymbol{\beta}}

\newcommand{\dv}{^{\dagger}}
\newcommand{\unt}{u}
\newcommand{\LB}{\mathfrak{LB}}
\newcommand{\LBu}{\mathfrak{LB}^{u}}

\newcommand{\fLu}{\mathfrak{L}^{u}}
\newcommand{\vrtcl}{v}
\newcommand{\fLv}{\mathfrak{L}^{\vrtcl}}
\newcommand{\fLuv}{\mathfrak{L}^{\unt,\vrtcl}}
\newcommand{\fL}{\mathfrak{L}}
\newcommand{\BM}{\mathrm{BM}}
\newcommand{\tw}{\mathrm{tw}}

\newcommand{\LCS}{\varGamma}
\newcommand{\N}{\mathscr{N}}
\newcommand{\obj}{\mathrm{Obj}}
\newcommand{\longhookrightarrow}{\lhook\joinrel\longrightarrow}
\NewDocumentCommand{\covr}{O{\bullet}}{%
\ensuremath{\mathrm{Cov}_{#1}}%
}
\NewDocumentCommand{\topr}{O{\bullet}}{%
\ensuremath{\mathrm{Top}_{#1}}%
}
\NewDocumentCommand{\lmod}{O{\bullet}}{%
\ensuremath{{}\textrm{-}\mathrm{Mod}}%
}
\newcommand{\Ring}{\ensuremath{\mathsf{Ring}}}
\newcommand{\Alg}{\ensuremath{\textrm{-}\mathsf{Alg}}}

\renewcommand{\geq}{\geqslant}
\renewcommand{\leq}{\leqslant}
\renewcommand{\footnoterule}{%
  \kern -3pt
  \hrule width \textwidth height 0.4pt
  \kern 2.6pt
}

\begin{document}
\title{\Large\bfseries Polynomiality of surface braid and mapping class group representations}
\author{\normalsize Martin Palmer and Arthur Souli{\'e}}
\date{\normalsize\todaysdate}
\maketitle
{\makeatletter
\renewcommand*{\BHFN@OldMakefntext}{}
\makeatother
}

\begin{abstract}
We study a wide range of homologically-defined representations of surface braid groups and of mapping class groups of surfaces, including the Lawrence-Bigelow representations of the classical braid groups. These representations naturally come in families, defining homological representation functors on categories associated to surface braid groups or all mapping class groups. We prove that many of these homological representation functors are \emph{polynomial}. This has applications to twisted homological stability and to understanding the structure of the representation theory of the associated families of groups. Our polynomiality results are consequences of more fundamental results establishing relations amongst the coherent representations that we consider via short exact sequences of functors. As well as polynomiality, these short exact sequences also have applications to understanding the kernels of the homological representations under consideration.
\end{abstract}

\section*{Introduction}

The representation theory of surface braid groups and of mapping class groups has been the subject of intensive study for several decades, and continues to be so; see for example the survey of Birman and Brendle \cite[\S 4]{BirmanBrendlesurvey} or the expository article of Margalit \cite{Margalit}. These groups naturally come in families -- we will consider the following ones, where all surfaces are assumed connected, compact and with one boundary component: the family of surface braid groups $\B_{n}(S)$ for each fixed surface $S$ (in particular the classical braid groups $\B_{n}$ when $S$ is the $2$-disc $\bD$), as well as the two families $\MCGo_{g,1}$ and $\MCGno_{h,1}$ of the mapping class groups of orientable and non-orientable surfaces respectively.

\paragraph*{Homological representation functors.}
One way to make the representation theory of these groups more tractable is to study coherent representations of each family of groups. Here, \emph{coherent representation} means a collection of one representation of each group in the family so that the whole collection of representations is compatible, in a certain sense, with the natural homomorphisms between the groups. This structure is encoded by a functor $\langle \cG,\cM \rangle \to R\lmod$, where $\langle \cG,\cM \rangle$ is a certain category whose automorphism groups are the family of groups in question. The categories $\langle \cG,\cM \rangle$ associated to the families of surface braid groups and mapping class groups of surfaces are described in \S\ref{ss:categorical_framework}. In particular, the objects of $\langle \cG,\cM \rangle$ are always indexed by non-negative integers $\tn$, whose automorphism group is either $\B_{n}(S)$, $\MCGo_{n,1}$ or $\MCGno_{n,1}$, depending on the context. For instance, for the classical braid groups $\B_{n}$, we take $\cG = \cM = \Beta$ where $\Beta$ is the braid groupoid.

The coherent representations of surface braid groups and of mapping class groups that we study in this paper are constructed systematically from natural actions on the homology of configuration spaces on the underlying surface, with coefficients twisted by certain local systems on these configuration spaces. Special cases of this construction recover, for example, the \emph{Lawrence-Bigelow representations} of the classical braid groups \cite{Lawrence1,BigelowHomrep} (Example~\ref{eg:LB_rep}), the \emph{An-Ko representations} of surface braid groups \cite{AnKo} (Example~\ref{eg:An-Ko}) and the \emph{Moriyama representations} of mapping class groups \cite{Moriyama} (Example~\ref{eg:Moriyama}). In general, our construction depends on a choice of an \emph{ordered partition $\lambda$ of a positive integer $k$} (see the definition below), corresponding to the number and partition of points in the configuration space, together with a positive integer $\ell$, corresponding to the depth of an associated lower central series quotient that determines the local system. This produces a functor
\begin{equation}
\label{eq:functor-from-general-construction}
\fL_{(\lambda,\ell)} \colon \langle \cG,\cM \rangle \longrightarrow R\lmod,
\end{equation}
called a \emph{homological representation functor}.
This construction (together with some variants) is described in \S\ref{ss:general_construction}--\S\ref{ss:examples_homological_representations_functors}, where we also explain how it fits into the larger framework of \cite{PSI}; we refer the reader there for a full introduction to this notion and we focus on presenting our results about these functors in the rest of the introduction.
We mention for the sake of accuracy that the target category of \eqref{eq:functor-from-general-construction} must in general be enlarged to the category $R\lmod^{\tw}$ of \emph{twisted} $R$-modules: this has the same objects as $R\lmod$ but morphisms are permitted to act on the underlying ring $R$ as well as on the modules; see \S\ref{sss:twisted-representations}. We will elide this subtlety in the introduction, although we are careful in the rest of the paper about when the target is $R\lmod^{\tw}$ and when we may restrict to $R\lmod$. (One may always compose with the functor $R\lmod^{\tw} \to \bZ\lmod$ that forgets module structures to avoid this twisting.) Homological representation functors have already proved themselves of key use for the questions of the \emph{linearity} of groups (see Example~\ref{eg:LB_rep}) and form a natural pathway for the construction of families of irreducible representations (see the forthcoming work \cite{PSIIi}). This paper aims to establish \emph{polynomiality} properties (see \S\ref{s:notions_polynomiality} for an introduction to these notions) for these functors.

When we are not working in a specific setting, we denote the homological representation functors that we construct by $\fL_{(\lambda,\ell)}$, as in \eqref{eq:functor-from-general-construction} above. When we are working in the setting of surface braid groups on a fixed surface $S$, we write $\fL_{(\lambda,\ell)} = \fL_{(\lambda,\ell)}(S)$. In the special case of classical braid groups ($S = \bD$) we also write $\fL_{(\lambda,\ell)}(\bD) = \LB_{(\lambda,\ell)}$, since they extend the $\mathfrak{L}$awrence-$\mathfrak{B}$igelow representations. In the setting of mapping class groups of orientable surfaces, we write $\fL_{(\lambda,\ell)} = \fL_{(\lambda,\ell)}(\MCGo)$; in the setting of mapping class groups of non-orientable surfaces, we write $\fL_{(\lambda,\ell)} = \fL_{(\lambda,\ell)}(\MCGno)$.

\paragraph*{Short exact sequences of functors.}
Our main result proves the existence of fundamental short exact sequences relating different homological representation functors; see Theorem~\ref{thm:ses}. Our polynomiality results, as well as results establishing other properties of these representations, are corollaries of these.

The short exact sequences depend on a ``translation'' operation $\tau_{\one}$ defined on functors \eqref{eq:functor-from-general-construction}, which is defined precisely in \S\ref{sss:translation_background}. Also, for $k \geq 1$ a positive integer, we recall that an \emph{ordered partition $\lambda\vdash k$} is an ordered $r$-tuple $\lambda = (\lambda_{1},\ldots,\lambda_{r})$ of integers $\lambda_{i} \geq 1$ (for some $r \geq 1$ called the \emph{length} of $\lambda$) such that $k = \sum_{1\leq i \leq r} \lambda_{i}$. (Note that we do not impose the condition that $\lambda_{i}\geq \lambda_{i+1}$.)
Also, for a fixed $\lambda\vdash k$, we write $\lambda[j]$ for the tuple obtained by subtracting $1$ from the $j$th term (and removing the $j$th term entirely, if it is now zero); see also Notation~\ref{not:sets_partitions}. Our main result is the following.

\begin{athm}[{Theorems~\ref{thm:key_SES_classical_braids}, \ref{thm:key_SES_surface_braid_groups} and \ref{thm:key_SES_mapping_class_groups}}]
\label{thm:ses}
For any positive integer $k\geq2$, any ordered partition $\lambda = (\lambda_{1},\ldots,\lambda_{r})\vdash k$ and any positive integer $\ell \geq 1$, there is a short exact sequence
\begin{equation}
\label{eq:key_SES_classical_braids_partitioned-intro}
\begin{tikzcd}
0 \ar[r] & \LB_{(\lambda,\ell)} \ar[r] & \tau_{\one}\LB_{(\lambda,\ell)} \ar[r] & \underset{1\leq j\leq r}{\bigoplus}\tau_{\one}\LB_{(\lambda[j],\ell)} \ar[r] & 0
\end{tikzcd}
\end{equation}
of functors $\langle \Beta , \Beta \rangle \to R\lmod$. There are analogous short exact sequences of functors in the settings of surface braid groups -- see \eqref{eq:key_SES_braid_surface} -- and mapping class groups of surfaces  -- see \eqref{eq:key_SES_MCG_orientable} and \eqref{eq:key_SES_MCG_non_orientable}. In the last case, these short exact sequences are moreover \textbf{split}.
\end{athm}

\paragraph*{Polynomiality.}
Our first corollary of Theorem~\ref{thm:ses}, and its analogues for surface braid groups and mapping class groups of surfaces, is that many of the functors \eqref{eq:functor-from-general-construction} are \emph{polynomial} in the senses recalled precisely in \S\ref{s:notions_polynomiality}.

\begin{acoro}
\label{coro:polynomiality}
Fix a positive integer $k \geq 1$, an ordered partition $\lambda\vdash k$ and a positive integer $\ell \geq 1$.
In the setting of the classical braid groups:
\begin{itemizeb}
\item The functor $\LB_{((1),\ell)}$ is strong polynomial of degree $2$ and weak polynomial of degree $1$.
\item For $k\geq 2$, the functor $\LB_{(\lambda,\ell)}$ is both very strong and weak polynomial of degree $k$.
\end{itemizeb}
In the setting of surface braid groups on $S = \Sigma_{g,1} \text{ or } \N_{h,1}$ for $g,h \geq 1$:
\begin{itemizeb}
\item The functor $\fL_{(\lambda,\ell)}(S)$ is weak polynomial of degree $k$ and strong polynomial of degree $k$ or $k+1$. 
\end{itemizeb}
In the setting of mapping class groups of surfaces:
\begin{itemizeb}
\item The functors $\fL_{(\lambda,\ell)}(\MCGo)$ and $\fL_{(\lambda,\ell)}(\MCGno)$ are both split and weak polynomial of degree $k$.
\end{itemizeb}
\end{acoro}

The indeterminacy between $k$ and $k+1$ for the strong polynomial degree of the functor $\fL_{(\lambda,\ell)}(S)$ is a side effect of the interactions between the $\B_{n}(S)$-representations for $n\leq2$ encoded by this functor; see Remark~\ref{rmk:polynomiality_L_one_functor}.
In each setting, we also define an alternative version of each of the functors $\fL_{(\lambda,\ell)}$, which we call its ``\emph{vertical-type alternative}'' functor, denoted by $\fL^{\vrtcl}_{(\lambda,\ell)}$. (This terminology refers to the shape of the homology cycles representing a basis for the underlying modules of these alternative representations; see Figure~\ref{fig:models-dual}.) In general, these functors have very different polynomiality behaviour:

\begin{athm}
\label{thm:polynomiality_vs_vertical}
Fix a positive integer $k \geq 1$, an ordered partition $\lambda\vdash k$ and a positive integer $\ell \geq 1$.
In the setting of the classical braid groups:
\begin{itemizeb}
\item The functor $\LB^{\vrtcl}_{(\lambda,\ell)}$ is not strong polynomial, but it is weak polynomial of degree $0$.
\end{itemizeb}
In the setting of surface braid groups on $S = \Sigma_{g,1} \text{ or } \N_{h,1}$ for $g,h \geq 1$:
\begin{itemizeb}
\item The functor $\fL^{\vrtcl}_{(\lambda,\ell)}(S)$ is not strong polynomial, but it is weak polynomial of degree $0$.
\end{itemizeb}
In the setting of mapping class groups of surfaces:
\begin{itemizeb}
\item For $\ell \in \{1,2\}$, the functors $\fL^{\vrtcl}_{(\lambda,\ell)}(\MCGo)$ and $\fL^{\vrtcl}_{(\lambda,\ell)}(\MCGno)$ are both split and weak polynomial of degree $k$.
\end{itemizeb}
\end{athm}

Corollary~\ref{coro:polynomiality} and Theorem~\ref{thm:polynomiality_vs_vertical} are proven in \S\ref{ss:polynomiality_surface_braid_groups} for the classical braid groups and surface braid groups and in \S\ref{ss:poly_mcg} for mapping class groups of surfaces.
Closely related to the functors $\fL^{\vrtcl}_{(\lambda,\ell)}$ are certain other functors $\fL_{(\lambda,\ell)}^{\vee}$, which are defined on automorphisms by the duals of the representations corresponding to the effect of $\fL_{(\lambda,\ell)}$ on automorphisms; see \S\ref{ss:dual-bases}.
The statement of Theorem~\ref{thm:polynomiality_vs_vertical} also holds for these functors, as we prove in Theorems~\ref{thm:dual_representation_functors_surface_braid_groups} and \ref{thm:dual_representation_functors_mcg}.

\paragraph*{Applications of polynomiality.}

Corollary~\ref{coro:polynomiality} and Theorem~\ref{thm:polynomiality_vs_vertical} have immediate consequences for \emph{twisted homological stability} of surface braid groups and mapping class groups of orientable or non-orientable surfaces. More precisely, it follows from the work of Randal-Williams and Wahl \cite[Theorems~D, I, $5.23$, $5.26$, $5.29$]{RWW} (see Theorem~\ref{thm:homological_stability}) that twisted homological stability holds in each of these settings with coefficients in each functor that is proven in Corollary~\ref{coro:polynomiality} or Theorem~\ref{thm:polynomiality_vs_vertical} to be \emph{strong}, \emph{very strong} or \emph{split} polynomial; see Corollaries~\ref{coro:HS_surface_braids} and \ref{coro:HS_MCG}.

Furthermore, functors of the form $\langle \cG,\cM \rangle \to R\lmod$ that are \emph{weak} polynomial of degree at most $d$ form a category $\mathcal{P}ol_{d}(\langle \cG , \cM \rangle)$ that is \emph{localising} as a subcategory of $\mathcal{P}ol_{d+1}(\langle \cG , \cM \rangle)$ in the sense of \cite[Chapitre~III, \S1]{gabriel}; see Proposition~\ref{prop:properties_polynomiality}. This allows one to define a sequence of quotient categories
\begin{equation}
\label{eq:sequence-of-quotient-categories}
\begin{tikzcd}
\cdots & \mathcal{P}ol_{d}(\langle \cG,\cM \rangle)/\mathcal{P}ol_{d-1}(\langle \cG,\cM \rangle) \ar[l,"{\delta_{\one}(d)}",swap] & \mathcal{P}ol_{d+1}(\langle \cG,\cM \rangle)/\mathcal{P}ol_{d}(\langle \cG,\cM \rangle) \ar[l,"{\delta_{\one}(d+1)}",swap] & \cdots \ar[l,"{\delta_{\one}(d+2)}",swap]
\end{tikzcd}
\end{equation}
where each functor $\delta_{\one}(i)$ is induced by the difference functor defined in \S\ref{sss:translation_background}. It then follows from Corollary~\ref{coro:polynomiality} that:
\begin{acoro}
\label{coro:quotient_categories}
Fix a positive integer $k \geq 1$, an ordered partition $\lambda= (\lambda_{1},\ldots,\lambda_{r})\vdash k$ and a positive integer $\ell \geq 1$. The $(\lambda,\ell)$-Lawrence-Bigelow functor $\LB_{(\lambda,\ell)}$ is a non-trivial element of
\[
\mathcal{P}ol_{k}(\langle \Beta , \Beta \rangle)/\mathcal{P}ol_{k-1}(\langle \Beta , \Beta \rangle).
\]
Moreover, for each $i\in\{1,\ldots,r\}$, the functor $\LB_{(\lambda[i],\ell)}$ is a direct summand of $\delta_{\one}(k)(\LB_{(\lambda,\ell)})$. Similar results hold for the other functors of Corollary~\ref{coro:polynomiality}, which encode representations of $\B_{n}(S)$, $\MCGo_{g,1}$ and $\MCGno_{h,1}$.
\end{acoro}
The notion of polynomiality thus provides a precise organising tool for coherent representations of these families of groups.

\paragraph*{Faithfulness.}
A second application of the short exact sequence \eqref{eq:key_SES_classical_braids_partitioned-intro} of Theorem~\ref{thm:ses} is to deduce the faithfulness of certain homological representations of the classical braid groups:

\begin{acoro}
\label{coro:faithfulness}
For any $\ell\geq 2$ and any ordered partition $\lambda = (\lambda_{1},\ldots,\lambda_{r})\vdash k$ where $\lambda_{i} \geq 2$ for at least one $i \in \{1,\ldots,r\}$, the $\B_{n}$-representation $\LB_{(\lambda,\ell)}(\tn+1)$ is faithful.
\end{acoro}

More precisely, the short exact sequence \eqref{eq:key_SES_classical_braids_partitioned-intro} of Theorem~\ref{thm:ses} implies the existence of inclusions between the kernels of different homological representations of the braid groups; see Proposition~\ref{prop:kernels}. Combining this with the celebrated result of Bigelow~\cite{bigelow2001braid} and Krammer~\cite{KrammerLK} on the faithfulness of certain homological representations of the braid groups, we deduce Corollary~\ref{coro:faithfulness}, which tells us that many other homological representations of the braid groups are also faithful.

\paragraph*{Analyticity.}
As a final application, we prove analyticity (and non-polynomiality) of a functor encoding certain quantum representations of the braid groups. 

There is a representation $\bV$ of $\bU_{q}(\mathfrak{sl}_{2})$, the quantum enveloping algebra of the Lie algebra $\mathfrak{sl}_{2}$, defined over the ring $\bL:=\bZ[\mathfrak{s}^{\pm 1},\mathfrak{q}^{\pm 1}]$, introduced by Jackson and Kerler \cite{Jackson_Kerler} and called the \emph{generic Verma module}. The structure of $\bU_{q}(\mathfrak{sl}_{2})$ as a quasitriangular Hopf algebra induces a $\B_{n}$-representation on its $n$th tensor power $\bV^{\otimes n}$, which we call the $n$th \emph{Verma module representation}.

\begin{acoro}[Corollary~\ref{coro:Verma_tauLB}]
\label{coro:analyticity}
There is a functor $\mathfrak{Ver} \colon \langle \Beta,\Beta \rangle \to \bL\lmod$ whose restriction to the automorphism group of the object $\tn$ of $\langle \Beta , \Beta \rangle$ is the $n$th Verma module representation. Moreover, this functor is analytic, i.e.~a colimit of polynomial functors, but it is not polynomial.
\end{acoro}

\begin{rmk}
Theorem~\ref{thm:polynomiality_vs_vertical} and Corollary~\ref{coro:analyticity} illustrate that polynomiality is not an ``automatic'' property of coherent representations, even in cases (such as the first two points of Theorem~\ref{thm:polynomiality_vs_vertical}) where the dimensions of the underlying modules of the representations grow polynomially with $\tn$. See also Remark~\ref{rmk:Magnus_Heisenberg} for more examples of non-polynomial homological representation functors.
\end{rmk}

\paragraph*{Module structure.}
Finally, a key tool to prove Theorem~\ref{thm:ses} is an explicit computation (see Theorem~\ref{thm:module-structure} below) of the underlying module structures of the homological representation functors.
Let $S$ be a compact, connected surface with one boundary component and let $A \subset S$ be either a finite subset of its interior or one point on its boundary. For an ordered partition $\lambda = (\lambda_{1},\ldots,\lambda_{r})\vdash k$, we consider the $\lambda$-partitioned configuration space
\[
C_{\lambda}(S \smallsetminus A) = \{ (x_{1},\ldots,x_{k}) \in (S \smallsetminus A)^{k} \mid x_{i} \neq x_{j} \text{ for } i \neq j \} / \Sym_{\lambda},
\]
where $\Sym_{\lambda} := \Sym_{\lambda_{1}} \times \cdots \times \Sym_{\lambda_{r}} \subseteq \Sym_{k}$. Let $\cL$ be a local system on $C_{\lambda}(S \smallsetminus A)$, defined over a ring $R$, and denote its fibre by $V$. The underlying modules of all of our representations are given by the twisted Borel-Moore homology modules $H_{k}^{\BM}(C_{\lambda}(S \smallsetminus A) ; \cL)$.

\begin{athm}[{Proposition~\ref{prop:module_structure_BM_homology}}]
\label{thm:module-structure}
The twisted Borel-Moore homology $H_{*}^{\BM}(C_{\lambda}(S \smallsetminus A) ; \cL)$ is trivial except in degree $k$. There is an isomorphism of $R$-modules
\begin{equation}
\label{eq:module-structure}
H_{k}^{\BM}(C_{\lambda}(S \smallsetminus A) ; \cL) \;\cong\; \bigoplus_{w} V ,
\end{equation}
where the direct sum on the right-hand side is indexed by the following combinatorial data. Let $\Gamma$ be an embedded graph in $S$ with set of vertices $A$, such that $S$ deformation retracts onto $\Gamma$ relative to $A$; see Figure~\ref{fig:models} for illustrations. There is then one copy of $V$ in the direct sum for each function $w$ assigning to each edge of $\Gamma$ a word in the alphabet $\{1,\ldots,r\}$ so that each letter $i \in \{1,\ldots,r\}$ appears precisely $\lambda_{i}$ times as $w$ runs over all edges of $\Gamma$.
\end{athm}

In each of our examples, $\cL$ will be a rank-$1$ local system, i.e.~$V \cong R$, so Theorem~\ref{thm:module-structure} says that $H_{k}^{\BM}(C_{\lambda}(S \smallsetminus A) ; \cL)$ is a free $R$-module, with a free generating set given by the set of functions $w$ described above.
We note that, in the special case when $\lambda = (k)$, the direct sum in \eqref{eq:module-structure} is indexed by functions $w$ assigning non-negative integers to each edge of $\Gamma$ that sum to $k$.

\begin{rmk}
We refer the reader to \cite[Chap.~V]{bredonsheaf} for a detailed introduction to Borel-Moore homology.
The principal reason why we work with Borel-Moore homology instead of ordinary homology is the structural result of Theorem~\ref{thm:module-structure}. In contrast, the ordinary (co)homology of configuration spaces on surfaces is in general much more complicated, and the few cases in which the computations are known lead to representations that are much harder to work with; see for instance \cite[Th.~$1.4$]{stavrou}.
\end{rmk}

\paragraph*{Outline.}
In \S\ref{s:background}, we explain the categorical framework for the families of groups that we work with (\S\ref{ss:categorical_framework}), construct the functors \eqref{eq:functor-from-general-construction} in this framework (\S\ref{ss:general_construction}) and then discuss this construction in more detail (\S\ref{ss:examples_homological_representations_functors}) in each of our three settings: classical braid groups, surface braid groups and mapping class groups of surfaces. In \S\ref{s:free_generating_sets}, we study the underlying module structure of these representations, proving Theorem~\ref{thm:module-structure}. In \S\ref{s:SES_homol_rep-func} we then construct the short exact sequences of Theorem~\ref{thm:ses}, recalling first the necessary background on \emph{translation}, \emph{difference} and \emph{evanescence} operations on functors (\S\ref{ss:background_preliminaries_SES}), and also prove Corollary~\ref{coro:faithfulness}. Finally, in \S\ref{s:poly_homol_rep-func} we prove our results on polynomiality (Corollary~\ref{coro:polynomiality} and Theorem~\ref{thm:polynomiality_vs_vertical}) and analyticity (Corollary~\ref{coro:analyticity}).

\paragraph*{General notation.}
We denote by $\bN$ the set of non-negative integers. For a small category $\cC$, we use the abbreviation $\mathrm{ob}(\cC)$ to denote the set of objects of $\cC$.
For $\cD$ any category and $\cC$ a small category, we denote by $\Fct(\cC,\cD)$ the category of functors from $\cC$ to $\cD$.
For $X$ a manifold with boundary, $\mathring{X}$ denotes its interior. For an ordered partition $\lambda = (\lambda_{1},\ldots,\lambda_{r}) \vdash k$ and a topological space $Y$, we denote by $C_{\lambda}(Y)$ the $\lambda$-configuration space $\left\{ \left(x_{1},\ldots,x_{k}\right)\in Y^{\times k}\mid x_{i}\neq x_{j}\textrm{ if }i\neq j\right\} /\Sym_{\lambda}$, with $\Sym_{\lambda}:=\Sym_{\lambda_{1}} \times \cdots \times \Sym_{\lambda_{r}}$, where $\Sym_{n}$ denotes the symmetric group on a set of $n$ elements.
Furthermore, the \emph{$\lambda$-partitioned braid group $\B_{\lambda_{1},\ldots,\lambda_{r}}(Y)$ on $k$ strings on $Y$} is the fundamental group $\pi_{1}(C_{\lambda}(Y),c_{0})$, where $c_{0} \in C_{\lambda}(Y)$ is a fixed $\lambda$-partitioned configuration.

For $R$ a non-zero unital ring, we denote by ${R}\lmod$ the category of left $R$-modules. For an $R$-module $M$, we denote by $\mathrm{Aut}_{R}(M)$ the group of $R$-module automorphisms of $M$. When $R=\bZ$, we omit it from the notation as long as there is no ambiguity.
The lower central series of a group $G$ is the descending chain of subgroups $\{ \LCS_{\ell}(G)\} _{\ell\geq 1}$ defined by $\LCS_{1}(G):=G$ and $\LCS_{\ell+1}(G):=[G,\LCS_{\ell}(G)]$, the subgroup of $G$ generated by the commutators $[g,h]$ for $g \in G$ and $h \in \LCS_{\ell}(G)$. For the sake of simplicity, each quotient $G/\LCS_{\ell}(G)$ is denoted by $G/\LCS_{\ell}$.

\paragraph*{Acknowledgements.}
The authors would like to thank Tara Brendle, Brendan Owens, Geoffrey Powell, Oscar Randal-Williams and Christine Vespa for illuminating discussions and questions. They would also like to thank Oscar Randal-Williams for inviting the first author to the University of Cambridge in November 2019, where the authors were able to make significant progress on the present article. In addition, they would like to thank Geoffrey Powell for a careful reading and many valuable comments on an earlier draft of this article.
Finally, the authors thank the anonymous referee for very helpful comments and corrections on earlier versions of this work.

The first author was partially supported by a grant of the Romanian Ministry of Education and Research, CNCS - UEFISCDI, project number PN-III-P4-ID-PCE-2020-2798, within PNCDI III. The second author was partially supported by a Rankin-Sneddon Research Fellowship of the University of Glasgow, by the Institute for Basic Science IBS-R003-D1 and by the ANR Projects ChroK ANR-16-CE40-0003 and AlMaRe ANR-19-CE40-0001-01. The authors were able to make significant progress on the present article thanks to research visits to Glasgow and Bucharest, funded respectively by the School of Mathematics and Statistics of the University of Glasgow and the above-mentioned grant PN-III-P4-ID-PCE-2020-2798.

\tableofcontents

\section{Background}\label{s:background}

This section recollects the construction of homological representation functors introduced in \cite{PSI}; see \S\ref{ss:general_construction}.
We first recall the underlying categorical framework in \S\ref{ss:categorical_framework} and then detail in \S\ref{ss:examples_homological_representations_functors} the outputs of the construction of \cite[\S 2--\S 3]{PSI} for the families of groups studied in this paper.

\subsection{Categorical framework for families of groups}\label{ss:categorical_framework}

We introduce here the categorical framework that is central to this paper to deal with families of groups.

\paragraph*{Preliminaries on categorical tools.}
We refer to \cite{MacLane1} for a complete introduction to the notions of strict monoidal categories and modules over them. We generically denote a strict monoidal category by $(\cC,\natural,\zero)$, where $\cC$ is a category, $\natural$ is the monoidal product and $\zero$ is the monoidal unit. If it is braided, then its braiding is denoted by $b_{A,B}^{\cC} \colon A\natural B\overset{\sim}{\to}B\natural A$ for all objects $A$ and $B$ of $\cC$. A left-module $(\cM,\sharp)$ over a (strict) monoidal category $(\cC,\natural,\zero)$ is a category $\cM$ with a functor $\sharp\colon\cC\times\cM\to\cM$ that is unital and associative. For instance, a monoidal category $(\cC,\natural,\zero)$ is equipped with a (strict) left-module structure over itself, induced by its monoidal product. Each left-module structure $\sharp$ in this paper is defined from some underlying monoidal structure $\natural$ (see \S\ref{ss:categories_examples}), so we abuse notation by using the same symbol $\natural$ for $\sharp$.

Considering the category of (small skeletal strict) braided monoidal groupoids $\mathfrak{BrG}$, there is always an arbitrary binary choice for the convention of the braiding. We may pass from one to the other by the following inversion of the braiding operator. Let $(-)\dv\colon\mathfrak{BrG}\to\mathfrak{BrG}$ be the endofunctor defined on each object $(\cG,\natural,\zero)$ by $(\cG,\natural,\zero)\dv=(\cG,\natural,\zero)$ as a monoidal groupoid but whose braiding is defined by the inverse of that of $(\cG,\natural,\zero)$, i.e.~$b^{\cG\dv}_{A,B}:=(b^{\cG}_{B,A})^{-1}$.

\subsubsection{The Quillen bracket construction}\label{ss:Quillen}

In this section, we describe a useful categorical construction to encode families of groups: the \emph{bracket construction} due to Quillen, which is a particular case of a more general construction described in \cite[p.219]{graysonQuillen}; see also \cite[\S 1]{RWW}.  A reader familiar with \cite{RWW} may skip this subsection.

Throughout \S\ref{s:background}, we fix an object $(\cG,\natural,\zero)$ of $\mathfrak{BrG}$ and a (small, strict) left-module $(\cM,\natural)$ over $\cG$. The \emph{Quillen bracket construction} $\langle \cG,\cM \rangle $ on the left-module $(\cM,\natural)$
over the groupoid $(\cG,\natural,\zero)$ is the category with the same objects as $\cM$ and whose morphisms are given by:
\[
\Hom_{\langle \cG,\cM \rangle }(A,B)=\underset{\cG}{\mathrm{colim}}[\Hom_{\cM}(-\natural A,B)].
\]
Thus, a morphism from $A$ to $B$ in $\langle \cG,\cM \rangle$ is an equivalence class of pairs $(X,\varphi)$, denoted by $[X,\varphi] \colon A\to B$, with $X$ an object of $\cG$ and $\varphi \colon X\natural A\to B$ a morphism in $\cM$, and where $[X,\varphi]\sim [X',\varphi']$ if $X=X'$ and there exists $\chi\in \Aut_{\cG}(X)$ such that $\varphi'\circ (\chi\natural \id_{A})=\varphi$.
Also, for two morphisms $[X,\varphi] \colon A\to B$ and $[Y,\psi] \colon B\to C$ in $\langle \cG,\cM \rangle $, the composition is defined by
\begin{equation}\label{eq:composition_rule}
[Y,\psi]\circ[X,\varphi]=[Y\natural X,\psi\circ(\id_{Y}\natural\varphi)].
\end{equation}
There is a canonical faithful functor $\cM\hookrightarrow\langle \cG,\cM \rangle $ defined as the identity on objects and sending $f\in\Hom_{\cM}(A,B)$ to $[\zero,f]$. 
From now on, we assume that $\cM$ is a groupoid, that $(\cG,\natural,\zero)$ has \emph{no zero divisors} -- i.e.~$X\natural Y\cong\zero$ if and only if $X\cong Y\cong\zero$ for $X,Y\in \obj(\cG)$ -- and that $\mathrm{Aut}_{\cG}(\zero)=\{ \id_{\zero}\} $. These properties are satisfied in each setting that we consider in this paper; see \S\ref{ss:categories_examples}.
Then $\cM$ is the maximal subgroupoid of $\langle \cG,\cM \rangle$ and, when considering elements of $\Hom_{\langle \cG,\cM \rangle}(A,A)$, we abuse notation and write $f$ for $[\zero,f]$ for each $A\in\obj(\cM)$ and $f\in\Aut_{\cM}(A)$; see \cite[Prop.~$1.7$]{RWW}.

\paragraph*{Module structure over $\cG$.}
The following discussion is a direct generalisation of \cite[Prop.~$1.8$]{RWW} to which we refer for further details.
The category $\langle \cG,\cM \rangle$ inherits a (strict) left-module structure over $(\cG,\natural,\zero)$ as follows. The module bifunctor $\natural$ extends to $\langle \cG,\cM \rangle $ with the same assignment on objects, by letting, for $\varphi\in \Aut_{\cG}(X)$ and $[Y,\psi]\in \Hom_{\langle \cG,\cM \rangle }(B,C)$:
\begin{equation}\label{eq:formula_morphism_{i}d_plus_morphism}
\varphi\natural[Y,\psi]=[Y,(\varphi\natural\psi)\circ((b_{X,Y}^{\cG})^{-1}\natural \id_{B})].
\end{equation}

\paragraph*{Extensions along the Quillen bracket construction.}
Note that precomposing by the canonical inclusion $\cM\hookrightarrow\langle \cG,\cM \rangle$ induces the restriction functor $\mathbf{Fct}(\langle \cG,\cM \rangle ,\cC)\to\mathbf{Fct}(\cM,\cC)$.
The following result provides a way to extend a functor out of the category $\cM$ to a functor with $\langle \cG,\cM \rangle$ as source category. Its proof repeats mutatis mutandis that of \cite[Lem.~$1.2$]{soulieLMgeneralised}.
\begin{lem}\label{lem:extend_functor_Quillen}
Let $\cC$ be a category and F an object of $\Fct(\cM,\cC)$. Assume that, for each $X\in\obj(\cG)$ and $A\in\obj(\cM)$, there exists a morphism $\alpha_{X,A} \colon F(A)\to F(X\natural A)$ such that $\alpha_{Y,X\natural A}\circ\alpha_{X,A}=\alpha_{Y\natural X,A}$ for all $Y\in\obj(\cG)$ and $\alpha_{\zero,A}=\id_{F(A)}$.
Then the assignments $F([X,\varphi])=F(\varphi)\circ\alpha_{X,A}$ to all morphisms $[X,\varphi] \colon A\to B$ of $\langle \cG,\cM \rangle$ extend the functor $F\colon\cM\to\cC$ to a functor $F\colon\langle \cG,\cM \rangle\to\cC$ if and only if for all $X\in\obj(\cG)$ and $A\in\obj(\cM)$, and for all $\varphi'\in\Aut_{\cG}(X)$ and $\varphi''\in\Aut_{\cM}(A)$, the following relation holds:
\begin{equation}\label{eq:criterion'}
    F(\varphi'\natural \varphi'') \circ \alpha_{X,A} = \alpha_{X,A} \circ F(\varphi'').
\end{equation}
\end{lem}
Similarly, we may extend a morphism in $\mathbf{Fct}(\cM,\cC)$ to a morphism in $\mathbf{Fct}(\langle \cG,\cM \rangle ,\cC)$ thanks to the following result, whose proof repeats verbatim that of \cite[Lem.~$1.12$]{soulieLM1}.
\begin{lem}\label{lem:criterionnaturaltransfo}
Let $\cC$ be a category, $F$ and $G$ objects of $\mathbf{Fct}(\langle \cG,\cM \rangle ,\cC)$ and $\eta \colon F\to G$ a natural transformation in $\mathbf{Fct}(\cM,\cC)$. Then $\eta$ is a natural transformation in the category $\mathbf{Fct}(\langle \cG,\cM \rangle ,\cC)$ if and only if for all $A,B\in \mathrm{Ob}(\cM)$ such that $B\cong X\natural A$ with $X\in \mathrm{Ob}(\cG)$:
\begin{equation}
\eta_{B}\circ F([X,\id_{B}])=G([X,\id_{B}])\circ\eta_{A}.\label{eq:criterion''}
\end{equation}
\end{lem}

\subsubsection{Categories for surface braid groups and mapping class groups}\label{ss:categories_examples}

We now describe the categories associated to the families of groups that we study.
For our purposes, we shall construct each of our skeletal categories synthetically, rather than distilling them from some more natural, larger categories. This will allow us to get a more concrete handle on them for direct calculations.

All of the categories that we consider will be of the form $\langle \cG , \cM \rangle$ for a braided monoidal (small) groupoid $\cG$ and left $\cG$-module (small) groupoid $\cM$, both of them skeletal and strict. Moreover, in each case we also have $\obj(\cG) = \obj(\cM) = \bN$, with the monoidal structure and left action given on objects by addition and both denoted by $\natural$. Thus we just have to describe the groupoids $\cG$ and $\cM$, in each case, at the level of morphisms. This consists in specifying, for all $m,n \in \bN$:
\begin{compactenum}[(1)]
\item\label{groupoid-data-1} groups $\rG_{n}$ and $\rM_{n}$;
\item\label{groupoid-data-2} group homomorphisms $\theta_{m,n} \colon \rG_{m} \times \rG_{n} \to \rG_{m+n}$ that are associative and unital;
\item\label{groupoid-data-3} group homomorphisms $\alpha_{m,n} \colon \rG_{m} \times \rM_{n} \to \rM_{m+n}$ that associate with $\theta_{m,n}$ and are unital;
\item\label{groupoid-data-4} elements $b_{m,n} \in \rG_{m+n}$ that conjugate $\theta_{m,n}(g_{1},g_{2})$ to $\theta_{n,m}(g_{2},g_{1})$ for each $g_{1} \in \rG_{m}$, $g_{2} \in \rG_{n}$.
\end{compactenum}

In fact, it is enough to specify the above for positive $m$ and $n$ (ignoring the unitality conditions) and then set $\rG_{0} = \rM_{0} = \{1\}$, extend $\theta_{m,n}$ and $\alpha_{m,n}$ by unitality and set $b_{0,n} = b_{n,0} = 1 \in \rG_{n}$. This is what we will do in each case. In some cases we will have $\cG = \cM$, which corresponds to $\rG_{n} = \rM_{n}$ and $\theta_{m,n} = \alpha_{m,n}$. In these cases we will therefore not separately specify $\rM_{n}$ and $\alpha_{m,n}$.

The groupoids $\M^{+}$, $\M^{-}$, $\Beta$ and $\Beta^{S}$ that we construct are equivalent to those of \cite[\S 5.6]{RWW} and \cite[\S 3.1]{soulieLMgeneralised}. The Quillen bracket category $\langle \M^{+} , \M^{+} \rangle$ is also similar to the category used in \cite{Ivanov} to index local coefficient systems. We also note here that all of these groupoids have no zero-divisors, since they are skeletal and their underlying monoid of objects is $\bN$.

\begin{figure}
    \centering
    \includegraphics[scale=0.7]{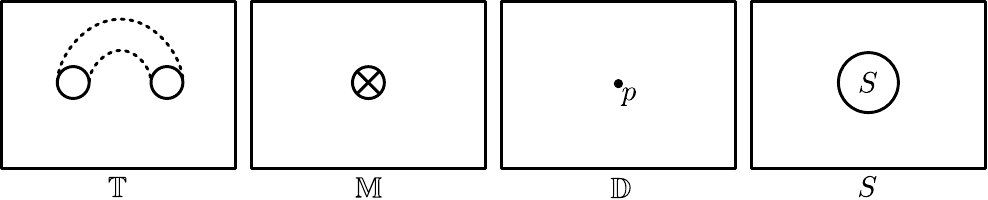}
    \caption{The four building blocks to construct the groupoids $\M^{+}$, $\M^{-}$, $\Beta$ and $\Beta^{S}$.}
    \label{fig:model-surfaces}
    \vspace{1em}
    \includegraphics[scale=0.7]{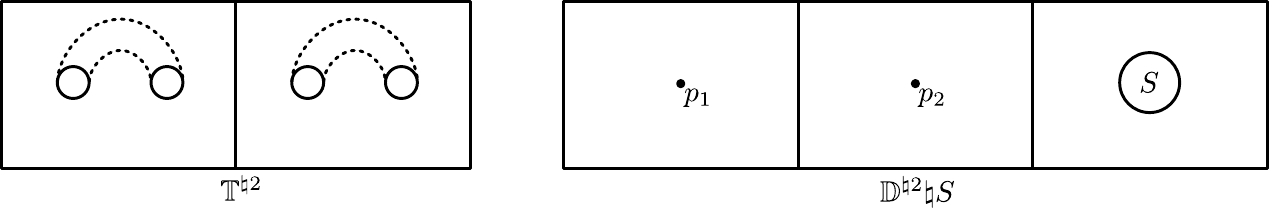}
    \caption{The monoidal structure on $\M^{+}$ and the left action of $\Beta$ on $\Beta^{S}$.}
    \label{fig:model-surfaces-gluing}
    \vspace{1em}
    \includegraphics[scale=0.7]{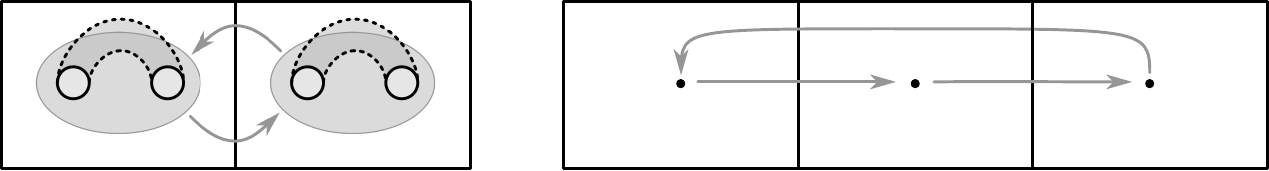}
    \caption{The braidings on $\M^{+}$ and on $\Beta$.}
    \label{fig:model-surfaces-braiding}
\end{figure}

\subsubsubsection{Some injectivity results}
\label{ssss:injectivity-results}

In each of the examples that we construct in \S\ref{sss:category_mcg} and \S\ref{sss:category_surface_braid_groups}, the group homomorphisms \eqref{groupoid-data-2} and \eqref{groupoid-data-3} will be injective. In order to prove this, we collect here some general injectivity results for homomorphisms of mapping class groups and surface braid groups. We use the following general notation for mapping class groups.

\begin{notation}
For a surface $S$, we write $\MCG(S)$ for its mapping class group, namely the group $\pi_{0}(\Diff_\partial(S))$ of isotopy classes of diffeomorphisms of $S$ that fix a neighbourhood of $\partial S$ pointwise. For example, in this notation we have $\MCGo_{g,1} = \MCG(\Sigma_{g,1})$ and $\MCGno_{h,1} = \MCG(\N_{h,1})$. For a non-negative integer $k$, we write $\MCG(S,k)$ for the mapping class group of $S$ with $k$ marked points, namely the group $\pi_{0}(\Diff_\partial(S,k))$ of isotopy classes of diffeomorphisms of $S$ that fix a neighbourhood of $\partial S$ pointwise and also fix a subset $\cP \subset \mathring{S}$ of cardinality $k$ setwise.
\end{notation}

\begin{lem}
\label{lem:injectivity_mcg}
Let $S_{1}$, $S_{2}$ be two compact surfaces, possibly with finitely many punctures, and let $S_{1} \natural S_{2}$ denote their boundary connected sum along specified intervals in their boundaries. Then the homomorphism $\MCG(S_{1}) \times \MCG(S_{2}) \to \MCG(S_{1} \natural S_{2})$ induced by extending diffeomorphisms by the identity is injective.
\end{lem}
\begin{proof}
For $i\in\{1,2\}$, let $f_{i}$ be a diffeomorphism of $S_{i}$ fixing a neighbourhood of its boundary pointwise, and suppose that $f_{1} \natural f_{2}$ is isotopic to the identity on $S_{1} \natural S_{2}$. We must show that $f_{i}$ is isotopic to the identity on $S_{i}$ for $i=1,2$. Let $C$ be any simple proper arc or simple closed curve in $S_{i}$ that does not bound a disc. By the Alexander method \cite[Prop.~$2.8$]{farbmargalit}, it is enough to show that $f_i(C)$ is isotopic to $C$. By a preliminary isotopy, if necessary, we may assume that the endpoints of $C$ (if it is an arc) do not lie on the interval along which the boundary connected sum is taken; thus we may view $C$ as a simple closed curve or simple arc in $S_{1} \natural S_{2}$. Since $f_{1} \natural f_{2}$ is isotopic to the identity, we deduce that $f_i(C)$ is isotopic to $C$ in $S_{1} \natural S_{2}$. Applying any retraction $S_{1} \natural S_{2} \to S_{i}$ to this isotopy, it follows that $f_i(C)$ is homotopic to $C$ in $S_{i}$. By \cite{Epstein} (see also \cite{Baer27,Baer28}), this implies that $f_i(C)$ is isotopic to $C$ in $S_{i}$.
\end{proof}

\begin{coro}
\label{coro:injectivity}
In the setting of Lemma~\ref{lem:injectivity_mcg}, consider the map $C_m(S_{1}) \times C_{n}(S_{2}) \to C_{m+n}(S_{1} \natural S_{2})$ defined by taking unions of configurations. Then the induced homomorphism $\B_m(S_{1}) \times \B_{n}(S_{2}) \to \B_{m+n}(S_{1} \natural S_{2})$ of braid groups is injective.
\end{coro}
\begin{proof}
There is a commutative square
\[
\begin{tikzcd}
\B_m(S_{1}) \times \B_{n}(S_{2}) \ar[rr] \ar[d,hook] && \B_{m+n}(S_{1} \natural S_{2}) \ar[d,hook] \\
\MCG(S_{1},m) \times \MCG(S_{2},n) \ar[rr] && \MCG(S_{1} \natural S_{2},m+n)
\end{tikzcd}
\]
where the vertical maps are the inclusions of the Birman exact sequence (see for instance \cite[Lem.~$4.16$]{farbmargalit} or \cite[Cor.~$1.34$]{PSI}), the top horizontal map is the homomorphism under consideration and the bottom horizontal map is the homomorphism of Lemma~\ref{lem:injectivity_mcg}. Since the bottom horizontal map is injective, by Lemma~\ref{lem:injectivity_mcg}, so is the top one.
\end{proof}

\subsubsubsection{For mapping class groups of surfaces}\label{sss:category_mcg}

The setting for mapping class groups of surfaces corresponds to two braided monoidal groupoids, $\M^{+}$ and $\M^{-}$, which we define by specifying the data \eqref{groupoid-data-1}--\eqref{groupoid-data-4} above. Denote by $\bT$ the one-holed torus, with its boundary identified with a fixed rectangle, as in Figure~\ref{fig:model-surfaces}. For an integer $n\geq 1$, write $\bT^{\natural n}$ for the result of gluing $n$ copies of $\bT$ side by side; see Figure~\ref{fig:model-surfaces-gluing} for the case $n=2$. For the groups \eqref{groupoid-data-1}, we set $\rG_{n}$ to be the mapping class group $\MCG(\bT^{\natural n}) = \pi_{0}(\Diff_{\partial}(\bT^{\natural n}))$, i.e.~the group of isotopy classes of self-diffeomorphisms of $\bT^{\natural n}$ restricting to the identity on a neighbourhood of $\partial (\bT^{\natural n})$. The monoidal structure \eqref{groupoid-data-2} is given by the homomorphisms
\begin{equation}
\label{eq:monoidal-structure-Mplus}
\theta_{m,n} \colon \MCG(\bT^{\natural m}) \times \MCG(\bT^{\natural n}) \longrightarrow \MCG(\bT^{\natural m+n}),
\end{equation}
for $m,n \geq 1$, that send $(\varphi_{1},\varphi_{2})$ to the diffeomorphism of $\bT^{\natural m+n}$ that acts by $\varphi_{1}$ on the left-hand $m$ copies of $\bT$ and by $\varphi_{2}$ on the right-hand $n$ copies of $\bT$. (This is a well-defined diffeomorphism since $\varphi_{1}$ and $\varphi_{2}$ agree on and are the identity in a neighbourhood of the interval in which their domains of definition intersect.) This operation is clearly associative. The homomorphisms \eqref{eq:monoidal-structure-Mplus} are also injective by Lemma~\ref{lem:injectivity_mcg}.

For $m,n \geq 1$, we specify the element $b_{m,n} \in \MCG(\bT^{\natural m+n})$ in \eqref{groupoid-data-4} to act as follows. First, for an interval of integers $A \subseteq \{1,\ldots,m+n\}$, let $S_A \subset \bT^{\natural m+n}$ denote the union of the $\lvert A \rvert$ consecutive copies of $\bT$ indexed by $A$ minus a collar neighbourhood of the boundary of this subsurface. For example, the left-hand side of Figure~\ref{fig:model-surfaces-braiding} illustrates the subsurfaces $S_{\{1\}}$ and $S_{\{2\}}$ in $\bT^{\natural 2}$. The diffeomorphism $b_{m,n}$ is then defined separately in three regions. It restricts to $S_{\{1,\ldots,m\}} \to S_{\{1+n,\ldots,m+n\}}$ by translating $n$ steps to the right, it restricts to $S_{\{m+1,\ldots,m+n\}} \to S_{\{1,\ldots,n\}}$ by translating $m$ steps to the left and it restricts to
\[
\bT^{\natural m+n} \smallsetminus (S_{\{1,\ldots,m\}} \sqcup S_{\{m+1,\ldots,m+n\}}) \longrightarrow \bT^{\natural m+n} \smallsetminus (S_{\{1,\ldots,n\}} \sqcup S_{\{1+n,\ldots,m+n\}})
\]
by a positive half-Dehn twist, i.e.~an anticlockwise half-twist. For example, the left-hand side of Figure~\ref{fig:model-surfaces-braiding} illustrates the diffeomorphism $b_{1,1}$ of $\bT^{\natural 2}$. This completes the definition of the braided monoidal groupoid $\M^{+}$.

The definition of $\M^{-}$ is almost identical, the only difference being that the basic building block is the Möbius band $\bM$ (see Figure~\ref{fig:model-surfaces}) instead of $\bT$.

We denote the mapping class groups $\MCG(\bT^{\natural n})$ and $\MCG(\bM^{\natural n})$ by $\MCGo_{n,1}$ and $\MCGno_{n,1}$ respectively. We thus say that the braided monoidal groupoid $\M^{+}$ encodes the sequence of orientable mapping class groups $\MCGo_{n,1}$ (together with the extra structure of the operations \eqref{eq:monoidal-structure-Mplus} and the braidings $b_{m,n}$); similarly, $\M^{-}$ encodes the sequence of non-orientable mapping class groups $\MCGno_{n,1}$.

\subsubsubsection{For braid groups on surfaces}\label{sss:category_surface_braid_groups}

The setting for surface braid groups corresponds to the braided monoidal groupoid $\Beta$ and the left $\Beta$-module groupoid $\Beta^{S}$ (for a surface $S$), which we define by specifying the data \eqref{groupoid-data-1}--\eqref{groupoid-data-4}. This definition makes sense for any surface $S$ with one boundary component, but we will be particularly interested in the cases $S = \bT^{\natural g}$ and $S = \bM^{\natural h}$, which we denote by $\Sigma_{g,1}$ and $\N_{h,1}$ respectively.

To begin, we recall the definition of partitioned surface braid groups. There are several ways to define these; see for example \cite[\S 6.2--6.3]{DPS} for an overview.
For an ordered partition $\lambda=(\lambda_{1},\ldots,\lambda_{r}) \vdash n$, we recall that the \emph{$\lambda$-partitioned braid group $\B_{\lambda_{1},\ldots,\lambda_{r}}(S)$ on $n$ strings on the surface $S$} is the fundamental group $\pi_{1}(C_{\lambda}(S),c_{0})$ of the $\lambda$-partitioned configuration space $C_{\lambda}(S)$, where $c_{0} \in C_{\lambda}(S)$ is a fixed $\lambda$-partitioned configuration.
The braid groups on the $2$-disc $\bD$ are the \emph{classical braid groups}; we write $\B_{\lambda_{1},\ldots,\lambda_{r}} = \B_{\lambda_{1},\ldots,\lambda_{r}}(\bD)$ in this case.
Full presentations of these groups are recalled in \cite[Prop.~$4.1$]{PSI} (see also \cite{LambropoulouOldenburg, Bellingeripresentations, BellingeriGodelleGuaschi}).

Let us fix a $2$-disc $\bD$ viewed as a rectangle and equipped with a marked point $p$ in its interior as in Figure~\ref{fig:model-surfaces}. For each integer $n\geq 1$ let us denote by $(\bD^{\natural n},p_{1},\ldots,p_{n})$ the result of gluing $n$ copies of $(\bD,p)$ side by side: topologically this is a $2$-disc with a set of $n$ marked points in its boundary. Similarly, if $S$ is a surface with one boundary component that we have identified with a fixed rectangle, as in Figure~\ref{fig:model-surfaces}, then for each $n\geq 1$ we denote by $(\bD^{\natural n} \natural S , p_{1},\ldots,p_{n})$ the result of gluing $n$ copies of $(\bD,p)$ and one copy of $S$ (on the right) side by side, as in Figure~\ref{fig:model-surfaces-gluing}.

For the groups \eqref{groupoid-data-1}, we set $\rG_{n} = \B_{n}(\bD^{\natural n})$ and $\rM_{n} = \B_{n}(\bD^{\natural n} \natural S)$, where the base configuration is $\{p_{1},\ldots,p_{n}\}$ in each case.
The monoidal structure \eqref{groupoid-data-2} is then given by the homomorphisms $\theta_{m,n} \colon \B_{m}(\bD^{\natural m}) \times \B_{m}(\bD^{\natural m}) \to \B_{m+n}(\bD^{\natural m+n})$ induced by the maps
\begin{equation}
\label{eq:monoidal-structure-Beta}
C_m(\bD^{\natural m}) \times C_{n}(\bD^{\natural n}) \longrightarrow C_{m+n}(\bD^{\natural m+n})
\end{equation}
that send a pair of configurations $(c_{1},c_{2})$ to the configuration in $\bD^{\natural m+n}$ that consists of $c_{1}$ in the left-hand $m$ copies of $\bD$ and $c_{2}$ in the right-hand $n$ copies of $\bD$. This operation is clearly associative. The operation \eqref{groupoid-data-3} is given by the homomorphisms $\alpha_{m,n} \colon \B_{m}(\bD^{\natural m}) \times \B_{m}(\bD^{\natural m} \natural S) \to \B_{m+n}(\bD^{\natural m+n} \natural S)$ induced by the maps
\begin{equation}
\label{eq:action-Beta-S}
C_m(\bD^{\natural m}) \times C_{n}(\bD^{\natural n} \natural S) \longrightarrow C_{m+n}(\bD^{\natural m+n} \natural S)
\end{equation}
that send a pair of configurations $(c_{1},c_{2})$ to the configuration in $\bD^{\natural m+n}$ that consists of $c_{1}$ in the left-hand $m$ copies of $\bD$ and $c_{2}$ in the right-hand $n$ copies of $\bD$ together with $S$. The homomorphisms $\theta_{m,n}$ and $\alpha_{m,n}$ are all injective, by Corollary~\ref{coro:injectivity}.

For integers $m,n \geq 1$, we specify the element $b_{m,n} \in \B_{m+n}(\bD^{\natural m+n})$ in \eqref{groupoid-data-4} to be the loop constructed as follows. First, the right-hand $n$ points in the base configuration $\{p_{1},\ldots,p_{m+n}\}$ move vertically upwards in $\bD^{\natural m+n}$; then they move $m$ steps to the left while the left-hand $m$ points in the base configuration move $n$ steps to the right; then the $n$ points in the interior (which are now on the left of the picture) move vertically downwards again. For example, Figure~\ref{fig:model-surfaces-braiding} illustrates the element $b_{2,1} \in \B_{3}(\bD^{\natural 3})$. This completes the definition of the braided monoidal groupoid $\Beta$ and the left $\Beta$-module groupoid $\Beta^{S}$.
To finish, we make explicit some conventions:

\begin{convention}\label{convention:braiding_braids}
For each $n\in\bN$, we denote the punctured disc $\bD^{\natural n} \smallsetminus \{p_{1},\ldots,p_{n}\}$ by $\bD_{n}$, with the punctures ordered from left to right, and we write $\B_{n} = \B_{n}(\bD^{\natural n})=\pi_{0}(\Diff_{\partial}(\bD^{\natural n} , \{p_{1},\ldots,p_{n}\}))$, i.e.~the group of isotopy classes of self-diffeomorphisms of $\bD^{\natural n}$ restricting to the identity on a neighbourhood of $\partial (\bD^{\natural n})$ and fixing the punctures $\{p_{1},\ldots,p_{n}\}$ setwise. We take the convention that the element $\sigma_{i} \in \B_{n}$ is the geometric braid that swaps the points $i$ and $i+1$ \emph{anticlockwise} (in other words it acts locally as $b_{1,1}$); see Figure~\ref{fig:action-of-braiding} for an illustration of $\sigma_{1}$.
\end{convention}

\subsection{Construction of homological representation functors}\label{ss:general_construction}

Here we explain the constructions introduced in \cite{PSI} of homological representation functors for surface braid groups and mapping class groups of surfaces. Namely, we follow the method and ideas of \cite[\S 2--\S 3]{PSI}, but follow a more direct method (although equivalent in the end) to simplify the presentation; see Remark~\ref{rmk:connection_PSI}. A reader familiar with \cite{PSI} may skip this subsection.

\subsubsection{Framework}\label{sss:framework_homological_representation_functor}
Let $\cG$ be one of the small strict braided monoidal groupoids $\{\M^{+},\M^{-},\Beta\}$ introduced in \S\ref{ss:categories_examples} along with any associated $\cG$-module $\cM$ defined there (typically $\M^{+}$, $\M^{-}$ or $\Beta^{S}$ respectively).
Recall from \S\ref{ss:categories_examples} that in each case $\obj(\cG) = \obj(\cM) = \bN$: from now on, we denote the objects of $\cG$ and $\cM$ by non-negative integers $\tn$.
We denote by $\{G_{n}\}_{n\in\bN}$ the family of automorphism groups $\{\Aut_{\cM}(\tn)\}_{\tn\in\obj(\cM)}$ encoded by $\cM$. (These automorphism groups were denoted in \S\ref{ss:categories_examples} by $\rM_{n}$, but henceforth we shall always write $G_{n}$.) They come equipped with canonical homomorphisms $\id_{\one}\natural (-)_{n} \colon G_{n} \to G_{n+1}$ induced by the $\cG$-module structure: precisely, these homomorphisms are $\alpha_{1,n}(\id_\one , -)$ in the notation of \S\ref{ss:categories_examples}. In each of the settings described in \S\ref{ss:categories_examples}, the homomorphisms $\alpha_{m,n}$ are injective (see Lemma~\ref{lem:injectivity_mcg} and Corollary~\ref{coro:injectivity}), and hence so are the maps $\id_{\one}\natural (-)_{n} = \alpha_{1,n}(\id_\one , -)$.
The first key ingredient to define homological representations is to find a family of spaces on which the family $\{G_{n}\}_{n\in\bN}$ acts. In the case of surface braid groups, this will involve considering ``partitioned versions'' of the groups $G_{n}$, which we define in a general context:

\begin{defn}[\emph{Partitioned groups.}]
\label{defn:partitioned}
Let $\mathsf{G}_{k}$ be a group equipped with a surjection $\mathsf{s}_{k} \colon  \mathsf{G}_{k} \twoheadrightarrow \Sym_{k}$. Given an ordered partition $\lambda = (\lambda_{1},\ldots,\lambda_{r}) \vdash k$, for $j \leq r$, we define $t_{j} := \sum_{i \leq j} \lambda_{i}$ (including $t_{0} = 0$). Then the set $\{t_{j-1}+1, \ldots , t_{j}\}$ is referred to as the \emph{$j$th block} of $\lambda$, and $\lambda_{i}$ is called the \emph{size} of the $i$th block. The preimage $\mathsf{G}_{\lambda} := \mathsf{s}_{k}^{-1}(\Sym_{\lambda})$ (where $\Sym_{\lambda}:=\Sym_{\lambda_{1}} \times \cdots \times \Sym_{\lambda_{r}}$) is called the $\lambda$-\emph{partitioned version} of $\mathsf{G}_{k}$. The extremal situations are the discrete partition $\lambda = (1,\ldots,1)$, which corresponds to the \emph{pure version} of the group $\mathsf{G}_{k}$, and the trivial case $\lambda = (k)$, which is simply the group $\mathsf{G}_{k}$ itself. The group $\mathsf{G}_{\lambda}$ fits into the short exact sequence: $1\to\mathsf{G}_{(1,\ldots,1)} \to \mathsf{G}_{\lambda}\to \Sym_{\lambda}\to 1$.
\end{defn}

In all the situations addressed in this paper, the parameter $k$ corresponds to the motion of $k$ points, while the surjection corresponds to the permutations of these points. For the remainder of \S\ref{ss:general_construction}, we consider an ordered partition $\lambda = (\lambda_{1},\ldots,\lambda_{r}) \vdash k$ of an integer $k\geq1$.
Furthermore, for each $n\in\bN$, we consider the surface $\cS_{n}$ defined from the object $\tn\in\obj(\cM)$ as follows:
\begin{itemizeb}
\item When $G_{n}=\B_{n}(S)$ (where $S = \Sigma_{g,1} \text{ or } \N_{h,1}$), we set $\cS_{n} := \bD_{n}\natural S$ and $G_{k,n} := \B_{k,n}(S)$ equipped with the evident surjection $\B_{k,n}(S) \twoheadrightarrow \Sym_{k}$; thus $G_{\lambda,n} = \B_{\lambda_{1},\ldots,\lambda_{r},n}(S)$.
\item When $G_{n}=\MCG(\sS^{\natural n})$ (where $\sS = \bT \text{ or } \bM$, see \S\ref{sss:category_mcg}), we set $\cS_{n}:=(\sS^{\natural n}) \smallsetminus I$, where $I$ denotes the closed interval in $\partial \sS^{\natural n}$ given by the bottom edge of the rectangle in Figure~\ref{fig:model-surfaces-gluing} (see Remark~\ref{rmk:Stavrou} for an explanation of this choice). We also set $G_{k,n} := \MCG(\sS^{\natural n} , k)$ (see \S\ref{ssss:injectivity-results}) equipped with the evident surjection $\MCG(\sS^{\natural n} , k) \twoheadrightarrow \Sym_{k}$; thus $G_{\lambda,n} = \MCG(\sS^{\natural n} , \lambda)$.
\end{itemizeb}
In each case, considering the $\lambda$-configuration space $C_{\lambda}(\cS_{n})$ and its fundamental group $\B_{\lambda_{1},\ldots,\lambda_{r}}(\cS_{n})$, there is a split short exact sequence
\begin{equation}\label{eq:Birman_Fadell_Neuwirth_SES}
\begin{tikzcd}
1 \ar[r] & \B_{\lambda_{1},\ldots,\lambda_{r}}(\cS_{n}) \ar[r] & G_{\lambda,n} \ar[r] & G_{n} \ar[l,bend right=30,dashed] \ar[r] & 1.
\end{tikzcd}
\end{equation}
Moreover, this split short exact sequence is functorial with respect to $n$, in the sense that the maps $\id_{\one}\natural (-)_{n} \colon G_{n} \to G_{n+1}$ lift to maps $G_{\lambda,n} \to G_{\lambda,n+1}$, and these together induce maps of split short exact sequences of the form \eqref{eq:Birman_Fadell_Neuwirth_SES}. In more detail:
\begin{itemizeb}
\item If $G_{n}=\B_{n}(S)$, the sequence \eqref{eq:Birman_Fadell_Neuwirth_SES} is the classical \emph{Fadell-Neuwirth exact sequence} (see for instance \cite[Prop.~$6.15$]{DPS} or \cite[Cor.~$1.30$]{PSI}), and its section, denoted by $s_{(\lambda,n)}$, is induced by the embedding of configuration spaces $C_{n}(\cS'_{n})\hookrightarrow C_{\lambda,n}(\cS_{n})$ defined by considering an isotopy equivalent proper subsurface $\cS'_{n}$ of $\cS_{n}$ and by fixing a $\lambda$-partitioned configuration of $k$ points in $\cS_{n}-\cS'_{n}$ (see also \cite[Prop.~$6.15$]{DPS} or \cite[Cor.~$1.30$]{PSI} for further details).
\item[] We note in passing that the composition of the section $s_{(\lambda,n)}$ with the canonical injection $G_{\lambda,n} \hookrightarrow G_{k+n}$ is equal to $\id_{k}\natural (-)_{n} \colon G_{n}\hookrightarrow G_{k+n}$.
\item[] The lift $G_{\lambda,n} \to G_{\lambda,n+1}$ is defined analogously to $\id_{\one}\natural (-)_{n} \colon G_{n} \to G_{n+1}$, taking the boundary connected sum with a disc containing an additional configuration point, where we specify that this new configuration point belongs to the `$n$' block of the partition $(\lambda,n)$.
\item If $G_{n}=\MCG(\sS^{\natural n})$, the sequence \eqref{eq:Birman_Fadell_Neuwirth_SES} is known as the \emph{Birman short exact sequence} (see for instance \cite[Lem.~$4.16$]{farbmargalit} or \cite[Cor.~$1.34$]{PSI}), and its section is induced by extending diffeomorphisms of $\sS^{\natural n}$ along the inclusion $\sS^{\natural n}\hookrightarrow \bD_{k}\natural \sS^{\natural n}$ by the identity on $\bD_{k}$, as in Lemma~\ref{lem:injectivity_mcg} (see also \cite[Cor.~$1.34$]{PSI} for further details).
\item[] We note that the kernel of \eqref{eq:Birman_Fadell_Neuwirth_SES} in this case is a priori equal to $\B_{\lambda_{1},\ldots,\lambda_{r}}(\sS^{\natural n})$, but removing an interval from the boundary of $\sS^{\natural n}$ does not change its isotopy type and so $\B_{\lambda_{1},\ldots,\lambda_{r}}(\sS^{\natural n})$ is identified with $\B_{\lambda_{1},\ldots,\lambda_{r}}(\sS^{\natural n} \smallsetminus I) = \B_{\lambda_{1},\ldots,\lambda_{r}}(\cS_{n})$.
\item[] The lift $G_{\lambda,n} \to G_{\lambda,n+1}$ is defined exactly like $\id_{\one}\natural (-)_{n} \colon G_{n} \to G_{n+1}$: at the level of diffeomorphisms, it is given by extending by the identity on the new copy of $\sS$.
\end{itemizeb}
Finally, we note that the short exact sequence \eqref{eq:Birman_Fadell_Neuwirth_SES} provides an action (by conjugation) of $G_{n}$ on $\B_{\lambda_{1},\ldots,\lambda_{r}}(\cS_{n})$.

\subsubsection{Twisted representations}
\label{sss:twisted-representations}

Another preliminary is the recollection of the notion of \emph{twisted} representations.

\begin{defn}[\emph{Category of twisted modules.}]
\label{def:cat_twisted_modules}
Let $\bA$ be a non-zero associative unital ring and let $R$ be an associative, unital $\bA$-algebra. The category of \emph{twisted $R$-modules}, denoted by $R\lmod^{\tw}$, is defined as follows. An object of $R\lmod^{\tw}$ is simply a left $R$-module $V$. A morphism $V \to V'$ is an automorphism $\psi \in \Aut_{\bA\Alg}(R)$ of unital $\bA$-algebras together with a morphism $\zeta \colon V \to \psi^{*} (V')$ of left $R$-modules.

We will henceforth set $\bA := \bZ$, so that $\bA$-algebras are just rings. From \S\ref{sss:transformation_groups} onwards, we will typically work with group rings $R := \bZ[Q]$ for a given group $Q$.
\end{defn}

A functor $F\colon\langle \cG,\cM \rangle \to R\lmod^{\tw}$ encodes \emph{twisted} $R$-representations. More precisely, at the level of group representations, it means that the action of $G_{n}$ on the corresponding $R$-module commutes with the $R$-module structure only up to a ``twist'', i.e.~an action $a_{n}\colon G_{n} \to \Aut_{\Ring}(R)$ where $\Ring$ is the category of associative, unital rings. When this action $a_{n}$ is trivial, we recover the classical notion of an $R$-representation (also called \emph{genuine} $R$-representation) of the group $G_{n}$.

Furthermore, the module category $R\lmod$ is by definition the subcategory of $R\lmod^{\tw}$ on the same objects and those morphisms $(\psi,\zeta)$ with $\psi=\id_{R}$, so there is a canonical embedding $R\lmod \subset R\lmod^{\tw}$.
In particular, a representation encoded by a functor $F\colon\langle \cG,\cM \rangle\to R\lmod^{\tw}$ is a \emph{genuine} $R$-representation if and only if $F$ factors through the module subcategory $R\lmod$:
\begin{equation}\label{eq:genuine_representation_functor}
    F\colon\langle \cG,\cM \rangle\longrightarrow R\lmod \longhookrightarrow R\lmod^{\tw}.
\end{equation}
In any case, representations encoded by functors $\langle \cG,\cM \rangle\to R\lmod^{\tw}$ may always be viewed as \emph{genuine} $\bZ$-module representations. Indeed, there is a forgetful functor $R\lmod^{\tw} \to {\bZ}\lmod$, where we forget the $R$-module structure on objects and the $\psi$ component of a morphism $(\psi,\zeta)$ in Definition~\ref{def:cat_twisted_modules}. Hence, we may always form the composite
\begin{equation}\label{eq:twisted_to_genuine_representation_functor}
    \langle \cG,\cM \rangle \longrightarrow R\lmod^{\tw} \longrightarrow {\bZ}\lmod,
\end{equation}
in order to view \emph{twisted} $R$-module representations as \emph{genuine} $\bZ$-module representations.

\subsubsection{Local coefficient systems}\label{sss:transformation_groups}

We now introduce the key parameter to define a homological representation functor. For the remainder of \S\ref{ss:general_construction}, we consider an integer $\ell\geq1$ corresponding to a lower central series index.
For each $n$, we use the short exact sequence \eqref{eq:Birman_Fadell_Neuwirth_SES} to define a group $Q_{(\lambda,\ell,n)}$ so that we have the commutative diagram:
\begin{equation}
\label{eq:input-diagram}
\begin{tikzcd}
1 \ar[r] & \B_{\lambda_{1},\ldots,\lambda_{r}}(\cS_{n}) \ar[d,two heads,"\phi_{(\lambda,\ell,n)}",labels=left] \ar[r] & G_{\lambda,n} \ar[d,two heads] \ar[r] & G_{n} \ar[d,two heads] \ar[l,bend right=25,dashed] \ar[r] & 1 \\
1 \ar[r] & Q_{(\lambda,\ell,n)} \ar[r] & G_{\lambda,n}/\LCS_{\ell} \ar[r] & G_{n}/\LCS_{\ell} \ar[l,bend right=25,dashed] \ar[r] & 1.
\end{tikzcd}
\end{equation}
More precisely, the right-exactness of the quotient $-/\LCS_{\ell}$ gives the right half of the bottom short exact sequence and ensures that the right-hand square of the diagram is commutative; the group $Q_{(\lambda,\ell,n)}$ is defined as the kernel of the surjection $G_{\lambda,n}/\LCS_{\ell} \twoheadrightarrow G_{n}/\LCS_{\ell}$; the map $\phi_{(\lambda,\ell,n)}\colon \B_{\lambda_{1},\ldots,\lambda_{r}}(\cS_{n})\twoheadrightarrow Q_{(\lambda,\ell,n)}$ is uniquely defined by the universal property of $\B_{\lambda_{1},\ldots,\lambda_{r}}(\cS_{n})$ as a kernel and its surjectivity follows from the splitting of the short exact sequence \eqref{eq:Birman_Fadell_Neuwirth_SES} and the fact that the lower central series defines a \emph{functorial quotient of groups} in the sense of \cite[Def.~$2.22$]{PSI}; see \cite[Lem.~$2.24$]{PSI}.
Furthermore, the functoriality of \eqref{eq:Birman_Fadell_Neuwirth_SES} with respect to $n$ and the universal property of $G_{\lambda,n}/\LCS_{\ell}$ and $G_{n}/\LCS_{\ell}$ as cokernels ensure that there exist unique maps $G_{\lambda,n}/\LCS_{\ell}\to G_{\lambda,n+1}/\LCS_{\ell}$ and $G_{n}/\LCS_{\ell}\to G_{n+1}/\LCS_{\ell}$ making the following square commutative:
\[
\begin{tikzcd}
G_{\lambda,n}/\LCS_{\ell} \ar[d] \ar[r, two heads] & G_{n}/\LCS_{\ell} \ar[d] \\
G_{\lambda,n+1}/\LCS_{\ell} \ar[r, two heads] & G_{n+1}/\LCS_{\ell}.
\end{tikzcd}
\]
Hence, by the universal property of a kernel, there exists a canonical map $q_{(\lambda,\ell,n)}\colon Q_{(\lambda,\ell,n)}\to Q_{(\lambda,\ell,n+1)}$ making the obvious diagram commutative. The colimit of the groups $\{(Q_{(\lambda,\ell,n)})\}_{n\in\bN}$ with respect to the maps $q_{(\lambda,\ell,n)}$ is denoted by $Q_{(\lambda,\ell)}$. Let us write $\phi_{(\lambda,\ell)} \colon \B_{\lambda_{1},\ldots,\lambda_{r}}(\cS_{n}) \to Q_{(\lambda,\ell)}$ for the composition of $\phi_{(\lambda,\ell,n)}$ with the map to the colimit. Since the configuration space $C_{\lambda}(\cS_{n})$ is path-connected, locally path-connected and semi-locally simply-connected, the map $\phi_{(\lambda,\ell)}$ defines a regular covering of $C_{\lambda}(\cS_{n})$ with deck transformation group $\Image(\phi_{(\lambda,\ell)}) \subseteq Q_{(\lambda,\ell)}$ by classical covering space theory (see for instance \cite[\S 1.3]{hatcheralgebraic}). Equivalently, it defines a rank-$1$ local system on $C_{\lambda}(\cS_{n})$ with fibre $\bZ[\Image(\phi_{(\lambda,\ell)})]$. We then take the fibrewise tensor product with respect to the inclusion $\bZ[\Image(\phi_{(\lambda,\ell)})] \subseteq \bZ[Q_{(\lambda,\ell)}]$, and thus change the ground ring to $\bZ[Q_{(\lambda,\ell)}]$. By abuse of notation, we denote the resulting rank-$1$ local system on $C_{\lambda}(\cS_{n})$ by the name of its fibre, i.e.~$\bZ[Q_{(\lambda,\ell)}]$.

Moreover, we deduce from \eqref{eq:input-diagram} that the group $G_{n}$ naturally acts by conjugation on the transformation group $Q_{(\lambda,\ell,n)}$. Via the inclusions $\id_{m}\natural (-)_{n} \colon G_{n}\hookrightarrow G_{m+n}$, it also acts (compatibly) by conjugation on $Q_{(\lambda,\ell,N)}$ for each $N\geq n$ and thus on the colimit $Q_{(\lambda,\ell)}$ of this direct system.

\paragraph*{Untwisted local system.} A natural goal (for the purpose of constructing \emph{genuine}, rather than \emph{twisted}, representations) is to choose transformation groups such that the actions of the groups $G_{n}$ on their colimit group are trivial. The optimal way to do this consists in taking the coinvariants of the group $Q_{(\lambda,\ell)}$ under the action of each $G_{n}$.

Namely, we consider for each $n$ the coinvariants $(Q_{(\lambda,\ell,n)})_{G_{n}}$, i.e.~the largest quotient of $Q_{(\lambda,\ell,n)}$ that collapses the orbits of the $G_{n}$-action. Let $Q^{\unt}_{(\lambda,\ell)}$ be the colimit of the groups $\{(Q_{(\lambda,\ell,n)})_{G_{n}}\}_{n\in\bN}$ with respect to the maps $(Q_{(\lambda,\ell,n)})_{G_{n}}\to (Q_{(\lambda,\ell,n+1)})_{G_{n+1}}$ induced by the canonical morphisms $\id_{\one}\natural (-)_{n} \colon G_{n}\hookrightarrow G_{n+1}$.
In particular, there is a canonical surjective morphism $Q_{(\lambda,\ell)} \twoheadrightarrow Q^{\unt}_{(\lambda,\ell)}$. The ``$\unt$'' in the notation stands for \emph{untwisted} since the $G_{n}$-actions for all $n$ on $Q^{\unt}_{(\lambda,\ell)}$ are trivial.

\begin{rmk}
Recall that coinvariants are particular instances of coequalisers (see \cite[\S III.3]{MacLane1} for instance), so they commute with colimits (see \cite[\S IX.8]{MacLane1} for instance). Therefore, the group $Q^{\unt}_{(\lambda,\ell)}$ is isomorphic to the coinvariant group $(Q_{(\lambda,\ell)})_{G_{\infty}}$, where $G_{\infty}$ is the colimit of the groups $\{G_{n}\}_{n\in\bN}$ with respect to the maps $\id_{\one}\natural (-)_{n}$. In particular, the quotient group $Q^{\unt}_{(\lambda,\ell)}$ of $Q_{(\lambda,\ell)}$ is optimal in the sense that any other \emph{untwisted} (i.e.~with trivial $G_{n}$-actions for all $n$) quotient $Q'$ of $Q_{(\lambda,\ell)}$ is a quotient of $Q^{\unt}_{(\lambda,\ell)}$ (in other words, it is the initial untwisted quotient of $Q_{(\lambda,\ell)}$).
\end{rmk}

\subsubsection{Definition of the homological representation functors}\label{sss:def_homological_rep_functors}

We may now define the homological representations and their associated functors. In \S\ref{sss:transformation_groups}, we introduced actions of the group $G_{n}$ on $\pi_{1}(C_{\lambda}(\cS_{n}))$ and on the associated rank-$1$ local system $\bZ[Q_{(\lambda,\ell)}]$, induced by the splittings of \eqref{eq:input-diagram}. Now we define from these a representation
\begin{equation}\label{eq:homological_representation_def}
G_{n}\longrightarrow\Aut_{\bZ[Q_{(\lambda,\ell)}]\lmod^{\tw}} \bigl( H_{k}^{\BM}(C_{\lambda}(\cS_{n});\bZ[Q_{(\lambda,\ell)}]) \bigr)
\end{equation}
using the functoriality of (twisted) Borel-Moore homology (see \cite[Chap.~V, \S 3]{bredonsheaf} for instance). (In fact, a priori, we need more: we need an action up to homotopy of $G_{n}$ on the based space $C_{\lambda}(\cS_{n})$ that induces the action on $\pi_{1}(-)$. However, since $C_{\lambda}(\cS_{n})$ is aspherical, i.e.~a classifying space for its fundamental group, by \cite[Cor.~$2.2$]{FadellNeuwirth1962}, this exists and is unique, so it comes ``for free''.) These combine to define a functor
\begin{equation}\label{eq:homological_representation_functor}
    \fL_{(\lambda,\ell)}\colon\cM\longrightarrow {\bZ[Q_{(\lambda,\ell)}]}\lmod^{\tw}.
\end{equation}
Alternatively, considering instead the untwisted transformation group $Q^{\unt}_{(\lambda,\ell)}$, each group $G_{n}$ acts trivially the rank-$1$ local system $\bZ[Q^{\unt}_{(\lambda,\ell)}]$ and thus the analogous representation to that of \eqref{eq:homological_representation_def} preserves the $\bZ[Q^{\unt}_{(\lambda,\ell)}]$-module structure of $H_{k}^{\BM}(C_{\lambda}(\cS_{n});\bZ[Q^{\unt}_{(\lambda,\ell)}])$. Therefore, the analogue of \eqref{eq:homological_representation_functor} using the untwisted transformation group $Q^{\unt}_{(\lambda,\ell)}$ is a functor
\begin{equation}\label{eq:homological_representation_functor_untwisted}
\fLu_{(\lambda,\ell)}\colon\cM\longrightarrow{\bZ[Q^{\unt}_{(\lambda,\ell)}]}\lmod \subset {\bZ[Q^{\unt}_{(\lambda,\ell)}]}\lmod^{\tw}.
\end{equation}

\begin{notation}\label{notation:generic_notation_star}
We generically denote by $\fL^{\star}_{(\lambda,\ell)}$ the functors \eqref{eq:homological_representation_functor} and \eqref{eq:homological_representation_functor_untwisted} and by ${\bZ[Q^{\star}_{(\lambda,\ell)}]}\lmod^{\tw}$ the associated target categories for simplicity, where $\star$ either stands for the blank space or $\star=\unt$.
\end{notation}

We now extend the functors \eqref{eq:homological_representation_functor} and \eqref{eq:homological_representation_functor_untwisted} along the canonical inclusion $\cM\hookrightarrow\langle \cG,\cM \rangle$ for the source category thanks to Lemma~\ref{lem:extend_functor_Quillen}.
In each situation described in \S\ref{sss:framework_homological_representation_functor}, for each $\tm\in\obj(\cG)$ and $\tn\in\obj(\cM)$, the morphism $[\tm,\id_{\tm\natural\tn}]$ of the category $\langle \cG,\cM \rangle$ corresponds to a proper embedding $\cS_{n}\hookrightarrow \cS_{m+n}$, which in turn induces a map $H_{k}^{\BM}(C_{\lambda}(\cS_{n});\bZ[Q^{\star}_{(\lambda,\ell)}])\to H_{k}^{\BM}(C_{\lambda}(\cS_{m+n});\bZ[Q^{\star}_{(\lambda,\ell)}])$, which we denote by $\iota_{\tm,\tn}$.

\begin{lem}\label{lem:extension_Quillen_source_homological_rep_functors}
Assigning $\fL^{\star}_{(\lambda,\ell)}([\tm,\id_{\tm\natural\tn}])$ to be $\iota_{\tm,\tn}$ for each $\tm\in\obj(\cG)$ and $\tn\in\obj(\cM)$, we extend the functor $\fL^{\star}_{(\lambda,\ell)}\colon\cM\to{\bZ[Q^{\star}_{(\lambda,\ell)}]}\lmod^{\tw}$ to a functor $\fL^{\star}_{(\lambda,\ell)}\colon\langle \cG,\cM \rangle\to{\bZ[Q^{\star}_{(\lambda,\ell)}]}\lmod^{\tw}$.
\end{lem}

\begin{proof}
By Lemma~\ref{lem:extend_functor_Quillen}, it is enough to prove that the compatibility relation \eqref{eq:criterion'} is satisfied. We consider $f\in\Aut_{\cG}(\tm)$ and $g\in\Aut_{\cM}(\tn)$, and denote by $\cS'_{m}$ the surface $\bD_{m}$ if $\cG=\Beta^{S}$ and $\sS^{\natural m}$ if $\cG=\M^{\pm}$ so that $\cS_{m+n} \cong \cS'_{m}\natural \cS_{n}$. We note that the image of $\iota_{\tm,\tn}$ consists of homology classes of configurations that are fully supported in the subsurface $\cS_{n}\hookrightarrow \cS_{m+n}$. Then, since the action of $f\natural \id_{\tn}$ is supported in the subsurface $\cS'_{m}\hookrightarrow \cS_{m+n}$, the map $\fL^{\star}_{(\lambda,\ell)}(f\natural \id_{\tn})$ acts trivially on the image of $\iota_{\tm,\tn}$, and so $\fL^{\star}_{(\lambda,\ell)}(f\natural \id_{\tn})\circ \iota_{\tm,\tn}=\iota_{\tm,\tn}$.

Furthermore, the above description of the image of $\iota_{\tm,\tn}$ implies that the action of $\fL^{\star}_{(\lambda,\ell)}(\id_{\tm}\natural g)$ on the image of $\iota_{\tm,\tn}$ is fully determined by the action of $\fL^{\star}_{(\lambda,\ell)}(g)$ on the homology classes of configurations supported in the subsurface $\cS_{n}\hookrightarrow \cS_{m+n}$ (because the action of $\id_{\tm}\natural g$ is supported in this subsurface). Hence $\fL^{\star}_{(\lambda,\ell)}(\id_{\tm}\natural g)\circ \iota_{\tm,\tn}=\iota_{\tm,\tn}\circ \fL^{\star}_{(\lambda,\ell)}(g)$.
Since $\fL^{\star}_{(\lambda,\ell)}(f\natural \id_{\tn})\circ \fL^{\star}_{(\lambda,\ell)}(\id_{\tm}\natural g)=\fL^{\star}_{(\lambda,\ell)}(f\natural g)$ (because of the compatibility of the monoidal structure $\natural$ with respect to composition and the fact that $\fL^{\star}_{(\lambda,\ell)}$ is a functor), we deduce that $\fL^{\star}_{(\lambda,\ell)}(f\natural g)\circ\iota_{\tm,\tn}=\iota_{\tm,\tn}\circ \fL^{\star}_{(\lambda,\ell)}(g)$ and so \eqref{eq:criterion'} is satisfied, which ends the proof.
\end{proof}

\begin{rmk}[\emph{Comparison with \cite{PSI}.}]\label{rmk:connection_PSI}
The way we extend the homological representation functors along the Quillen bracket construction $\langle \cG,\cM \rangle$ in Lemma~\ref{lem:extension_Quillen_source_homological_rep_functors} may seem a little ad hoc since we make an apparently arbitrary choice for this extension.
In \cite{PSI}, there is a more conceptual (although equivalent) method of the construction of the homological representation functors $\fL^{\star}_{(\lambda,\ell)}$. In particular, the fact that these functors are well-defined on the category $\langle \cG,\cM \rangle$ is already encoded in the method of \cite[\S 2--\S 3]{PSI}, and our choice for the morphism $[\tm,\id_{\tm\natural\tn}]$ in Lemma~\ref{lem:extension_Quillen_source_homological_rep_functors} matches with this alternative definition.
We refer to \cite[\S 2--\S 3]{PSI} for further details.
\end{rmk}

Finally, we note that the homological representations obtained with the parameters $\ell\in\{1,2\}$ are always untwisted:
\begin{lem}\label{lem:Qu_ell=2_useless}
There are equalities $Q^{\unt}_{(\lambda,\ell)} = Q_{(\lambda,\ell)}$ and $\fLu_{(\lambda,\ell)} =\fL_{(\lambda,\ell)}$ for $\ell \leq 2$.
\end{lem}
\begin{proof}
The result for $\ell=1$ is obvious since $Q_{(\lambda,1)}=0$. For $\ell = 2$, the $G_{n}$-action on $Q_{(\lambda,2,n)}$ is trivial for each $n$, since this is induced by conjugation in the abelian group $G_{\lambda,n}/\LCS_{2}$. Hence the surjection $Q_{(\lambda,\ell)}\twoheadrightarrow Q^{\unt}_{(\lambda,\ell)}$ is an equality. The result for $\fLu_{(\lambda,\ell)}$ then follows by construction.
\end{proof}    

\subsubsection{The vertical-type alternatives}\label{ss:vertical_alternatives}

Finally, we describe an important general modification that we may make in the parameters of the construction. We recall that we consider the configuration space $C_{\lambda}(\cS_{n})$ of $k$ points in a surface $\cS_{n}$, which is obtained from a compact surface by removing finitely many punctures from its interior or by removing a closed interval (equivalently, one puncture) from its boundary. For such surfaces, we introduce the associated notions of \emph{blow-up} and \emph{dual} surfaces:

\begin{defn}[\emph{Dual surfaces.}]
\label{defn:Sdoubleprime}
Consider a finite-type surface $S \smallsetminus \cP$, namely a compact surface $S$ minus a finite subset $\cP \subset S$. Its \emph{blow-up} $\overline{S}$ is then obtained from $S$ by blowing up each $p \in \cP$ to a new boundary component (if $p \in S \smallsetminus \partial S$) or an interval (if $p \in \partial S$). Furthermore, its \emph{dual surface} $\check{S}$ is obtained by removing from $\overline{S}$ the original boundary $\partial (S \smallsetminus \cP)$. Note that $(\overline{S};S \smallsetminus \cP,\check{S})$ is a manifold triad.
\end{defn}

Hence, we may alternatively use the \emph{dual surface} $\check{\cS}_{n}$ instead of $\cS_{n}$ and repeat mutatis mutandis the construction of \S\ref{sss:framework_homological_representation_functor}--\S\ref{sss:def_homological_rep_functors}.
This modification has a deep impact on the module structures of the representations, in particular for the basis we obtain for the modules for surface braid group representations; see \S\ref{ss:free-bases}. We single this variant out by calling it the \emph{vertical-type alternative} as a reference to the shape of the homology classes in the alternative module basis (see Figure~\ref{fig:models-dual}), and we denote it by $\fLv_{(\lambda,\ell)}$ (where ``$\vrtcl$'' stands for ``vertical'').

\subsection{Applications for surface braid groups and mapping class groups}\label{ss:examples_homological_representations_functors}

We now review the application of the construction of \S\ref{ss:general_construction} to produce homological representation functors for classical braid groups (see \S\ref{sss:representations_classical_braid_groups}), surface braid groups (see \S\ref{sss:representations_surface_braid_groups}) and mapping class groups of surfaces (see \S\ref{sss:representations_mapping_class_groups}).
Throughout \S\ref{ss:examples_homological_representations_functors}, we consider an integer $k\geq1$ and an ordered partition $\lambda = (\lambda_{1},\ldots,\lambda_{r}) \vdash k$ and we denote by $r'$ the number of indices $i \leq r$ in $\lambda$ such that $\lambda_{i} \geq 2$.

\subsubsection{Classical braid groups}\label{sss:representations_classical_braid_groups}

We apply the construction of \S\ref{ss:general_construction} in the setting $G_{n}=\B_{n}$, $\cS_{n}=\bD_{n}$ and $\cG=\cM=\Beta$, denoting by $Q_{(\lambda,\ell)} = Q_{(\lambda,\ell)}(\bD)$ the colimit transformation group defined in \S\ref{sss:transformation_groups} with these assignments.
Taking quotients by the $\LCS_{\ell}$ terms for each $\ell\geq1$, the construction of \S\ref{ss:general_construction} provides functors
\begin{equation}\label{def:Lawrence_bigelow_further_ell}
\LB_{(\lambda,\ell)}\colon\langle \Beta, \Beta\rangle \longrightarrow {\bZ[Q_{(\lambda,\ell)}(\bD)]}\lmod^{\tw}
\,\text{ and }\,
\LBu_{(\lambda,\ell)}\colon\langle \Beta, \Beta\rangle \longrightarrow {\bZ[Q^{\unt}_{(\lambda,\ell)}(\bD)]}\lmod,
\end{equation}
which we call the twisted and untwisted \emph{$(\lambda,\ell)$-Lawrence-Bigelow functors}. 

\begin{eg}[\emph{The Lawrence-Bigelow representations \cite{Lawrence1,BigelowHomrep}.}]\label{eg:LB_rep}
This terminology for the above functors comes from the fact that, when $\lambda=(k)$ and $\ell=2$, the functor $\LB_{((k),2)}$ encodes the $k$th family of the Lawrence-Bigelow representations; see \cite[Th.~$3.5$]{PSI}. These representations were originally introduced by Lawrence \cite{Lawrence1} as representations of Hecke algebras and then by Bigelow \cite{BigelowHomrep} via topological methods.
The Burau representations originally introduced in \cite{burau} are encoded by the functor $\LB_{((1),2)}$, while the \emph{Lawrence-Krammer-Bigelow} representations that Bigelow \cite{bigelow2001braid} and Krammer \cite{KrammerLK} independently proved to be faithful are encoded by the functor $\LB_{((2),2)}$; see \cite[\S 3.2.1]{PSI}.
Also, each functor $\LB_{((k),1)}$ corresponds to the trivial specialisation $\bZ[Q_{((k),2)}(\bD)]\twoheadrightarrow \bZ$ of the functor $\LB_{((k),2)}$, and Lawrence \cite[\S 3.4]{Lawrence1} proves that it encodes the representations factoring through $\B_{n}\twoheadrightarrow \Sym_{n}$.
\end{eg}

\begin{rmk}[\emph{Calculations of transformation groups and dependence on $\ell$.}]
\label{rmk:properties_Lawrence-Bigelow_functors}
By \cite[Lem.~$4.3$]{PSI}, we have $Q_{(\lambda,2)}(\bD) \cong \bZ^{r'} \times \bZ^{r(r-1)/2} \times \bZ^{r}$. If $\lambda_{i}\geq3$ for all $1\leq i\leq r$ or $\lambda$ is either $1$ or $(1,1)$, it follows from \cite[Th.~$3.6$]{DPS} that $Q_{(\lambda,\ell)}(\bD)=Q^{\unt}_{(\lambda,2)}(\bD)=Q_{(\lambda,2)}(\bD)$, and a fortiori that $\LB_{(\lambda,\ell)}=\LBu_{(\lambda,\ell)}=\LB_{(\lambda,2)}$ by construction. In contrast, it follows from \cite[Tab.~$2$]{PSIN} that as soon as $\lambda$ is of the form $(2,\lambda')$, $(1,1,1,\lambda')$, $(2,2,\lambda')$ or $(1,2,\lambda')$, then $\LB_{(\lambda,\ell)}\neq \LB_{(\lambda,\ell+1)}$ for each $\ell\geq1$.
Furthermore, when $\lambda=(2,\lambda')$ for $\lambda'$ such that each $\lambda'_{l}\geq3$, the transformation group $Q_{(\lambda,\ell)}(\bD)$ is computed in \cite[Prop.~$4.5$]{PSI}.
We prove in \cite[\S 5]{PSIN} that the representations are untwisted in this case, and a fortiori that $\LB_{(\lambda,\ell)}=\LBu_{(\lambda,\ell)}$. We may also compute the explicit formulas of the $\B_{n}$-actions; see \cite[Tab.~$1$ and Rem.~$4.9$]{PSIN}.    
\end{rmk}

Considering the dual surface $\check{\cS}_{n} = \check{\bD}_{n}$ rather than $\cS_{n} = \bD_{n}$, the construction of \S\ref{ss:vertical_alternatives} defines for each $\ell\geq1$ the \emph{vertical} Lawrence-Bigelow functors $\LB^{\vrtcl}_{(\lambda,\ell)}\colon\langle \Beta, \Beta\rangle \to{\bZ[Q_{(\lambda,\ell)}(\bD)]}\lmod^{\tw}$ and $\LB^{\unt,\vrtcl}_{(\lambda,\ell)}\colon\langle \Beta, \Beta\rangle \to{\bZ[Q^{\unt}_{(\lambda,\ell)}(\bD)]}\lmod$. The properties discussed in Remark~\ref{rmk:properties_Lawrence-Bigelow_functors} for the functors \eqref{def:Lawrence_bigelow_further_ell} are the same for these vertical-type alternatives.

\subsubsection{Braid groups on surfaces different from the disc}\label{sss:representations_surface_braid_groups}

We fix two integers $g\geq1$ and $h\geq1$, and a surface $S$ that is either $\Sigma_{g,1}$ or else $\N_{h,1}$ defined in \S\ref{sss:category_surface_braid_groups}.
We apply the construction of \S\ref{ss:general_construction} in the setting $G_{n}=\B_{n}(S)$, $\cS_{n}=\bD_{n}\natural S$, $\cG=\Beta$ and $\cM=\Beta^{S}$, denoting by $Q_{(\lambda,\ell)} = Q_{(\lambda,\ell)}(S)$ the colimit transformation group defined in \S\ref{sss:transformation_groups} with these assignments.
Taking quotients by the $\LCS_{\ell}$ terms for each $\ell\geq1$, the construction of \S\ref{ss:general_construction} provides homological representation functors, for $S\in\{\Sigma_{g,1},\N_{h,1}\}$:
\begin{equation}\label{eq:ell_nilpotent_hom_rep_braid_surface}
\fL_{(\lambda,\ell)}(S)\colon\langle\Beta,\Beta^{S}\rangle \longrightarrow \bZ[Q_{(\lambda,\ell)}(S)]\lmod^{\tw}
\,\text{ and }\,
\fLu_{(\lambda,\ell)}(S)\colon\langle\Beta,\Beta^{S}\rangle \longrightarrow \bZ[Q^{\unt}_{(\lambda,\ell)}(S)]\lmod.
\end{equation}

\begin{eg}[\emph{The An-Ko representations \cite{AnKo}.}]\label{eg:An-Ko}
For orientable surfaces, the trivial partition $\lambda=(k)$ and $\ell=3$, the $\B_{n}(\Sigma_{g,1})$-representation $\fL_{((k),3)}(\Sigma_{g,1})(\tn) \otimes_{\bZ[Q_{((k),3)}(\Sigma_{g,1})]} \bZ[\B_{k,n}(\Sigma_{g,1})/\LCS_{3}]$ is isomorphic to the one introduced by An and Ko in \cite[Th.~$3.2$]{AnKo}; see \cite[Ex.~$3.6$]{PSI}. The group $Q_{((k),3)}(\Sigma_{g,1})$ is abstractly defined in \cite{AnKo} in terms of group presentations to satisfy certain technical homological constraints, while \cite[\S 4]{BellingeriGodelleGuaschi} explains all of the connections to the third lower central series quotient. On the other hand, the untwisted representations encoded by the functor $\fLu_{((k),3)}(\Sigma_{g,1})$ are specific to \cite[\S 3.2.1]{PSI}.
\end{eg}

\begin{rmk}[\emph{Calculations of transformation groups and dependence on $\ell$.}]
\label{rmk:properties_surface_braid_groups_functors}
We know from \cite[Lem.~$4.3$]{PSI} that $Q_{(\lambda,2)}(S) \cong (\bZ/2)^{r'} \times H_{1}(S;\bZ)^{\times r}$ for $S\in\{\Sigma_{g,1},\N_{h,1}\}$.
If $\lambda_{i}\geq3$ for all $1\leq i\leq r$, it follows from \cite[Th.~$6.52$ and Prop.~$6.62$]{DPS} that $\fL_{(\lambda,\ell)}(S)=\fLu_{(\lambda,\ell)}(S)=\fL_{(\lambda,3)}(S)$ for $\ell\geq 4$. Moreover, we explicitly compute the transformation groups $Q_{(\lambda,3)}(\Sigma_{g,1})$, $Q^{\unt}_{(\lambda,3)}(\Sigma_{g,1})$, $Q_{(\lambda,3)}(\N_{h,1})$ and $Q^{\unt}_{(\lambda,3)}(\N_{h,1})$ in \cite[Prop.~$4.5$]{PSI}. In particular, we deduce that $\fLu_{(\lambda,\ell)}(S)\neq\fL_{(\lambda,\ell)}(S)$ if $\lambda_{i}\geq3$ for all $1\leq i\leq r$.
In contrast, it follows from \cite[Tab.~$2$]{PSIN} that if $\lambda$ is of the form $(2,\lambda')$ or $(1,\lambda')$ (assuming that $S\neq \bM$ for the latter), then $\fL_{(\lambda,\ell)}(S)\neq \fL_{(\lambda,\ell+1)}(S)$ for each $\ell\geq3$.
It is unclear whether $\fL_{(\lambda,\ell)}(S) = \fLu_{(\lambda,\ell)}(S)$ in this situation; see \cite[Rem.~$4.6$]{PSI}.
\end{rmk}

Considering the dual surface $\check{\cS}_{n} = (\bD_{n}\natural S)\,\check{}$ instead of $\cS_{n} = \bD_{n}\natural S$, the construction of \S\ref{ss:vertical_alternatives} defines the vertical homological representation functors $\fLuv_{(\lambda,\ell)}(\Sigma_{g,1})$, $\fLv_{(\lambda,\ell)}(\Sigma_{g,1})$, $\fLuv_{(\lambda,\ell)}(\N_{h,1})$ and $\fLv_{(\lambda,\ell)}(\N_{h,1})$ for each $\ell\geq1$. Their source and target categories are the same as for their non-vertical counterparts, and the properties discussed in Remark~\ref{rmk:properties_surface_braid_groups_functors} for the functors \eqref{eq:ell_nilpotent_hom_rep_braid_surface} are the same for these vertical-type alternatives.

\subsubsection{Mapping class groups of surfaces}\label{sss:representations_mapping_class_groups}

We apply the construction of \S\ref{ss:general_construction} in the setting $G_{n}=\MCGo_{n,1}$ or $\MCGno_{n,1}$, $\cS_{n}=(\sS^{\natural n}) \smallsetminus I$ where $\sS=\bT \text{ or }\bM$ respectively and $\cG=\cM=\M^{+}$ or $\M^{-}$ respectively. We denote by $Q_{(\lambda,\ell)} = Q_{(\lambda,\ell)}(\sS)$ the colimit transformation group defined in \S\ref{sss:transformation_groups} with these assignments.

\begin{rmk}\label{rmk:Stavrou}
A more natural assignment for applying the construction of \S\ref{ss:general_construction} would be to take $\cS_{n} = \sS^{\natural n}$, i.e.~not to remove the subinterval $I \subset \partial \sS^{\natural n}$ given by the bottom edge of the rectangle in Figure~\ref{fig:model-surfaces-gluing}. We do however choose $(\sS^{\natural n}) \smallsetminus I$ instead because it is necessary for applying Theorem~\ref{thm:free_generating_sets_rep} in order to compute the underlying modules of the representations; see \S\ref{ss:free-bases}.

Otherwise, the calculations of the representations using $\cS_{n} = \sS^{\natural n}$ are much more complicated. See for instance the work of Stavrou~\cite[Th.~$1.4$]{stavrou}, who computes the $\MCGo_{n,1}$-representation equivalent to that obtained from the construction of \S\ref{ss:general_construction} with $\cS_{n} = \bT^{\natural n}$, $\ell=1$, taking $\bQ$ as ground ring and using classical homology instead of Borel-Moore homology.
\end{rmk}

Taking quotients by the $\LCS_{\ell}$ terms for each $\ell\geq1$, the construction of \S\ref{ss:general_construction} defines homological representation functors
\begin{equation}\label{eq:hom_rep_functor_rep_mcg_o}
\fL_{(\lambda,\ell)}(\MCGo)\colon\langle\M^{+},\M^{+}\rangle\to \bZ[Q_{(\lambda,\ell)}(\bT)]\lmod^{\tw}
\,\text{ and }\,
\fLu_{(\lambda,\ell)}(\MCGo)\colon\langle\M^{+},\M^{+}\rangle\to \bZ[Q^{\unt}_{(\lambda,\ell)}(\bT)]\lmod,
\end{equation}
\begin{equation}\label{eq:hom_rep_functor_rep_mcg_no_ell}
\fL_{(\lambda,\ell)}(\MCGno)\colon\langle\M^{-},\M^{-}\rangle\to \bZ[Q_{(\lambda,\ell)}(\bM)]\lmod^{\tw}
\,\text{ and }\,
\fLu_{(\lambda,\ell)}(\MCGno)\colon\langle\M^{-},\M^{-}\rangle\to \bZ[Q^{\unt}_{(\lambda,\ell)}(\bM)]\lmod.
\end{equation}

\begin{eg}[\emph{The Moriyama representations \cite{Moriyama}.}]\label{eg:Moriyama}
For orientable surfaces, the discrete partition $\lambda=(1,\ldots,1)$ and $\ell=1$, the functor $\fL_{((1,\ldots,1),1)}(\MCGo)$ encodes the mapping class group representations introduced by Moriyama \cite{Moriyama}; see \cite[Prop.~$3.9$]{PSI} or \cite[Prop.~$2.1$]{Annex}. It is thus called the \emph{$k$th Moriyama functor}. In particular, the representations encoded by the functor $\fL_{((1),1)}(\MCGo)$ are equivalent to the standard representations on $H_{1}(\Sigma_{g,1};\bZ)$, which factor through the symplectic groups $\mathrm{Sp}_{2g}(\bZ)$.
\end{eg}

We record here the computations of the transformation groups for the functors \eqref{eq:hom_rep_functor_rep_mcg_o} and \eqref{eq:hom_rep_functor_rep_mcg_no_ell} when $\ell=2$, the proofs of which are elementary (see \cite[Cor.~$4.9$]{PSI} for instance). They will be of key use later; see Lemma~\ref{lem:computations_MCG_braiding}.

\begin{lem}[{\cite[Cor.~$4.9$]{PSI}, \cite[Prop.~$1.1$, Rem.~$1.3$]{Annex}}]
\label{lem:transformation_groups_MCG_ell_2}
We have $Q_{(\lambda,2)}(\bT) \cong (\bZ/2)^{r'}$ and $Q_{(\lambda,2)}(\bM) \cong (\bZ/2)^{r'}\times(\bZ/2)^{r}$. More precisely:
\begin{itemizeb}
\item Each of the (first) $r'$ $\bZ/2$-summands is generated by the image in the abelianisation of a standard braid generator interchanging two points in the $\rho$-th block of the partition, for each $\rho \in \{1,\ldots, r\}$ such that $\lambda_{\rho}\geq2$. This is known as the \emph{writhe} (modulo $2$) of the $\rho$-th block of strands; see \textup{\cite[Prop.~$1.1$]{Annex}}.
\item The last $r$ $\bZ/2$-summands in $Q_{(\lambda,2)}(\bM)$ measure the number of times that a strand from the $\rho$-th block of the partition passes through a crosscap, for each $\rho \in \{1,\ldots, r\}$; see \textup{\cite[Rem.~$1.3$]{Annex}}.
\end{itemizeb}
\end{lem}

\begin{rmk}[\emph{Further computations of transformation groups and dependence on $\ell$.}]
\label{rmk:properties_MCG_functors}
For orientable surfaces, it follows from \cite[Cor.~$3.8$ and $3.9$]{PSI} and \cite[Prop.~$1.2$]{Annex} that, for all $\ell\geq3$, we have $Q_{(\lambda,\ell)}(\bT)=Q_{(\lambda,2)}(\bT)$. A fortiori, $\fL_{(\lambda,\ell)}(\MCGo)=\fL_{(\lambda,2)}(\MCGo)$ for all $\ell\geq3$ by construction.
On the other hand, for non-orientable surfaces, it is unclear whether $Q_{(\lambda,\ell)}(\bM) = Q_{(\lambda,2)}(\bM)$ for $\ell\geq 3$; if not, the functors $\fL_{(\lambda,\ell)}(\MCGno)$ will give rise to more sophisticated sequences of representations of the mapping class groups of non-orientable surfaces; see \cite[Rem.~$1.4$]{Annex}.
\end{rmk}

Finally, we may consider the dual surface $\check{\cS}_{n} = ((\sS^{\natural n}) \smallsetminus I)\,\check{}$ instead of $\cS_{n} = (\sS^{\natural n}) \smallsetminus I$ (as before, $\sS$ is either $\bT$ or $\bM$). In other words, instead of removing the interval $I$ from the (rectangular) boundary of $\sS^{\natural n}$, we remove the complementary interval, i.e.~the closure of $\partial (\sS^{\natural n})\smallsetminus I$.
However, in this case, we also change our convention on the braiding for the groupoid $\M$ by choosing its opposite:

\begin{convention}
In this setting, we apply the construction of \S\ref{ss:general_construction} taking $\cG=\cM$ to be equal to one of the braided monoidal groupoids $(\M^{+})\dv$ or $(\M^{-})\dv$ (depending on the case, orientable or non-orientable), instead of the braided monoidal groupoids $\M^{+}$ or $\M^{-}$. Recall from the beginning of \S\ref{ss:categorical_framework} that this simply consists in choosing the opposite convention for the braiding.
This purely arbitrary choice is motivated by the construction of short exact sequences; see Theorem~\ref{thm:SES_MCG_alternatives}. These rely on computations explained in \S\ref{ss:mcg_hom_rep_poly_preliminary} that would not be satisfied defining these functors over $\M^{+}$ and $\M^{-}$; see Remarks~\ref{rmk:interaction-with-braiding-dagger} and \ref{rmk:SES_alternative_dagger_MCG}.
\end{convention}

Then, for each $\ell\geq1$, the construction of \S\ref{ss:vertical_alternatives} defines the vertical homological representation functors $\fLv_{(\lambda,\ell)}(\MCGo)\colon\langle(\M^{+})\dv,(\M^{+})\dv\rangle\to \bZ[Q_{(\lambda,\ell)}(\bT)]\lmod^{\tw}$ and $\fLv_{(\lambda,\ell)}(\MCGno)\colon\langle(\M^{-})\dv,(\M^{-})\dv\rangle\to{\bZ[Q_{(\lambda,\ell)}(\bM)]}\lmod^{\tw}$ as well as their untwisted versions $\fLuv_{(\lambda,\ell)}(\MCGo)$ and $\fLuv_{(\lambda,\ell)}(\MCGno)$. The properties of Lemma~\ref{lem:transformation_groups_MCG_ell_2} and Remark~\ref{rmk:properties_MCG_functors} are exactly the same for these vertical-type alternatives.

\section{Module structure}\label{s:free_generating_sets}

The homological representations described above (see \S\ref{sss:def_homological_rep_functors}) are constructed from actions on the twisted Borel-Moore homology of configuration spaces on surfaces. In this section, we study the underlying module structure of these representations.

In \S\ref{ss:isomorphism-criterion} we prove a general criterion implying that the (possibly twisted) Borel-Moore homology of configuration spaces on a given underlying space is isomorphic to the Borel-Moore homology of configuration spaces on a subspace. Roughly, this works when the underlying space has a metric and the subspace is a ``skeleton'' onto which it deformation retracts in a controlled, non-expanding way. See Theorem~\ref{thm:free_generating_sets_rep} for the precise statement and Examples~\ref{eg:examples-BM-homology-lemma} for several examples corresponding to the underlying modules of representations of surface braid groups, mapping class groups, loop braid groups and related groups.

In \S\ref{ss:free-bases} we study several applications of Theorem~\ref{thm:free_generating_sets_rep} in more detail, describing explicit free generating sets for certain Borel-Moore homology modules. In \S\ref{ss:dual-bases} we then describe their ``dual bases'' with respect to certain perfect pairings. These dual bases, together with some diagrammatic reasoning, are used to prove some key lemmas needed in our arguments of \S\ref{s:SES_homol_rep-func}.

In total, this gives us a detailed understanding of the underlying module structure of the surface braid group and mapping class group representations that we consider. One may then attempt to derive explicit formulas for the group action in these models. We shall not pursue this here (beyond the qualitative diagrammatic arguments referred to above), since such explicit formulas are not needed to prove our polynomiality results.

\subsection{An isomorphism criterion for twisted Borel-Moore homology}\label{ss:isomorphism-criterion}

The main goal of this section is to prove the following criterion for an inclusion of metric spaces to induce isomorphisms on the (possibly twisted) Borel-Moore homology of their associated configuration spaces.
This generalises previously-known results described in Examples~\ref{eg:examples-BM-homology-lemma}.

\begin{thm}\label{thm:free_generating_sets_rep}
Let $M$ be a compact metric space with closed subspaces $A \subseteq B \subseteq M$, where $M$ and $B$ are locally compact. Suppose that there exists a strong deformation retraction $h$ of $M$ onto $B$, in other words a map $h \colon [0,1] \times M \to M$ satisfying the following two conditions:
\begin{itemizeb}
\item $h(t,x)=x$ whenever $t=0$ or $x \in B$,
\item $h(1,x) \in B$ for all $x \in M$,
\end{itemizeb}
such that moreover the following two additional conditions hold:
\begin{itemizeb}
\item $h(t,-)$ is non-expanding for all $t$, i.e.~$d(x,y) \geq d(h(t,x),h(t,y))$ for all $x,y \in M$,
\item $h(t,-)$ is a topological self-embedding of $M$ for all $t<1$.
\end{itemizeb}
Then, for all $k \in \bN$ and partitions $\lambda\vdash k$, the inclusion of configuration spaces
\[
C_{\lambda}(B \smallsetminus A) \lhook\joinrel\longrightarrow C_{\lambda}(M \smallsetminus A)
\]
induces isomorphisms on Borel-Moore homology in all degrees and for all local coefficient systems that extend to $C_{\lambda}(M)$.
\end{thm}

The point of this theorem, for the present paper, is that the Borel-Moore homology of the configuration space $C_{\lambda}(M \smallsetminus A)$ is the underlying module of a representation that we are studying, whereas the Borel-Moore homology of its subspace $C_{\lambda}(B \smallsetminus A)$ is easily computable.

\begin{rmk}
\label{rmk:condition-on-local-systems}
The condition that the local coefficient systems under consideration must extend to the larger space $C_{\lambda}(M)$ is automatically satisfied in all of the examples that we shall consider, since in these examples the inclusion $C_{\lambda}(M \smallsetminus A) \hookrightarrow C_{\lambda}(M)$ is a homotopy equivalence. Indeed, this holds whenever $M$ is a manifold and $A \subseteq M$ is a subset of its boundary.
Notice also that the hypotheses on $A$ are rather weak in Theorem~\ref{thm:free_generating_sets_rep}: it is simply any closed subset of $B$; the non-trivial hypothesis is the existence of a controlled deformation retraction of $M$ onto $B$, without reference to $A$. We will apply Theorem~\ref{thm:free_generating_sets_rep} in situations where $M = S$ is a surface that deformation retracts onto an embedded graph $B = \Gamma \subset S$.
\end{rmk}

\begin{proof}[Proof of Theorem~\ref{thm:free_generating_sets_rep}]
For $t \in [0,1]$, we write $h_{t} = h(t,-) \colon M \to M$ and recall that $h_{0} = \id$ and $h_{1}(M)=B$. For $\epsilon > 0$, we define
\[
C_{\epsilon} \coloneqq \bigl\lbrace [c_{1},\ldots,c_{k}] \in C_{\lambda}(M) \mid d(c_{i},c_{j})<\epsilon \text{ for some } i\neq j \text{ or } d(c_{i},a)<\epsilon \text{ for some } a \in A \bigr\rbrace .
\]
For each $t \in [0,1]$, every compact subspace of $C_{\lambda}(h_{t}(M) \smallsetminus A)$ is disjoint from $C_{\epsilon}$ for some $\epsilon > 0$, so we may write its Borel-Moore homology as the inverse limit
\[
H_{*}^{\BM} \bigl( C_{\lambda}(h_{t}(M) \smallsetminus A);\cL \bigr) \cong \underset{\epsilon\to 0}{\mathrm{lim}}\, H_{*} \bigl( C_{\lambda}(h_{t}(M) \smallsetminus A) , C_{\lambda}(h_{t}(M) \smallsetminus A) \cap C_{\epsilon} ; \cL \bigr)
\]
for any local system $\cL$. In particular, it suffices to show that the inclusion of pairs
\begin{equation}
\label{eq:inclusion-of-pairs}
(C_{\lambda}(B \smallsetminus A) , C_{\lambda}(B \smallsetminus A) \cap C_{\epsilon}) \lhook\joinrel\longrightarrow (C_{\lambda}(M \smallsetminus A) , C_{\lambda}(M \smallsetminus A) \cap C_{\epsilon})
\end{equation}
induces isomorphisms on twisted homology in all degrees for all local systems extending to $C_{\lambda}(M)$, for all $\epsilon > 0$. This fits into a diagram of inclusions of pairs of spaces
\begin{equation}
\label{eq:square-of-inclusions}
\begin{tikzcd}
(C_{\lambda}(B \smallsetminus A) , C_{\lambda}(B \smallsetminus A) \cap C_{\epsilon}) \ar[r,hook] \ar[d,hook] & (C_{\lambda}(M \smallsetminus A) , C_{\lambda}(M \smallsetminus A) \cap C_{\epsilon}) \ar[d,hook] \\
(C_{\lambda}(B) , C_{\lambda}(B) \cap C_{\epsilon}) \ar[r,hook] & (C_{\lambda}(M) , C_{\lambda}(M) \cap C_{\epsilon}).
\end{tikzcd}
\end{equation}
The vertical inclusions in \eqref{eq:square-of-inclusions} induce isomorphisms on twisted homology in all degrees by the excision theorem; see \cite[Th.~$2.20$]{hatcheralgebraic} (recalling that excision holds also with local coefficients, see \cite[Th.~$5.13$]{daviskirk} for example).
Hence, abbreviating $C^{t} := C_{\lambda}(h_{t}(M))$ and $C := C^{0}$, it will suffice to show that the inclusion of pairs $(C^1 , C^1 \cap C_{\epsilon}) \hookrightarrow (C , C_{\epsilon})$ induces isomorphisms on twisted homology in all degrees, for all $\epsilon > 0$.

Let us now fix $\epsilon > 0$. The hypothesis that $h_{t} \colon M \to M$ is a topological self-embedding for $t<1$ implies that it induces well-defined maps of configuration spaces that define a strong deformation retraction of $C$ onto $C^{t}$ for any $t<1$. Moreover, the hypothesis that $h_{t}$ is \emph{non-expanding} means that these maps of configuration spaces preserve the subspace $C_{\epsilon}$, so we in fact have a strong deformation retraction of the pair $(C,C_{\epsilon})$ onto the pair $(C^{t} , C^{t} \cap C_{\epsilon})$ for any $t<1$. On the other hand, we cannot conclude the same statement for $t=1$, since $h_{1} \colon M \to M$ is not assumed to be an embedding (and in our key examples it will not be). In order to continue the deformation retraction of configuration spaces, we first pass to a subspace: for any $t<1$, we define
\[
\check{C}^{t} \coloneqq \bigl\lbrace [c_{1},\ldots,c_{k}] \in C^{t} \mid h_{1}(h_{t}^{-1}(c_{i})) \neq h_{1}(h_{t}^{-1}(c_{j})) \text{ for each } i\neq j \bigr\rbrace .
\]
This additional condition precisely ensures that points do not collide if we continue applying the deformation retraction $h_{t}$ to configurations until time $t=1$. Thus there is a strong deformation retraction of the pair $(\check{C}^{t} , \check{C}^{t} \cap C_{\epsilon})$ onto the pair $(C^1 , C^1 \cap C_{\epsilon})$ for any $t<1$. It therefore remains to show that there exists some $t<1$ (depending on $\epsilon$) such that the inclusion
\[
(\check{C}^{t} , \check{C}^{t} \cap C_{\epsilon}) \lhook\joinrel\longrightarrow (C^{t} , C^{t} \cap C_{\epsilon})
\]
induces isomorphisms on twisted homology in all degrees. By excision, it suffices to show that $\check{C}^{t}$ and $C^{t} \cap C_{\epsilon}$ form an open covering of $C^{t}$. It is clear that these are both open subspaces, so we just have to show that there exists some $t<1$ such that $\check{C}^{t} \cup (C^{t} \cap C_{\epsilon}) = C^{t}$, or equivalently such that $C^{t} \smallsetminus \check{C}^{t} \subseteq C_{\epsilon}$.

By continuity of $h$ and compactness of $M$, there exists $\delta < 1$ such that $d(h_\delta(x),h_{1}(x)) < \epsilon / 2$ for all $x \in M$. By the argument so far, it suffices to show that $C^{\delta} \smallsetminus \check{C}^{\delta} \subseteq C_{\epsilon}$. Let $c = [c_{1},\ldots,c_{k}]$ be a configuration in $C^{\delta} \smallsetminus \check{C}^{\delta}$, in other words we have $c_{i} = h_\delta(x_{i})$ for some configuration $[x_{1},\ldots,x_{k}]$ in $C = C_{\lambda}(M)$ and $h_{1}(x_{i}) = h_{1}(x_{j})$ for some $i\neq j$. The distance from $c_{i}$ to $c_{j}$ is therefore at most the sum of the distances from $c_{i} = h_\delta(x_{i})$ to $h_{1}(x_{i})$ and from $h_{1}(x_{i}) = h_{1}(x_{j})$ to $h_\delta(x_{j}) = c_{j}$. These latter distances are both less than $\epsilon / 2$ by our choice of $\delta$, so we have $d(c_{i},c_{j}) < \epsilon$ and hence $c \in C_{\epsilon}$. Thus we complete the excision argument in the previous paragraph with $t=\delta$.

In summary, we have proved Theorem~\ref{thm:free_generating_sets_rep} by showing that, in the diagram
\begingroup
\small
\begin{equation*}
\begin{tikzcd}
& (C_{\lambda}(B \smallsetminus A) , C_{\lambda}(B \smallsetminus A) \cap C_{\epsilon}) \ar[r,hook] \ar[dl,hook',"{(*)}",swap] & (C_{\lambda}(M \smallsetminus A) , C_{\lambda}(M \smallsetminus A) \cap C_{\epsilon}) \ar[dr,hook,"{(*)}"] & \\
(C^1 , C^1 \cap C_{\epsilon}) \ar[r,hook,"{(\ddagger)}"] & (\check{C}^{t} , \check{C}^{t} \cap C_{\epsilon}) \ar[r,hook,"{(**)}"] & (C^{t} , C^{t} \cap C_{\epsilon}) \ar[r,hook,"{(\ddagger)}"] & (C , C_{\epsilon}),
\end{tikzcd}
\end{equation*}
\endgroup
the arrows $(*)$ induce isomorphisms on twisted homology in all degrees (by excision), the arrows $(\ddagger)$ are homotopy equivalences and for each $\epsilon > 0$ there exists $t \in (0,1)$ such that the arrow $(**)$ induces isomorphisms on twisted homology in all degrees (again by excision).
\end{proof}
\begin{egs}
\label{eg:examples-BM-homology-lemma}
We describe several examples of nested subspaces $A \subseteq B \subseteq M$ satisfying the hypotheses of Theorem~\ref{thm:free_generating_sets_rep} and the corresponding inclusions of configuration spaces
\begin{equation}
\label{eq:inclusion-of-config-spaces}
C_{\lambda}(B \smallsetminus A) \lhook\joinrel\longrightarrow C_{\lambda}(M \smallsetminus A).
\end{equation}

\begin{itemizeb}
\item (\emph{Configurations on punctured discs.})
Let us first consider the $n$-holed disc $M = \Sigma_{0,n+1}$, let $A$ be the union of the $n$ inner boundary components and let $B$ be the union of $A$ with $n-1$ arcs connecting the consecutive components of $A$. With respect to an appropriate metric, this satisfies the hypotheses of Theorem~\ref{thm:free_generating_sets_rep}. Moreover, since $A$ is part of the boundary of $M$, all local coefficient systems on $C_{\lambda}(M\smallsetminus A)$ extend to $C_{\lambda}(M)$. Thus Theorem~\ref{thm:free_generating_sets_rep} implies that \eqref{eq:inclusion-of-config-spaces} induces isomorphisms on twisted Borel-Moore homology in all degrees. This special case recovers \cite[Lem.~$3.1$]{BigelowHomrep}, which may also be deduced from the work of Kohno in \cite[Th.~$1$]{Kohno_hom_loc_systems_hyperplanes} and \cite[Prop.~$3.2$]{Kohno_one-parameter_family}. In this setting, $M \smallsetminus A$ is the $n$-punctured $2$-disc and $B \smallsetminus A$ is a disjoint union of $n-1$ open arcs.

\item (\emph{Configurations on non-closed surfaces.})
Generalising the previous point, we take $M=S$ to be any compact surface with non-empty boundary and $B=\Gamma \subseteq S$ to be an embedded finite graph onto which it deformation retracts. Choosing an appropriate metric, this satisfies the hypotheses of Theorem~\ref{thm:free_generating_sets_rep}. If we then take $A$ to be any closed subset of $\Gamma$, we conclude that the inclusion \eqref{eq:inclusion-of-config-spaces} induces isomorphisms on twisted Borel-Moore homology in all degrees (for local systems on $C_{\lambda}(S \smallsetminus A)$ that extend to $C_{\lambda}(S)$; this is automatic if $A \subseteq \Gamma \cap \partial S$). In the case $S = \Sigma_{g,1}$, this recovers \cite[Lem.~$3.3$]{AnKo}, \cite[Th.~$6.6$]{anghelpalmer} and \cite[Th.~A(a)]{BlanchetPalmerShaukat}.

\item (\emph{Higher dimensions.})
Extending to higher dimensions, we may take $M$ to be the manifold $W_{g,1} = (S^n \times S^n)^{\sharp g} \smallsetminus \mathring{D}^{2n}$, where $\sharp$ denotes the connected sum. This deformation retracts onto a subspace $B \subseteq W_{g,1}$ that is homeomorphic to $\vee^{2g} S^n$, with the basepoint of the wedge sum corresponding to a point $p$ in the boundary of $W_{g,1}$. Taking $A = \{p\}$, Theorem~\ref{thm:free_generating_sets_rep} then implies that the twisted Borel-Moore homology of configurations in $W_{g,1} \smallsetminus \{p\}$ is given by the twisted Borel-Moore homology of configurations in disjoint unions of Euclidean spaces.
\end{itemizeb}
\end{egs}

\subsection{Free bases}\label{ss:free-bases}

\begin{figure}[tbp]
    \centering
    \begin{subfigure}[b]{0.48\textwidth}
        \centering
        \includegraphics[scale=0.65]{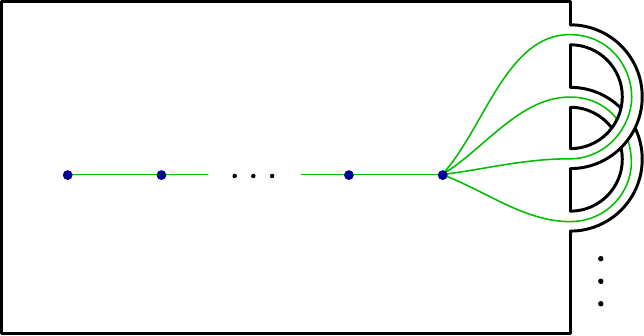}
        \caption{The model for orientable surface braid groups.}
        \label{fig:model-or-braids}
    \end{subfigure}
    \hfill
    \begin{subfigure}[b]{0.48\textwidth}
        \centering
        \includegraphics[scale=0.65]{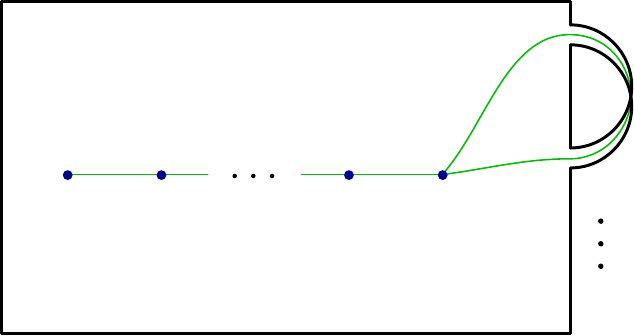}
        \caption{The model for non-orientable surface braid groups.}
        \label{fig:model-nor-braids}
    \end{subfigure}
    \\[3ex]
    \begin{subfigure}[b]{0.48\textwidth}
        \centering
        \includegraphics[scale=0.65]{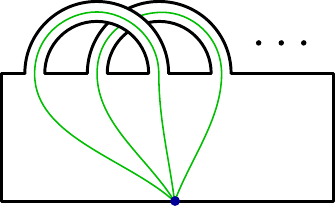}
        \caption{The model for orientable mapping class groups.}
        \label{fig:model-or-mcg}
    \end{subfigure}
    \hfill
    \begin{subfigure}[b]{0.48\textwidth}
        \centering
        \includegraphics[scale=0.65]{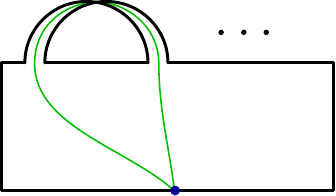}
        \caption{The model for non-orientable mapping class groups.}
        \label{fig:model-nor-mcg}
    \end{subfigure}
    \caption{Four examples of the setting of Theorem~\ref{thm:free_generating_sets_rep} where $S$ is a compact, connected surface with one boundary component, $\Gamma$ is the embedded graph (in green) and $A$ is its set of vertices (blue).}
    \label{fig:models}
\end{figure}

The key setting for the rest of the paper will be the second point of Examples~\ref{eg:examples-BM-homology-lemma}, which we now consider in more detail. Let $S$ be a connected, compact surface with one boundary component, let $\Gamma$ be the embedded graph pictured in Figure~\ref{fig:models} and let $A$ denote the set of vertices of $\Gamma$.

To apply Theorem~\ref{thm:free_generating_sets_rep}, it will be convenient to modify these spaces a little in cases (a) and (b) of Figure~\ref{fig:models}, where the vertices lie in the interior of $S$. In these cases, let $\overline{S}$ be the result of blowing up each vertex of $\Gamma$ to a boundary component (so that the total number of boundary components of $\overline{S}$ is $\lvert A \rvert + 1$), let $\overline{\Gamma}$ be the result of replacing each vertex $v$ of $\Gamma$ with a circle (coinciding with the corresponding new boundary component of $\overline{S}$) subdivided into $\nu(v)$ vertices and $\nu(v)$ edges, where $\nu(v)$ is the valence of $v$, and finally let $\overline{A} \subset \overline{\Gamma}$ be the union of these circles (equivalently, the new boundary components of $\overline{S}$). We clearly have homeomorphisms $\overline{S} \smallsetminus \overline{A} \cong S \smallsetminus A$ and $\overline{\Gamma} \smallsetminus \overline{A} \cong \Gamma \smallsetminus A$. In cases (c) and (d) of Figure~\ref{fig:models}, we simply take $\overline{S}=S$, $\overline{\Gamma}=\Gamma$ and $\overline{A}=A$.

By Theorem~\ref{thm:free_generating_sets_rep}, the inclusion
\begin{equation}
\label{eq:inclusion-of-configuration-spaces}
C_{\lambda}(\Gamma \smallsetminus A) \cong C_{\lambda}(\overline{\Gamma} \smallsetminus \overline{A}) \longhookrightarrow C_{\lambda}(\overline{S} \smallsetminus \overline{A}) \cong C_{\lambda}(S \smallsetminus A)
\end{equation}
induces isomorphisms on Borel-Moore homology for all local coefficient systems on $C_{\lambda}(\overline{S} \smallsetminus \overline{A})$ that extend to $C_{\lambda}(\overline{S})$. But $\overline{A}$ is contained in the boundary of $\overline{S}$ (the purpose of replacing $S,\Gamma,A$ with $\overline{S},\overline{\Gamma},\overline{A}$ was precisely to ensure this) so, by Remark~\ref{rmk:condition-on-local-systems}, the inclusion \eqref{eq:inclusion-of-configuration-spaces} induces isomorphisms on Borel-Moore homology with all local coefficient systems.

The twisted Borel-Moore homology of $C_{\lambda}(S\smallsetminus A)$ may therefore be computed from the twisted Borel-Moore homology of $C_{\lambda}(\Gamma\smallsetminus A)$, where we may now consider $\Gamma$ as an abstract graph (forgetting its embedding into $S$) with vertex set $A$, as depicted in Figure~\ref{fig:graphs-configurations}. Since the complement $\Gamma\smallsetminus A$ is simply the disjoint union of the (open) edges of the graph $\Gamma$, its configuration space $C_{\lambda}(\Gamma\smallsetminus A)$ is a disjoint union of $k$-dimensional open simplices, one for each choice of:
\begin{itemizeb}
\item the number of points that lie on each edge of $\Gamma$;
\item for each edge of $\Gamma$, an ordered list of blocks of $\lambda$, prescribing which blocks of the partition the configuration points that lie on this edge must belong to, as we pass from left to right along the edge (with respect to an arbitrary orientation of the edge, chosen once and for all).
\end{itemizeb}
We summarise this combinatorial information as follows.

\begin{notation}\label{not:set_and_lengths_words}
For a set $X$, write $\tM(X)$ for the free monoid of words on $X$. For a graph $\Gamma$ and an ordered partition $\lambda = (\lambda_{1},\ldots,\lambda_{r}) \vdash k$, denote by $E(\Gamma)$ the set of edges of $\Gamma$ and define $\cW_{\lambda}(\Gamma)$ to be the set of all functions $w \colon E(\Gamma) \to \tM(\{1,\ldots,r\})$ such that each $i = 1,\ldots,r$ appears precisely $\lambda_{i}$ times as a letter in the collection of words $\{ w(e) \mid e \in E(\Gamma) \}$ (thus the total length of these words is $k$). For each word $w(e)$, we denote by $\lvert w(e)\rvert$ its length.
\end{notation}

\begin{figure}[tbp]
    \centering
    \begin{subfigure}[b]{0.48\textwidth}
        \centering
        \includegraphics[scale=0.65]{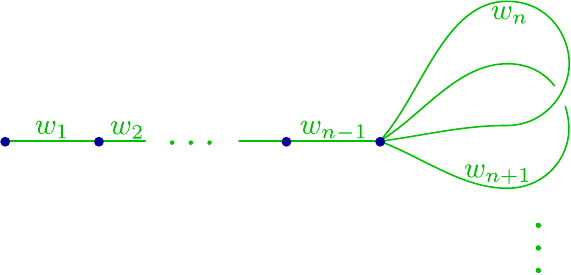}
        \caption{The graph for orientable surface braid groups.}
        \label{fig:graph-or-braids}
    \end{subfigure}
    \hfill
    \begin{subfigure}[b]{0.48\textwidth}
        \centering
        \includegraphics[scale=0.65]{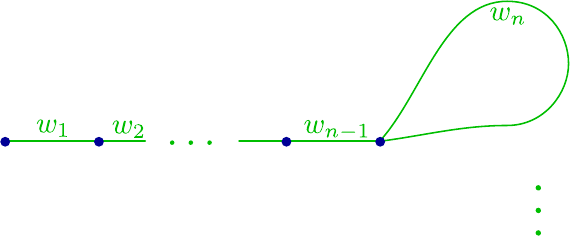}
        \caption{The graph for non-orientable surface braid groups.}
        \label{fig:graph-nor-braids}
    \end{subfigure}
    \\[3ex]
    \begin{subfigure}[b]{0.48\textwidth}
        \centering
        \includegraphics[scale=0.65]{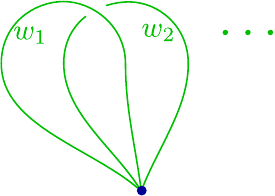}
        \caption{The graph for orientable mapping class groups.}
        \label{fig:graph-or-mcg}
    \end{subfigure}
    \hfill
    \begin{subfigure}[b]{0.48\textwidth}
        \centering
        \includegraphics[scale=0.65]{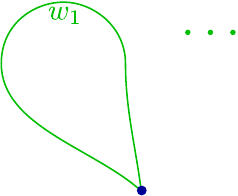}
        \caption{The graph for non-orientable mapping class groups.}
        \label{fig:graph-nor-mcg}
    \end{subfigure}
    \caption{The graphs from Figure~\ref{fig:models}, considered now as abstract graphs and equipped with labels, viewed as generators of the Borel-Moore homology of $C_{\lambda}(\Gamma \smallsetminus A)$; equivalently, by Proposition~\ref{prop:module_structure_BM_homology}, the Borel-Moore homology of $C_{\lambda}(S \smallsetminus A)$.}
    \label{fig:graphs-configurations}
\end{figure}

In this notation, the labelled graphs depicted in Figure~\ref{fig:graphs-configurations} correspond to the different components of the configuration space $C_{\lambda}(\Gamma\smallsetminus A)$ indexed by $\cW_{\lambda}(\Gamma)$.

\begin{eg}\label{eg:set_and_lengths_words}
For example, if $\lambda = (2,1,3,1)$ and the graph $\Gamma$ has four edges enumerated as $E(\Gamma) = \{e_{1},e_{2},e_{3},e_{4}\}$, then the assignment $(w(e_{1}),w(e_{2}),w(e_{3}),w(e_{4})) = (34,12,\varnothing,313)$ specifies an element $w \in \cW_{\lambda}(\Gamma)$, where $\varnothing$ denotes the empty word.
\end{eg}

We summarise this discussion in the following result.

\begin{prop}\label{prop:module_structure_BM_homology}
Let $S$ be a connected, compact surface with one boundary component, let $A$ be either a finite subset of its interior or a single point on its boundary and let $\lambda$ be a partition of a positive integer $k$. Let $\Gamma$ be the abstract graph depicted in Figure~\ref{fig:graphs-configurations}. Then there is a proper map
\begin{equation}
\label{eq:map-from-union-of-open-balls}
\bigsqcup_{w \in \cW_{\lambda}(\Gamma)} \mathring{\Delta}^{k} \longrightarrow C_{\lambda}(S \smallsetminus A),
\end{equation}
where $\mathring{\Delta}^{k}$ denotes the $k$-dimensional open simplex, that induces isomorphisms on Borel-Moore homology in all degrees and with coefficients in any local system $\cL$ on $C_{\lambda}(S\smallsetminus A)$ defined over a ring $R$.
Thus the Borel-Moore homology $H_{*}^{\BM}(C_{\lambda}(S\smallsetminus A);\cL)$ is concentrated in degree $k$ and the $R$-module
\begin{equation}
\label{eq:twisted-BM-homology-group}
H_{k}^{\BM}(C_{\lambda}(S\smallsetminus A);\cL)
\end{equation}
decomposes as a direct sum of $\lvert \cW_{\lambda}(\Gamma) \rvert$ copies of the fibre of $\cL$.
\end{prop}

\begin{notation}\label{nota:tail_and_wedge}
It will be convenient later to fix some standard notation for the different parts of the graphs $\Gamma$ appearing in Proposition~\ref{prop:module_structure_BM_homology} and depicted in Figure~\ref{fig:graphs-configurations}. In cases (a) and (b), assuming that there are $n$ punctures, i.e.~$\lvert A \rvert = n$, let us write $\bI_{n}$ for the linear (or ``tail'') part of the graph, which is a linear graph with $n$ vertices and $n-1$ edges. When the surface $S$ is orientable (cases (a) and (c)), we write $\bW^{\Sigma}_{g}$ for the ``wedge'' part of the graph, which is a graph with one vertex and $2g$ edges, where $g$ is the genus of $S$. When the surface $S$ is non-orientable (cases (b) and (d)), we write instead $\bW^{\N}_{h}$ for the ``wedge'' part of the graph, which is a graph with one vertex and $h$ edges, where $h$ is the non-orientable genus of $S$. The elements of $\cW_{\lambda}(\Gamma)$ indexing the decomposition of \eqref{eq:twisted-BM-homology-group} will typically be denoted by
\begin{equation}
\label{eq:generator_surface_braid_hom_rep_orientable}
(w_{1},\ldots,w_{n-1},[w_{n},w_{n+1}],\ldots,[w_{n+2g-2},w_{n+2g-1}])
\end{equation}
when $S=\Sigma_{g,1}$ and by
\begin{equation}
\label{eq:generator_surface_braid_hom_rep_non-orientable}
(w_{1},\ldots,w_{n-1},[w_{n}],\ldots,[w_{n+h-1}])
\end{equation}
when $S=\N_{h,1}$. The first $n-1$ terms are the values of $w$ on $\bI_{n}$ and the remaining $2g$ (resp.~$h$) terms in square brackets are the values of $w$ on $\bW^{\Sigma}_{g}$ (resp.~$\bW^{\N}_{h}$).
\end{notation}

Recall from Definition~\ref{defn:Sdoubleprime} the notion of the \emph{dual surface} of a punctured surface, which we will apply to the surfaces depicted in Figure~\ref{fig:models}. In cases (a) and (b), the blow-up $\overline{S}$ is obtained from the punctured surface $S \smallsetminus A$ by blowing up each (interior) puncture in $A$ to a new boundary component and the dual surface $\check{S}$ is given by removing the original boundary component $\partial S$ from $\overline{S}$ but keeping the $\lvert A \rvert$ new boundary components.
In cases (c) and (d), the blow-up $\overline{S}$ simply replaces the single boundary puncture $A$ in $\partial S$ with a closed interval and the dual surface $\check{S}$ is the union of the interior of $S$ with this closed interval in the boundary of $\overline{S}$. We may also take the dual of the graph $\Gamma$:

\begin{notation}
\label{notation:dual-graph}
We denote by $\check{\Gamma}$ the embedded dual graph of $\Gamma$ as illustrated in Figure~\ref{fig:models-dual}. (For the purposes of this description of $\check{\Gamma}$, each collection of parallel green arcs in Figure~\ref{fig:models-dual} labelled by $w_i$ is to be considered as a single edge; the collections of parallel arcs will become relevant only later in \S\ref{ss:dual-bases}, for Definition~\ref{def:dual-basis-elements}.) Precisely, $\check{\Gamma}$ is an embedded disjoint union of edges, with one edge $e'$ for each edge $e$ of $\Gamma$, intersecting $e$ transversely exactly once and disjoint from the other edges of $\Gamma$, and with its endpoints on the boundary of $S$.
\end{notation}

The twisted Borel-Moore homology $H_{*}^{\BM}(C_{\lambda}(\check{S});\cL)$, considered in \S\ref{ss:vertical_alternatives} to define vertical-type alternatives, has an explicit description as a module similar to that of $H_{*}^{\BM}(C_{\lambda}(S \smallsetminus A);\cL)$ in Proposition~\ref{prop:module_structure_BM_homology}. This is another direct application of Theorem~\ref{thm:free_generating_sets_rep}, this time applied to the triple $(M',B',A')$ where we set $M' = \overline{S}$, $A' = \partial(S \smallsetminus A)$ (so that $M' \smallsetminus A' = \check{S}$) and $B' = A' \cup \check{\Gamma} = \partial(S \smallsetminus A) \cup \check{\Gamma}$. We record this as follows:

\begin{prop}
\label{lem:module_structure_BM_homology-check}
Let $S$ and $A$ be as in Proposition~\ref{prop:module_structure_BM_homology} and define $\check{S}$ as in Definition~\ref{defn:Sdoubleprime}. Let $\lambda$ be a partition of a positive integer $k$. Then the statement of Proposition~\ref{prop:module_structure_BM_homology} holds verbatim with $S \smallsetminus A$ replaced with $\check{S}$, and the graph $\Gamma$ replaced with its dual $\check{\Gamma}$.
\end{prop}

\begin{rmk}
\label{rmk:dual-graph}
The embedded graphs $\Gamma$ and $\check{\Gamma}$ are dual in the sense that each edge of $\Gamma$ intersects exactly one edge of $\check{\Gamma}$ and vice versa. This determines a bijection $E(\Gamma) \cong E(\check{\Gamma})$, which induces a bijection $\cW_{\lambda}(\Gamma) \cong \cW_{\lambda}(\check{\Gamma})$ (see Notation~\ref{not:set_and_lengths_words}) and hence, by Propositions~\ref{prop:module_structure_BM_homology} and \ref{lem:module_structure_BM_homology-check}, a module isomorphism between the twisted Borel-Moore homology of $C_{\lambda}(S \smallsetminus A)$ and the twisted Borel-Moore homology of $C_{\lambda}(\check{S})$.
\end{rmk}

\subsection{Dual bases}\label{ss:dual-bases}

\begin{figure}[tb]
    \centering
    \begin{subfigure}[b]{0.48\textwidth}
        \centering
        \includegraphics[scale=0.65]{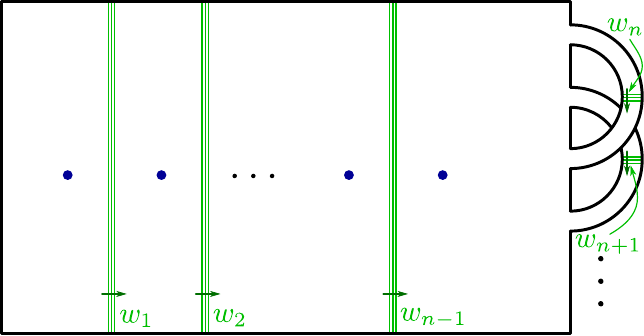}
        \caption{Dual basis for orientable surface braid groups.}
        \label{fig:model-or-braids-dual}
    \end{subfigure}
    \hfill
    \begin{subfigure}[b]{0.48\textwidth}
        \centering
        \includegraphics[scale=0.65]{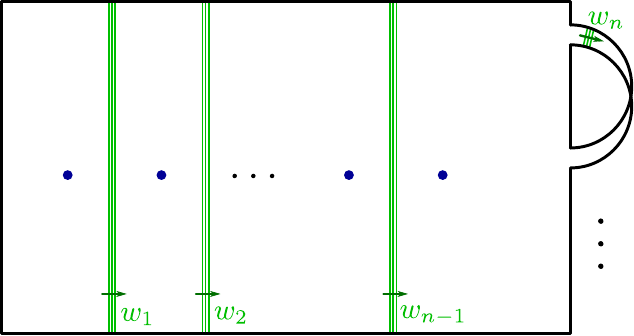}
        \caption{Dual basis for non-orientable surface braid groups.}
        \label{fig:model-nor-braids-dual}
    \end{subfigure}
    \\[3ex]
    \begin{subfigure}[b]{0.48\textwidth}
        \centering
        \includegraphics[scale=0.65]{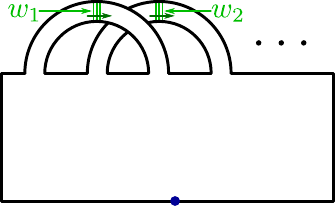}
        \caption{Dual basis for orientable mapping class groups.}
        \label{fig:model-or-mcg-vertical-dual}
    \end{subfigure}
    \hfill
    \begin{subfigure}[b]{0.48\textwidth}
        \centering
        \includegraphics[scale=0.65]{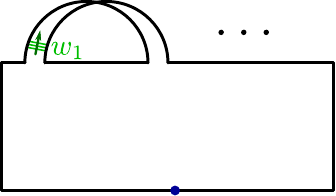}
        \caption{Dual basis for non-orientable mapping class groups.}
        \label{fig:model-nor-mcg-vertical-dual}
    \end{subfigure}
    \caption{A linearly independent set of elements of $H_{k}(C_{\lambda}(S \smallsetminus A) , \partial ; \cL\otimes\cO)$ whose span is isomorphic to the dual of $H_{k}^{\BM}(C_{\lambda}(S \smallsetminus A); \cL)$ via the perfect pairing \eqref{eq:perfect-pairing}. See Definition~\ref{def:dual-basis-elements}.}
    \label{fig:models-dual}
\end{figure}

We now describe, using Poincar{\'e}-Lefschetz duality, a perfect pairing between the $R$-module $\eqref{eq:twisted-BM-homology-group} = H_{k}^{\BM}(C_{\lambda}(S\smallsetminus A);\cL)$ and another naturally-defined homology $R$-module, for which we describe a ``dual'' basis. In order to apply Poincar{\'e}-Lefschetz duality, we consider the orientation local system $\cO$ of the manifold $C_{\lambda}(S \smallsetminus A)$; in particular, when $S$ is orientable, $\cO$ is the trivial local system $\bZ$.

Let us now consider the relative homology group $H_{k}(C_{\lambda}(S \smallsetminus A) , \partial ; \cL\otimes\cO)$, where $\partial$ is an abbreviation of $\partial C_{\lambda}(S \smallsetminus A)$, the boundary of the topological manifold $C_{\lambda}(S \smallsetminus A)$, which consists of all configurations that non-trivially intersect the boundary of $S \smallsetminus A$. In cases (c) and (d) of Figure~\ref{fig:models} (i.e.~for the mapping class group representations), we implicitly make a small modification here: we replace $A$, which is a single point in $\partial S$, with a small closed interval in $\partial S$ and we correspondingly replace $S \smallsetminus A$ with the closure in $S$ of the complement of this small closed interval. In other words, similarly to the modification that we made in cases (a) and (b) in \S\ref{ss:free-bases}, we are \emph{blowing up} the (unique) vertex of the graph $\Gamma$ on $\partial S$.

For this subsection, we assume that $\cL$ is a \emph{rank-one} local system; i.e.~its fibre over each point is a free module of rank one over the ground ring $R$. In this case, the direct sum decomposition of Proposition~\ref{prop:module_structure_BM_homology} corresponds to a free basis for $H_{k}^{\BM}(C_{\lambda}(S \smallsetminus A); \cL)$ over $R$. There is a naturally corresponding set of elements of the relative homology group $H_{k}(C_{\lambda}(S \smallsetminus A) , \partial ; \cL\otimes\cO)$, depicted in Figure~\ref{fig:models-dual}, indexed by the set $\cW_{\lambda}(\check{\Gamma})$ (see Notations~\ref{not:set_and_lengths_words} and \ref{notation:dual-graph}).

\begin{defn}
\label{def:dual-basis-elements}
Consider one of the labelled graphs of Figure~\ref{fig:models-dual}, whose labels specify an element $w \in \cW_{\lambda}(\Gamma)$; in particular, each collection of parallel arcs is labelled by a word $w_{i}$ in the monoid $\tM(\{1,\ldots,r\})$. This first of all indicates that the number of parallel arcs in the collection is $\lvert w_{i} \rvert$, and each individual arc in the collection inherits a label which is the corresponding letter (in $\{1,\ldots,r\}$) of this word (written in the direction specified by the small perpendicular arrows in Figure~\ref{fig:models-dual}). The relative homology class in $H_{k}(C_{\lambda}(S \smallsetminus A) , \partial ; \cL\otimes\cO)$ depicted by this figure is represented by the relative cycle given by the subspace of configurations where exactly one point lies on each arc and this point belongs to the block of $\lambda$ specified by the label of the arc.
\end{defn}

\begin{convention}
The bijection $\cW_{\lambda}(\Gamma) \cong \cW_{\lambda}(\check{\Gamma})$ of Remark~\ref{rmk:dual-graph} gives a bijection between (the indexing set of) a basis for the free $R$-module $H_{k}^{\BM}(C_{\lambda}(S \smallsetminus A); \cL)$ (by Proposition~\ref{prop:module_structure_BM_homology}) and the set of elements of $H_{k}(C_{\lambda}(S \smallsetminus A) , \partial ; \cL\otimes\cO)$ described in Definition~\ref{def:dual-basis-elements}. We will henceforth implicitly identify these two indexing sets via this bijection.
\end{convention}

As explained in \cite[Th.~A]{anghelpalmer} (in the orientable setting, which generalises verbatim to the non-orientable setting in the presence of the orientation local system $\cO$), the relative cap product and Poincar{\'e}-Lefschetz duality induce a pairing
\begin{equation}
\label{eq:pairing}
H_{k}^{\BM}(C_{\lambda}(S \smallsetminus A); \cL) \otimes_R H_{k}(C_{\lambda}(S \smallsetminus A) , \partial ; \cL\otimes\cO) \longrightarrow R
\end{equation}
whose evaluation on a basis element of $H_{k}^{\BM}(C_{\lambda}(S \smallsetminus A); \cL)$ (from Proposition~\ref{prop:module_structure_BM_homology}) together with an element of Definition~\ref{def:dual-basis-elements} is equal to $1$ if the two elements are indexed by the same $w \in \cW_{\lambda}(\Gamma)$ and equal to $0$ otherwise. It follows that the submodule spanned by the elements of Definition~\ref{def:dual-basis-elements} is freely spanned by them (i.e.~they are linearly independent), and the pairing \eqref{eq:pairing} restricts to a \emph{perfect pairing} when we restrict to this submodule of the right-hand factor of its domain.

\begin{notation}\label{not:second_alternative}
We write $H_{k}^{\partial}(C_{\lambda}(S \smallsetminus A); \cL\otimes\cO)$ for the $R$-submodule of $H_{k}(C_{\lambda}(S \smallsetminus A) , \partial ; \cL\otimes\cO)$ (freely) spanned by the elements defined in Definition~\ref{def:dual-basis-elements}. In general, for a module $W$ over a ring $R$, we denote by $W^{\vee}$ its linear dual $R$-module $\Hom_{R}(W,R)$.
\end{notation}

With this notation, the discussion above may be summarised as follows.

\begin{coro}
\label{coro:perfect-pairing}
Let $S$ be a connected, compact, orientable surface with one boundary component, let $A$ be either a finite subset of its interior or a closed interval in its boundary and let $\lambda$ be a partition of a positive integer $k$. Choose any rank-one local system $\cL$ on $C_{\lambda}(S\smallsetminus A)$ defined over a ring $R$. Then the $R$-module $H_{k}^{\partial}(C_{\lambda}(S \smallsetminus A); \cL\otimes\cO)$ is freely generated over $R$ by the elements of Definition~\ref{def:dual-basis-elements}, indexed by $\cW_{\lambda}(\Gamma)$. Moreover, there is a perfect pairing
\begin{equation}
\label{eq:perfect-pairing}
H_{k}^{\BM}(C_{\lambda}(S \smallsetminus A); \cL) \otimes_R H_{k}^{\partial}(C_{\lambda}(S \smallsetminus A); \cL\otimes\cO) \longrightarrow R,
\end{equation}
given by the relative cap product and Poincar{\'e}-Lefschetz duality, whose matrix with respect to the two bases that we have described is the identity matrix. In particular, we have
\begin{equation}
\label{eq:isomorphism-with-dual}
H_{k}^{\partial}(C_{\lambda}(S \smallsetminus A); \cL\otimes\cO) \;\cong\; \left( H_{k}^{\BM}(C_{\lambda}(S \smallsetminus A); \cL) \right)^{\vee}.
\end{equation}
\end{coro}

The perfect pairing \eqref{eq:perfect-pairing} and the dual basis described in Definition~\ref{def:dual-basis-elements} (and illustrated in Figure~\ref{fig:models-dual}) will be used in some diagrammatic proofs in the next section, including in the proof of the ``Cloud lemma'' (see Lemma~\ref{lem:cloud}).

\paragraph*{Relation to the vertical-type alternative.}

The dual of the representation $H_{k}^{\BM}(C_{\lambda}(S \smallsetminus A); \cL)$ and its vertical-type alternative $H_{k}^{\BM}(C_{\lambda}(\check{S});\cL)$ are closely related via the isomorphism \eqref{eq:isomorphism-with-dual}, as we now explain. For the purpose of this paragraph, we shall abbreviate these representations as $V := H_{k}^{\BM}(C_{\lambda}(S \smallsetminus A); \cL)$ and $\check{V} := H_{k}^{\BM}(C_{\lambda}(\check{S});\cL)$.
Applying Corollary~\ref{coro:perfect-pairing} and its analogue when $S \smallsetminus A$ is replaced by $\check{S}$ (which is also part of \cite[Thm.~A]{anghelpalmer}), it follows from \cite[Thm.~B]{anghelpalmer} that we have embeddings of representations
\begin{equation}
\label{eq:embeddings-duals}
V^{\vee} \longhookrightarrow \check{V} \qquad\text{and}\qquad (\check{V})^{\vee} \longhookrightarrow V
\end{equation}
under the mild assumption that the local coefficient system $\cL$ is \emph{$u$-homogeneous} for $u \in R^\times$ such that the quantum factorials $[n]_u!$ are non-zero-divisors for all $n\geq 1$; see \cite[Def.~$2.14$]{anghelpalmer}. In this setting, if $R$ is an integral domain, we have isomorphisms
\begin{equation}
\label{eq:iso-duals}
V^{\vee} \otimes \F(R) \cong \check{V} \otimes \F(R) \qquad\text{and}\qquad (\check{V})^{\vee} \otimes \F(R) \cong V \otimes \F(R)
\end{equation}
after tensoring \eqref{eq:embeddings-duals} with the field of fractions $\F(R)$.

\begin{eg}\label{eg:LB_dual}
This applies in particular to the representations $V = \LB_{((k),2)}(n)$ and their vertical-type alternatives $\LB^{\vrtcl}_{((k),2)}(n)$ described in \S\ref{ss:vertical_alternatives}, where the ground ring $R$ is $\bZ[q^{\pm 1}]$ if $k=1$ and $\bZ[t^{\pm 1},q^{\pm 1}]$ if $k 
\geq 2$. In this example, the local coefficient system $\cL$ is $1$-homogeneous if $k=1$ and $t$-homogeneous if $k \geq 2$. We therefore have embeddings
\[
\LB_{((k),2)}(n)^{\vee} \longhookrightarrow \LB^{\vrtcl}_{((k),2)}(n) \qquad\text{and}\qquad \LB^{\vrtcl}_{((k),2)}(n)^{\vee} \longhookrightarrow \LB_{((k),2)}(n)
\]
that become isomorphisms after tensoring with the field of rational functions $\F(R) = \bQ(q) \text{ or } \bQ(t,q)$.
The representations $\LB^{\vrtcl}_{((k),2)}(n)$ are part of a functor $\LB^{\vrtcl}_{((k),2)}$ defined on $\langle \Beta,\Beta \rangle$, described in \S\ref{ss:vertical_alternatives}. The dual representations $\LB_{((k),2)}(n)^{\vee}$ may similarly be extended to a functor defined on $\langle \Beta,\Beta \rangle$ (see the end of \S\ref{ss:SES_surface_braids_alternative}), so the left-hand embedding above may be thought of as an embedding of representations of the category $\langle \Beta,\Beta \rangle$ (that becomes an isomorphism after composing with $- \otimes \F(R) \colon R\lmod \to \F(R)\lmod$). One may similarly upgrade the right-hand embedding to an embedding of representations of $\langle \Beta,\Beta \rangle$.
\end{eg}

\subsection{Tethers}
\label{ss:tethers}

In the case when the local system $\cL$ has rank one (i.e.~its fibres are free modules of rank one over the ground ring), the homology module $H_{k}^{\BM}(C_{\lambda}(S\smallsetminus A);\cL)$ is free of rank $\lvert \cW_{\lambda}(\Gamma) \rvert$, by Proposition~\ref{prop:module_structure_BM_homology}. However, the map \eqref{eq:map-from-union-of-open-balls} does not quite specify a free generating set: each open simplex $\mathring{\Delta}^{k}$ in its domain has a fundamental class in Borel-Moore homology, but in order to push it forward to an element of $H_{k}^{\BM}(C_{\lambda}(S\smallsetminus A);\cL)$, we need to choose a lift of the proper map $\mathring{\Delta}^{k} \to C_{\lambda}(S \smallsetminus A)$ to a map of local systems (where $\mathring{\Delta}^{k}$ is equipped with a trivial local system). To finish this section, we explain how to specify this choice via \emph{tethers}.

\begin{defn}
\label{defn:tether}
Let $f \colon X \to Y$ be a map of spaces and choose $x_{0} \in X$ and $y_{0} \in Y$. A \emph{tether} is a path in $Y$, up to endpoint-preserving homotopy, from $y_{0}$ to $f(x_{0})$.
\end{defn}

\begin{lem}
\label{lem:tethers-lifting}
Let $X$ be a simply-connected space and let $Y$ be a path-connected space admitting a universal cover, both equipped with basepoints $x_{0} \in X$ and $y_{0} \in Y$. Let $f \colon X \to Y$ be a map, not necessarily preserving basepoints, and let $\xi$ be a bundle of $R$-modules over $Y$. Denote by $V$ the fibre of $\xi$ over $y_{0}$ and denote by $\tau$ the trivial bundle of $R$-modules over $X$ with fibre $V$.
\begin{itemizeb}
\item Each morphism $\rT_{0} \colon y_{0} \to f(x_{0})$ in the fundamental groupoid $\Pi_{1}(Y)$ determines a lift of $f$ to a bundle map $\tau \to \xi$.
\item The two bundle maps $\tau \to \xi$ determined by two morphisms $\rT_{0},\rT'_{0} \colon y_{0} \to f(x_{0})$ differ by the automorphism of the trivial bundle $\tau$ given by the monodromy $\xi(\rT_{0}^{-1} \circ \rT'_{0}) \in \mathrm{Aut}_R(V)$.
\end{itemizeb}
\end{lem}
\begin{proof}
By unique path lifting, bundles of $R$-modules over $X$ (resp.~$Y$) are in one-to-one correspondence with functors $\Pi_{1}(X) \to R\lmod$ (resp.~$\Pi_{1}(Y) \to R\lmod$). Under this identification, a bundle map $\tau \to \xi$ lifting $f$ corresponds (uniquely) to a natural transformation $T \colon \tau \to \xi \circ \Pi_{1}(f)$. Since $\tau$ is the trivial functor at $V \in R\lmod$, this is determined by specifying, for each $x \in X$, a homomorphism $T_x \colon V \to \xi(f(x))$. Since $X$ is simply-connected, there is a unique morphism $x_{0} \to x$ in $\Pi_{1}(X)$, which we denote by $\gamma_{x_{0},x}$. We then define
\[
T_x := \xi \bigl( \Pi_{1}(f)(\gamma_{x_{0},x}) \circ \rT_{0} \bigr) \colon V = \xi(y_{0}) \longrightarrow \xi(f(x)).
\]
It is straightforward to check that this defines a natural transformation $T \colon \tau \to \xi \circ \Pi_{1}(f)$, using the fact that $\Pi_{1}(X)$ is an indiscrete groupoid (i.e.~it has a unique morphism between any two objects).

Replacing $\rT_{0}$ with $\rT'_{0}$ in the construction of the bundle map (natural transformation) has the effect of precomposing with the automorphism $\xi(\rT_{0}^{-1} \circ \rT'_{0})$ of $V$, which we may view as an automorphism of the trivial bundle $\tau$.
\end{proof}

\begin{rmk}
\label{rmk:tethers-lifting}
In the terminology of Definition~\ref{defn:tether}, the first point of Lemma~\ref{lem:tethers-lifting} says that, if $X$ is simply-connected and $Y$ is path-connected admitting a universal cover, then for any bundle $\xi \colon E \to Y$ of $R$-modules, a choice of tether determines a lift of $f \colon X \to Y$ to a bundle map of the form $X \times E|_{y_{0}} \to E$.
\end{rmk}

\begin{construction}
\label{construction:tethers-basis}
In the setting of Proposition~\ref{prop:module_structure_BM_homology}, we may consider the restriction of the map \eqref{eq:map-from-union-of-open-balls} to any open simplex of its domain, which we denote by $\eta_{\bw} \colon \mathring{\Delta}^{k} \to C_{\lambda}(S \smallsetminus A)$ for $\bw \in \cW_{\lambda}(\Gamma)$. Let us choose, once and for all, a base configuration $c_{0} \in C_{\lambda}(S \smallsetminus A)$ contained in the boundary $\partial(S \smallsetminus A)$. We take the barycentre $b_{0} \in \mathring{\Delta}^k$ to be the basepoint of the open simplex. Finally, we choose a tether for $\eta_{\bw}$ (with respect to $b_{0}$ and $c_{0}$), namely a path $\rT_{\bw}$ in $C_{\lambda}(S \smallsetminus A)$, up to endpoint-preserving homotopy, from $c_{0}$ to the image $\eta_{\bw}(b_{0})$ of the barycentre of the open simplex. By the first point of Lemma~\ref{lem:tethers-lifting} (and Remark~\ref{rmk:tethers-lifting}), this determines a lift of $\eta_{\bw}$ to a map of local systems $\mathring{\Delta}^{k} \times V \to \cL$, where $V = \cL|_{c_{0}}$ denotes the fibre of $\cL$ over $c_{0}$. Since $\eta_{\bw}$ is a proper map, we obtain an induced map on twisted Borel-Moore homology
\begin{equation}
\label{eq:induced-homomorphism-tether}
(\eta_{\bw},\rT_{\bw})_* \colon H_k^{\BM}(\mathring{\Delta}^{k} ; V) \longrightarrow H_k^{\BM}(C_{\lambda}(S \smallsetminus A) ; \cL).
\end{equation}
In particular, if $\cL$ is a local system of rank one, we may identify $V$ with the ground ring $R$ and the left-hand side of \eqref{eq:induced-homomorphism-tether} has a canonical generator given by the fundamental class $[\mathring{\Delta}^{k}]$. In this setting, we may define
\begin{equation}
\label{eq:explicit-basis-element}
e_{\bw} := (\eta_{\bw},\rT_{\bw})_*([\mathring{\Delta}^{k}]) \in H_k^{\BM}(C_{\lambda}(S \smallsetminus A) ; \cL).
\end{equation}
By Proposition~\ref{prop:module_structure_BM_homology}, the set $\{ e_{\bw} \mid \bw \in W_{\lambda}(\Gamma) \}$ is a free basis for $H_k^{\BM}(C_{\lambda}(S \smallsetminus A) ; \cL)$ over $R$.
\end{construction}

\begin{notation}\label{notation:basis_notation}
For brevity, we will often denote this homology class simply by $\bw$ instead of $e_{\bw}$, whenever this would not lead to confusion.
\end{notation}

\begin{rmk}
\label{rmk:tethers-basis-choices}
To be explicit, the choices involved in Construction~\ref{construction:tethers-basis} are the base configuration $c_{0} \in C_{\lambda}(S \smallsetminus A)$, the identification $\cL|_{c_{0}} \cong R$ and one tether $\rT_{\bw}$ for each $\bw \in W_{\lambda}(\Gamma)$. We will specify these choices later, when they are needed to do explicit computations; see for example Figures~\ref{fig:cloud-lemma}, \ref{fig:basis-alternative-picture} and \ref{fig:basis-alternative-picture-vertical} for illustrations of tethers $\rT_{\bw}$.
\end{rmk}

\begin{lem}
\label{lem:changing-tether}
In the setting of Construction~\ref{construction:tethers-basis} with $\cL$ a local system of rank one, changing the choice of tether $\rT_{\bw}$ has the effect of multiplying the homology class \eqref{eq:explicit-basis-element} by a unit of the ground ring $R$. More precisely, for two tethers $\rT_{\bw}$ and $\rT'_{\bw}$, we have
\[
(\eta_{\bw},\rT'_{\bw})_*([\mathring{\Delta}^{k}]) \;=\; \mu(\rT'_{\bw} . \bar{\rT}_{\bw}) (\eta_{\bw},\rT_{\bw})_*([\mathring{\Delta}^{k}]),
\]
where $\bar{\rT}_{\bw}$ denotes the reverse of the path $\rT_{\bw}$, so that $\rT'_{\bw} . \bar{\rT}_{\bw}$ is an element of $\pi_{1}(C_{\lambda}(S \smallsetminus A) , c_{0})$, and $\mu \colon \pi_{1}(C_{\lambda}(S \smallsetminus A) , c_{0}) \to R^\times$ denotes the monodromy of $\cL$ at $c_{0}$, using the identification $\cL|_{c_{0}} \cong R$.
\end{lem}
\begin{proof}
By the second point of Lemma~\ref{lem:tethers-lifting}, the two maps of local systems $\mathring{\Delta}^{k} \times V \to \cL$ differ by precomposition by the automorphism $\xi(\rT_{\bw}^{-1} \circ \rT'_{\bw}) \in \mathrm{Aut}_R(V)$. Here, $\xi$ is the local system $\cL$ viewed as a functor out of the fundamental groupoid, so its restriction to the automorphism group of $y_{0} = c_{0}$ is the monodromy action $\mu$. We may therefore rewrite this automorphism as $\mu(\rT'_{\bw} . \bar{\rT}_{\bw})$, where we have also switched to the usual notation convention for composition of paths, which goes from left to right and where $\bar{\phantom{\rT}}$ denotes the reverse of a path. Since we have $V = R$, this automorphism lies in $\mathrm{Aut}_R(R) = R^\times$ and it therefore acts via the canonical action of $R^\times$ on the $R$-module of maps of local systems from $\mathring{\Delta}^{k} \times R$ to $\cL$. Hence the two maps of local systems differ by multiplication by the scalar $\mu(\rT'_{\bw} . \bar{\rT}_{\bw}) \in R^\times$. It follows that the same is true for the induced maps \eqref{eq:induced-homomorphism-tether} on twisted Borel-Moore homology and thus in particular for the images \eqref{eq:explicit-basis-element} of $[\mathring{\Delta}^k]$ under these maps.
\end{proof}

\begin{construction}
\label{construction:tethers-basis-dual}
One may similarly specify a free generating set of the dual homology module $H_{k}^{\partial}(C_{\lambda}(S \smallsetminus A); \cL\otimes\cO)$ (see Corollary~\ref{coro:perfect-pairing}). In this case, rather than a proper map from an open $k$-simplex, each $\bw \in \cW_{\lambda}(\Gamma)$ indexes a map from a closed $k$-cube taking its boundary to the boundary of the manifold $C_{\lambda}(S \smallsetminus A)$. The construction is parallel to that of Construction~\ref{construction:tethers-basis}, using the same base configuration $c_{0} \in C_{\lambda}(S \smallsetminus A)$ as before, an identification $(\cL\otimes\cO)|_{c_{0}} \cong R$ and one tether $\rT_{\bw}$ for each $\bw \in W_{\lambda}(\Gamma)$. Since we have already chosen an identification $\cL|_{c_{0}} \cong R$ (see Remark~\ref{rmk:tethers-basis-choices}), an identification $(\cL\otimes\cO)|_{c_{0}} \cong R$ is determined by an identification $\cO|_{c_{0}} \cong R$, in other words a local orientation of the manifold $C_{\lambda}(S \smallsetminus A)$ at $c_{0}$.
\end{construction}

\section{Short exact sequences}\label{s:SES_homol_rep-func}

In this section, we construct the fundamental short exact sequences for homological representation functors of Theorem~\ref{thm:ses}. We start by recalling the categorical background of these short exact sequences in \S\ref{ss:background_preliminaries_SES}. Then we construct the short exact sequences for the functors of surface braid groups in \S\ref{ss:SES_surface_braid_groups}, and in \S\ref{ss:SES_mcg} for those of mapping class groups of surfaces.
Throughout \S\ref{s:SES_homol_rep-func}, we consider homological representation functors indexed by an ordered partition $\lambda = (\lambda_{1},\ldots,\lambda_{r}) \vdash k$ of an integer $k\geq1$ and by the depth $\ell\geq1$ of a lower central series.

\subsection{Background and preliminaries}\label{ss:background_preliminaries_SES}

This section recollects the key categorical tools that define the setting in which we unearth the short exact sequences of homological representation functors of \S\ref{ss:SES_surface_braid_groups} and \S\ref{ss:SES_mcg}. We also prepare the work of these sections with an important foreword in \S\ref{ss:preliminaries_SES}.

\subsubsection{Short exact sequences induced from the categorical framework}\label{sss:translation_background}

We recollect here the notions and first properties of \emph{translation}, \emph{difference} and \emph{evanescence} functors, which give rise to the key natural short exact sequences that we study for homological representation functors in \S\ref{ss:SES_surface_braid_groups} and \S\ref{ss:SES_mcg}.
The following definitions and results extend verbatim to the present slightly larger framework from the previous literature on this topic; see \cite{DV3}, \cite[\S 3]{soulieLM1} and \cite[\S 4]{soulieLMgeneralised} for instance. The various proofs are straightforward generalisations of this previous work. Hence the content of \S\ref{sss:translation_background} is just stated, without detailed justification, and we refer the reader to these sources for a comprehensive introduction to the following material.

For the remainder of \S\ref{ss:background_preliminaries_SES}, we fix an abelian category $\cA$, a strict left-module groupoid $(\cM,\natural)$ over a braided strict monoidal groupoid $(\cG,\natural,\zero)$, such that $(\cG,\natural,\zero)$ has no zero divisors and that $\mathrm{Aut}_{\cG}(\zero)=\{ \id_{\zero} \}$. We also assume that $\cM$ and $\cG$ are both small and skeletal, have the same set of objects identified with the non-negative integers $\bN$ with the standard notation $\tn$ to denote these objects as in \S\ref{sss:framework_homological_representation_functor}, and that both the monoidal and module structures $\natural$ are given on objects by addition.
In particular, this is consistent with the fact that $\zero$ is the unit for the monoidal structure of $\cG$. One quickly checks that all of the examples of $\cM$ and $\cG$ defined in \S\ref{ss:examples_homological_representations_functors} satisfy all of these assumptions.

For an object $\tn$ of $\cG$, let $\tau_{\tn}$ be the endofunctor of the functor category $\Fct(\langle \cG,\cM \rangle ,\cA)$ defined by $\tau_{\tn}(F):=F(\tn\natural-)$, called the \emph{translation} functor.
Let $i_{\tn} \colon \Identity\to\tau_{\tn}$ be the natural transformation of $\Fct(\langle \cG,\cM \rangle , \cA)$ defined by precomposition with the morphisms $[\tn,\id_{\tn\natural \tm}] \colon \tm\to \tn\natural \tm$ for each $\tm\in \obj(\cM)$. We define $\delta_{\tn}:=\mathrm{coker}(i_{\tn})$, called the \emph{difference} functor, and $\kappa_{\tn}:=\mathrm{ker}(i_{\tn})$, called the \emph{evanescence} functor; the associated canonical natural inclusion $\kappa_{\tn}\hookrightarrow\Identity$ and natural  projection $\tau_{\tn}\twoheadrightarrow\delta_{\tn}$ are denoted by $\Omega_{\tn}$ and $\Delta_{\tn}$ respectively. We also denote by $\tau_{\tn}^{d}$ and $\delta_{\tn}^{d}$ the $d$-fold iterations $\tau_{\tn}\cdots\tau_{\tn}\tau_{\tn}$ and $\delta_{\tn}\cdots\delta_{\tn}\delta_{\tn}$ respectively. The translation functor $\tau_{\tn}^{d}$ is by definition naturally isomorphic to $\tau_{d\tn}$.

The translation functor $\tau_{\tn}$ is exact and induces the following exact sequence of endofunctors of $\Fct(\langle \cG,\cM \rangle ,\mathcal{A})$:
\begin{equation}\label{eq:ESCaract}
0\longrightarrow\kappa_{\tn}\overset{\Omega_{\tn}}{\longrightarrow}\Identity\overset{i_{\tn}}{\longrightarrow}\tau_{\tn}\overset{\Delta_{\tn}}{\longrightarrow}\delta_{\tn}\longrightarrow0.
\end{equation}
Moreover, for a short exact sequence $0\to F\to G\to H\to 0$ in the category $\Fct(\langle \cG,\cM \rangle ,\mathcal{A})$, there is a natural exact sequence defined from the snake lemma:
\begin{equation}\label{eq:LESkappadelta}
0\longrightarrow\kappa_{\tn}F\longrightarrow\kappa_{\tn}G\longrightarrow\kappa_{\tn}H\longrightarrow\delta_{\tn}F\longrightarrow\delta_{\tn}G\longrightarrow\delta_{\tn}H\longrightarrow0.
\end{equation}
Finally, for $\tn,\tm\in\obj(\cG)$, $\tau_{\tn}$ and $\tau_{\tm}$ commute up to natural isomorphism coming from the braiding and they commute with limits and colimits; $\delta_{\tn}$ and $\delta_{\tm}$ commute up to natural isomorphism (induced by the braiding) and they commute with colimits; $\kappa_{\tn}$ and $\kappa_{\tm}$ commute up to natural isomorphism (induced by the braiding) and they commute with limits; and $\tau_{\tn}$ commute with the functors $\delta_{\tm}$ and $\kappa_{\tm}$ up to natural isomorphism.

\subsubsection{Preliminaries for the homological representation functors}\label{ss:preliminaries_SES}

Throughout \S\ref{ss:preliminaries_SES}, we consider any one of the homological representation functors of \S\ref{ss:examples_homological_representations_functors}. Following Notation~\ref{notation:generic_notation_star}, we denote it by $\fL^{\star}_{(\lambda,\ell)}$, with $\bZ[Q^{\star}_{(\lambda,\ell)}]$ the ground ring of the target module category, where $\star$ either stands for the blank space or $\star=\unt$.

\begin{rmk}
The short exact sequences that we exhibit in \S\ref{ss:SES_surface_braid_groups}--\S\ref{ss:SES_mcg} are applications of the exact sequence \eqref{eq:ESCaract}, with $\tn=\one$, to each homological representation functor of \S\ref{ss:examples_homological_representations_functors}. With a little more work, one could deduce analogous (though slightly more complex) results from \eqref{eq:ESCaract} for any object $\tn$. However, only the case $\tn=\one$ will be needed in \S\ref{s:poly_homol_rep-func} to prove our polynomiality results, so we shall not pursue this generalisation here.
\end{rmk}

\paragraph*{Subpartitions.}
Our descriptions of the difference functor $\delta_{\one}\fL^{\star}_{(\lambda,\ell)}$ in \S\ref{ss:SES_surface_braid_groups}--\S\ref{ss:SES_mcg} make key use of some appropriate partitions of $k-1$ obtained from $\lambda\vdash k$.

\begin{notation}\label{not:sets_partitions}
Let $k,k'\geq1$ be integers such that $k\geq k'$. For an ordered partition $\lambda = (\lambda_{1},\ldots,\lambda_{r})\vdash k$, we denote by $\{\lambda - k'\}$ the set of ordered partitions
\[
\left\{(\lambda_{1}-\lambda'_{1},\ldots,\lambda_{r}-\lambda'_{r}) \bigm| 0\leq \lambda'_{i}\leq \lambda_{i} \textrm{ such that } \sum_{1\leq l\leq r}\lambda'_{l}=k'\right\}.
\]
When $k'=1$, we denote by $\lambda[i]$ the element $(\lambda_{1},\ldots, \lambda_{i}-1,\ldots,\lambda_{r})$ of $\{\lambda - 1\}$, for each $1\leq i\leq r$ such that $\lambda_{i}\geq1$.

When $k'=2$, we denote by $\lambda[i,j]$ the element $(\lambda_{1},\ldots, \lambda_{i}-1,\ldots, \lambda_{j}-1,\ldots,\lambda_{r})$ of $\{\lambda - 2\}$, for each $1\leq i < j\leq r$ such that $\lambda_{i}\geq1$ and $\lambda_{j}\geq1$. We similarly denote by $\lambda[i,i]$ the element $(\lambda_{1},\ldots, \lambda_{i}-2,\ldots,\lambda_{r})$ of $\{\lambda - 2\}$, for each $1\leq i\leq r$ such that $\lambda_{i}\geq2$.

For partitions $\lambda = (\lambda_{1},\ldots,\lambda_{r})\vdash k$ and $\lambda' = (\lambda'_{1},\ldots,\lambda'_{s})\vdash k'$, we write $\lambda' \prec \lambda$ if for each $1\leq i\leq s$ we have $\lambda'_{i} \leq \lambda_{j_{i}}$ for $1\leq j_{1} < \cdots < j_{s} \leq r$. This is a partial ordering on tuples of positive integers.
\end{notation}

Furthermore, we deal with partitions where some blocks may be null as follows.
\begin{notation}\label{nota:partition_0_block}
Let $\upsilon = (\upsilon_{1},\ldots,\upsilon_{r})\vdash h$ with $h\geq0$ be a partition such that $0\leq \upsilon_{l}\leq h$ for all $1\leq l\leq r$. We denote by $\overline{\upsilon}$ the partition of $h$ obtained from $\upsilon$ by removing the $0$-blocks. Then, we always identify $\fL^{\star}_{(\upsilon,\ell)}$ with $\fL^{\star}_{(\overline{\upsilon},\ell)}$ since these functors are obviously isomorphic.
\end{notation}

\paragraph*{Twisted functors.}
Let us now consider a homological representation functor $\fL^{\star}_{(\lambda,\ell)} = \fL_{(\lambda,\ell)}$ that is \emph{twisted}, i.e.~where $\fL_{(\lambda,\ell)}\neq \fLu_{(\lambda,\ell)}$ has a category of twisted modules ${\bZ[Q_{(\lambda,\ell)}]}\lmod^{\tw}$ as its target; see Definition~\ref{def:cat_twisted_modules}. (Recall that we sometimes have $\fL_{(\lambda,\ell)}= \fLu_{(\lambda,\ell)}$, e.g.~if $\ell\leq 2$; see Lemma~\ref{lem:Qu_ell=2_useless}.)
The problem arising in this setting is that the category ${\bZ[Q_{(\lambda,\ell)}]}\lmod^{\tw}$ is not abelian (because it lacks a null object), while this is a necessary condition of the categorical framework to introduce the exact sequences of \S\ref{sss:translation_background} and later to define polynomiality in \S\ref{s:notions_polynomiality}.
However, this subtlety does not impact the core ideas we deal with, and we solve this minor issue by adopting the following convention.
\begin{convention}\label{convention:twisted_polynomiality}
Throughout \S\ref{s:SES_homol_rep-func}, when we consider a \emph{twisted} homological representation functor $\fL_{(\lambda,\ell)}$, we always postcompose it by the forgetful functor ${\bZ[Q_{(\lambda,\ell)}]}\lmod^{\tw} \to {\bZ}\lmod$ as in \eqref{eq:twisted_to_genuine_representation_functor}.
A fortiori, the target category of $\fL^{\star}_{(\lambda,\ell)}$ is either ${\bZ}\lmod$ if $\fL^{\star}_{(\lambda,\ell)}$ is \emph{twisted}, and it is ${\bZ[Q_{(\lambda,\ell)}]}\lmod$ otherwise.
Following Notation~\ref{notation:generic_notation_star}, we denote by ${\bZ[Q^{\star}_{(\lambda,\ell)}]}\lmod^{\bullet}$ the target category of $\fL^{\star}_{(\lambda,\ell)}$ under this convention.
\end{convention}

\begin{rmk}\label{rmk:impact_ell}
For the sake of simplicity, we do not fully detail the target categories of the functors of the short exact sequences of Theorems~\ref{thm:key_SES_classical_braids}, \ref{thm:key_SES_surface_braid_groups} and \ref{thm:key_SES_mapping_class_groups}. We however record here that this target is of the form $\bZ[Q_{(\lambda,\ell)}]\lmod$ if $\ell\leq 2$ (see Lemma~\ref{lem:Qu_ell=2_useless}), $\bZ[Q^{\unt}_{(\lambda,\ell)}]\lmod)$ if $\ell\geq 3$ and $\star = u$, and $\bZ\lmod$ if $\ell\geq 3$ and $\star$ is the blank space.
\end{rmk}

\paragraph*{Change of rings operations.}
Finally, we explain some key manipulations of the transformation groups associated to homological representation functors. Let $R$ be an associative unital ring.
For a category $\cC$ and a ring homomorphism $f\colon\bZ[Q]\to R$, the \emph{change of rings operation} on a functor $F\colon \cC\to {\bZ[Q]}\lmod^{*}$ consists in composing with the \emph{induced module functor} $f_{!} \colon {\bZ[Q]}\lmod^{*}\to R\lmod^{*}$, also known as the tensor product functor $R\otimes_{f}-$, where $*$ either stands for the blank space or $*=\tw$.
A key use of the change of rings operations is the following natural modification of the ground rings of homological representation functors with respect to partitions.

We consider partitions $\lambda = (\lambda_{1},\ldots,\lambda_{r})\vdash k$ and $\lambda' = (\lambda_{1}',\ldots,\lambda_{r}')\vdash k'$ such that $\lambda' \prec \lambda$ (see Notation~\ref{not:sets_partitions}).
In the notation of \S\ref{sss:framework_homological_representation_functor}, for each $1\leq i \leq r$ such that $\lambda'_{i} < \lambda_{i}$, there is an evident analogue of the short exact sequence \eqref{eq:Birman_Fadell_Neuwirth_SES} with $G_{\lambda'_{i},n}$ as quotient and $G_{(\lambda_{i}-\lambda'_{i},\lambda'_{i}),n}$ as the middle term. The section $s_{(\lambda_{i}-\lambda'_{i},\lambda'_{i}),n}$ of this short exact sequence provides an injection $G_{\lambda'_{i},n}\hookrightarrow G_{(\lambda_{i}-\lambda'_{i},\lambda'_{i}),n}$. Composing this with the canonical injection $G_{(\lambda_{i}-\lambda'_{i},\lambda'_{i}),n}\hookrightarrow G_{\lambda_{i},n}$, we obtain an injection $G_{\lambda'_{i},n} \hookrightarrow G_{\lambda_{i},n}$. Applying this procedure for each block, we obtain a canonical injection $G_{\lambda',n}\hookrightarrow G_{\lambda,n}$.
Now, for some fixed $\ell\geq1$, we consider the transformation groups $Q^{\star}_{(\lambda,\ell)}$ and $Q^{\star}_{(\lambda',\ell)}$ associated to homological representation functors $\fL^{\star}_{(\lambda,\ell)}$ and $\fL^{\star}_{(\lambda',\ell)}$ respectively.

\begin{lem}\label{lem:canonical_map_transformation_groups}
The canonical injection $G_{\lambda',n}\hookrightarrow G_{\lambda,n}$ induces a well-defined group homomorphism $\cQ_{(\lambda'\to\lambda,\ell)}\colon Q^{\star}_{(\lambda',\ell)} \to Q^{\star}_{(\lambda,\ell)}$. This homomorphism also makes the following square commute (where we temporarily shorten the notation $\B_{\lambda_{1},\ldots,\lambda_{r}}(\cS_{n})$ to $\B_{\lambda}(\cS_{n})$), where the top horizontal arrow is the stabilisation map that adds trivial strands to a braid:
\begin{equation}
\label{eq:phi-and-stabilisation}
\begin{tikzcd}
\B_{\lambda'}(\cS_{n}) \ar[rr] \ar[dd,two heads,"\phi_{(\lambda',\ell)}",swap] && \B_{\lambda}(\cS_{n}) \ar[dd,two heads,"\phi_{(\lambda,\ell)}"] \\
\\
Q_{(\lambda',\ell)} \ar[rr,"{\cQ_{(\lambda'\to\lambda,\ell)}}"] && Q_{(\lambda,\ell)}.
\end{tikzcd}
\end{equation}
\end{lem}
\begin{proof}
The canonical injection $G_{\lambda',n}\hookrightarrow G_{\lambda,n}$ induces the following commutative triangle:
\[\begin{tikzcd}
G_{\lambda',n}/\LCS_{\ell}\ar[d] \ar[r, two heads] &  G_{n}/\LCS_{\ell}\\
G_{\lambda,n}/\LCS_{\ell}  \ar[ur, two heads].
\end{tikzcd}\]
Taking kernels and the colimit as $n\to\infty$, we uniquely define the morphism $\cQ_{(\lambda'\to\lambda,\ell)}$.

The stabilisation map $\B_{\lambda'}(\cS_{n}) \to \B_{\lambda}(\cS_{n})$ along the top of \eqref{eq:phi-and-stabilisation} is the restriction (along the left-hand map of \eqref{eq:Birman_Fadell_Neuwirth_SES}) of the canonical injection $G_{\lambda',n}\hookrightarrow G_{\lambda,n}$. Instead of quotienting by $\LCS_\ell$ and then passing to kernels, we may equivalently (by the universal property of kernels) pass to the kernels first (to obtain the stabilisation map at the top of the diagram) and then pass to the quotients (to obtain $\cQ_{(\lambda'\to\lambda,\ell)}$ at the bottom of the diagram).
\end{proof}

In \S\ref{ss:SES_surface_braid_groups}--\S\ref{ss:SES_mcg}, we will apply a change of rings operation using morphisms of the type $\cQ_{(\lambda'\to\lambda,\ell)}$ and use the following property:

\begin{lem}\label{lem:change_of_rings_operations}
The change of rings operation $(\bZ[\cQ_{(\lambda'\to\lambda,\ell)}])_{!} \colon \bZ[Q^{\star}_{(\lambda',\ell)}]\lmod^{*} \to \bZ[Q^{\star}_{(\lambda,\ell)}]\lmod^{*}$ gives $\fL^{\star}_{(\lambda',\ell)}$ the same ground ring as $\fL^{\star}_{(\lambda,\ell)}$. In the case when $* = \tw$, it also gives $\fL^{\star}_{(\lambda',\ell)}$ the same action of each group $G_{n}$ on $\bZ[Q^{\star}_{(\lambda',\ell)}]$ as $\fL^{\star}_{(\lambda,\ell)}$.
\end{lem}

Hence, the change of rings operation $(\bZ[\cQ_{(\lambda'\to\lambda,\ell)}])_{!}$ allows us to canonically switch the module structure of $\fL^{\star}_{(\lambda',\ell)}$ from $\bZ[Q^{\star}_{(\lambda',\ell)}]$ to $\bZ[Q^{\star}_{(\lambda,\ell)}]$, as well as the potential twisted structure, i.e.~actions of the groups $G_{n}$ on these modules. In particular, we use this type of operation to identify $\fL^{\star}_{(\lambda',\ell)}$ as a summand of the difference functor $\delta_{\one}\fL^{\star}_{(\lambda,\ell)}$ in \S\ref{ss:SES_surface_braid_groups}--\S\ref{ss:SES_mcg}.
This change of ground ring map is the identity in many situations and, regardless, it does not impact the key underlying structures of $\fL^{\star}_{(\lambda',\ell)}$ for our work.
Because these subtleties are minor observations and do not affect the key points of the reasoning, we choose to use the following conventions to simplify the notation:

\begin{convention}\label{conv:transformation_group_summand}
Throughout \S\ref{ss:SES_surface_braid_groups}--\S\ref{ss:SES_mcg}, some change of ground ring operations $(\bZ[\cQ_{(\lambda'\to\lambda,\ell)}])_{!}$ must sometimes be applied in order to properly identify the functor $\fL^{\star}_{(\lambda',\ell)}$ as a summand of $\delta_{\one}\fL^{\star}_{(\lambda,\ell)}$. We note here that this change of rings operation is \emph{non-zero}, i.e.~it does not factor through the zero module. However, we generally keep the notation $\fL^{\star}_{(\lambda',\ell)}$ for this modified functor for the sake of simplicity and to avoid overloading the notation, whenever these change of rings operations are clear from the context.
\end{convention}

The following lemma will be the key point justifying that a change of rings operation does not impact the results of \S\ref{ss:SES_surface_braid_groups}--\S\ref{ss:SES_mcg}; see Corollary~\ref{coro:change_of_ring}.

\begin{lem}\label{lem:change_of_rings_SES}
Let $Q$ be a group, $R$ a ring and $f\colon\bZ[Q]\to R$ a ring homomorphism. We consider a functor $F\colon \langle \cG,\cM \rangle\to {\bZ[Q]}\lmod$ such that $\kappa_{\one}F = 0$ and so that $\delta_{\one}F(\tn)$ is a free $\bZ[Q]$-module for each $\tn\in \obj (\cM)$. Then the exact sequence \eqref{eq:ESCaract} applied to the functor $f_{!}F$ induces the short exact sequence:
\begin{equation}\label{eq:SES_f!F}
0 \longrightarrow f_{!}F \longrightarrow \tau_{\one}(f_{!}F) \longrightarrow \delta_{\one}(f_{!}F)\longrightarrow 0.
\end{equation}

The same statement also holds if $F$ takes values in $\bZ[Q]\lmod^{\tw}$. In this case, the change of rings functor $f_{!}$ takes place at the level of twisted module categories $\bZ[Q]\lmod^{\tw} \to R\lmod^{\tw}$ but, as explained in Convention~\ref{convention:twisted_polynomiality}, all of the statements take place in the category of functors into $\bZ\lmod$, via post-composing with the forgetful functor.
\end{lem}
\begin{proof}
We first note that the change of rings functor $f_{!}$ is right-exact and clearly commutes with the translation functor $\tau_{\one}$, so $\tau_{\one}(f_{!}F)\cong f_{!}(\tau_{\one}F)$ and $\delta_{\one}(f_{!}F)\cong f_{!}(\delta_{\one}F)$. Now, for each $\tn\in \obj (\cM)$, since $\kappa_{\one}F(\tn) = 0$ and the $\bZ[Q]$-module $\delta_{\one}F(\tn)$ is projective, the exact sequence \eqref{eq:ESCaract} applied to the functor $F$ induces a \emph{split} short exact sequence $0 \to F(\tn) \to \tau_{\one}F(\tn) \to \delta_{\one}F(\tn)\to 0$. The result then follows from the fact that the functor $f_{!}$ turns split short exact sequences of $\bZ[Q]$-modules into split short exact sequences of $R$-modules.
\end{proof}

\begin{coro}\label{coro:change_of_ring}
The results of Theorems~\ref{thm:key_SES_classical_braids}, \ref{thm:key_SES_surface_braid_groups}, \ref{thm:key_SES_mapping_class_groups} and \ref{thm:SES_MCG_alternatives} hold after any change of rings operation.
\end{coro}
\begin{proof}
Each one of the short exact sequences \eqref{eq:key_SES_classical_braids_classical}, \eqref{eq:key_SES_braid_surface} and \eqref{eq:key_SES_braid_coro}, the isomorphisms \eqref{eq:key_SES_MCG_orientable} and \eqref{eq:key_SES_MCG_non_orientable} (and their alternatives of Theorem~\ref{thm:SES_MCG_alternatives}) follow from the exact sequence \eqref{eq:ESCaract} applied to some homological representation functor $F$. In each case, it follows from the results on the module structures from \S\ref{ss:free-bases}--\S\ref{ss:dual-bases} that $F(\tn)$, $\tau_{\one}F(\tn)$ and $\delta_{\one}F(\tn)$ are free $\bZ[Q]$-modules for all $\tn\in \obj (\cM)$, and that $\kappa_{\one}F = 0$, so we may apply Lemma~\ref{lem:change_of_rings_SES}.
\end{proof}

\subsection{For surface braid group functors}\label{ss:SES_surface_braid_groups}

We prove in \S\ref{ss:SES_classical_braids} and \S\ref{ss:SES_surface_braids} our results on short exact sequences of Theorem~\ref{thm:ses} for surface braid group functors. The proofs of these results require certain diagrammatic arguments, which we explain first in \S\ref{sss:preliminary_SES_surface_braid_groups}.
For the remainder of \S\ref{ss:SES_surface_braid_groups}, we consider any one of the homological representation functors of \eqref{def:Lawrence_bigelow_further_ell} and \eqref{eq:ell_nilpotent_hom_rep_braid_surface} with the classical (i.e.~non-vertical) setting. Following Notation~\ref{notation:generic_notation_star}, we generically denote it by $\fL^{\star}_{(\lambda,\ell)}(S)$ where $\star$ either stands for the blank space or $\star=\unt$, $S\in\{\bD,\Sigma_{g,1},\N_{h,1}\}$ with $g\geq1$ and $h\geq1$, and the associated transformation group is denoted by $Q^{\star}_{(\lambda,\ell)}(S)$.

\subsubsection{First properties and diagrammatic arguments}\label{sss:preliminary_SES_surface_braid_groups}

We recall that we have introduced model graphs $\bI_{n}$, $\bW^{\Sigma}_{g}$ and $\bW_h^{\N}$ in Notation~\ref{nota:tail_and_wedge}, modelled by Figures~\ref{fig:model-or-braids} and \ref{fig:model-nor-braids}. Let us abbreviate by writing $\bW^{S}=\bW^{\Sigma}_{g}$ if $S=\Sigma_{g,1}$ and $\bW^{S}=\bW_h^{\N}$ if $S=\N_{h,1}$.

\begin{rmk}
For convenience, the diagrammatic arguments below illustrated in Figures~\ref{fig:cloud}--\ref{fig:cloud-splitting} are drawn only with the case $S = \Sigma_{g,1}$. Indeed, only the planar parts of these figures are relevant to the arguments.
\end{rmk}

We follow the notation of \S\ref{ss:free-bases} and consider, for each $\tn\in \obj(\Beta^{S})$, the surface braid group $\B_{n}(S)$-representation $\fL^{\star}_{(\lambda,\ell)}(S)(\tn)= H_{k}^{\BM}(C_{\lambda}(\bD_{n}\natural S);\bZ[Q^{\star}_{(\lambda,\ell)}(S)])$ where $\bZ[Q^{\star}_{(\lambda,\ell)}(S)]$ is a rank-one local system explained in the general construction of \S\ref{ss:general_construction}. Proposition~\ref{prop:module_structure_BM_homology} describes a free basis for the underlying $\bZ[Q^{\star}_{(\lambda,\ell)}(S)]$-module of $\fL^{\star}_{(\lambda,\ell)}(S)(\tn)$ indexed by labellings of the embedded graph $\bI_{n} \vee \bW^{S} \subset S$ by words in the blocks of $\lambda$.
We use the following slight simplification of Notation~\ref{nota:tail_and_wedge} for \S\ref{ss:SES_surface_braid_groups}.

\begin{notation}\label{nota:genus_basis}
Let $g_{S}$ denote the integer $2g$ if $S = \Sigma_{g,1}$ and $h$ if $S = \N_{h,1}$.
We will generically write the basis elements \eqref{eq:generator_surface_braid_hom_rep_orientable} and \eqref{eq:generator_surface_braid_hom_rep_non-orientable} as $(w_{1},\ldots,w_{g_{S}+n-1})$.
Also, choosing an ordering of the edges of the embedded graph $\bI_{n} \vee \bW^{S} \subset \bD_{n}\natural S$, we write $(w_{1},\ldots,w_{g_{S}+n-1})\vdash \lambda$ to indicate that the basis element $(w_{1},\ldots,w_{g_{S}+n-1})$ is associated to the ordered partition $\lambda$.
\end{notation}

The representation $\tau_{\one}\fL^{\star}_{(\lambda,\ell)}(S)(\tn)=H_{k}^{\BM}(C_{\lambda}(\bD_{1+n}\natural S);\bZ[Q^{\star}_{(\lambda,\ell)}(S)])$ has a very similar description as a free module. More precisely, the only difference with respect to $\fL^{\star}_{(\lambda,\ell)}(S)(\tn)$ is that there is one extra edge of the embedded graph $\bI_{1+n} \vee \bW^{S} \subset \bD_{1+n}\natural S$.
We write this as $(w_{0},w_{1},\ldots,w_{g_{S}+n-1}) \vdash \lambda$, where $w_{0}$ is the label of the extra edge. Now, we recall from Lemma~\ref{lem:extension_Quillen_source_homological_rep_functors} that the image of the canonical morphism $[\one,\id_{\one\natural\tn}]\in \langle \Beta,\Beta^{S} \rangle(\tn,\one\natural\tn)$ under $\fL^{\star}_{(\lambda,\ell)}(S)$ is the map $H_{k}^{\BM}(C_{\lambda}(\bD_{n}\natural S);\bZ[Q^{\star}_{(\lambda,\ell)}(S)]) \to H_{k}^{\BM}(C_{\lambda}(\bD_{1+n}\natural S);\bZ[Q^{\star}_{(\lambda,\ell)}(S)])$ induced by the evident inclusion of configuration spaces $C_{\lambda}(\bD_{n}\natural S)\hookrightarrow C_{\lambda}(\bD_{1}\natural(\bD_{n}\natural S))$ coming from the boundary connected sum with the left-most copy of $\bD_{1}$. In terms of the above free generating sets coming from Proposition~\ref{prop:module_structure_BM_homology}, the map $\fL^{\star}_{(\lambda,\ell)}(S)([\one,\id_{\one\natural\tn}])$ is thus the injection defined by
\begin{equation}
\label{eq:image_injection_hom_rep_functor_surface}
(w_{1},\ldots,w_{g_{S}+n-1}) \longmapsto (\varnothing,w_{1},\ldots,w_{g_{S}+n-1}).
\end{equation}
The cokernel $\delta_{\one}\fL^{\star}_{(\lambda,\ell)}(S)(\tn)$ of $i_{\one}(\fL^{\star}_{(\lambda,\ell)}(S))_{\tn}$ may therefore be described as the free $\bZ[Q^{\star}_{(\lambda,\ell)}(S)]$-module generated by all edge-labellings $(w_{0},w_{1},\ldots,w_{g_{S}+n-1}) \vdash \lambda$ of $\bI_{1+n} \vee \bW^{S}$ such that $w_{0}$ is not the empty word. On the other hand, we have $\kappa_{\one}\fL^{\star}_{(\lambda,\ell)}(S)(\tn)=0$.

Furthermore, if $k\geq2$, the direct sum $\bigoplus_{i=1}^r \tau_{\one}\fL^{\star}_{(\lambda[i],\ell)}(S)(\tn)$ has a basis indexed by pairs $(i,(w_{0},w_{1},\ldots,w_{g_{S}+n-1}))$, where $1\leq i\leq r$ and $(w_{0},w_{1},\ldots,w_{g_{S}+n-1}) \vdash \lambda[i]$. There is an evident bijection between the basis for $\delta_{\one}\fL^{\star}_{(\lambda,\ell)}(S)(\tn)$ and this basis for $\bigoplus_{i=1}^r \tau_{\one}\fL^{\star}_{(\lambda[i],\ell)}(S)(\tn)$ given by
\begin{equation}
\label{eq:definition-of-pkln}
(w_{0},w_{1},\ldots,w_{g_{S}+n-1}) \longmapsto (i,(w'_{0},w_{1},\ldots,w_{g_{S}+n-1})),
\end{equation}
where $w_{0} = iw'_{0}$ with $w_{0}$ and $w'_{0}$ words in the monoid $\tM(\{1,\ldots,r\})$ (see Notation~\ref{not:set_and_lengths_words}) while $i$ is seen as a letter in the alphabet $\{1,\ldots,r\}$. Extending by linearity, this bijection determines an isomorphism of free $\bZ[Q^{\star}_{(\lambda,\ell)}(S)]$-modules
\begin{equation}
\label{eq:isom-of-modules_braids}
\delta_{\one}\fL^{\star}_{(\lambda,\ell)}(S)(\tn) \overset{\cong}{\longrightarrow} \bigoplus_{i=1}^r \tau_{\one}\fL^{\star}_{(\lambda[i],\ell)}(S)(\tn).
\end{equation}

\paragraph*{The cloud lemma.}
In order to construct the fundamental short exact sequences of Theorem~\ref{thm:ses} for surface braid groups, we will need to show that the isomorphism of modules \eqref{eq:isom-of-modules_braids} is an isomorphism of representations. (In fact, we will show that it is moreover an isomorphism of functors as $\tn$ varies.) Assuming that $\tn\geq \two$, the key ingredient to prove this is Lemma~\ref{lem:cloud} below, which is pictorially summarised in Figure~\ref{fig:cloud}. To give the precise statement of the lemma, we first have to describe precisely what this figure is illustrating.

\begin{defn}
\label{defn:nu}
Given a word $w$ in the monoid $\tM(\{1,\ldots,r\})$, write $\nu(w) = (\nu(w)_1 , \ldots , \nu(w)_r)$, where $\nu(w)_i$ is the number of copies of the letter $i$ in $w$.
\end{defn}

\begin{defn}
\label{defn:varsigma}
Suppose we are given a properly-embedded, oriented open arc $\alpha$ in the surface $\bD_{1+n}\natural S$ equipped with a path from some point on $\alpha$ to the basepoint of $\bD_{1+n}\natural S$. Suppose further that we are also given a word $w$ as in Definition~\ref{defn:nu}. We then define $\varsigma(\alpha,w)$ to be the Borel-Moore cycle on $C_{\nu(w)}(\bD_{1+n}\natural S)$ given by the open singular simplex consisting of all configurations that lie on $\alpha$ and whose labels spell the word $w$ following the orientation of $\alpha$. As explained in \S\ref{ss:tethers}, in order for this to specify a Borel-Moore cycle for twisted homology, one needs to specify also a tether, namely a path from the base configuration to a configuration lying on $\alpha$. The base configuration consists of $\lvert w \rvert$ points contained in the boundary of $\bD_{1+n}\natural S$ close to the basepoint of the surface, and we specify the tether to be the path of configurations that follows $\lvert w \rvert$ paths parallel to the given path from $\alpha$ to the basepoint.
\end{defn}

The left-hand side of Figure~\ref{fig:cloud} depicts a Borel-Moore cycle on the partitioned configuration space $C_{\lambda}(\bD_{1+n}\natural S)$, representing an element of $\tau_{\one}\fL^{\star}_{(\lambda,\ell)}(S)(\tn) = H_{k}^{\BM}(C_{\lambda}(\bD_{1+n}\natural S);\bZ[Q^{\star}_{(\lambda,\ell)}(S)])$ and thus determining an element of the quotient $\delta_{\one}\fL^{\star}_{(\lambda,\ell)}(S)(\tn)$ of $\tau_{\one}\fL^{\star}_{(\lambda,\ell)}(S)(\tn)$. This cycle is assumed to be of a specific form, depending on the following two choices:
\begin{itemizeb}
\item[(1)] A non-empty word $w_{0}$ in the monoid $\tM(\{1,\ldots,r\})$. This determines a partition $\nu(w_{0})$ (see Definition~\ref{defn:nu}) and hence also $\lambda - \nu(w_{0})$ given by componentwise subtraction: for $1\leq i\leq r$ its $i$th term is $\lambda_{i} - \nu(w_{0})_{i}$. Note that this is non-negative since $\lambda_{i}$ is the number of copies of the letter $i$ in the concatenation $w_{0} w_{1} \cdots w_{g_S + n - 1}$.
\item[(2)] A cycle $\xi$ on $C_{\lambda - \nu(w_{0})}(\bD_{1+n}\natural S)$ supported in the blue shaded region (the ``\emph{cloud}'').
\end{itemizeb}
Denote by $\alpha$ the open green arc depicted on the left-hand side of Figure~\ref{fig:cloud}, equipped with the grey path to the basepoint (on the boundary of the surface). By Definition~\ref{defn:varsigma}, together with the word $w_{0}$, this determines a cycle $\varsigma(\alpha,w_{0})$ on $C_{\nu(w_{0})}(\bD_{1+n}\natural S)$. The product $\xi \times \varsigma(\alpha,w_{0})$ is then a cycle on $C_{\lambda}(\bD_{1+n}\natural S)$; this is the cycle that we consider on the left-hand side of Figure~\ref{fig:cloud}.

The right-hand side has a similar description, where we decompose the non-empty word $w_{0}$ as $iw'_{0}$. The ``$i$'' component simply says that the element lies in the $i$th summand on the right-hand side of \eqref{eq:isom-of-modules_braids}. The second, pictorial component depicts the Borel-Moore cycle $\xi \times \varsigma(\alpha,w'_{0})$ on $C_{\lambda[i]}(\bD_{1+n}\natural S)$, representing an element of $\tau_{\one}\fL^{\star}_{(\lambda[i],\ell)}(S)(\tn)$.

\begin{figure}[ht]
    \centering
    \includegraphics[scale=0.5]{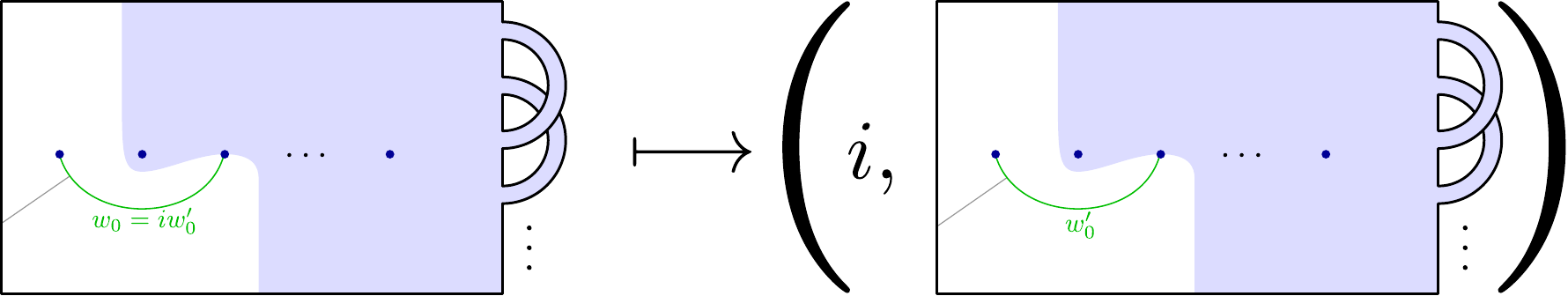}
    \caption{The Borel-Moore cycles considered in Lemma~\ref{lem:cloud}.}
    \label{fig:cloud}
\end{figure}

With these descriptions and notation, we may now give the precise statement of the lemma.

\begin{lem}[``\emph{Cloud lemma}''.]
\label{lem:cloud}
For any choices of $w_{0} = iw'_{0}$ and $\xi$ as above, the module isomorphism \eqref{eq:isom-of-modules_braids} sends the element $[\xi \times \varsigma(\alpha,w_{0})]$ to the element $(i,[\xi \times \varsigma(\alpha,w'_{0})])$.
\end{lem}

\begin{proof}
Let us write $\bw = (w_{0},w_{1},\ldots,w_{g_{S}+n-1}) \vdash \lambda$ and define an operation $(-)^{\rfl}$ on such tuples of words by setting $\bw^{\rfl} = (w^{\rfl}_{0},w_{1},\ldots,w_{g_{S}+n-1})$, where $w_{0} = jw^{\rfl}_{0}$; in other words the operation $(-)^{\rfl}$ removes the first letter of the first word of $\bw$. Note that $\bw^{\rfl} \vdash \lambda[j]$.

For the sake of clarity, we prefer in this proof to denote by $e_{\bw}$ (rather than $\bw$, see Notation~\ref{notation:basis_notation}) the standard basis element corresponding to the tuple $\bw$, depicted in Figure~\ref{fig:graphs-configurations}. Then we denote by $e'_{\bw}$ the corresponding dual basis element depicted in Figure~\ref{fig:models-dual}. As explained in \S\ref{ss:tethers}, to specify these fully, we must choose a tether for each $e'_{\bw}$: we choose these tethers to be paths to the base configuration (in the boundary of the surface) given by sending all of the configuration points along parallel horizontal trajectories, as depicted in Figure~\ref{fig:cloud-lemma}. Using this notation, by definition, the isomorphism \eqref{eq:isom-of-modules_braids} takes $e_{\bw}$ to the element $e_{\bw^{\rfl}}$ in the $j$th summand of the right-hand side.

Let us first decompose the element represented by the left-hand side of Figure~\ref{fig:cloud} as
\begin{equation}
\label{eq:cloud-lemma-LHS}
\mathrm{LHS} := [\xi \times \varsigma(\alpha,w_{0})] = \sum_{\bw \vdash \lambda} \varrho_{\bw} . e_{\bw} = \sum_{j=1}^r \sum_{\substack{\bw \vdash \lambda \\ w_{0} = jw^{\rfl}_{0}}} \varrho_{\bw} . e_{\bw} ,
\end{equation}
where $\varrho_{\bw} = \langle \mathrm{LHS} , e'_{\bw} \rangle \in \bZ[Q^{\star}_{(\lambda,\ell)}(S)]$ is the value of the intersection pairing \eqref{eq:pairing} evaluated on the left-hand side of Figure~\ref{fig:cloud} and the dual basis element $e'_{\bw}$. This is illustrated on the left-hand side of Figure~\ref{fig:cloud-lemma}. From that figure, it is clear that $\varrho_{\bw} = 0$ unless $j=i$, so we may remove the outer sum and set $j=i$ in the formula \eqref{eq:cloud-lemma-LHS}. The image of the element \eqref{eq:cloud-lemma-LHS} under the map \eqref{eq:isom-of-modules_braids} may therefore be written as
\begin{equation}
\label{eq:cloud-lemma-LHS-image}
\sum_{\substack{\bw \vdash \lambda \\ w_{0} = iw^{\rfl}_{0}}} \varrho_{\bw} . (i, e_{\bw^{\rfl}}) .
\end{equation}
On the other hand, the element represented by the right-hand side of Figure~\ref{fig:cloud} decomposes as
\begin{equation}
\label{eq:cloud-lemma-RHS}
\mathrm{RHS} := (i,[\xi \times \varsigma(\alpha,w'_{0})]) = \sum_{\bv \vdash \lambda[i]} \mu_{\bv} . (i,e_{\bv}) ,
\end{equation}
where $\mu_{\bv} = \langle \mathrm{RHS} , e'_{\bv} \rangle  \in \bZ[Q^{\star}_{(\lambda,\ell)}(S)]$ is the value of the intersection pairing \eqref{eq:pairing} evaluated on the right-hand side of Figure~\ref{fig:cloud} and the dual basis element $e'_{\bv}$. This is illustrated on the right-hand side of Figure~\ref{fig:cloud-lemma}. There is clearly a bijection between the two indexing sets of the sums above given by sending $\bw$ to $\bv = \bw^{\rfl}$. Thus, in order to prove that $\eqref{eq:cloud-lemma-LHS-image} = \eqref{eq:cloud-lemma-RHS}$, as desired, it remains to show that we have an equality of coefficients; in other words, we must prove that
\begin{equation}
\label{eq:equality-of-coefficients}
\langle \mathrm{LHS} , e'_{\bw} \rangle = \varrho_{\bw} = \mu_{\bw^{\rfl}} = \langle \mathrm{RHS} , e'_{\bw^{\rfl}} \rangle
\end{equation}
for each $\bw \vdash \lambda$ such that $w_{0} = iw^{\rfl}_{0}$.

To prove the equality \eqref{eq:equality-of-coefficients}, we first recall, in the next few paragraphs below, some details about how the intersection pairings $\langle \mathrm{LHS} , e'_{\bw} \rangle$ and $\langle \mathrm{RHS} , e'_{\bw^{\rfl}} \rangle$ may be computed. This is well-established in the literature, but we include it in order to fix the notation and conventions that we will use in our calculations. For further details, see \cite[\S 2.1]{bigelow2001braid} or \cite[\S 4.3]{PSIN} (when the surface is a disc) or \cite[\S 7]{BlanchetPalmerShaukat} (for more general orientable surfaces).

\emph{\textbf{Calculating intersection pairings in general.}}
We first consider $\langle \mathrm{LHS} , e'_{\bw} \rangle$. We may assume without loss of generality that the Borel-Moore homology class denoted by $\mathrm{LHS}$ (the left-hand side of Figure~\ref{fig:cloud}) is represented, similarly to Definition~\ref{defn:varsigma}, by configuration spaces on a collection of pairwise disjoint, properly embedded, open arcs (each equipped with a tether, i.e.~a path to the basepoint of the surface), one of these being the arc depicted and the others being contained in the shaded ``cloud''. Indeed, this is due to the basis that we have described in \S\ref{s:free_generating_sets}. The dual basis element $e'_{\bw}$ is represented by the cycle given by the red vertical (and horizontal, in the handles) arcs on the left-hand side of Figure~\ref{fig:cloud-lemma} (also equipped with tethers). We assume that these intersect the arcs representing $\mathrm{LHS}$ transversely, in particular at finitely many points. This ensures that the cycles on the configuration space $C_{\lambda}(\bD_{1+n}\natural S)$ representing the homology classes $\mathrm{LHS}$ and $e'_{\bw}$ intersect at finitely many points. We shall generally denote these intersection points by $p$, and we emphasise that each $p$ is a \emph{configuration}, so it consists of $k$ points on $\bD_{1+n}\natural S$.

\begin{figure}[tb]
    \centering
    \includegraphics[scale=0.6]{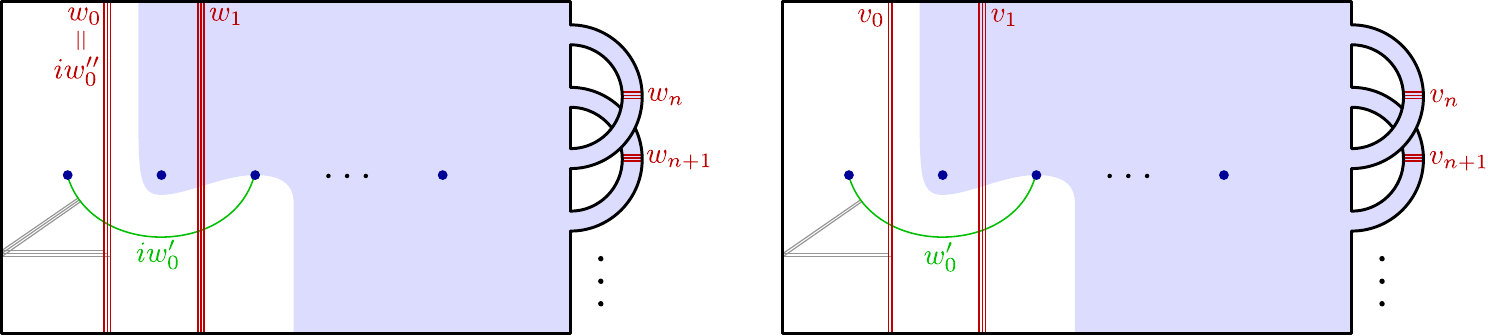}
    \caption{An illustration of the intersection pairings $\varrho_{\bw} = \langle \mathrm{LHS} , e'_{\bw} \rangle$ (left) and $\mu_{\bw^{\rfl}} = \langle \mathrm{RHS} , e'_{\bw^{\rfl}} \rangle$ (right). Here, $\mathrm{LHS}$ and $\mathrm{RHS}$ refer to the Borel-Moore homology classes depicted on the left-hand side and right-hand side of Figure~\ref{fig:cloud} respectively.}
    \label{fig:cloud-lemma}
\end{figure}

The value of the pairing $\langle \mathrm{LHS} , e'_{\bw} \rangle \in \bZ[Q^{\star}_{(\lambda,\ell)}(S)]$ is then a sum of terms $\epsilon_{p} \phi_{(\lambda,\ell)}(\gamma_{p})$ indexed by these (finitely many) intersection configurations $p$. Here $\epsilon_{p} \in \{\pm 1\}$ is a sign, $\gamma_{p}$ is a based loop in the configuration space $C_{\lambda}(\bD_{1+n}\natural S)$ and $\phi_{(\lambda,\ell)} \colon \pi_{1}(C_{\lambda}(\bD_{1+n}\natural S)) \to Q^{\star}_{(\lambda,\ell)}(S)$ is the homomorphism determining the local system in the definition of $\fL^{\star}_{(\lambda,\ell)}$; see \S\ref{sss:transformation_groups}.
The loop $\gamma_{p}$ is constructed as the concatenation of four paths of configurations
\begin{equation}
\label{eq:concatenation}
* \rightsquigarrow x \rightsquigarrow p \rightsquigarrow y \rightsquigarrow * ,
\end{equation}
where $*$ is the base configuration. The first path $* \rightsquigarrow x$ is a tether for the submanifold representing the homology class $\mathrm{LHS}$ (so $x$ is a point on this submanifold), the last path $y \rightsquigarrow *$ is the reverse of a tether for $e'_{\bw}$ (so $y$ is a point on $e'_{\bw}$), the second is any path in $\mathrm{LHS}$ from $x$ to the intersection configuration $p$ and the third is any path in $e'_{\bw}$ from $p$ to $y$.

The intersection pairing $\langle \mathrm{RHS} , e'_{\bw^{\rfl}} \rangle$ has an almost identical description, the only difference being that one vertical red arc (containing a single point in the $i$th block of $\lambda$) has been removed and the green arc labelled by $iw'_{0}$ is now labelled by $w'_{0}$, so its left-most point (in the $i$th block of $\lambda$) has also been removed. The tethers have also been modified correspondingly: one of the parallel grey paths from the basepoint to the green arc has been removed, as has the grey path to the (removed) left-most vertical red arc. This means, a priori, that we use the quotient map $\phi_{\lambda[i],\ell}$ rather than $\phi_{\lambda,\ell}$ and the pairing takes values in $\bZ[Q^{\star}_{(\lambda[i],\ell)}(S)]$ rather than $\bZ[Q^{\star}_{(\lambda,\ell)}(S)]$. However, by Convention~\ref{conv:transformation_group_summand}, we implicitly apply the ring homomorphism $\bZ[\cQ_{(\lambda[i]\to\lambda,\ell)}]$ in order to interpret this as an element of $\bZ[Q^{\star}_{(\lambda,\ell)}(S)]$. Hence in fact we are using the composition $\cQ_{(\lambda[i]\to\lambda,\ell)} \circ \phi_{\lambda[i],\ell}$ instead of $\phi_{\lambda,\ell}$ and the pairing takes values in $\bZ[Q^{\star}_{(\lambda,\ell)}(S)]$.

\emph{\textbf{Calculating the intersection pairings in our setting.}}
To compare the two elements of $\bZ[Q^{\star}_{(\lambda,\ell)}(S)]$ corresponding to the left-hand side and right-hand side of \eqref{eq:equality-of-coefficients}, first notice that there is a bijection of intersection points $\mathrm{RHS} \cap e'_{\bw^{\rfl}} \to \mathrm{LHS} \cap e'_{\bw}$ given by $p \mapsto \bar{p} = \{p_{0}\} \cup p$, where $p_{0} \in \bD_{1+n}\natural S$ is the unique intersection point between the left-most vertical red arc and the curved (green) arc on the left-hand side of Figure~\ref{fig:cloud-lemma}; see the zoomed-in Figure~\ref{fig:cloud-lemma-zoomed}. It therefore suffices to check that we have $\epsilon_{\bar{p}} = \epsilon_{p}$ and $\phi_{(\lambda,\ell)}(\gamma_{\bar{p}}) = \cQ_{(\lambda[i]\to\lambda,\ell)}(\phi_{(\lambda[i],\ell)}(\gamma_{p}))$, where the values with a subscript $\bar{p}$ are computed using the left-hand side of Figure~\ref{fig:cloud-lemma} and those with a subscript $p$ are computed using the right-hand side of Figure~\ref{fig:cloud-lemma}.

The loops $\gamma_{p}$ and $\gamma_{\bar{p}}$ are both constructed as concatenations of four paths of configurations of the form \eqref{eq:concatenation}. Since the configuration $\bar{p}$ is obtained from $p$ by adjoining one point $p_{0} \in \bD_{1+n}\natural S$, it follows that $\gamma_{\bar{p}}$ is obtained from $\gamma_{p}$ by adjoining one loop in $\bD_{1+n}\natural S$ passing through $p_{0}$. This loop is highlighted in Figure~\ref{fig:cloud-lemma-zoomed}. By the commutativity of the square \eqref{eq:phi-and-stabilisation} of Lemma~\ref{lem:canonical_map_transformation_groups} (with $n$ replaced by $n+1$ and setting $\lambda' := \lambda[i]$), the equation $\phi_{\lambda,\ell}(\gamma_{\bar{p}}) = \cQ_{(\lambda[i]\to\lambda,\ell)}(\phi_{\lambda[i],\ell}(\gamma_p))$ that we need to prove will follow if we prove that the stabilisation map along the top of \eqref{eq:phi-and-stabilisation} sends $\gamma_{p}$ to $\gamma_{\bar{p}}$. The stabilisation map simply adjoins a \emph{stationary} point in the boundary of the surface, so what we must show is that adjoining the loop highlighted in Figure~\ref{fig:cloud-lemma-zoomed} to $\gamma_{p}$ is the same, up to homotopy, as adjoining a stationary point in the boundary. In other words, writing $s(\gamma_{p})$ for the stabilisation of $\gamma_{p}$, we must show that the based loops $s(\gamma_{p})$ and $\gamma_{\bar{p}}$ are homotopic in $C_{\lambda}(\bD_{1+n}\natural S)$.

To see this, we will restrict attention to the small (light blue) shaded region on the left of Figure~\ref{fig:cloud-lemma-zoomed}, denoted by $D$. We first note that, at all times during the loops $s(\gamma_{p})$ and $\gamma_{\bar{p}}$, there are exactly $\lvert w_{0} \rvert$ configuration points in $D$ and the other $k - \lvert w_{0} \rvert$ configuration points remain disjoint from $D$. (The number $\lvert w_{0} \rvert$ arises because it is the number of intersection points between the curved (green) arc and the vertical (red) arcs on the right of $D$.) Moreover, the loops $s(\gamma_{p})$ and $\gamma_{\bar{p}}$ are identical outside of $D$, so if we write $s(\gamma_{p})|_{D}$ and $\gamma_{\bar{p}}|_{D}$ for the restrictions of $s(\gamma_{p})$ and $\gamma_{\bar{p}}$ to $\lvert w_{0} \rvert$-point subconfigurations supported on $D$, it will suffice to show that $s(\gamma_{p})|_{D}$ is homotopic to $\gamma_{\bar{p}}|_{D}$ in $C_{\lvert w_{0} \rvert}(D)$. During $\gamma_{\bar{p}}|_{D}$, the configuration points travel in parallel anticlockwise around the boundary of $D$ without `twisting' (i.e.~we may think of the configuration as lying on a small interval, which travels around the boundary of $D$ without rotating).
This is homotopic to the constant loop by shrinking the loop along which the configuration points travel within $D$ until it is constant (here it is important that there is no twisting). During $s(\gamma_{p})|_{D}$, one of the points remains stationary and the other $\lvert w_{0} \rvert - 1$ points travel in parallel anticlockwise around the boundary of $D$ without twisting. This is homotopic to the constant loop for exactly the same reason. Composing these two homotopies, we see that $s(\gamma_{p})|_{D} \simeq \gamma_{\bar{p}}|_{D}$ in $C_{\lvert w_{0} \rvert}(D)$. Extending this to the $\lambda$-partitioned configuration space by a constant homotopy outside of $D$ (i.e.~by simply not modifying the trajectories of the other $k - \lvert w_{0} \rvert$ configuration points), we deduce that $s(\gamma_{p})$ and $\gamma_{\bar{p}}$ are homotopic in $C_{\lambda}(\bD_{1+n}\natural S)$. This completes the proof that we have $\phi_{\lambda,\ell}(\gamma_{\bar{p}}) = \cQ_{(\lambda[i]\to\lambda,\ell)}(\phi_{\lambda[i],\ell}(\gamma_p))$.

\begin{figure}[htb]
    \centering
    \includegraphics[scale=1]{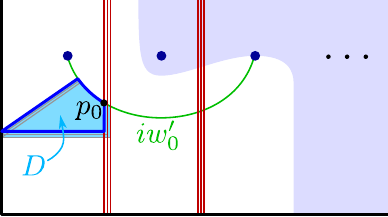}
    \caption{A zoomed-in portion of the left-hand side of Figure~\ref{fig:cloud-lemma}. The loop of configurations $\gamma_{\bar{p}}$ is obtained from the loop of configurations $\gamma_{p}$ by adjoining the loop (in the surface) highlighted in blue.}
    \label{fig:cloud-lemma-zoomed}
\end{figure}

Finally, we recall that the sign $\epsilon_{p}$ is the product of the local signs of the intersections of arcs in the surface at each point of $p = \{p_{1},\ldots,p_{k}\}$ together with an additional sign recording the parity of the permutation of the base configuration induced by $\gamma_{p}$. The local sign of the intersection at $p_{0}$ is $+1$, so adjoining $p_{0}$ does not change the product of the local signs. In addition, as a consequence of the paragraph above, the permutation induced by $\gamma_{\bar{p}}$ is obtained from the permutation induced by $\gamma_{p}$ by adjoining a fixed point; in particular they both have the same parity. Thus $\epsilon_{\bar{p}} = \epsilon_{p}$.
\end{proof}

\paragraph*{A decomposition property.}
We also will need the following (more elementary) diagrammatic fact in \S\ref{ss:SES_classical_braids}--\S\ref{ss:SES_surface_braids}. This is an identity, depicted in Figure~\ref{fig:cloud-splitting}, taking place in the Borel-Moore homology group $\fL^{\star}_{(\lambda,\ell)}(S)(\tn) = H_{k}^{\BM}(C_{\lambda}(\bD_{n}\natural S);\bZ[Q^{\star}_{(\lambda,\ell)}(S)])$.

The left-hand side of the figure represents a Borel-Moore cycle on $C_{\lambda}(\bD_{n}\natural S)$ defined as follows. Choose a word $w$ in the monoid $\tM(\{1,\ldots,r\})$ and a cycle $\xi$ on $C_{\lambda - \nu(w)}(\bD_{n}\natural S)$ supported in the blue shaded region. (Recall from Definition~\ref{defn:nu} the notation $\nu(w)$.) If we denote by $\alpha$ the green arc on the left-hand side of Figure~\ref{fig:cloud-splitting}, then the Borel-Moore cycle is $\xi \times \varsigma(\alpha,w)$ in the notation of Definition~\ref{defn:varsigma}.

The right-hand side of the figure represents a sum of Borel-Moore cycles, taken over all possible decompositions $w = w_{1}w_{2}$ of the word $w$ as a concatenation of two words. The Borel-Moore cycle in this sum corresponding to $(w_{1},w_{2})$ is $\xi \times \varsigma(\alpha_1,w_1) \times \varsigma(\alpha_2,w_2)$, where $\alpha_1 , \alpha_2$ are the two green arcs depicted on the right-hand side of Figure~\ref{fig:cloud-splitting}.

\begin{figure}[tb]
    \centering
    \includegraphics[scale=0.5]{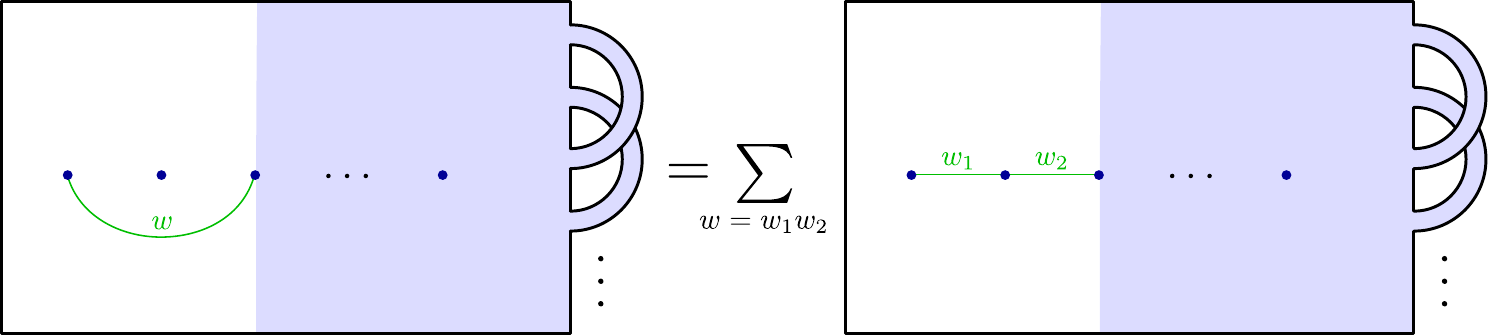}
    \caption{The identity of Lemma~\ref{lem:cloud-splitting} in the orientable case. The identity in the non-orientable case is the obvious analogue; only the left-hand side of each diagram, which is planar, is important.}
    \label{fig:cloud-splitting}
\end{figure}

\begin{lem}
\label{lem:cloud-splitting}
For any choices of a cycle $\xi$ and a word $w$, we have the relation
\[
[\xi \times \varsigma(\alpha,w)] \;\; = \sum_{w=w_{1}w_{2}}[\xi \times \varsigma(\alpha_1,w_1) \times \varsigma(\alpha_2,w_2)]
\]
in $\fL^{\star}_{(\lambda,\ell)}(S)(\tn) = H_{k}^{\BM}(C_{\lambda}(\bD_{n}\natural S);\bZ[Q^{\star}_{(\lambda,\ell)}(S)])$.
\end{lem}

\begin{proof}
We follow the notation of the proof of Lemma~\ref{lem:cloud}. The result follows immediately by verifying that each side of the equation evaluates to the same element of the ground ring when applying the intersection pairing $\langle - , e'_{\bv} \rangle$ with the dual basis element $e'_{\bv}$ for each $\bv = (v_{1},v_{2},v_{3}\ldots) \vdash \lambda$. (Details of how these intersection pairings are computed are explained in the proof of Lemma~\ref{lem:cloud} above.)
\end{proof}

\subsubsection{For classical braid groups}\label{ss:SES_classical_braids}

Our aim here is to prove Theorem~\ref{thm:ses} for each $(\lambda,\ell)$-Lawrence-Bigelow functor $\LB_{(\lambda,\ell)}$ and its untwisted version $\LBu_{(\lambda,\ell)}$ of \S\ref{sss:representations_classical_braid_groups}.
The arguments for this are exactly the same for both $\LB_{(\lambda,\ell)}$ and $\LBu_{(\lambda,\ell)}$. So, following \S\ref{sss:representations_classical_braid_groups} and \S\ref{sss:preliminary_SES_surface_braid_groups}, we henceforth speak of the functor $\fL^{\star}_{(\lambda,\ell)}(\bD)=\LB^{\star}_{(\lambda,\ell)}$ where $\star$ either stands for the blank space or $\star=\unt$.

From now on in \S\ref{ss:SES_classical_braids}, we assume that $k\geq2$. For any $\tn\in\obj(\Beta)$, we recall from the preliminary study of \S\ref{sss:preliminary_SES_surface_braid_groups} that $\kappa_{\one}\LB^{\star}_{(\lambda,\ell)}(\tn)=0$ and $\delta_{\one}\LB^{\star}_{(\lambda,\ell)}(\tn)$ is the free $\bZ[Q^{\star}_{(\lambda,\ell)}(\bD)]$-module with basis given by the tuples $(w_{0},w_{1},\ldots,w_{n-1})\vdash \lambda$ such that $\vert w_{0}\vert\geq 1$, which identifies as a $\bZ[Q^{\star}_{(\lambda,\ell)}(\bD)]$-module to the direct sum $\bigoplus_{1\leq j\leq r} \tau_{\one}\LB^{\star}_{(\lambda[j],\ell)}(\tn)$ via the isomorphism \eqref{eq:isom-of-modules_braids} (if $\tn\geq\one$; it is trivial otherwise). From now on, we denote this isomorphism by $(\mathfrak{p}_{(\lambda,\ell)})_{\tn}$, setting $(\mathfrak{p}_{(\lambda,\ell)})_{\zero}$ to be the null morphism. We now prove the main result of \S\ref{ss:SES_classical_braids}:

\begin{thm}\label{thm:key_SES_classical_braids}
For each $\lambda\vdash k\geq2$ and $\ell\geq1$, the exact sequence \eqref{eq:ESCaract} induces the short exact sequences:
\begin{equation}
\label{eq:key_SES_classical_braids_classical}
\begin{tikzcd}
0 \ar[r] & \LB^{\star}_{(\lambda,\ell)} \ar[r] & \tau_{\one}\LB^{\star}_{(\lambda,\ell)} \ar[r] & \underset{1\leq j\leq r}{\bigoplus}\tau_{\one}\LB^{\star}_{(\lambda[j],\ell)} \ar[r] & 0
\end{tikzcd}
\end{equation}
in $\Fct(\langle \Beta, \Beta\rangle ,{\bZ[Q^{\unt}_{(\lambda,\ell)}(S)]}\lmod)$ if $\star = u$, and in $\Fct(\langle \Beta, \Beta \rangle ,{\bZ[Q_{(\lambda,\ell)}(S)]}\lmod^{\bullet})$ if $\star$ is the blank space.
\end{thm}

\begin{proof}
As a consequence of the preliminary study above, we already know that the evaluation of \eqref{eq:key_SES_classical_braids_classical} at each object $\tn$ is a short exact sequence of $\bZ[Q^{\star}_{(\lambda,\ell)}(\bD)]$-modules. Hence the key point is to prove the compatibility with respect to the morphisms of $\langle \Beta, \Beta\rangle$. To achieve this, it suffices to show that the $\bZ[Q^{\star}_{(\lambda,\ell)}(\bD)]$-module isomorphisms $\{(\mathfrak{p}_{(\lambda,\ell)})_{\tn}\}_{\tn\in\obj(\Beta)}$ define an isomorphism $\mathfrak{p}_{(\lambda,\ell)}\colon\delta_{\one}\LB^{\star}_{(\lambda,\ell)}\overset{\sim}{\to}\bigoplus_{1\leq j\leq r}\tau_{\one}\LB^{\star}_{(\lambda[j],\ell)}$ in $\Fct(\langle \Beta, \Beta\rangle ,{\bZ[Q^{\star}_{(\lambda,\ell)}(\bD)]}\lmod^{\bullet})$.

As a first step, we prove that assembling these module isomorphisms defines an isomorphism in the category $\Fct(\Beta,{\bZ[Q^{\star}_{(\lambda,\ell)}(\bD)]}\lmod^{\bullet})$, in other words that each $(\mathfrak{p}_{(\lambda,\ell)})_{\tn}$ is a $\B_{n}$-module morphism. The braid group $\B_{n}$ being trivial for $n\in\{0,1\}$, we consider $\tn\geq\two$ and prove the commutation of $(\mathfrak{p}_{(\lambda,\ell)})_{\tn}$ with respect to the action of any Artin generator $\sigma_{i}$ of $\B_{n}$ with $1\leq i\leq n-1$.
Since $\tau_{\one}$ is the left-to-right translation-by-one operator, the morphisms $\tau_{\one}\LB^{\star}_{(\lambda[j],\ell)}(\sigma_{i})$ for all $1\leq j\leq r$ and $\delta_{\one}\LB^{\star}_{(\lambda,\ell)}(\sigma_{i})$ are defined by the action of the generator $\sigma_{i+1}$ on the Borel-Moore homology classes supported on the embedded graph $\bI_{1+n}$; see Figure~\ref{fig:model-or-braids} with $g=0$.

For $i=1$, we use Figure~\ref{fig:cloud} (with $g=0$) to illustrate the actions of $\tau_{\one}\LB^{\star}_{(\lambda[j],\ell)}(\sigma_{1})$ (the right-hand side of Figure~\ref{fig:cloud}) and of $\delta_{\one}\LB^{\star}_{(\lambda,\ell)}(\sigma_{1})$ (the left-hand side of Figure~\ref{fig:cloud}) on the basis elements $(w'_{0},\ldots,w_{n})$ and $(j w'_{0},\ldots,w_{n})$. Namely, by the choice of Convention~\ref{convention:braiding_braids}, the labelled green arc is the image of the left-most edge of the graph $\bI_{1+n}$, which we will denote by $(1,2)$ (using the left-to-right enumeration of its vertices), under the action of $\sigma_{2}$, which swaps the points $2$ and $3$ anticlockwise, while the images of the other edges are concentrated in the blue shaded region. It thus follows from Lemma~\ref{lem:cloud} and the definition of $(\mathfrak{p}_{(\lambda,\ell)})_{\tn}$ that $(\mathfrak{p}_{(\lambda,\ell)})_{\tn}\circ \delta_{\one}\LB^{\star}_{(\lambda,\ell)}(\tn)(\sigma_{1}) =(\bigoplus_{1\leq j\leq r}\tau_{\one}\LB^{\star}_{(\lambda[j],\ell)}(\sigma_{1}))\circ (\mathfrak{p}_{(\lambda,\ell)})_{\tn}$.

For $i\geq 2$, we note that the action of $\sigma_{i+1}$ is supported in the subsurface containing the right-most subgraph $\bI_{n}\subset\bI_{1+n}$ (disjoint from the left-most edge $(1,2)$). (This subsurface is the one whose image under $\sigma_{1}$ is the blue shaded region of Figure~\ref{fig:cloud-lemma}.) Hence the generator $\sigma_{i+1}$ acts trivially on the parts of the cycles representing the basis elements $(w'_{0},\ldots,w_{n-1})$ and $(jw'_{0},\ldots,w_{n-1})$ corresponding to the words $w'_{0}$ and $jw'_{0}$.
It thus follows from the definition of $(\mathfrak{p}_{(\lambda,\ell)})_{\tn}$ (see \eqref{eq:definition-of-pkln}) that we have $(\mathfrak{p}_{(\lambda,\ell)})_{\tn}\circ \delta_{\one}\LB^{\star}_{(\lambda,\ell)}(\tn)(\sigma_{i}) =(\bigoplus_{1\leq j\leq r}\tau_{\one}\LB^{\star}_{(\lambda[j],\ell)}(\sigma_{i}))\circ (\mathfrak{p}_{(\lambda,\ell)})_{\tn}$ for $i\geq2$. We have thus verified that $\mathfrak{p}_{(\lambda,\ell)}$ is a natural transformation in $\Fct(\Beta,{\bZ[Q^{\star}_{(\lambda,\ell)}(\bD)]}\lmod^{\bullet})$.

Now, by Lemma~\ref{lem:criterionnaturaltransfo}, it suffices to check relation \eqref{eq:criterion''} in order to prove that $\mathfrak{p}_{(\lambda,\ell)}$ extends to a natural transformation in $\Fct(\langle \Beta, \Beta\rangle,{\bZ[Q^{\star}_{(\lambda,\ell)}(\bD)]}\lmod^{\bullet})$. We consider an object $\tn\geq\one$, the proof being trivial for $\tn=\zero$. We note from the formula \eqref{eq:formula_morphism_{i}d_plus_morphism}, from the composition rule \eqref{eq:composition_rule}, from the fact that $\natural$ is a strict monoidal structure and from the functoriality of $\LB^{\star}_{(\lambda,\ell)}$, that $\tau_{\one}\LB^{\star}_{(\lambda,\ell)}([\one,\id_{\one\natural\tn}])=\LB^{\star}_{(\lambda,\ell)}(\sigma_{1}^{-1})\circ\LB^{\star}_{(\lambda,\ell)}([\one,\id_{\two\natural\tn}])$. (Here, we recall from \S\ref{sss:category_surface_braid_groups} that we have the canonical identification $\sigma_{1}=b_{\one,\one}^{\Beta}\natural \id_{\tn}$ defined by the braiding $b_{\one,\one}^{\Beta} \colon \one\natural \one\cong \one\natural \one$.)

The map $\LB^{\star}_{(\lambda,\ell)}([\one,\id_{\two\natural\tn}])$ clearly has a similar description to that of $\LB^{\star}_{(\lambda,\ell)}([\one,\id_{\one\natural\tn}])$ in \S\ref{sss:preliminary_SES_surface_braid_groups}, simply by adding a right-most extra edge and so replacing $n$ with $n+1$ in the defining assignment \eqref{eq:image_injection_hom_rep_functor_surface}. Specifically, the map $\LB^{\star}_{(\lambda,\ell)}([\one,\id_{\two\natural\tn}])$ is the morphism induced by the embedding of $\bI_{1+n}$ into $\bI_{2+n}$ defined by sending each edge $(i,i+1)$ for $1\leq i \leq n $ to the edge $(i+1,i+2)$.
Hence there are no configuration points on the left-most edge $(1,2)$ of $\bI_{2+n}$ in the image of $\LB^{\star}_{(\lambda,\ell)}([\one,\id_{\two\natural\tn}])$. Also, the morphism $\LB^{\star}_{(\lambda,\ell)}(\sigma_{1}^{-1})$ is defined by the action of $\sigma_{1}^{-1}$ on $\bI_{2+n}$, so the image of $(2,3)$ is the green arc on the left-hand side of Figure~\ref{fig:cloud-splitting}.
For each $1\leq j\leq r$, we thus deduce from Lemma~\ref{lem:cloud-splitting} that $((\mathfrak{p}_{(\lambda,\ell)})_{\one\natural\tn}\circ\delta_{\one}\LB^{\star}_{(\lambda,\ell)}([\one,\id_{\one\natural\tn}]))(j w'_{0},\ldots,w_{n})$ and $\tau_{\one}\LB^{\star}_{(\lambda[j],\ell)}(\tn)([\one,\id_{\one\natural\tn}])(w'_{0},\ldots,w_{n})$ are both equal to $\sum_{u_{0}v_{0}=w'_{0}}(u_{0},v_{0},w_{1},\ldots,w_{n})$.
Therefore, it follows from the definition of $(\mathfrak{p}_{(\lambda,\ell)})_{\tn}$ and an evident induction on $\tm$ (the base case $\tm = \one$ being the previous paragraph) that $(\bigoplus_{1\leq j\leq r}\tau_{\one}\LB^{\star}_{(\lambda[j],\ell)}([\tm,\id_{\tm\natural\tn}]))\circ (\mathfrak{p}_{(\lambda,\ell)})_{\tn}$ is equal to $(\mathfrak{p}_{(\lambda,\ell)})_{\tm\natural\tn}\circ\delta_{\one}\LB^{\star}_{(\lambda,\ell)}([\tm,\id_{\tm\natural\tn}])$.
Hence relation \eqref{eq:criterion''} is satisfied for each $\tn\in\obj(\Beta)$, so Lemma~\ref{lem:criterionnaturaltransfo} implies that $\mathfrak{p}_{(\lambda,\ell)}$ is a natural isomorphism on $\langle \Beta, \Beta\rangle$, as desired.

For the above arguments, we stress that, whenever we deal with a \emph{twisted} homological representation functor, the action of $\delta_{\one}\LB^{\star}_{(\lambda,\ell)}$ on the ground ring $\bZ[Q^{\star}_{(\lambda,\ell)}(\bD)]$ does not affect any point of the reasoning, since $\bigoplus_{1\leq j\leq r}\tau_{\one}\LB^{\star}_{(\lambda[j],\ell)}$ is automatically equipped with the same action via the implicit change of rings of Convention~\ref{conv:transformation_group_summand} for each summand; see Lemma~\ref{lem:change_of_rings_operations}.
\end{proof}

\begin{rmk}\label{rmk:no_clear_splitting_SES_calssical_braid_groups}
There is no obvious splitting for the short exact sequences of functors \eqref{eq:key_SES_classical_braids_classical}.
For instance, there is an obvious splitting of \eqref{eq:key_SES_classical_braids_classical} at the level of modules, given by sending each basis element $(w_{0},\ldots,w_{n})\in \tau_{\one}\LB_{(\lambda[j],\ell)}(\tn)$ for $1\leq j\leq r$ to $(j w_{0},\ldots,w_{n})$. However, one may check that this splitting does not commute with the action of $\sigma_{1}$; in particular it is not a natural transformation of functors on $\langle \Beta, \Beta\rangle$.
\end{rmk}

\paragraph*{Applications for kernels of representations.}
A remarkable consequence of the short exact sequences of Theorem~\ref{thm:key_SES_classical_braids} is the following inclusion \eqref{eq:inclusion_kernels} of kernels of homological representations encoded by Lawrence-Bigelow functors, which then allows us to prove Corollary~\ref{coro:faithfulness}.

We note that, by construction, any $\LB_{(\lambda,\ell)}(\tn+1)$ is a representation of $\B_{n+1}$, which we may consider as a representation of $\B_{n}$ by restriction.

\begin{prop}
\label{prop:kernels}
Whenever $\lambda' \prec \lambda$ (see Notation~\ref{not:sets_partitions}) and $\ell\geq1$, we have an inclusion
\begin{equation}
\label{eq:inclusion_kernels}
\mathrm{ker}\bigl(\LB_{(\lambda,\ell)}(\tn+1)\bigr) \subseteq \mathrm{ker}\bigl(\LB_{(\lambda',\ell)}(\tn+1)\bigr)
\end{equation}
of kernels of $\B_{n}$-representations.
\end{prop}
\begin{proof}
By Theorem~\ref{thm:key_SES_classical_braids}, there is an epimorphism of functors $\tau_{\one}\LB_{(\lambda,\ell)} \twoheadrightarrow \tau_{\one}\LB_{(\lambda[j],\ell)}$ for any $1\leq j\leq r$. Repeating this finitely many times, we therefore obtain an epimorphism $\tau_{\one}\LB_{(\lambda,\ell)} \twoheadrightarrow \tau_{\one}\LB_{(\lambda',\ell)}$. Restricting to the automorphism group of the object $\tn$, we obtain a surjection of $\B_{n}$-representations $\LB_{(\lambda,\ell)}(\tn+1) \twoheadrightarrow \LB_{(\lambda',\ell)}(\tn+1)$, which implies the claimed inclusion of kernels.
\end{proof}

\begin{rmk}
Inclusions of kernels of representations of $\B_{n}(S)$, $\MCGo_{g,1}$ and $\MCGno_{h,1}$ analogous to those of Proposition~\ref{prop:kernels} follow by the same reasoning from the short exact sequences \eqref{eq:key_SES_braid_surface}, \eqref{eq:key_SES_MCG_orientable} and \eqref{eq:key_SES_MCG_non_orientable}.
\end{rmk}

\begin{proof}[Proof of Corollary~\ref{coro:faithfulness}]
By hypothesis we have $(2) \prec \lambda$, so it suffices by Proposition~\ref{prop:kernels} to prove that $\LB_{((2),\ell)}(\tn+1)$ is faithful as a $\B_{n}$-representation. For $\ell = 2$, it is proven in \cite{bigelow2001braid} (see also \cite{KrammerLK}) that $\LB_{((2),2)}(\tn+1)$ is faithful as a $\B_{n+1}$-representation and hence, by restriction, also as a $\B_{n}$-representation. In \cite[Rem.~$4.8$]{PSIN}, it is explained how to deduce from this that $\LB_{((2),\ell)}(\tn+1)$ is faithful for all $\ell\geq 2$.
\end{proof}

\subsubsection{For braid groups on surfaces different from the disc}\label{ss:SES_surface_braids}

We deal here with the short exact sequences for the homological representation functors $\fL_{(\lambda,\ell)}(\Sigma_{g,1})$, $\fLu_{(\lambda,\ell)}(\Sigma_{g,1})$, $\fL_{(\lambda,\ell)}(\N_{h,1})$ and $\fLu_{(\lambda,\ell)}(\N_{h,1})$ of \S\ref{sss:representations_surface_braid_groups}; see Theorem~\ref{thm:key_SES_surface_braid_groups}.
The arguments for our work in this section are analogous regardless of which of the homological representation functors amongst this list we consider. For the sake of simplicity and to avoid repetition, we thus pool the key steps and common arguments for the remainder of \S\ref{ss:SES_surface_braids}, only emphasising the (minor) differences when necessary. Following \S\ref{sss:representations_surface_braid_groups} and \S\ref{sss:preliminary_SES_surface_braid_groups}, we use the standard notation $\fL^{\star}_{(\lambda,\ell)}(S)$, where $\star$ either stands for the blank space or $\star=\unt$, $S$ is either $\Sigma_{g,1}$ or $\N_{h,1}$ with $g,h\geq1$, $Q^{\star}_{(\lambda,\ell)}(S)$ is the associated transformation group and $\Beta^{S}$ is the associated groupoid.

Furthermore, it will be convenient to consider various ``cut versions'' of these homological representation functors defined on $\langle \Beta , \Beta^{S} \rangle$. In fact, this definition makes sense more generally:

\begin{defn}[\emph{Cut functors.}]
\label{defn:cut-functors}
Let $\cC$ be a category whose objects form a totally-ordered set and in which there are no morphisms $a \to b$ if $a>b$. For such a category $\cC$, a functor $F \colon \cC \to R\lmod$ and an object $c$ of $\cC$, we define the truncation $F_{\mid\geq c} \colon \cC \to R\lmod$ on objects by $F_{\mid\geq c}(a) = F(a)$ for $a\geq c$ and $F_{\mid\geq c}(a) = 0$ for $a<c$ and on morphisms by $F_{\mid\geq c}(f) = F(f)$ if the domain of $f$ is $\geq c$ and $F_{\mid\geq c}(f) = 0$ otherwise.
\end{defn}

In the case of $\cC = \langle \Beta , \Beta^{S} \rangle$, the objects form the totally-ordered set $\bN$. This ``cut'' alteration is negligible for our later study of polynomiality (see the proof of Corollary~\ref{coro:polynomiality} for surface braid groups in \S\ref{ss:polynomiality_surface_braid_groups}), while the ``cut'' subfunctors are much more convenient to deal with (see Remark~\ref{rmk:conjsecture_result_without_cut}).

From now on in \S\ref{ss:SES_surface_braids}, we assume that $k\geq2$. We recall from the preliminary study of \S\ref{sss:preliminary_SES_surface_braid_groups} that $\kappa_{\one}\fL^{\star}_{(\lambda,\ell)}(S)(\tn)=0$ and that $\delta_{\one} \fL^{\star}_{(\lambda,\ell)}(S)(\tn)$ is the free $\bZ[Q^{\star}_{(\lambda,\ell)}(S)]$-module with basis given by the tuples of words $(w_{0},w_{1},\ldots,w_{g_{S}+n-1})\vdash \lambda$ such that $\vert w_{0}\vert\geq 1$ for each $\tn\in\obj(\Beta^{S})$. For $\tn\geq \two$, we denote by $(\mathfrak{p}_{(\lambda,\ell)})_{\tn}$ the $\bZ[Q^{\star}_{(\lambda,\ell)}(S)]$-module isomorphism \eqref{eq:isom-of-modules_braids}, which may be written as
\begin{equation}
\label{eq:isom-of-modules-truncated}
\delta_{\one}(\fL^{\star}_{(\lambda,\ell)}(S)_{\mid\geq2})(\tn) \overset{\cong}{\longrightarrow} \underset{1\leq j\leq r}{\bigoplus} (\tau_{\one}\fL^{\star}_{(\lambda[j],\ell)}(S))_{\mid\geq2}(\tn)
\end{equation}
since the truncations do not make any difference when $\tn\geq \two$.
We also set $(\mathfrak{p}_{(\lambda,\ell)})_{\zero}$ to be the trivial morphism; this gives an isomorphism of the form \eqref{eq:isom-of-modules-truncated} for $\tn=\zero$. However, for $\tn=\one$ there is no isomorphism of the form \eqref{eq:isom-of-modules-truncated}, since the right-hand side is zero, whereas we have $\delta_{\one}(\fL^{\star}_{(\lambda,\ell)}(S)_{\mid\geq2})(\one) \cong (\tau_{\one}\fL^{\star}_{(\lambda,\ell)}(S)_{\mid\geq2})(\one)$ as $\B_{1}(S)$-representations over $\bZ[Q^{\star}_{(\lambda,\ell)}(S)]$.

Our first goal in this section is to promote \eqref{eq:isom-of-modules-truncated} to a natural isomorphism of functors on $\langle \Beta , \Beta^{S} \rangle$, so we first need to correct the right-hand side on the object $\tn=\one$. To do this, we choose a certain extension of functors, via the following lemma:

\begin{lem}
\label{lem:extension-of-functors}
Let $\cM$ be a module over a braided monoidal groupoid $\cG$ that on objects is given by the monoid $\bN$ as a module over itself.
Let $F,G \colon \langle \cG , \cM \rangle \to R\lmod$ be two functors with $F(n)=0$ for $n\leq c-1$ and $G(n)=0$ for $n\geq c$ for an integer $c\geq 1$. Then there is a one-to-one correspondence between extensions of $G$ by $F$, i.e.~short exact sequences $0 \to F \to {?} \to G \to 0$, and morphisms $G(c-1) \to F(c)$ in $R\lmod$, given by evaluating the extension functor at $[\one,\id_c]$.
\end{lem}
\begin{proof}
Since $F$ and $G$ have disjoint support, there is no choice about the action of any such extension on objects and on automorphisms; in other words, there is a unique extension of $G$ by $F$ when restricting the domain to the subgroupoid $\cG$. Lemma~\ref{lem:extend_functor_Quillen} describes the data and conditions required to extend a functor out of $\cG$ to a functor out of $\langle \cG , \cM \rangle$. In light of the requirement that this is an extension of $G$ by $F$, the only remaining choice is the value assigned to the morphism $[\one,\id_c]$; conversely, any such choice determines an extension.
\end{proof}

\begin{defn}[\emph{An extension by an atomic functor.}]
\label{defn:extension-of-atomic}
Denote by $(\tau_{\one}\fL^{\star}_{(\lambda,\ell)}(S)_{\mid\geq2})(\one)$ the ``\emph{atomic}'' functor out of $\langle \Beta , \Beta^{S} \rangle$ whose value on the object $\one$ is $(\tau_{\one}\fL^{\star}_{(\lambda,\ell)}(S)_{\mid\geq2})(\one)=\fL^{\star}_{(\lambda,\ell)}(S)_{\mid\geq2}(\two)$ and whose value on all other objects is the zero module.
Denote by $\widetilde{\bigoplus}_{1\leq j\leq r} (\tau_{\one}\fL^{\star}_{(\lambda[j],\ell)}(S))_{\mid\geq2}$ the unique extension of this atomic functor by the functor $\bigoplus_{1\leq j\leq r}(\tau_{\one}\fL^{\star}_{(\lambda[j],\ell)}(S))_{\mid\geq2}$ whose value on $[\one,\id_{\two}]$ is:
\begin{center}
\[
\begin{tikzcd}
(\tau_{\one}\fL^{\star}_{(\lambda,\ell)}(S)_{\mid\geq2})(\one) \ar[rrrr,"{\tau_{\one}(\fL^{\star}_{(\lambda,\ell)}(S)_{\mid\geq2})([\one,\id_{\two}])}"]
&&&& \tau_{\one}(\fL^{\star}_{(\lambda,\ell)}(S)_{\mid\geq2})(\two) \ar[rr,"{\Delta_{\one}(\fL^{\star}_{(\lambda,\ell)}(S)_{\mid\geq2})(\two)}"]
&& \delta_{\one}(\fL^{\star}_{(\lambda,\ell)}(S)_{\mid\geq2})(\two) \ar[d,"{\eqref{eq:isom-of-modules-truncated}}"] \\
&&&&&& \underset{1\leq j\leq r}{\bigoplus} (\tau_{\one}\fL^{\star}_{(\lambda[j],\ell)}(S))_{\mid\geq2}(\two).
\end{tikzcd}
\]
\end{center}
We also denote by $(\mathfrak{p}_{(\lambda,\ell)})_{1}$ the isomorphism
\begin{equation}\label{eq:isomorphism-in-degree-one}
\begin{tikzcd}
\delta_{\one}(\fL^{\star}_{(\lambda,\ell)}(S)_{\mid\geq2})(\one) \ar[rrrr,"{(\Delta_{\one}(\fL^{\star}_{(\lambda,\ell)}(S)_{\mid\geq2})(\one))^{-1}}"] &&&& (\tau_{\one}\fL^{\star}_{(\lambda,\ell)}(S)_{\mid\geq2})(\one) = \left(\underset{1\leq j\leq r}{\widetilde{\bigoplus}} (\tau_{\one}\fL^{\star}_{(\lambda[j],\ell)}(S))_{\mid\geq2}\right)(\one).
\end{tikzcd}
\end{equation}
\end{defn}

Using the extension of Definition~\ref{defn:extension-of-atomic}, we may now upgrade \eqref{eq:isom-of-modules-truncated} to an isomorphism of functors:

\begin{thm}\label{thm:key_SES_surface_braid_groups}
For any $S=\Sigma_{g,1} \text{ or } \N_{h,1}$, $\lambda\vdash k\geq2$ and $\ell\geq 1$, the exact sequence \eqref{eq:ESCaract} induces a short exact sequence
\begin{equation}
\label{eq:key_SES_braid_surface}
\begin{tikzcd}
0 \ar[r] & \fL^{\star}_{(\lambda,\ell)}(S)_{\mid\geq2} \ar[r] & \tau_{\one}(\fL^{\star}_{(\lambda,\ell)}(S)_{\mid\geq2}) \ar[r] & \underset{1\leq j\leq r}{\widetilde{\bigoplus}} (\tau_{\one}\fL^{\star}_{(\lambda[j],\ell)}(S))_{\mid\geq2} \ar[r] & 0
\end{tikzcd}
\end{equation}
of functors in $\Fct(\langle \Beta, \Beta^{S}\rangle ,{\bZ[Q^{\unt}_{(\lambda,\ell)}(S)]}\lmod)$ if $\star = u$, and in $\Fct(\langle \Beta, \Beta^{S} \rangle ,{\bZ[Q_{(\lambda,\ell)}(S)]}\lmod^{\bullet})$ if $\star$ is the blank space.
\end{thm}
\begin{proof}
The roadmap of this proof is similar to that of Theorem~\ref{thm:key_SES_classical_braids}, whose arguments are reused below for the analogous steps.

As a consequence of the above preliminary study, we already have that the evaluation of \eqref{eq:key_SES_braid_surface} at any $\tn$ is a short exact sequences of $\bZ[Q^{\star}_{(\lambda,\ell)}(S)]$-modules. Hence the key point is to prove the compatibility with respect to the morphisms of $\langle \Beta, \Beta^{S}\rangle$. To achieve this, it suffices to show that the $\bZ[Q^{\star}_{(\lambda,\ell)}(S)]$-module isomorphisms $\{(\mathfrak{p}_{(\lambda,\ell)})_{\tn}\}_{\tn\in\obj(\Beta^{S})}$ assemble to an isomorphism in $\Fct(\langle \Beta, \Beta^{S}\rangle ,{\bZ[Q^{\star}_{(\lambda,\ell)}(S)]}\lmod^{\bullet})$.

As a first step, we prove that assembling these module isomorphisms defines an isomorphism in the category $\Fct(\Beta^{S} ,{\bZ[Q^{\star}_{(\lambda,\ell)}(S)]}\lmod^{\bullet})$, in other words that each $(\mathfrak{p}_{(\lambda,\ell)})_{\tn}$ is a $\B_{n}(S)$-module morphism. We first consider the case where $\tn\geq\two$ and use the classical generating set for $\B_{n}(S)$, extending the Artin generators $\sigma_{i}$ of $\B_{n}$, recalled in \cite[Prop.~$4.1$]{PSI}.
Since $\tau_{\one}$ is the left-to-right translation-by-one operator, for each generator $\rho\in\B_{n}(S)$, the morphisms $(\tau_{\one}\fL^{\star}_{(\lambda[j],\ell)}(S))_{\mid\geq2}(\rho)$ for all $1\leq j\leq r$ and $\delta_{\one}(\fL^{\star}_{(\lambda,\ell)}(S)_{\mid\geq2})(\rho)$ are defined by the action of the generator $\id_{\one}\natural\rho$ on the Borel-Moore homology classes supported on the embedded graph $\bI_{1+n} \vee \bW^{S} \subset \bD_{1+n}\natural S$; see Figures~\ref{fig:model-or-braids} and \ref{fig:model-nor-braids}.

For $\rho = \sigma_{1}$, we prove that $(\mathfrak{p}_{(\lambda,\ell)})_{\tn}\circ \delta_{\one}(\fL^{\star}_{(\lambda,\ell)}(S)_{\mid\geq2})(\sigma_{1}) = \bigl( \widetilde{\bigoplus}_{1\leq j\leq r} (\tau_{\one}\fL^{\star}_{(\lambda[j],\ell)}(S))_{\mid\geq2}(\sigma_{1}) \bigr) \circ (\mathfrak{p}_{(\lambda,\ell)})_{\tn}$ by repeating verbatim the corresponding step in the proof of Theorem~\ref{thm:key_SES_classical_braids} which relies on Lemma~\ref{lem:cloud}.

For $\rho \neq \sigma_{1}$, we note that the action of $\id_{\one}\natural\rho$ is supported on a subsurface containing the right-most subgraph $\bI_{n} \vee \bW^{S} \subset\bI_{1+n} \vee \bW^{S}$ (disjoint from the left-most edge).
Hence the generator $\id_{\one}\natural\rho$ acts trivially on the parts of the cycles representing the basis elements $(w'_{0},\ldots)$ and $(jw'_{0},\ldots)$ corresponding to the words $w'_{0}$ and $jw'_{0}$. Using the (simpler) analogue of Lemma~\ref{lem:cloud} similarly to the case of $\rho = \sigma_{1}$, it follows from the definition of $(\mathfrak{p}_{(\lambda,\ell)})_{\tn}$ (see \eqref{eq:isom-of-modules-truncated} and \eqref{eq:isomorphism-in-degree-one}) that we have $(\mathfrak{p}_{(\lambda,\ell)})_{\tn}\circ \delta_{\one}(\fL^{\star}_{(\lambda,\ell)}(S)_{\mid\geq2})(\rho) = \bigl( \widetilde{\bigoplus}_{1\leq j\leq r} (\tau_{\one}\fL^{\star}_{(\lambda[j],\ell)}(S))_{\mid\geq2}(\rho) \bigr) \circ (\mathfrak{p}_{(\lambda,\ell)})_{\tn}$.

Furthermore, the analogous relations trivially hold also for $\tn=\zero$ (because $\fL^{\star}_{(\lambda,\ell)}(S)_{\mid\geq2}(\zero)=0$) and for $\tn=\one$ (because the isomorphism $(\mathfrak{p}_{(\lambda,\ell)})_{1} = \eqref{eq:isomorphism-in-degree-one}$ is $\B_{1}(S)$-equivariant by construction).
We have thus verified that $\mathfrak{p}_{(\lambda,\ell)}$ is a natural transformation in $\Fct(\Beta^{S} ,{\bZ[Q^{\star}_{(\lambda,\ell)}(S)]}\lmod^{\bullet})$.

Now, by Lemma~\ref{lem:criterionnaturaltransfo}, it suffices to check relation \eqref{eq:criterion''} in order to prove that $\mathfrak{p}_{(\lambda,\ell)}$ extends to a natural transformation in $\Fct(\langle \Beta , \Beta^{S} \rangle,{\bZ[Q^{\star}_{(\lambda,\ell)}(S)]}\lmod^{\bullet})$.
We consider an object $\tn\geq\one$, the proof being trivial for $\tn=\zero$. We note from the formula \eqref{eq:formula_morphism_{i}d_plus_morphism}, from the composition rule \eqref{eq:composition_rule}, from the fact that $\natural$ is a strict monoidal structure and from the functoriality of $\fL^{\star}_{(\lambda,\ell)}(S)_{\mid\geq2}$, that $\tau_{\one}(\fL^{\star}_{(\lambda,\ell)}(S)_{\mid\geq2})([\one,\id_{\one\natural\tn}])=\fL^{\star}_{(\lambda,\ell)}(S)_{\mid\geq2}(\sigma_{1}^{-1})\circ\fL^{\star}_{(\lambda,\ell)}(S)_{\mid\geq2}([\one,\id_{\two\natural\tn}])$, using the canonical identification $b_{\one,\one}^{\Beta}\natural \id_{\tn}=\sigma_{1}$ defined by the braiding $b_{\one,\one}^{\Beta} \colon \one\natural \one\cong \one\natural \one$; see \S\ref{sss:category_surface_braid_groups}.

The map $\fL^{\star}_{(\lambda,\ell)}(S)_{\mid\geq2}([\one,\id_{\two\natural\tn}])$ clearly has a similar description to that of $\fL^{\star}_{(\lambda,\ell)}(S)_{\mid\geq2}([\one,\id_{\one\natural\tn}])$ in \S\ref{sss:preliminary_SES_surface_braid_groups}. Namely, it is the morphism induced by the embedding of $\bI_{1+n}$ into $\bI_{2+n}$ defined by sending each edge $(i,i+1)$ for $1\leq i \leq n $ to the edge $(i+1,i+2)$, and by the identity on the wedge $\bW^{S}$. In other words, it is explicitly obtained by replacing $n$ with $n+1$ in the defining assignment \eqref{eq:image_injection_hom_rep_functor_surface}. 
Then, applying Lemma~\ref{lem:cloud-splitting} with the illustration of Figure~\ref{fig:cloud-splitting} and by an evident induction on $\tm$, the corresponding reasoning in the proof of Theorem~\ref{thm:key_SES_classical_braids} repeats mutatis mutandis here, proving that ${(\mathfrak{p}_{(\lambda,\ell)})}_{\tm\natural\tn}\circ\delta_{\one}(\fL^{\star}_{(\lambda,\ell)}(S)_{\mid\geq2})([\tm,\id_{\tm\natural\tn}])=\left(\widetilde{\bigoplus}_{1\leq j\leq r} (\tau_{\one}\fL^{\star}_{(\lambda[j],\ell)}(S))_{\mid\geq2}([\tm,\id_{\tm\natural\tn}])\right)\circ(\mathfrak{p}_{(\lambda,\ell)})_{\tn}$ for each $\tm\geq1$ and $\tn\geq\two$.
The analogous relation follows in the exact same way for $\tn=\one$, the point being that $\widetilde{\bigoplus}_{1\leq j\leq r} (\tau_{\one}\fL^{\star}_{(\lambda[j],\ell)}(S))_{\mid\geq2}([\one,\id_{\two}]) = \delta_{\one}\fL_{(\lambda,\ell)\mid\geq2}^{\star}([\one,\id_{\two}]) \circ \Delta_{\one}\fL_{(\lambda,\ell)\mid\geq2}^{\star}(\one)$. In addition, this relation also holds trivially for $\tn=\zero$ because $\fL^{\star}_{(\lambda,\ell)}(S)_{\mid\geq2}(\zero)=0$.
Hence relation \eqref{eq:criterion''} is satisfied for each $\tn\in\obj(\Beta^{S})$, so Lemma~\ref{lem:criterionnaturaltransfo} implies that $\mathfrak{p}_{(\lambda,\ell)}$ is a natural isomorphism on $\langle \Beta,\Beta^{S}\rangle$, as desired.

In the above arguments, we stress that, whenever we deal with a \emph{twisted} homological representation functor, the action of $\delta_{\one}(\fL^{\star}_{(\lambda,\ell)}(S)_{\mid\geq2})$ on the ground ring $\bZ[Q_{(\lambda,\ell)}(S)]$ does not affect any point of the reasoning, thanks to the implicit change of rings of Convention~\ref{conv:transformation_group_summand}.
\end{proof}

\begin{coro}\label{coro:key_SES_surface_braid_groups}
For any $S=\Sigma_{g,1} \text{ or } \N_{h,1}$, $\lambda\vdash k\geq2$ and $\ell\geq 1$, the short exact sequence \eqref{eq:key_SES_braid_surface} induces a short exact sequence
\begin{equation}
\label{eq:key_SES_braid_coro}
\begin{tikzcd}
0 \ar[r] & \fL^{\star}_{(\lambda,\ell)}(S)_{\mid\geq2} \ar[r] & (\tau_{\one}\fL^{\star}_{(\lambda,\ell)}(S))_{\mid\geq2} \ar[r] & \underset{1\leq j\leq r}{\bigoplus} (\tau_{\one}\fL^{\star}_{(\lambda[j],\ell)}(S))_{\mid\geq2} \ar[r] & 0
\end{tikzcd}
\end{equation}
of functors in $\Fct(\langle \Beta, \Beta^{S}\rangle ,{\bZ[Q^{\unt}_{(\lambda,\ell)}(S)]}\lmod)$ if $\star = u$, and $\Fct(\langle \Beta, \Beta^{S} \rangle ,{\bZ[Q_{(\lambda,\ell)}(S)]}\lmod^{\bullet})$ if $\star$ is the blank space.
\end{coro}
\begin{proof}
It straightforwardly follows from the definition of $\tau_{\one}$ and from Definition~\ref{defn:cut-functors} that the cut functor $(\tau_{\one}\fL^{\star}_{(\lambda,\ell)}(S))_{\mid\geq2}$ is a subfunctor of $\tau_{\one}(\fL^{\star}_{(\lambda,\ell)}(S)_{\mid\geq2})$: namely, the associated natural transformation $\mathsf{i}\colon  (\tau_{\one}\fL^{\star}_{(\lambda,\ell)}(S))_{\mid\geq2}\hookrightarrow\tau_{\one}(\fL^{\star}_{(\lambda,\ell)}(S)_{\mid\geq2})$ is defined by the identity on the objects $\tn\geq\two$ and $\zero$, while $\mathsf{i}_{\one}:0\to \fL^{\star}_{(\lambda,\ell)}(S)(\two)$ is the zero map. Furthermore, since the objects of $\langle \Beta , \Beta^{S} \rangle$ form the totally-ordered set $\bN$, modifying the natural transformation $i_{\one}\fL^{\star}_{(\lambda,\ell)}(S)$ by assigning the zero map for $\tn\leq\one$ defines a natural transformation $(i_{\one}\fL^{\star}_{(\lambda,\ell)}(S))_{\mid\geq2}\colon \fL^{\star}_{(\lambda,\ell)}(S)_{\mid\geq2} \to(\tau_{\one}\fL^{\star}_{(\lambda,\ell)}(S)_{\mid\geq2})$. We deduce from these definitions that the composite $\mathsf{i}_{\tn}\circ ((i_{\one}\fL^{\star}_{(\lambda,\ell)}(S))_{\mid\geq2})_{\tn}$ is equal to $(i_{\one}(\fL^{\star}_{(\lambda,\ell)}(S)_{\mid\geq2}))_{\tn}$ for $\tn\geq\two$ (which is the left-hand map in the exact sequence \eqref{eq:key_SES_braid_surface}), while it is the zero map for $\tn\leq \one$. The result thus follows from Theorem~\ref{thm:key_SES_surface_braid_groups}, by using the universal property of a cokernel and the definition of ${\widetilde{\bigoplus}}_{1\leq j\leq r} (\tau_{\one}\fL^{\star}_{(\lambda[j],\ell)}(S))_{\mid\geq2}$.
\end{proof}

\begin{rmk}
Similarly to Remark~\ref{rmk:no_clear_splitting_SES_calssical_braid_groups}, we note that there is no obvious splitting for the short exact sequence of functors \eqref{eq:key_SES_braid_surface}.
\end{rmk}

\begin{rmk}\label{rmk:conjsecture_result_without_cut}
We conjecture that Theorem~\ref{thm:key_SES_surface_braid_groups} holds also for the functors $\fL^{\star}_{(\lambda,\ell)}(S)$, i.e., without truncating to the functors $\fL^{\star}_{(\lambda,\ell)}(S)_{\mid\geq2}$ via Definition~\ref{defn:cut-functors}. For the trivial partition $\lambda=(k)$, we have verified this for the functor $\fL_{((k),2)}(\Sigma_{g,1})$ for each $k\geq2$ by using explicit formulas for the action of $\B_{n}(\Sigma_{g,1})$ from \cite{PSIIi}. However, it seems significantly more difficult to prove this in general for any ordered partition $\lambda$.
\end{rmk}

\subsection{For mapping class group functors}\label{ss:SES_mcg}

We construct here the short exact sequences for the functors associated to mapping class groups defined in \S\ref{sss:representations_mapping_class_groups}, i.e.~the functors \eqref{eq:hom_rep_functor_rep_mcg_o} and \eqref{eq:hom_rep_functor_rep_mcg_no_ell}, as well as their vertical-type alternatives, for any $\lambda\vdash k\geq2$ and $\ell\geq1$.
The results in the classical (i.e.~non-vertical) setting are in \S\ref{sss:SES_classical_MCG}, those in the vertical setting in \S\ref{sss:SES_MCG_alternative}, preceded by preliminary work and diagrammatic lemmas in \S\ref{ss:mcg_hom_rep_poly_preliminary}.

The arguments being analogous for orientable and non-orientable surfaces, we pool the key steps and common arguments for these two cases. Following Notation~\ref{notation:generic_notation_star}, we use the generic notation $\sS$ for either $\bT \cong \Sigma_{1,1}$ or $\bM \cong \N_{1,1}$, $\MCG(\sS^{\natural n})$ for either $\MCGo_{n,1}$ or $\MCGno_{n,1}$, $\fL^{\star}_{(\lambda,\ell)}$ for any one of the functors \eqref{eq:hom_rep_functor_rep_mcg_o} and \eqref{eq:hom_rep_functor_rep_mcg_no_ell}, $\fL^{\star,\vrtcl}_{(\lambda,\ell)}$ for the vertical-type alternative, $Q^{\star}_{(\lambda,\ell)}(\sS)$ for the associated transformation group and $\cM$ for either $\M^{+}$ or $\M^{-}$.

\subsubsection{First properties and diagrammatic arguments}\label{ss:mcg_hom_rep_poly_preliminary}

For the purposes of \S\ref{sss:SES_classical_MCG} and \S\ref{sss:SES_MCG_alternative}, we begin by proving qualitative properties of the representations, including a disjoint support argument in the case of boundary connected sums of two surfaces and some calculations of the actions of various braiding actions.

Let us first focus on the classical (i.e.~non-vertical) setting for homological representation functors.
We follow the notation of \S\ref{ss:free-bases} and consider, for each $\tn\in \obj(\cM)$, the mapping class group $\MCG(\sS^{\natural n})$ representation $\fL^{\star}_{(\lambda,\ell)}(\tn)= H_{k}^{\BM}(C_{\lambda}(\sS^{\natural n} \smallsetminus I);\bZ[Q^{\star}_{(\lambda,\ell)}(\sS)])$ where $\bZ[Q^{\star}_{(\lambda,\ell)}(\sS)]$ is the rank-one local system explained in the general construction of \S\ref{ss:general_construction}. We recall that we introduce model graphs $\bW^{\Sigma}_{g}$ and $\bW_h^{\N}$ in Notation~\ref{nota:tail_and_wedge}, which are illustrated in Figures~\ref{fig:model-or-mcg} and \ref{fig:model-nor-mcg}. Let us write $\bW_{n}^{\sS} := \bW^{\Sigma}_{n}$ in the orientable setting and $\bW^{\sS}_{n} := \bW^{\N}_{n}$ in the non-orientable setting.
By Proposition~\ref{prop:module_structure_BM_homology}, the $\bZ[Q^{\star}_{(\lambda,\ell)}(\sS)]$-module $\fL^{\star}_{(\lambda,\ell)}(\tn)$ is free with basis indexed by labellings of the embedded graph $\bW_{n} ^{\sS} \subset \sS^{\natural n}$ by words in the blocks of $\lambda$. In other words, the basis is the set of tuples $\bw$ of the form \eqref{eq:generator_surface_braid_hom_rep_orientable} if $\sS=\bT$ and \eqref{eq:generator_surface_braid_hom_rep_non-orientable} if $\sS=\bM$ (ignoring the initial tuple of words corresponding to the linear part of the graph, which does not exist in the mapping class group setting).

\paragraph*{Boundary connected sums.}
We will use the following general principle for representations of the mapping class group of a surface that splits as a boundary connected sum. Let $\hat{\fL}_{(\lambda,\ell)}$ be either $\fL^{\star}_{(\lambda,\ell)}$ or its vertical-type alternative $\fL^{\star,\vrtcl}_{(\lambda,\ell)}$. We recall from \S\ref{sss:category_mcg} that each object $\tn$ of $\cM$ is the surface $\sS^{\natural n}$.

\begin{lem}
\label{lem:invariance_translation_block_MCG_action}
Let $\tn,\tm\in\obj(\cM)$. Let $\bw$ be a basis element (see Notation~\ref{notation:basis_notation}) of the representation $\hat{\fL}_{(\lambda,\ell)}(\tn\natural\tm)$, using the bases described in \S\ref{s:free_generating_sets} and write the tuple $\bw$ as $\bw = (\bw',\bw'')$, where the entries of $\bw'$ correspond to arcs supported in $\sS^{\natural n}$ and the entries of $\bw''$ correspond to arcs supported in $\sS^{\natural m}$. Then, for $\varphi_{\tm} \in \MCG(\sS^{\natural m})$, the element $\hat{\fL}_{(\lambda,\ell)}(\id_{\tn} \natural \varphi_{\tm})(\bw)$ is a linear combination of basis elements of the form $(\bw',\bv'')$, where $\bv''$ runs over all possible labellings of arcs supported in $\sS^{\natural m}$.
\end{lem}
\begin{proof}
As in the proof of Lemma~\ref{lem:cloud}, we prefer to denote by $e_{\bw}$ (rather than $\bw$) the basis element corresponding to the tuple $\bw$ for the sake of clarity.
Let $e'_{\bv}$ be an arbitrary dual basis element and write $\bv = (\bv',\bv'')$ similarly to the decomposition $\bw = (\bw',\bw'')$. It suffices to show that $\langle \hat{\fL}_{(\lambda,\ell)}(\id_{\tn} \natural \varphi_{\tm})(e_{\bw}) , e'_{\bv} \rangle = 0$ unless $\bv' = \bw'$. To see this, recall that the homology class $e_{\bw}$ is represented by some configurations on embedded arcs in $\sS^{\natural (n+m)}$. Since $\id_{\tn} \natural \varphi_{\tm}$, by construction, is supported in $\sS^{\natural m}$, the homology class $\hat{\fL}_{(\lambda,\ell)}(\id_{\tn} \natural \varphi_{\tm})(e_{\bw})$ may be represented by some configurations on embedded arcs that are identical on the boundary connected summand $\sS^{\natural n}$ to those representing $e_{\bw}$. The intersection pairing with $e'_{\bv}$ must therefore be zero unless $\bv' = \bw'$.
\end{proof}

\paragraph*{Interaction with the braiding.}

To discuss elements of mapping class groups that act by ``braiding'' handles or crosscaps of the surface $\tn=\sS^{\natural n}$, it is convenient to pass, in this section, to a different way of representing $\sS^{\natural n}$ diagrammatically. Instead of a rectangle to which we have glued a finite number of strips (as in, for example, Figure~\ref{fig:models}), we will represent this surface as a rectangle from which we have either erased the interiors of $2g$ discs and glued their boundaries in pairs (when considering $\sS^{\natural g} \cong \Sigma_{g,1}$) or erased the interiors of $h$ discs and glued each resulting boundary component to itself by a degree-$2$ map (when considering  $\sS^{\natural h} \cong \N_{h,1}$).

Each basis element $\bw$ of the representation $\fL^{\star}_{(\lambda,\ell)}(\tn)$ (see Figures~\ref{fig:graph-or-mcg} and \ref{fig:graph-nor-mcg}) looks as illustrated in Figure~\ref{fig:basis-alternative-picture} in this picture, where we have also included explicit choices of \emph{tethers} (see \S\ref{ss:tethers}), i.e.~paths from the base configuration to a point on the submanifold representing the homology class. Similarly, we denote by $\bw^{\vrtcl}$ the basis elements of the vertical-type alternative representations $\fL^{\star,\vrtcl}_{(\lambda,\ell)}(\tn)$ (see Figures~\ref{fig:model-or-mcg-vertical-dual} and \ref{fig:model-nor-mcg-vertical-dual}), which look as illustrated in Figure~\ref{fig:basis-alternative-picture-vertical} in this picture, where again we have included explicit choices of tethers.

\begin{figure}[tbp]
    \centering
    \begin{subfigure}[b]{\textwidth}
        \centering
        \includegraphics[scale=0.65]{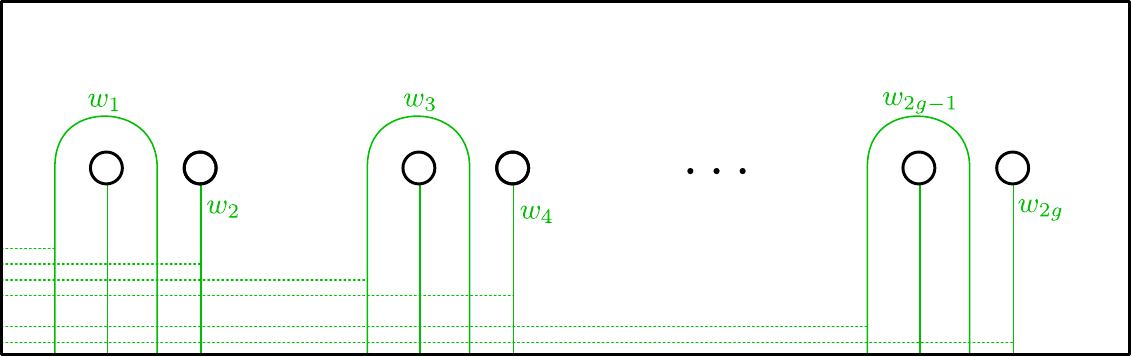}
        \caption{The orientable case.}
        \label{fig:basis-alternative-picture-or}
    \end{subfigure}
    \\[3ex]
    \begin{subfigure}[b]{\textwidth}
        \centering
        \includegraphics[scale=0.65]{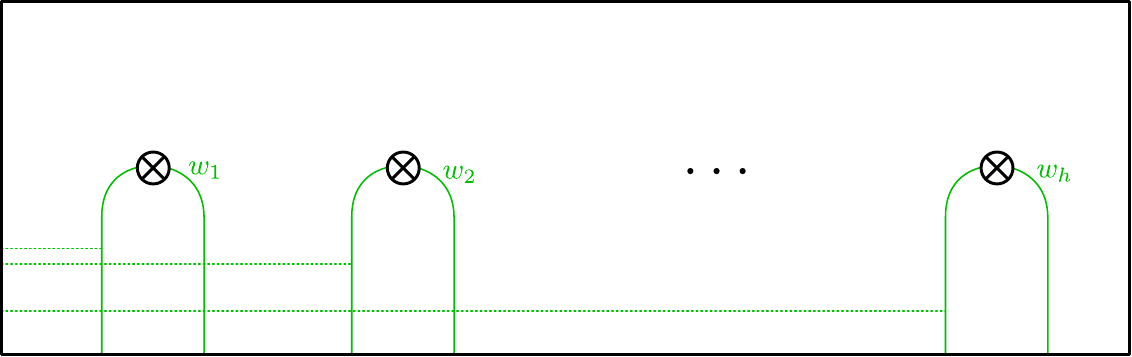}
        \caption{The non-orientable case.}
        \label{fig:basis-alternative-picture-nor}
    \end{subfigure}
    \caption{Another perspective on the basis elements depicted in Figures~\ref{fig:graph-or-mcg} and \ref{fig:graph-nor-mcg}.}
    \label{fig:basis-alternative-picture}
\end{figure}

\begin{figure}[tbp]
    \centering
    \begin{subfigure}[b]{\textwidth}
        \centering
        \includegraphics[scale=0.65]{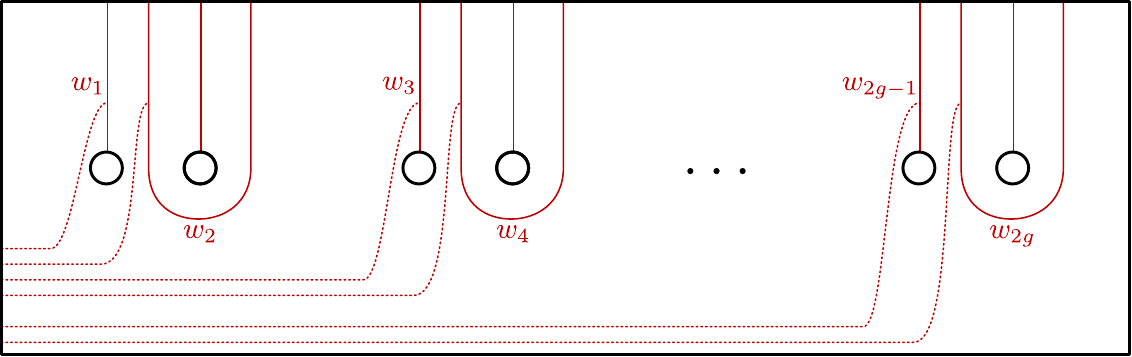}
        \caption{The orientable case.}
        \label{fig:basis-alternative-picture-or-vertical}
    \end{subfigure}
    \\[3ex]
    \begin{subfigure}[b]{\textwidth}
        \centering
        \includegraphics[scale=0.65]{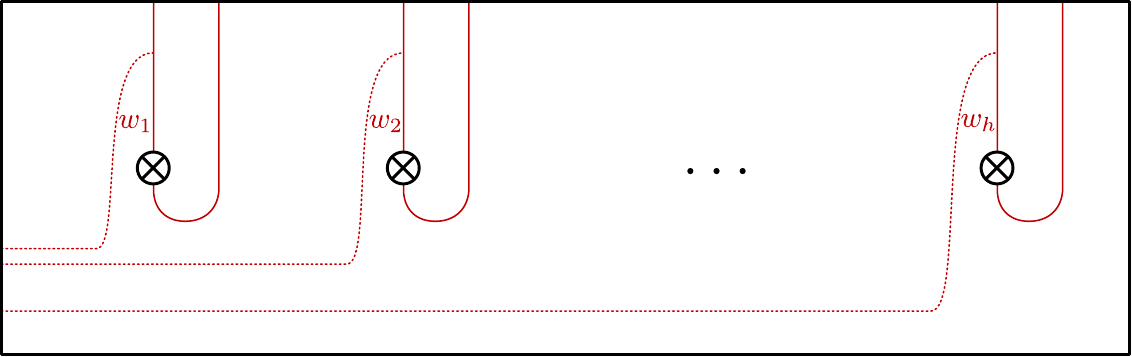}
        \caption{The non-orientable case.}
        \label{fig:basis-alternative-picture-nor-vertical}
    \end{subfigure}
    \caption{Another perspective on the basis elements depicted in Figures~\ref{fig:model-or-mcg-vertical-dual} and \ref{fig:model-nor-mcg-vertical-dual}.}
    \label{fig:basis-alternative-picture-vertical}
\end{figure}

\begin{notation}
We denote by $\sigma_{1} \in \MCG(\sS^{\natural n})$ the mapping class illustrated in Figure~\ref{fig:braiding-mcg}: it braids the left-most two handles if $\sS = \bT$ and it braids the left-most two crosscaps if $\sS = \bM$.
\end{notation}

\begin{figure}[htb]
    \centering
    \includegraphics[scale=0.5]{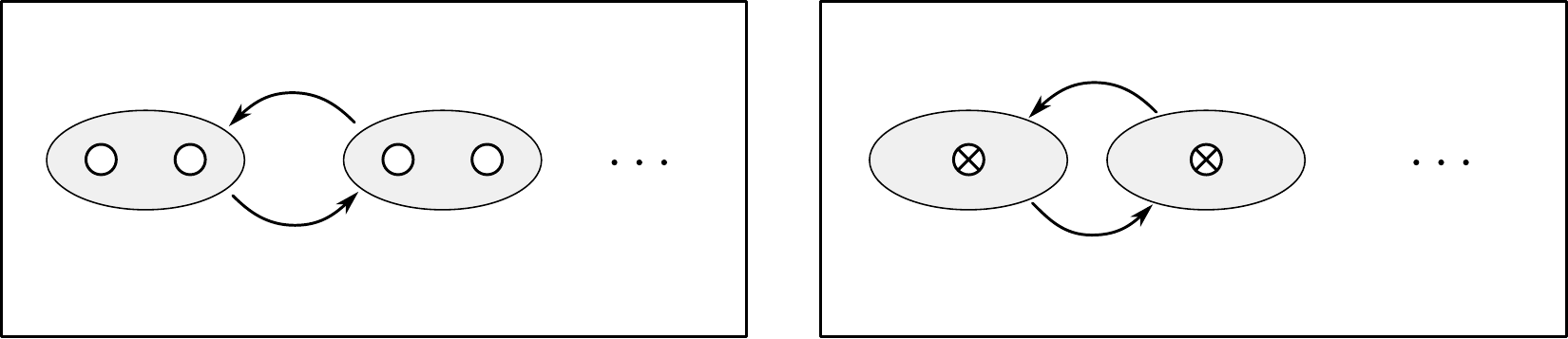}
    \caption{The braiding element $\sigma_{1} \in \MCG(S)$ when $S = \Sigma_{g,1}$ (left) and $S = \N_{h,1}$ (right).}
    \label{fig:braiding-mcg}
\end{figure}

\begin{lem}
\label{lem:computations_MCG_braiding}
The following identities hold in $\fL^{\star}_{(\lambda,\ell)}(\tn)$ for $\ell\geq1$:
\begin{align}
\sigma_{1}^{-1}\left(([\varnothing,\varnothing],[w_{3},w_{4}],\ldots)\right) &= ([w_{3},w_{4}],[\varnothing,\varnothing],\ldots) \label{eq:identity-braiding-or} \\
\sigma_{1}^{-1}\left(([\varnothing],[w_{2}],\ldots)\right) &= ([w_{2}],[\varnothing],\ldots) \label{eq:identity-braiding-nor}
\end{align}
and in $\fL^{\vrtcl}_{(\lambda,\ell')}(\tn)$ for $\ell'\leq 2$ (recall that $\fLv_{(\lambda,\ell')}=\fLuv_{(\lambda,\ell')}$ by Lemma~\ref{lem:Qu_ell=2_useless}):
\begin{align}
\sigma_{1}\left(([\varnothing,\varnothing],[w_{3},w_{4}],\ldots)^{\vrtcl}\right) &= ([w_{3},w_{4}],[\varnothing,\varnothing],\ldots)^{\vrtcl} \label{eq:identity-braiding-or-vertical} \\
\sigma_{1}\left(([\varnothing],[w_{2}],\ldots)^{\vrtcl}\right) &= ([w_{2}],[\varnothing],\ldots)^{\vrtcl} \label{eq:identity-braiding-nor-vertical}
\end{align}
\end{lem}

\begin{proof}
Equations \eqref{eq:identity-braiding-or} and \eqref{eq:identity-braiding-nor} are clear from the diagrams of the basis element (see Figure~\ref{fig:basis-alternative-picture}) and the action of $\sigma_{1}$ (see Figure~\ref{fig:braiding-mcg}), using the facts that the label(s) corresponding to the left-most handle or crosscap are empty and that the tethers are the same up to homotopy.

On the other hand, we see from the diagrams that the left-hand side of equation~\eqref{eq:identity-braiding-or-vertical} is equal to the element illustrated in Figure~\ref{fig:identity-braiding-or-vertical}. This differs from the right-hand side of equation~\eqref{eq:identity-braiding-or-vertical} only in the choice of tether. As explained in Lemma~\ref{lem:changing-tether}, changing the tether has the effect of multiplying the homology class by the unit in the ground ring given by the monodromy action of the based loop of configurations given by the difference between the two tethers. In our setting, the monodromy action is the quotient onto $Q_{(\lambda,\ell')}(\bT)$ followed by its inclusion into $\bZ[Q_{(\lambda,\ell')}(\bT)]^{\times}$. We therefore need to show that this based loop of configurations, which is illustrated in Figure~\ref{fig:identity-braiding-or-vertical-loop}, projects to the trivial element of the group $Q_{(\lambda,\ell')}(\bT)$. This is obvious for $\ell'=1$ since the group $Q_{(\lambda,1)}(\bT)$ is always trivial. For $\ell'=2$, we recall from Lemma~\ref{lem:transformation_groups_MCG_ell_2} that $Q_{(\lambda,2)}(\bT)$ is simply a product of copies of $\bZ/2$, one for each block of $\lambda = (\lambda_{1},\ldots,\lambda_{r})$ with $\lambda_{i} \geq 2$. The projection onto $Q_{(\lambda,2)}(\bT)$ records the writhe (modulo $2$) of each block of strands (in a surface of positive genus the writhe is only well-defined modulo $2$). It is clear that the writhe of the loop of configurations illustrated in Figure~\ref{fig:identity-braiding-or-vertical-loop} is trivial for each block. This establishes equation~\eqref{eq:identity-braiding-or-vertical}. 

We argue similarly for equation~\eqref{eq:identity-braiding-nor-vertical}. (Again, the case $\ell'=1$ being obvious, we just consider $\ell'=2$.) The left-hand side is equal to the element illustrated in Figure~\ref{fig:identity-braiding-nor-vertical}, which differs from the right-hand side only by its choice of tether; the difference between the two tethers forms the based loop of configurations illustrated in Figure~\ref{fig:identity-braiding-nor-vertical-loop}. We therefore just have to show that this projects to the trivial element of the group $Q_{(\lambda,2)}(\bM)$. This time the group $Q_{(\lambda,2)}(\bM)$ is a product of $r'+r$ copies of $\bZ/2$, where $r'$ denotes the number of blocks of $\lambda$ with $\lambda_{i} \geq 2$; see Lemma~\ref{lem:transformation_groups_MCG_ell_2}. The first $r'$ copies of $\bZ/2$, in the projection to $Q_{(\lambda,2)}(\bM)$ of a loop of configurations, record the writhe of each block of strands; see \cite[Prop.~$1.1$]{Annex}. The remaining $r$ copies of $\bZ/2$ record, for each block of strands, the number of times modulo $2$ that a strand from that block passes through a crosscap; \cite[Prop.~$1.1$]{Annex}. As before, it is clear that the writhe of the loop of configurations illustrated in Figure~\ref{fig:identity-braiding-nor-vertical-loop} is trivial for each block; thus the first $r'$ coordinates of its projection to $Q_{(\lambda,2)}(\bM)$ are zero. Moreover, each strand in this loop of configurations passes around a crosscap an integer number of times, which corresponds to passing through a crosscap an even number of times; thus the last $r$ coordinates of its projection to $Q_{(\lambda,2)}(\bM)$ are also zero. This establishes equation~\eqref{eq:identity-braiding-nor-vertical}.
\end{proof}

\begin{rmk}
Equations \eqref{eq:identity-braiding-or} and \eqref{eq:identity-braiding-nor} of Lemma~\ref{lem:computations_MCG_braiding} hold for all $\ell \geq 1$. On the other hand, we used the explicit structure of the quotient group $Q_{(\lambda,2)}(\sS)$ (and the fact that $Q_{(\lambda,1)}(\sS)$ is trivial) to prove equations \eqref{eq:identity-braiding-or-vertical} and \eqref{eq:identity-braiding-nor-vertical}. For $\ell' \geq 3$ the proof shows that these equations hold up to a unit scalar, which is the image in $Q_{(\lambda,\ell')}(\sS)$ of the loops in Figures~\ref{fig:identity-braiding-or-vertical-loop} and \ref{fig:identity-braiding-nor-vertical-loop} for $\sS = \bT$ and $\sS = \bM$ respectively. We do not know whether this scalar is trivial in these cases.
\end{rmk}

\begin{figure}[tbp]
    \centering
    \begin{subfigure}[b]{0.48\textwidth}
        \centering
        \includegraphics[scale=0.5]{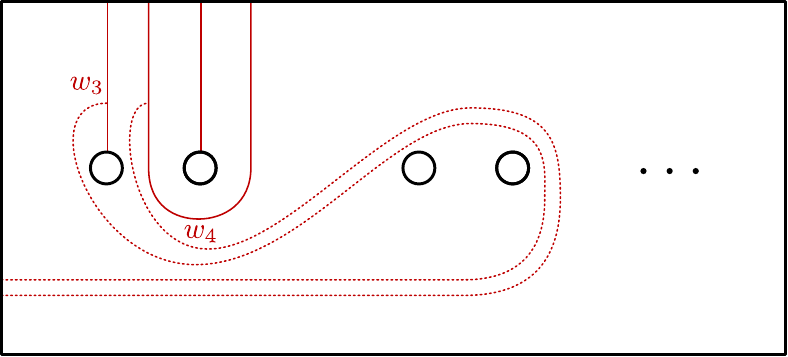}
        \caption{The left-hand side of equation~\eqref{eq:identity-braiding-or-vertical}.}
        \label{fig:identity-braiding-or-vertical}
    \end{subfigure}
    \hfill
    \begin{subfigure}[b]{0.48\textwidth}
        \centering
        \includegraphics[scale=0.5]{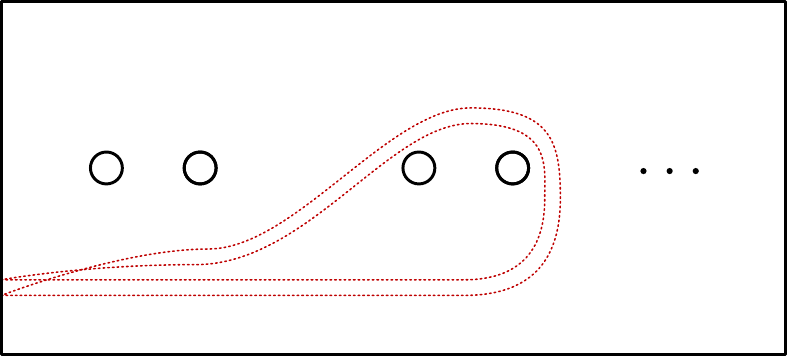}
        \caption{A loop given by the difference of two tethers.}
        \label{fig:identity-braiding-or-vertical-loop}
    \end{subfigure}
    \caption{The left-hand side of equation~\eqref{eq:identity-braiding-or-vertical} differs from the right-hand side of equation~\eqref{eq:identity-braiding-or-vertical} by the scalar in $\bZ[Q_{(\lambda,2)}(\bT)]$ given by the image in $Q_{(\lambda,2)}(\bT)$ of the loop illustrated on the right.}
\end{figure}

\begin{figure}[tbp]
    \centering
    \begin{subfigure}[b]{0.48\textwidth}
        \centering
        \includegraphics[scale=0.5]{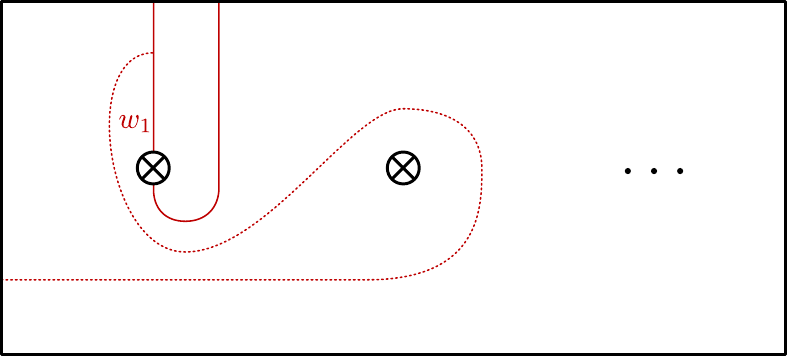}
        \caption{The left-hand side of equation~\eqref{eq:identity-braiding-nor-vertical}.}
        \label{fig:identity-braiding-nor-vertical}
    \end{subfigure}
    \hfill
    \begin{subfigure}[b]{0.48\textwidth}
        \centering
        \includegraphics[scale=0.5]{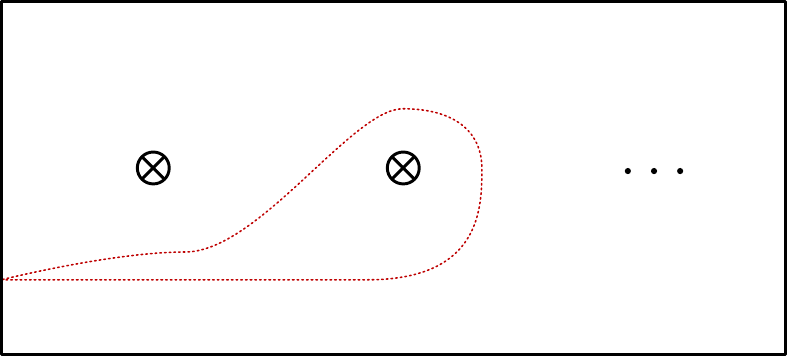}
        \caption{A loop given by the difference of two tethers.}
        \label{fig:identity-braiding-nor-vertical-loop}
    \end{subfigure}
    \caption{The left-hand side of equation~\eqref{eq:identity-braiding-nor-vertical} differs from the right-hand side of equation~\eqref{eq:identity-braiding-nor-vertical} by the scalar in $\bZ[Q_{(\lambda,2)}(\bM)]$ given by the image in $Q_{(\lambda,2)}(\bM)$ of the loop illustrated on the right.}
\end{figure}

\begin{rmk}
\label{rmk:interaction-with-braiding-dagger}
We will view equations \eqref{eq:identity-braiding-or-vertical} and \eqref{eq:identity-braiding-nor-vertical} as being statements about the action of $(\sigma_{1}^{-1})\dv = \sigma_{1}$, where $(-)\dv$ is the operation that inverts the braiding of a braided monoidal category; see the beginning of \S\ref{ss:categorical_framework}.
\end{rmk}

\paragraph*{Preliminary study of the difference functor.}
Similarly to $\fL^{\star}_{(\lambda,\ell)}(\tn)$, the $\MCG(\sS^{\natural n})$-module $\tau_{\one}\fL^{\star}_{(\lambda,\ell)}(\tn) = H_{k}^{\BM}(C_{\lambda}(\sS^{\natural (1+n)} \smallsetminus I);\bZ[Q^{\star}_{(\lambda,\ell)}(\sS)])$ is a free $\bZ[Q^{\star}_{(\lambda,\ell)}(\sS)]$-module with basis indexed by labellings of the embedded graph $\bW_{1+n} ^{\sS} \subset\sS^{\natural (1+n)}$ by words in the blocks of $\lambda$. Recall from Lemma~\ref{lem:extension_Quillen_source_homological_rep_functors} that the image of the morphism $[\one,\id_{\one\natural\tn}]\colon \tn\to\one\natural\tn$ of $\langle \cM,\cM \rangle$ under $\fL^{\star}_{(\lambda,\ell)}$ is the map $H_{k}^{\BM}(C_{\lambda}(\sS^{\natural n} \smallsetminus I);\bZ[Q^{\star}_{(\lambda,\ell)}(\sS)]) \to H_{k}^{\BM}(C_{\lambda}(\sS^{\natural (1+n)} \smallsetminus I);\bZ[Q^{\star}_{(\lambda,\ell)}(\sS)])$
induced by the inclusion of configuration spaces $C_{\lambda}(\sS^{\natural n} \smallsetminus I) \hookrightarrow C_{\lambda}(\sS^{\natural (1+n)} \smallsetminus I)$ coming from the boundary connected sum with the left-most copy of $\sS$. It thus follows that the map $\fL^{\star}_{(\lambda,\ell)}([\one,\id_{\one\natural\tn}])$ is the injection defined on basis elements by adjoining the empty word as the label of the left-most edge, or edges, of $\bW_{1+n}^{\sS}$, explicitly given by 
\begin{equation}
\label{eq:image_injection_hom_rep_functor_mcg}
\begin{cases}
([w_{1},w_{2}],\ldots,[w_{2n-1},w_{2n}])\mapsto ([\varnothing,\varnothing],[w_{1},w_{2}],\ldots,[w_{2n-1},w_{2n}]) & \textrm{if $\sS=\bT$;}\\
([w_{1}],\ldots,[w_{n}])\mapsto ([\varnothing],[w_{1}],\ldots,[w_{n}]) & \textrm{if $\sS=\bM$.}
\end{cases}
\end{equation}
Hence $\kappa_{\one}\fL^{\star}_{(\lambda,\ell)}(\tn)=0$ and $\delta_{\one}\fL^{\star}_{(\lambda,\ell)}(\tn)$ is the free $\bZ[Q^{\star}_{(\lambda,\ell)}(\sS)]$-module with generating set given by the tuples $([w_{0},w'_{0}],[w_{1},w_{2}],\ldots,[w_{2n-1},w_{2n}])$ such that $\vert w_{0} \vert + \vert w'_{0}\vert\geq 1$ if $\sS=\bT$, and the tuples $([w_{0}],[w_{1}],\ldots,[w_{n}])$ such that $\vert w_{0}\vert\geq 1$ if $\sS=\bM$.
Then, we define a $\bZ[Q^{\star}_{(\lambda,\ell)}(\sS)]$-module injection
\begin{equation}\label{eq:Delta'_def}
(\Delta'_{\one}\fL^{\star}_{(\lambda,\ell)})_{\tn}\colon \delta_{\one}\fL^{\star}_{(\lambda,\ell)}(\tn) \longhookrightarrow \tau_{\one}\fL^{\star}_{(\lambda,\ell)}(\tn)
\end{equation}
by sending each generating tuple $\bw=([w_{0}],[w_{1}],\ldots,[w_{n}])$ of $\delta_{\one}\fL^{\star}_{(\lambda,\ell)}(\tn)$ to itself in $\tau_{\one}\fL^{\star}_{(\lambda,\ell)}(\tn)$.

\paragraph*{Difference functor decomposition.}
From now on in \S\ref{ss:mcg_hom_rep_poly_preliminary}, we assume that $k\geq2$. In order to identify $\delta_{\one}\fL^{\star}_{(\lambda,\ell)}(\tn)$ in terms of some $\fL^{\star}_{(\lambda',\ell)}(\tn)$ where $\lambda'\prec \lambda$, we define some $\bZ[Q^{\star}_{(\lambda,\ell)}(\sS)]$-module morphisms depending on $\sS$ as follows.

\textbf{\emph{For $\sS=\bT$.}}
For each object $\tn\geq\zero$, we have to define several injections into $\delta_{\one}\fL^{\star}_{(\lambda,\ell)}(\MCGo)(\tn)$.
First, for each fixed pair of positive integers $(j_{1},j_{2})$ such that $1\leq j_{1} , j_{2}\leq r$, we define an injection $\tau_{\one}\fL^{\star}_{(\lambda[j_{1},j_{2}],\ell)}(\MCGo)(\tn)\hookrightarrow\delta_{\one}\fL^{\star}_{(\lambda,\ell)}(\MCGo)(\tn)$ by mapping each element $([w_{1},w_{2}],\ldots,[w_{2n+1},w_{2n+2}])$ to $([j_{1}w_{1},j_{2}w_{2}],\ldots,[w_{2n+1},w_{2n+2}])$. We denote by $(\mathfrak{i}_{(\lambda,\ell)}(\MCGo))_{\tn}^{\{\lambda - 2\}}$ the direct sum over $1\leq j_{1} , j_{2}\leq r$ of these injections.

Furthermore, we introduce two subfunctors of $\tau_{\one}\fL^{\star}_{(\lambda[j],\ell)}(\MCGo)$ (see Proposition~\ref{prop:tau_on_L_varnothing}), in order to define some more injections into $\delta_{\one}\fL^{\star}_{(\lambda,\ell)}(\MCGo)(\tn)$. For each object $\tm\geq\zero$, let $\tau_{\one}\fL^{\star}_{(\lambda[j],\ell)}(\MCGo)_{[-,\varnothing]} (\tm)$ and $\tau_{\one}\fL^{\star}_{(\lambda[j],\ell)}(\MCGo)_{[\varnothing,-]}(\tm)$ be the free $\bZ[Q^{\star}_{(\lambda,\ell)}(\bT)]$-modules with basis all the generators of $\tau_{\one}\fL^{\star}_{(\lambda[j],\ell)}(\MCGo)(\tm)$ of the form $([w_{1},\varnothing],\ldots,[w_{2m+1},w_{2m+2}])$ and $([\varnothing,w_{2}],\ldots,[w_{2m+1},w_{2m+2}])$ respectively. These modules are direct summands of $\tau_{\one}\fL^{\star}_{(\lambda[j],\ell)}(\MCGo)(\tm)$.
\begin{prop}\label{prop:tau_on_L_varnothing}
By gathering the $\bZ[Q^{\star}_{(\lambda,\ell)}(\bT)]$-modules $\{\tau_{\one}\fL^{\star}_{(\lambda[j],\ell)}(\MCGo)_{[-,\varnothing]} (\tm)\}_{\tm\in\obj(\M^{+})}$ and $\{\tau_{\one}\fL^{\star}_{(\lambda[j],\ell)}(\MCGo)_{[\varnothing,-]}(\tm)\}_{\tm\in\obj(\M^{+})}$ respectively, and with the assignment of the functor $\tau_{\one}\fL^{\star}_{(\lambda[j],\ell)}(\MCGo)$ on the morphisms of $\langle\M^{+},\M^{+}\rangle$, we define functors $\tau_{\one}\fL^{\star}_{(\lambda[j],\ell)}(\MCGo)_{[-,\varnothing]}$ and $\tau_{\one}\fL^{\star}_{(\lambda[j],\ell)}(\MCGo)_{[\varnothing,-]}$ of the category $\Fct(\langle\M^{+},\M^{+}\rangle,{\bZ[Q^{\star}_{(\lambda,\ell)}(\bT)]}\lmod^{\bullet})$. Moreover, these are direct summands of the functor $\tau_{\one}\fL^{\star}_{(\lambda[j],\ell)}(\MCGo)$.
\end{prop}
\begin{proof}
We focus here on the proof for $\tau_{\one}\fL^{\star}_{(\lambda[j],\ell)}(\MCGo)_{[-,\varnothing]}$, that for $\tau_{\one}\fL^{\star}_{(\lambda[j],\ell)}(\MCGo)_{[\varnothing,-]}$ being a mutatis mutandis replication.
For each $\tm\geq\zero$, we denote by $\mathtt{i}_{[-,\varnothing]}(\tm)$ (resp.~$\mathtt{p}_{[-,\varnothing]}(\tm)$) the canonical injection (resp.~projection) associated to the direct summand $\tau_{\one}\fL^{\star}_{(\lambda[j],\ell)}(\MCGo)_{[-,\varnothing]}(\tm)$ of $\tau_{\one}\fL^{\star}_{(\lambda[j],\ell)}(\MCGo)(\tm)$. For all $\varphi \in \MCGo_{m,1}$, it follows from Lemma~\ref{lem:invariance_translation_block_MCG_action} that $\fL^{\star}_{(\lambda[j],\ell)}(\MCGo) (\id_{\one}\natural \varphi)$ does not affect the first two entries of each generator $([w_{1},\varnothing],\ldots,[w_{2m+1},w_{2m+2}])$, so the image of the composite $\tau_{\one}\fL^{\star}_{(\lambda[j],\ell)}(\MCGo) (\varphi)\circ \mathtt{i}_{[-,\varnothing]}(\tm)$ belongs to the module $\tau_{\one}\fL^{\star}_{(\lambda[j],\ell)}(\MCGo)_{[-,\varnothing]}(\tm)$. We deduce that the $\bZ[Q^{\star}_{(\lambda,\ell)}(\bT)]$-modules $\{\tau_{\one}\fL^{\star}_{(\lambda[j],\ell)}(\MCGo)_{[-,\varnothing]} (\tm)\}_{\tm\in\obj(\M^{+})}$ define a functor $\tau_{\one}\fL^{\star}_{(\lambda[j],\ell)}(\MCGo)_{[-,\varnothing]}\colon \M^{+}\to {\bZ[Q^{\star}_{(\lambda,\ell)}(\bT)]}\lmod^{\bullet}$ by assigning the $\MCGo_{m,1}$-module structure of $\tau_{\one}\fL^{\star}_{(\lambda[j],\ell)}(\MCGo)(\tm)$ for each $\tm\geq\zero$.

We now extend the functor $\tau_{\one}\fL^{\star}_{(\lambda[j],\ell)}(\MCGo)_{[-,\varnothing]}$ along the inclusion $\M^{+}\hookrightarrow \langle\M^{+},\M^{+}\rangle$. Recall from \eqref{eq:formula_morphism_{i}d_plus_morphism} that $\tau_{\one}\fL^{\star}_{(\lambda[j],\ell)}(\MCGo)([\one,\id_{\one \natural \tm}])=\fL^{\star}_{(\lambda[j],\ell)}(\MCGo)(\sigma_{1}^{-1})\circ\fL^{\star}_{(\lambda[j],\ell)}(\MCGo)([\one,\id_{\two \natural \tm}])$, where $\sigma_{1}=b^{\M^{+}}_{\one,\one}\natural \id_{\tm}$ and $b^{\M^{+}}_{\one,\one}$ is the braiding of $\M^{+}$ (see Figure~\ref{fig:braiding-mcg}). Note that $\fL^{\star}_{(\lambda[j],\ell)}(\MCGo)([\one,\id_{\two \natural \tm}])$ has a similar description to that of $\fL^{\star}_{(\lambda[j],\ell)}(\MCGo)([\one,\id_{\one\natural\tm}])$, by simply replacing $n$ with $n+1$ in the defining assignment \eqref{eq:image_injection_hom_rep_functor_mcg}. Then, we deduce that $(\fL^{\star}_{(\lambda[j],\ell)}(\MCGo)([\one,\id_{\two \natural \tm}]))([w_{1},\varnothing],\ldots,[w_{2m+1},w_{2m+2}])$ is equal to $([\varnothing,\varnothing],[w_{1},\varnothing],\ldots,[w_{2m+1},w_{2m+2}])$ for each element of $\tau_{\one}\fL^{\star}_{(\lambda[j],\ell)}(\MCGo)_{[-,\varnothing]} (\tm)$. So it follows from equation \eqref{eq:identity-braiding-or} of Lemma~\ref{lem:computations_MCG_braiding} that $\tau_{\one}\fL^{\star}_{(\lambda[j],\ell)}(\MCGo)([\one,\id_{\one \natural \tm}])([w_{1},\varnothing],\ldots,[w_{2m+1},w_{2m+2}])=([w_{1},\varnothing],[\varnothing,\varnothing],\ldots,[w_{2m+1},w_{2m+2}])$, and thus belongs to $\tau_{\one}\fL^{\star}_{(\lambda[j],\ell)}(\MCGo)_{[-,\varnothing]} (\one \natural\tm)$. Relation \eqref{eq:criterion'} is then automatically satisfied since $\tau_{\one}\fL^{\star}_{(\lambda[j],\ell)}(\MCGo)$ is a functor out of $\langle\M^{+},\M^{+}\rangle$; it thus follows from Lemma~\ref{lem:extend_functor_Quillen} that $\tau_{\one}\fL^{\star}_{(\lambda[j],\ell)}(\MCGo)_{[-,\varnothing]}$ extends to an object of $\Fct(\langle\M^{+},\M^{+}\rangle,{\bZ[Q^{\star}_{(\lambda,\ell)}(\bT)]}\lmod^{\bullet})$.

Finally, a straightforward adaptation of the above reasoning proves that all of the projection maps $\{\mathtt{p}_{[-,\varnothing]}(\tm)\}_{\tm\in\obj(\M^{+})}$ are equivariant with respect to all of the morphisms of the category $\langle\M^{+},\M^{+}\rangle$. Therefore, the subfunctor $\tau_{\one}\fL^{\star}_{(\lambda[j],\ell)}(\MCGo)_{[-,\varnothing]}$ is a direct summand of the functor $\tau_{\one}\fL^{\star}_{(\lambda[j],\ell)}(\MCGo)$.
\end{proof}

Now, for each $1\leq j\leq r$, we define an injection $\tau_{\one}\fL^{\star}_{(\lambda[j],\ell)}(\MCGo)_{[-,\varnothing]}(\tn)\hookrightarrow\delta_{\one}\fL^{\star}_{(\lambda,\ell)}(\MCGo)(\tn)$ by mapping each element $([w_{1},\varnothing],\ldots,[w_{2n+1},w_{2n+2}])$ to $([jw_{1},\varnothing],\ldots,[w_{2n+1},w_{2n+2}])$, and an injection $\tau_{\one}\fL^{\star}_{(\lambda[j],\ell)}(\MCGo)_{[\varnothing,-]}(\tn)\hookrightarrow\delta_{\one}\fL^{\star}_{(\lambda,\ell)}(\MCGo)(\tn)$ by mapping each element $([\varnothing,w_{2}],\ldots,[w_{2n+1},w_{2n+2}])$ to $([\varnothing,jw_{2}],\ldots,[w_{2n+1},w_{2n+2}])$.
We denote by $(\mathfrak{i}_{(\lambda,\ell)}(\MCGo))_{\tn}^{\{\lambda - 1\}}$ the direct sum over $1\leq j\leq r$ of these two injections, and by $(\mathfrak{i}_{(\lambda,\ell)}(\MCGo))_{\tn}$ the direct sum of the maps $(\mathfrak{i}_{(\lambda,\ell)}(\MCGo))_{\tn}^{\{\lambda - 1\}}$ and $(\mathfrak{i}_{(\lambda,\ell)}(\MCGo))_{\tn}^{\{\lambda - 2\}}$. Finally, we use the notation $\tau_{\one}\fL^{\star}_{(\lambda[j],\ell)}(\MCGo)_{[\varnothing]}$ for the subfunctor $\tau_{\one}\fL^{\star}_{(\lambda[j],\ell)}(\MCGo)_{[-,\varnothing]}\oplus \tau_{\one}\fL^{\star}_{(\lambda[j],\ell)}(\MCGo)_{[\varnothing,-]}$ of $\tau_{\one}\fL^{\star}_{(\lambda[j],\ell)}(\MCGo)$.

\textbf{\emph{For $\sS=\bM$.}} For each $1\leq j\leq r$ and each $\tn\geq\zero$, we define an injection $\tau_{\one}\fL^{\star}_{(\lambda[j],\ell)}(\MCGno)(\tn)\hookrightarrow\delta_{\one}\fL^{\star}_{(\lambda,\ell)}(\MCGno)(\tn)$ by mapping each element $([w_{1}],\ldots,[w_{n+1}])$ to $([j w_{1}],\ldots,[w_{n+1}])$. We denote by $(\mathfrak{i}_{(\lambda,\ell)}(\MCGno))_{\tn}$ the direct sum over $1\leq j\leq r$ of these injections.

\begin{notation}
To abbreviate and unify our notation, we set $(\mathfrak{i}_{(\lambda,\ell)})_{\tn}$ to be the map $(\mathfrak{i}_{(\lambda,\ell)}(\MCGo))_{\tn}$ when $\sS=\bT$, and the map $(\mathfrak{i}_{(\lambda,\ell)}(\MCGno))_{\tn}$ when $\sS=\bM$. We also denote by $\bigoplus_{\lambda'}\tau_{\one}\fL^{\star}_{(\lambda',\ell)}$ the functor $(\bigoplus_{1\leq j_{1},j_{2}\leq r}\tau_{\one}\fL^{\star}_{(\lambda[j_{1},j_{2}],\ell)}(\MCGo)) \oplus (\bigoplus_{1\leq j\leq r}\tau_{\one}\fL^{\star}_{(\lambda[j],\ell)}(\MCGo)_{[\varnothing]})$ for the case when $\sS=\bT$, and the functor $\bigoplus_{1\leq j\leq r}\tau_{\one}\fL^{\star}_{(\lambda[j],\ell)}(\MCGno)$ for the case when $\sS=\bM$.
\end{notation}

An elementary check of the bases of $\delta_{\one}\fL^{\star}_{(\lambda,\ell)}(\tn)$ and $\bigoplus_{\lambda'}\tau_{\one}\fL^{\star}_{(\lambda',\ell)}(\tn)$ shows that the morphism $(\mathfrak{i}_{(\lambda,\ell)})_{\tn}$ is an isomorphism of (free) $\bZ[Q^{\star}_{(\lambda,\ell)}(\sS)]$-modules for each $\tn\geq\zero$:
\begin{equation}
\label{eq:isom-of-modules_mcg}
(\mathfrak{i}_{(\lambda,\ell)})_{\tn}\colon \bigoplus_{\lambda'}\tau_{\one}\fL^{\star}_{(\lambda',\ell)}(\tn) \overset{\cong}{\longrightarrow} \delta_{\one}\fL^{\star}_{(\lambda,\ell)}(\tn).
\end{equation}

A similar study to the above holds in the vertical setting. Namely, Proposition~\ref{lem:module_structure_BM_homology-check} provides $\fLv_{(\lambda,\ell)}(\tn)$ with the analogous free $\bZ[Q^{\star}_{(\lambda,\ell)}(\sS)]$-module structures, indexed by the set of tuples labelling the embedded ``vertical'' graphs modelled by Figures~\ref{fig:model-or-mcg-vertical-dual} and \ref{fig:model-nor-mcg-vertical-dual}. The above reasoning then repeats mutatis mutandis to provide an analogous isomorphism $(\mathfrak{i}^{\vrtcl}_{(\lambda,\ell)})_{\tn}$ to \eqref{eq:isom-of-modules_mcg}.

\subsubsection{Classical homological representation functors}\label{sss:SES_classical_MCG}

We prove here Theorem~\ref{thm:ses} for the homological representation functors of mapping class groups in the classical (non-vertical) setting. Following our notation so far in \S\ref{ss:SES_mcg}, we make all further reasoning, as much as possible, on the (generic) functor $\fL^{\star}_{(\lambda,\ell)}$, which denotes any one of the functors \eqref{eq:hom_rep_functor_rep_mcg_o} and \eqref{eq:hom_rep_functor_rep_mcg_no_ell}.

\begin{thm}\label{thm:key_SES_mapping_class_groups}
For each ordered partition $\lambda\vdash k\geq2$, each integer $\ell\geq1$ and $\star$ either the blank space or $\star=\unt$, the exact sequence \eqref{eq:ESCaract} induces the following isomorphisms in the functor categories $\Fct(\langle\M^{+},\M^{+}\rangle,{\bZ[Q^{\star}_{(\lambda,\ell)}(\bT)]}\lmod^{\bullet})$ and $\Fct(\langle\M^{-},\M^{-}\rangle,{\bZ[Q^{\star}_{(\lambda,\ell)}(\bM)]}\lmod^{\bullet})$ respectively:
\begin{align}
\begin{split}
\label{eq:key_SES_MCG_orientable}
\tau_{\one}\fL^{\star}_{(\lambda,\ell)}(\MCGo) &\cong \fL^{\star}_{(\lambda,\ell)}(\MCGo) \mediumoplus \left(\underset{1\leq j_{1}, j_{2}\leq r}{\bigoplus}\tau_{\one}\fL^{\star}_{(\lambda[j_{1},j_{2}],\ell)}(\MCGo)\right) \mediumoplus \left(\underset{1\leq j\leq r}{\bigoplus} \tau_{\one}\fL^{\star}_{(\lambda[j],\ell)}(\MCGo)_{[\varnothing]}\right),
\end{split}
\\
\begin{split}
\label{eq:key_SES_MCG_non_orientable}
\tau_{\one}\fL^{\star}_{(\lambda,\ell)}(\MCGno) &\cong \fL^{\star}_{(\lambda,\ell)}(\MCGno) \mediumoplus \left( \underset{1\leq j\leq r}{\bigoplus}\tau_{\one}\fL^{\star}_{(\lambda[j],\ell)}(\MCGno)\right).
\end{split}
\end{align}
\end{thm}

\begin{proof}
As a consequence of the above preliminary study of \S\ref{ss:mcg_hom_rep_poly_preliminary}, it suffices to show that the $\bZ[Q^{\star}_{(\lambda,\ell)}(\sS)]$-module isomorphisms $\{(\mathfrak{i}_{(\lambda,\ell)})_{\tn}\}_{\tn\in \obj(\cM)}$ (see \eqref{eq:isom-of-modules_mcg}) assemble into an isomorphism $\mathfrak{i}_{(\lambda,\ell)} \colon \bigoplus_{\lambda'}\tau_{\one}\fL^{\star}_{(\lambda',\ell)} \overset{\sim}{\to} \delta_{\one}\fL^{\star}_{(\lambda,\ell)}$ in $\Fct(\langle\cM,\cM\rangle,{\bZ[Q^{\star}_{(\lambda,\ell)}(\sS)]}\lmod^{\bullet})$, while the $\bZ[Q^{\star}_{(\lambda,\ell)}(\sS)]$-module injections $\{(\Delta'_{\one}\fL^{\star}_{(\lambda,\ell)})_{\tn}\}_{\tn\in \obj(\cM)}$ (see \eqref{eq:Delta'_def}) assemble into a natural transformation that is a section of $\Delta_{\one}\fL^{\star}_{(\lambda,\ell)}$.

First, we prove the commutation of $(\mathfrak{i}_{(\lambda,\ell)})_{\tn}$ and of $(\Delta'_{\one}\fL^{\star}_{(\lambda,\ell)})_{\tn}$ with respect to the action of $\MCG(\sS^{\natural n})$.
For each element $\rho$ of $\MCG(\sS^{\natural n})$ and for each direct summand $\tau_{\one}\fL^{\star}_{(\lambda',\ell)}$ of the source of $(\mathfrak{i}_{(\lambda,\ell)})_{\tn}$ (see \eqref{eq:isom-of-modules_mcg}), the morphisms $\delta_{\one}\fL^{\star}_{(\lambda,\ell)}(\rho)$ and $\tau_{\one}\fL^{\star}_{(\lambda',\ell)}(\rho)$ are induced by the actions of $\fL^{\star}_{(\lambda,\ell)}(\id_{\one} \natural \rho)$ and of $\fL^{\star}_{(\lambda',\ell)}(\id_{\one} \natural \rho)$ respectively on the Borel-Moore homology classes supported on the embedded graph $\bW^{\sS}_{1+n} \subset \sS^{\natural (1+n)}$.
It follows from Lemma~\ref{lem:invariance_translation_block_MCG_action} that the actions of $\fL^{\star}_{(\lambda,\ell)}(\id_{\one} \natural \rho)$ and of $\fL^{\star}_{(\lambda',\ell)}(\id_{\one} \natural \rho)$ do not affect the first two (if $\sS=\bT$) or one (if $\sS=\bM$) entries of a tuple corresponding to a generator.
Since $(\mathfrak{i}_{(\lambda,\ell)})_{\tn}$ and $(\Delta'_{\one}\fL^{\star}_{(\lambda,\ell)})_{\tn}$ both only affect the first two (if $\sS=\bT$) or one (if $\sS=\bM$) entries of a tuple by their definitions, we deduce that $\delta_{\one}\fL^{\star}_{(\lambda,\ell)}(\rho) \circ (\mathfrak{i}_{(\lambda,\ell)})_{\tn}= (\mathfrak{i}_{(\lambda,\ell)})_{\tn}\circ\bigoplus_{\lambda'}\tau_{\one}\fL^{\star}_{(\lambda',\ell)}(\rho)$ and $(\Delta'_{\one}\fL^{\star}_{(\lambda,\ell)})_{\tn}\circ \delta_{\one}\fL^{\star}_{(\lambda,\ell)}(\rho) = \tau_{\one}\fL^{\star}_{(\lambda,\ell)}(\rho)\circ(\Delta'_{\one}\fL^{\star}_{(\lambda,\ell)})_{\tn}$.
Therefore, the morphisms $\{(\mathfrak{i}_{(\lambda,\ell)})_{\tn}\}_{\tn\in\obj(\cM)}$ and $\{(\Delta'_{\one}\fL^{\star}_{(\lambda,\ell)})_{\tn}\}_{\tn\in\obj(\cM)}$ define natural transformations $\mathfrak{i}_{(\lambda,\ell)}$ and $\Delta'_{\one}\fL^{\star}_{(\lambda,\ell)}$ in $\Fct(\cM,{\bZ[Q^{\star}_{(\lambda,\ell)}(\sS)]}\lmod^{\bullet})$.

Now, by Lemma~\ref{lem:criterionnaturaltransfo}, it suffices to check relation \eqref{eq:criterion''} in order to prove that $\mathfrak{i}_{(\lambda,\ell)}$ and $\Delta'_{\one}\fL^{\star}_{(\lambda,\ell)}$ are actually natural transformations of functors on $\langle\cM,\cM\rangle$. We consider an object $\tn\geq\one$, the proof being trivial for $\tn=\zero$. We note from the formula \eqref{eq:formula_morphism_{i}d_plus_morphism}, from the composition rule \eqref{eq:composition_rule}, from the fact that $\natural$ is a strict monoidal structure and from the functoriality of $\fL^{\star}_{(\lambda,\ell)}$, that $\tau_{\one}\fL^{\star}_{(\lambda,\ell)}([\one,\id_{\one \natural \tn}])=\fL^{\star}_{(\lambda,\ell)}(\sigma_{1}^{-1})\circ\fL^{\star}_{(\lambda,\ell)}([\one,\id_{\two \natural \tn}])$, where $\sigma_{1}\in \Aut_{\cM}(\sS\natural \sS\natural \sS^{n})$ is the element $b^{\cM}_{\one,\one}\natural \id_{\tn}$ defined by the braiding $b^{\cM}_{\one,\one} \colon \one\natural \one\cong \one\natural \one$ (see Figure~\ref{fig:braiding-mcg}).
Note that $\fL^{\star}_{(\lambda,\ell)}([\one,\id_{\two \natural \tn}])$ has a similar description to that of $\fL^{\star}_{(\lambda,\ell)}([\one,\id_{\one\natural\tn}])$, by simply replacing $n$ with $n+1$ in the defining assignment \eqref{eq:image_injection_hom_rep_functor_mcg}. More precisely, it is the map induced by the embedding of $\bW^{\sS}_{1+n}$ into $\bW^{\sS}_{2+n}$ given by sending the $i$th edge $(\bS^1-\mathrm{pt})_{i}$ to the $(i+2)$nd edge $(\bS^1-\mathrm{pt})_{i+2}$ (if $\sS=\bT$) or the $(i+1)$st edge $(\bS^1-\mathrm{pt})_{i+1}$ (if $\sS=\bM$).
In particular this implies that, in the image of $\fL^{\star}_{(\lambda,\ell)}([\one,\id_{\two \natural \tn}])$, there are no configuration points on the two first edges $(\bS^1-\mathrm{pt})_{1}$ and $(\bS^1-\mathrm{pt})_{2}$ if $\sS=\bT$ or on the first edge $(\bS^1-\mathrm{pt})_{1}$ if $\sS=\bM$. Then the morphism $\fL^{\star}_{(\lambda,\ell)}(\sigma_{1}^{-1})$ corresponds to the action of $\sigma_{1}^{-1}$ on $\bW^{\sS}_{2+n}$. It therefore follows from equations \eqref{eq:identity-braiding-or} and \eqref{eq:identity-braiding-nor} of Lemma~\ref{lem:computations_MCG_braiding} that for any generator $\bw$ of $\bigoplus_{\lambda'}\tau_{\one}\fL^{\star}_{(\lambda',\ell)}(\tn)$, both $\delta_{\one}\fL^{\star}_{(\lambda,\ell)}([\one,\id_{\one \natural \tn}])((\mathfrak{i}_{(\lambda,\ell)})_{\tn}(\bw))$ and $(\mathfrak{i}_{(\lambda,\ell)})_{\one+\tn}(\bigoplus_{\lambda'}\tau_{\one}\fL^{\star}_{(\lambda',\ell)}([\one,\id_{\one \natural \tn}])(\bw))$ are equal to $\mathfrak{o}((\mathfrak{i}_{(\lambda,\ell)})_{\tn}(\bw))$, where $\mathfrak{o}$ denotes the operation
\[
\begin{cases}
([w_1,w_2],[w_3,w_4],\ldots) \longmapsto ([w_1,w_2],[\varnothing,\varnothing],[w_3,w_4],\ldots) & \text{if } \sS=\bT \\
([w_1],[w_2],\ldots) \longmapsto ([w_1],[\varnothing],[w_2],\ldots) & \text{if } \sS=\bM.
\end{cases}
\]

The same arguments using Lemma~\ref{lem:computations_MCG_braiding} also prove that $\tau_{\one}\fL^{\star}_{(\lambda,\ell)}([\one,\id_{\one \natural \tn}])((\Delta'_{\one}\fL^{\star}_{(\lambda,\ell)})_{\tn}(\bw'))$ and $(\Delta'_{\one}\fL^{\star}_{(\lambda,\ell)})_{\one+\tn}(\delta_{\one}\fL^{\star}_{(\lambda,\ell)}([\one,\id_{\one \natural \tn}])(\bw'))$ are both equal to $\mathfrak{o}(\bw')$ for any generator $\bw'$ of $\delta_{\one}\fL^{\star}_{(\lambda,\ell)}(\tn)$.
It then follows from the above equalities and induction on $\tm\geq\one$ that the collections of morphisms $\{(\mathfrak{i}_{(\lambda,\ell)})_{\tn}\}_{\tn\in\obj(\cM)}$ and $\{(\Delta'_{\one}\fL^{\star}_{(\lambda,\ell)})_{\tn}\}_{\tn\in\obj(\cM)}$ commute with the action of $[\tm,\id_{\tm\natural\tn}]$ for each $\tm\geq\one$.
Hence relation \eqref{eq:criterion''} is satisfied for all $\tn\in\obj(\cM)$ and Lemma~\ref{lem:criterionnaturaltransfo} implies that $\mathfrak{i}_{(\lambda,\ell)}$ and $\Delta'_{\one}\fL^{\star}_{(\lambda,\ell)}$ are natural transformations of functors out of $\langle\cM,\cM\rangle$. In particular, we deduce that $\mathfrak{i}_{(\lambda,\ell)}$ is an isomorphism in $\Fct(\langle\cM,\cM\rangle,{\bZ[Q^{\star}_{(\lambda,\ell)}(\sS)]}\lmod^{\bullet})$ between the functors $\delta_{\one}\fL^{\star}_{(\lambda,\ell)}$ and $\bigoplus_{\lambda'}\tau_{\one}\fL^{\star}_{(\lambda',\ell)}$. Also, it follows from the definition of $(\Delta'_{\one}\fL^{\star}_{(\lambda,\ell)})_{\tn}$ (see \eqref{eq:Delta'_def}) that $(\Delta_{\one}\fL^{\star}_{(\lambda,\ell)})_{\tn}\circ (\Delta'_{\one}\fL^{\star}_{(\lambda,\ell)})_{\tn}=\id_{\delta_{\one}\fL^{\star}_{(\lambda,\ell)}(\tn)}$, and so $\Delta'_{\one}\fL^{\star}_{(\lambda,\ell)}$ is a section of $\Delta_{\one}\fL^{\star}_{(\lambda,\ell)}$ in $\Fct(\langle\cM,\cM\rangle,{\bZ[Q^{\star}_{(\lambda,\ell)}(\sS)]}\lmod^{\bullet})$. Since $\kappa_{\one}\fL^{\star}_{(\lambda,\ell)}=0$ (because $i_{\one}(\fL^{\star}_{(\lambda,\ell)})_{\tn}$ is clearly injective for each $\tn\in \cM$; see \eqref{eq:image_injection_hom_rep_functor_mcg}), we deduce that the exact sequence \eqref{eq:ESCaract} is a split short exact sequence, which provides the isomorphisms \eqref{eq:key_SES_MCG_orientable} and \eqref{eq:key_SES_MCG_non_orientable}.

In the above arguments, whenever we deal with a \emph{twisted} homological representation functor, we stress that the action on the ground ring $\bZ[Q^{\star}_{(\lambda,\ell)}(\sS)]$ does not affect any of the above reasoning, since $\fL^{\star}_{(\lambda,\ell)}$ and $\bigoplus_{\lambda'}\tau_{\one}\fL^{\star}_{(\lambda',\ell)}$ are equipped with the same action as $\tau_{\one}\fL^{\star}_{(\lambda,\ell)}$ via the change of rings operation of Convention~\ref{conv:transformation_group_summand}; see Lemma~\ref{lem:change_of_rings_operations}.
\end{proof}

\subsubsection{Vertical-type alternatives}\label{sss:SES_MCG_alternative}

We now deal with the vertical-type alternatives of the homological representation functors for the mapping class groups of surfaces introduced in \S\ref{sss:representations_mapping_class_groups}.
We consider the functors $\fLv_{(\lambda,\ell')}(\MCGo)$ for orientable surfaces and the functors $\fLv_{(\lambda,\ell')}(\MCGno)$ for non-orientable surfaces, only for $\ell'\leq 2$. (We recall from Lemma~\ref{lem:Qu_ell=2_useless} that $\fLv_{(\lambda,\ell')}=\fLuv_{(\lambda,\ell')}$ in this case.) In particular, we do not consider the functors $\fLv_{(\lambda,\ell)}(\MCGno)$ with $\ell\geq3$ for non-orientable surfaces. This is because the proof of Theorem~\ref{thm:SES_MCG_alternatives} below relies on the identities proven in Lemma~\ref{lem:computations_MCG_braiding} using the specific structure of the transformation groups $Q_{(\lambda,2)}(\bT)$ and $Q_{(\lambda,2)}(\bM)$ (see Lemma~\ref{lem:transformation_groups_MCG_ell_2}); in contrast, the groups $Q_{(\lambda,\ell)}(\bM)$ are not known for $\ell\geq 3$. We however conjecture that all of the following arguments, including the results of Theorem~\ref{thm:SES_MCG_alternatives} and Theorem~\ref{thm:polynomiality_vs_vertical}, also hold for $\ell\geq 3$.

\begin{thm}\label{thm:SES_MCG_alternatives}
For each ordered partition $\lambda\vdash k\geq2$, the exact sequence \eqref{eq:ESCaract} induces the analogous isomorphisms to \eqref{eq:key_SES_MCG_orientable} and \eqref{eq:key_SES_MCG_non_orientable} for the functors $\fLv_{(\lambda,1)}(\MCGo)$, $\fLv_{(\lambda,2)}(\MCGo)$, $\fLv_{(\lambda,1)}(\MCGno)$ and $\fLv_{(\lambda,2)}(\MCGno)$.
\end{thm}

\begin{proof}
We fix $\ell'\in\{1,2\}$. The analogous isomorphisms to \eqref{eq:key_SES_MCG_orientable} and \eqref{eq:key_SES_MCG_non_orientable} for the vertical-type alternatives follow mutatis mutandis from the proof of Theorem~\ref{thm:key_SES_mapping_class_groups} by defining analogues $(\mathfrak{i}^{\vrtcl}_{(\lambda,\ell')})_{\tn}$ and $(\Delta'_{\one}\fLv_{(\lambda,\ell')})_{\tn}$ to the morphisms $(\mathfrak{i}_{(\lambda,\ell')})_{\tn}$ and $(\Delta'_{\one}\fL_{(\lambda,\ell')})_{\tn}$ for each $\tn\in\obj(\cM)$. The proof that these define natural transformations $\mathfrak{i}^{\vrtcl}_{(\lambda,\ell')}$ and $\Delta'_{\one}\fLv_{(\lambda,\ell')}$ in $\Fct(\cM,{\bZ[Q_{(\lambda,\ell')}]}\lmod)$ is a verbatim repetition of the first part of the proof of Theorem~\ref{thm:key_SES_mapping_class_groups}, again using the disjoint support argument of Lemma~\ref{lem:invariance_translation_block_MCG_action}. Then, the proof that $\mathfrak{i}^{\vrtcl}_{(\lambda,\ell')}$ and $\Delta'_{\one}\fLv_{(\lambda,\ell')}$ are natural transformations in $\Fct(\langle\cM\dv,\cM\dv\rangle,{\bZ[Q_{(\lambda,\ell')}]}\lmod)$ is the same as the second part of the proof of Theorem~\ref{thm:key_SES_mapping_class_groups}, except that we now use equations \eqref{eq:identity-braiding-or-vertical} and \eqref{eq:identity-braiding-nor-vertical} of Lemma~\ref{lem:computations_MCG_braiding} to understand the action of the braiding.
\end{proof}

\begin{rmk}\label{rmk:SES_alternative_dagger_MCG}
We could define the functors $\fLv_{(\lambda,1)}(\MCGo)$, $\fLv_{(\lambda,2)}(\MCGo)$, $\fLv_{(\lambda,1)}(\MCGno)$ and $\fLv_{(\lambda,2)}(\MCGno)$ on the categories $\langle\M^{+},\M^{+}\rangle$ and $\langle\M^{-},\M^{-}\rangle$ respectively (i.e.~without the opposite convention for the braiding of $\M$ induced by the $\dv$ endofunctor; see \S\ref{sss:representations_mapping_class_groups}). However, in this setting it is not clear that there are isomorphisms analogous to \eqref{eq:key_SES_MCG_orientable} and \eqref{eq:key_SES_MCG_non_orientable}.
\end{rmk}

\section{Polynomiality}\label{s:poly_homol_rep-func}

In this section, we recollect the theory of polynomial functors in \S\ref{s:notions_polynomiality}, then prove in \S\ref{ss:polynomiality_surface_braid_groups}--\S\ref{ss:poly_mcg} the polynomiality results of Corollary~\ref{coro:polynomiality} and Theorem~\ref{thm:polynomiality_vs_vertical}, and finally prove Corollary~\ref{coro:analyticity} in \S\ref{ss:analyticity}.
\subsection{Notions of polynomiality}\label{s:notions_polynomiality}

We review here the notions and basic properties of \emph{strong}, \emph{very strong}, \emph{split} and \emph{weak} polynomial functors.
The first definitions of polynomial functors date back to Eilenberg and Mac Lane in \cite{EilenbergMacLane} for functors on module categories. This notion has progressively been extended to deal with a more general framework, and has been the object of intensive study because of its applications in representation theory (see Djament, Touzé and Vespa \cite{djamenttouzevespa}), group cohomology (see Franjou, Friedlander, Scorichenko and Suslin \cite{FranjouFriedlanderScorichenkoSuslin}) and homological stability with twisted coefficients (see Randal-Williams and Wahl \cite{RWW}). In particular, Djament and Vespa \cite[\S 1]{DV3} introduce the notions of \emph{strong} and \emph{weak} polynomial functors in the context of a functor category $\Fct(\cC,\cA)$, where $\cC$ is a symmetric monoidal category where the unit is an initial object and $\cA$ is a Grothendieck category (see the definition below). They are then extended to the case where $\cC$ is pre-braided monoidal in \cite{soulieLM1,soulieLMgeneralised}, which also introduce the notion of \emph{very strong} polynomial functor. See also \cite{palmer2017comparison} for a comparison of the various instances of polynomial functors.
All these notions extend verbatim to the present slightly larger framework from the previous literature on this topic (see \cite[\S 4]{soulieLMgeneralised} for instance), the various proofs being mutatis mutandis generalisations of these previous works.
We also define the notion of \emph{split} polynomial functor, a particular kind of very strong polynomial functor, following an analogous notion from \cite{RWW}.

For the remainder of \S\ref{s:notions_polynomiality}, we fix a strict left-module groupoid $(\cM,\natural)$ over a braided strict monoidal groupoid $(\cG,\natural,\zero)$ satisfying the same assumptions as in \S\ref{sss:translation_background}: $(\cG,\natural,\zero)$ has no zero divisors, $\mathrm{Aut}_{\cG}(\zero)=\{ \id_{\zero}\}$, $\cM$ and $\cG$ are both small and skeletal, have the same set of objects identified with the non-negative integers $\bN$ with the standard notation $\tn$ to denote an object, and both the monoidal and module structures $\natural$ are given on objects by addition. For example, one quickly checks that all of the examples of $\cM$ and $\cG$ defined in \S\ref{ss:examples_homological_representations_functors} satisfy all of these assumptions. We also consider a Grothendieck category $\cA$, i.e.~a cocomplete abelian category, which admits a generator, and in which filtered colimits of exact sequences are exact. In particular, we recall from \cite[\S1.7, d)]{GrothendieckTohoku} that the functor category $\Fct(\langle \cG,\cM \rangle ,\cA)$ is a Grothendieck category.

\paragraph*{Strong, very strong and split polynomial functors.}

The category of \emph{strong} polynomial functors of degree at most $d\in \bN$, denoted by $\mathcal{P}ol^{str}_{d}(\langle \cG,\cM \rangle ,\cA)$, is the full subcategory of $\Fct(\langle \cG,\cM \rangle ,\cA)$ defined by $\mathcal{P}ol^{str}_{d}(\langle \cG,\cM \rangle ,\cA)=\{ 0\}$ if $d<0$ and the objects of $\mathcal{P}ol^{str}_{d}(\langle \cG,\cM \rangle ,\cA)$ for $d\in \bN$ are the functors $F$ such that the functor $\delta_{\one}(F)$ is an object of $\mathcal{P}ol_{d-1}^{str}(\langle \cG,\cM \rangle ,\cA)$. The smallest integer $d\in \bN$ for which an object $F$ of $\Fct(\langle \cG,\cM \rangle ,\cA)$ is an object of $\mathcal{P}ol^{str}_{d}(\langle \cG,\cM \rangle ,\cA)$ is called the strong degree of $F$, and is denoted by $\sdeg(F)$.

The category of \emph{very strong} polynomial functors of degree at most $d\in \bN$, denoted by $\mathcal{VP}ol_{d}(\langle \cG,\cM \rangle ,\cA)$, is the full subcategory of $\mathcal{P}ol^{str}_{d}(\langle \cG,\cM \rangle ,\cA)$ of the objects $F$ such that $\kappa_{\one}(F)=0$ and the functor $\delta_{\one}(F)$ is an object of $\mathcal{\mathcal{VP}}ol_{d-1}(\langle \cG,\cM \rangle ,\cA)$.

The category of \emph{split} polynomial functors of degree at most $d\in \bN$, denoted by $\mathcal{SP}ol_{d}(\langle \cG,\cM \rangle ,\cA)$, is the full subcategory of $\mathcal{P}ol^{str}_{d}(\langle \cG,\cM \rangle ,\cA)$ of the objects $F$ such that the translation map $i_{\one}F \colon F\to\tau_{\one}F$ is split injective in $\Fct(\langle \cG,\cM \rangle ,\cA)$.

\begin{rmk}\label{rmk:strong_polynomiality_comparisons}
It follows from the definitions that split polynomiality implies very strong polynomiality, while very strong polynomiality implies strong polynomiality. Also, for $F$ an object of $\Fct(\langle \cG,\cM \rangle ,\cA)$, if $F$ is very strong (resp.~split) polynomial, then its strong degree $\sdeg(F)$ is the smallest integer $d\in \bN$ such that $F$ is an object of $\mathcal{VP}ol_{d}(\langle \cG,\cM \rangle ,\cA)$ (resp.~of $\mathcal{SP}ol_{d}(\langle \cG,\cM \rangle ,\cA)$).
\end{rmk}

A valuable application of these notions of polynomiality is that of \emph{homological stability with twisted coefficients} for families of groups; see \cite{RWW} for a detailed introduction to this notion. Namely, let us consider any pair of groupoids $\cM$ and $\cG$ introduced in \S\ref{sss:category_mcg} and \S\ref{sss:category_surface_braid_groups}, such that the family of automorphism groups of the Quillen bracket construction $\langle \cG,\cM \rangle$ is $\{\B_{n}(S)\}_{n\in\bN}$ for a fixed $S\in\{\bD,\Sigma_{g,1},\N_{h,1}\}$, or $\{\MCGo_{n,1}\}_{n\in\bN}$, or $\{\MCGno_{n,1}\}_{n\in\bN}$.
\begin{thm}[{\cite[Theorems~D, I, $5.23$, $5.26$, $5.29$]{RWW}}]\label{thm:homological_stability}
Twisted homological stability holds for the family of automorphism groups of $\langle \cG,\cM \rangle$ with coefficients in any object $F$ of $\Fct(\langle \cG,\cM \rangle,\bZ\lmod)$ that is \emph{strong} polynomial (of finite degree $d\in \bN$), and for which there exists $\mathtt{N}\in \obj(\cM)$ such that, for all $r\in\{0,\ldots,d\}$, $\kappa_{\one}(\delta^{r}_{\one}F)(\tn)=0$ for each $\tn\geq \mathtt{N} - r$; see \textup{\cite[Def.~$4.10$]{RWW}} for further details on these conditions. For instance, these conditions hold whenever $F$ is \emph{very strong} polynomial or \emph{split} polynomial.
\end{thm}

\begin{rmk}
In the case of mapping class groups of orientable surfaces, similar twisted homological stability results to those of Theorem~\ref{thm:homological_stability} from \cite{RWW} were proven earlier by Ivanov \cite{Ivanov} and Boldsen \cite{Boldsen}.

On another note, the target category of a homological representation functor $\fL^{\star}_{(\lambda,\ell)}$ is generally of the form ${\bZ[Q^{\star}_{(\lambda,\ell)}]}\lmod^{\bullet}$ (see Convention~\ref{convention:twisted_polynomiality}). To rigorously apply Theorem~\ref{thm:homological_stability} (see Corollaries~\ref{coro:HS_surface_braids} and \ref{coro:HS_MCG}), we implicitly postcompose the homological representation functors by the forgetful functor ${\bZ[Q^{\star}_{(\lambda,\ell)}]}\lmod^{\bullet} \to {\bZ}\lmod$; such operations do not affect the polynomiality properties of $\fL^{\star}_{(\lambda,\ell)}$.
\end{rmk}

\paragraph*{Weak polynomial functors.}

Let $F$ be an object of $\Fct(\langle \cG,\cM \rangle ,\cA)$. For all objects $\tn$ and $\tm$ of $\cM$, it follows from a clear diagram chase, the universal property of a kernel and the $4$-lemma that the morphism $[\tm,\id_{\tm\natural\tn}]$ induces a canonical inclusion $\kappa_{\tn}(F)\hookrightarrow \kappa_{\tm\natural\tn}(F)$. These inclusions provide an ascending filtration on the evanescence functors $\{\kappa_{\tn}\}_{\tn\in\obj(\cM)}$. We denote by $\kappa(F)$ the filtered colimit $\sum_{\tn\in \obj(\cM)}\kappa_{\tn}(F)$, which is a subfunctor of $F$.

Let $K(\langle \cG,\cM \rangle ,\cA)$ be the full subcategory of $\Fct(\langle \cG,\cM \rangle ,\cA)$ of all those objects $F$ such that $\kappa(F)=F$. The category $K(\langle \cG,\cM \rangle ,\cA)$ is a thick subcategory of $\Fct(\langle \cG,\cM \rangle ,\cA)$ and it is closed under colimits; see \cite[Prop.~$4.6$]{soulieLMgeneralised}.

\begin{rmk}\label{rmk:Grothendieck}
It would be enough to assume that $\cA$ is abelian to define strong, very strong and split polynomiality. However, we need $\cA$ to be Grothendieck in order to work with the notion of weak polynomiality and the associated quotient category of the functor category $\Fct(\langle \cG,\cM \rangle ,\cA)$. Indeed, it is necessary to assume that the filtered colimits in the category $\Fct(\langle \cG,\cM \rangle ,\cA)$ are exact in order to prove \cite[Prop.~$4.6$]{soulieLMgeneralised}. See also \cite[Rem.~$4.7$]{soulieLMgeneralised} for further explanations.
\end{rmk}

Therefore, the subcategory $K(\langle \cG,\cM \rangle ,\cA)$ of the Grothendieck category $\Fct(\langle \cG,\cM \rangle ,\cA)$ is \emph{localising}, and we may define the associated quotient category denoted by $\mathbf{St}(\langle \cG,\cM \rangle ,\cA)$, along with the associated left adjoint quotient functor $\pi_{\langle \cG,\cM \rangle}$, which is exact,
essentially surjective and commutes with all colimits; see \cite[Chapitre~III, \S1]{gabriel}.

For each object $\tn$ of $\cM$, the translation functor $\tau_{\tn}$ and the difference functor $\delta_{\tn}$ in the category $\Fct(\langle \cG,\cM \rangle ,\cA)$ induce exact endofunctors of $\mathbf{St}(\langle \cG,\cM \rangle ,\cA)$, which commute with colimits, respectively called again the translation functor $\tau_{\tn}$ and the difference functor $\delta_{\tn}$. In addition, we have the commutation relations $\delta_{\tn}\circ\pi_{\langle \cG,\cM \rangle}=\pi_{\langle \cG,\cM \rangle}\circ\delta_{\tn}$ and $\tau_{\tn}\circ\pi_{\langle \cG,\cM \rangle}=\pi_{\langle \cG,\cM \rangle}\circ\tau_{\tn}$. Therefore, the exact sequence \eqref{eq:ESCaract} induces a short exact sequence $\Identity\hookrightarrow\tau_{\tn}\twoheadrightarrow\delta_{\tn}$ for the induced endofunctors of $\mathbf{St}(\langle \cG,\cM \rangle ,\cA)$.
Finally, the endofunctors $\delta_{\one}$, $\delta_{\tn}$, $\tau_{\one}$ and $\tau_{\tn}$ of $\mathbf{St}(\langle \cG,\cM \rangle ,\cA)$ pairwise commute up to natural isomorphism coming from the braiding.

We then define, inductively on $d\in\bN$, the category of polynomial functors of degree at most $d$, denoted by $\mathcal{P}ol_{d}(\langle \cG,\cM \rangle ,\cA)$, to be the full subcategory of $\mathbf{St}(\langle \cG,\cM \rangle ,\cA)$ as follows. If $d<0$, $\mathcal{P}ol_{d}(\langle \cG,\cM \rangle ,\cA)=\{ 0\}$; if $d\geq0$, the objects of $\mathcal{P}ol_{d}(\langle \cG,\cM \rangle ,\cA)$ are the functors $F$ such that the functor $\delta_{\one}(F)$ is an object of $\mathcal{P}ol_{d-1}(\langle \cG,\cM \rangle ,\cA)$. For an object $F$ of $\mathbf{St}(\langle \cG,\cM \rangle ,\cA)$ that is polynomial of degree at most $d\in\bN$, the smallest integer $n\leq d$ for which $F$ is an object of $\mathcal{P}ol_{n}(\langle \cG,\cM \rangle ,\cA)$ is called the degree of $F$. An object $F$ of $\Fct(\langle \cG,\cM \rangle ,\cA)$ is weak polynomial of degree at most $d$ if its image $\pi_{\langle \cG,\cM \rangle}(F)$ is an object of $\mathcal{P}ol_{d}(\langle \cG,\cM \rangle ,\cA)$.
The degree of polynomiality of $\pi_{\langle \cG,\cM \rangle}(F)$ is called the weak degree of $F$, and is denoted by $\wdeg(F)$.

\begin{lem}\label{lem:strong_weak_polynomiality_comparisons}
Let $F$ be an object of $\Fct(\langle \cG,\cM \rangle ,\cA)$. If $F$ is strong polynomial of strong degree $\sdeg(F)$, then it is weak polynomial of weak degree $\wdeg(F) \leq \sdeg(F)$.
\end{lem}
\begin{proof}
The result is a straightforward consequence of the commutation relation $\delta^{p}_{\one}(\pi_{\langle \cG,\cM \rangle}(F))=\pi_{\langle \cG,\cM \rangle}(\delta^{p}_{\one}(F))$ for all $p\geq 1$.
\end{proof}

\begin{rmk}\label{rmk:polynomiality_n=1}
Since each object $\tn$ is equal to $\one^{\natural n}$ in the category $\langle \cG,\cM \rangle$, our definitions of the notions of strong, very strong, split and weak polynomiality are equivalent to the more classical ones with the analogous criteria on the functors $\tau_{\tn}$, $\delta_{\tn}$ and $\kappa_{\tn}$ for all $\tn\in\obj(\cG)$ (instead of just $\tn=\one$); see \cite[Prop.~$1.8$]{DV3}, \cite[Prop.~$3.9$]{soulieLM1} and \cite[Prop.~$4.4$, (2)]{soulieLMgeneralised} for further details.
\end{rmk}

Furthermore, for a short exact sequence $0\to F\to G\to H\to 0$ in the category $\Fct(\langle \cG,\cM \rangle ,\mathcal{A})$, the snake lemma induces the following short exact sequence in $\mathbf{St}(\langle \cG,\cM \rangle ,\cA)$, which is the image of \eqref{eq:LESkappadelta}:
\begin{equation}\label{eq:LESdelta_stable}
0\longrightarrow\delta_{\tn}\circ\pi_{\langle \cG,\cM \rangle }(F)\longrightarrow\delta_{\tn}\circ\pi_{\langle \cG,\cM \rangle }(G)\longrightarrow\delta_{\tn}\circ\pi_{\langle \cG,\cM \rangle }(H)\longrightarrow0.
\end{equation}
Finally, we recall useful properties of the categories associated with the different types of polynomial functors, which are proven in \cite[Props.~$4.4$, $4.10$]{soulieLMgeneralised} (split polynomial functors are not considered there, but their study follows repeating mutatis mutandis this reference).

\begin{prop}\label{prop:properties_polynomiality}
Let $d\geq0$ be an integer. The categories $\mathcal{P}ol^{str}_{d}(\langle \cG,\cM \rangle ,\cA)$, $\mathcal{VP}ol_{d}(\langle \cG,\cM \rangle ,\cA)$ and $\mathcal{SP}ol_{d}(\langle \cG,\cM \rangle ,\cA)$ are closed under the translation functor, direct sum and direct summand.
The categories $\mathcal{P}ol^{str}_{d}(\langle \cG,\cM \rangle ,\cA)$ and $\mathcal{VP}ol_{d}(\langle \cG,\cM \rangle ,\cA)$ are closed under extensions.
The category $\mathcal{P}ol^{str}_{d}(\langle \cG,\cM \rangle ,\cA)$ is closed under colimits.
The categories $\mathcal{VP}ol_{d}(\langle \cG,\cM \rangle ,\cA)$ and $\mathcal{SP}ol_{d}(\langle \cG,\cM \rangle ,\cA)$ are closed under normal subobjects (i.e.~kernels of epimorphisms).
As a subcategory of $\mathbf{St}(\langle \cG,\cM \rangle ,\cA)$, the category $\mathcal{P}ol_{d}(\langle \cG,\cM \rangle ,\cA)$ is thick, complete and cocomplete.
\end{prop}

\subsection{For surface braid group functors}\label{ss:polynomiality_surface_braid_groups}

In this section, we prove the polynomiality properties of Corollary~\ref{coro:polynomiality} and Theorem~\ref{thm:polynomiality_vs_vertical} for the homological representation functors for classical braid groups and surface braid groups defined in \S\ref{sss:representations_classical_braid_groups} and \S\ref{sss:representations_surface_braid_groups}.
Throughout \S\ref{ss:polynomiality_surface_braid_groups}, we consider homological representation functors indexed by an ordered partition $\lambda\vdash k$ of an integer $k\geq1$ and by the depth $\ell\geq1$ of a lower central series.

\subsubsection{Classical homological representation functors}\label{ss:poly_braids}

We prove here Corollary~\ref{coro:polynomiality} in the classical (i.e.~non-vertical) setting for homological representation functors. These polynomiality results actually hold both for the standard functors (i.e.~$\LB_{(\lambda,\ell)}$ and $\fL_{(\lambda,\ell)}(S)$) as well as for their \emph{untwisted} versions (i.e.~$\LB^{\unt}_{(\lambda,\ell)}$ and $\fL^{\unt}_{(\lambda,\ell)}(S)$).

\begin{proof}[Proof of Corollary~\ref{coro:polynomiality} for classical braid groups]
Following \S\ref{ss:SES_classical_braids}, we consider the functor $\LB^{\star}_{(\lambda,\ell)}$ where $\star$ either stands for the blank space or $\star=\unt$.

We first consider the case of $k=1$, and so $\lambda=(1)$. We recall that the preliminary study of \S\ref{sss:preliminary_SES_surface_braid_groups} (except \eqref{eq:definition-of-pkln} and \eqref{eq:isom-of-modules_braids}) holds for the functor $\LB^{\star}_{((1),\ell)}$. Following Notation~\ref{nota:genus_basis}, for each $\tn\in\obj(\Beta)$, the $\B_{n}$-representation $\delta_{\one}\LB^{\star}_{((1),\ell)}(\tn)$ is the free $\bZ[Q^{\star}_{((1),\ell)}(\bD)]$-module of rank one, with generators given by the tuples $(w_{0},w_{1},\ldots,w_{n-1})$ such that $\vert w_{0}\vert= 1$, while $\kappa_{\one}\LB^{\star}_{((1),\ell)}=0$.
By expressing $\delta_{\one}\LB^{\star}_{((1),\ell)}([\one,\id_{\one\natural\tn}])$ as a quotient map of $\LB^{\star}_{((1),\ell)}(\sigma_{1}^{-1})\circ\LB^{\star}_{((1),\ell)}([\one,\id_{\one\natural\tn}])$, it follows from Lemma~\ref{lem:cloud-splitting} and from the description \eqref{eq:image_injection_hom_rep_functor_surface} of $\LB^{\star}_{((1),\ell)}([\one,\id_{\one\natural\tn}])$ that $i_{\one}(\delta_{\one}\LB^{\star}_{((1),\ell)})_{\tn}$ is an isomorphism for $\tn\geq1$, while $i_{\one}(\delta_{\one}\LB^{\star}_{((1),\ell)})_{\zero}$ is the zero map. So $\kappa_{\one}\delta_{\one}\LB^{\star}_{((1),\ell)}=0$, while $\delta_{\one}^{2}\LB^{\star}_{((1),\ell)}$ is the atomic functor whose unique non-null value is $\delta_{\one}^{2}\LB^{\star}_{((1),\ell)}(\zero)=\bZ[Q^{\star}_{((1),\ell)}(\bD)]$. A fortiori, $\LB^{\star}_{((1),\ell)}$ is strong polynomial of strong degree $2$ and weak polynomial of weak degree $1$.
Furthermore, it follows from the commutation property of $\delta_{\one}$ with $\tau_{\one}$ that $\delta_{\one}(\tau_{\one}\LB^{\star}_{((1),\ell)})(\tn)=\bZ[Q^{\star}_{((1),\ell)}(\bD)]$ for each $\tn \geq \zero$. By repeating mutatis mutandis the above argument, we show that $\kappa_{\one}(\delta_{\one}(\tau_{\one}\LB^{\star}_{((1),\ell)}))=\delta^{2}_{\one}(\tau_{\one}\LB^{\star}_{((1),\ell)})=0$, and thus $\tau_{\one}\LB^{\star}_{((1),\ell)}$ is both very strong and weak polynomial, of both strong and weak degrees $1$. Moreover, by applying Lemma~\ref{lem:change_of_rings_SES} and then repeating verbatim the above argument, we deduce that this polynomiality result for $\LB^{\star}_{((1),\ell)}$ still holds after any non-zero change of rings operation.

Now, we proceed by induction on $k\geq2$, reasoning on each ordered partition $\lambda\vdash k$. First, we consider the case of $k=2$ (with an ordered partition $\lambda\vdash 2$). It follows from Theorem~\ref{thm:key_SES_classical_braids} that $\kappa_{\one}(\LB^{\star}_{(\lambda,\ell)})=0$, while the difference functor $\delta_{\one}\LB^{\star}_{(\lambda,\ell)}$ is determined by $\tau_{\one}\LB^{\star}_{((1),\ell)}$ (potentially up to a non-zero change of rings; see Convention~\ref{conv:transformation_group_summand}). By the above polynomiality results on $\tau_{\one}\LB^{\star}_{((1),\ell)}$, we deduce that the functor $\LB^{\star}_{(\lambda,\ell)}$ is both very strong and weak polynomial, of both strong and weak degrees $2$. By using Corollary~\ref{coro:change_of_ring} and repeating verbatim this argument, this polynomiality result for $\LB^{\star}_{(\lambda,\ell)}$ with $\lambda\vdash 2$ also holds after any non-zero change of rings operation.

We do the inductive step on any fixed $\lambda\vdash k\geq2$. Namely, we assume that for each $\lambda'\in\{\lambda - 1\}$ (see Notation~\ref{not:sets_partitions}), the functor $\LB^{\star}_{(\lambda',\ell)}$ is both very strong and weak polynomial, of both strong and weak degrees $k-1$, and that these properties still hold after any non-zero change of rings operation. Since $\mathcal{VP}ol_{d}(\langle \Beta,\Beta \rangle ,{\bZ[Q^{\star}_{(\lambda,\ell)}(\bD)]}\lmod^{\bullet})$ is closed under $\tau_{\one}$ and under normal subobjects by Proposition~\ref{prop:properties_polynomiality}, this inductive assumption implies that $\tau_{\one}\LB^{\star}_{(\lambda',\ell)}$ is both very strong and weak polynomial, of both strong and weak degrees $k-1$. Now, we deduce from Theorem~\ref{thm:key_SES_classical_braids} that $\kappa_{\one}\LB^{\star}_{(\lambda,\ell)}=0$, while the difference functor $\delta_{\one}\LB^{\star}_{(\lambda,\ell)}$ is a direct sum of functors of the form $\tau_{\one}\LB^{\star}_{(\lambda',\ell)}$ where $\lambda'\in\{\lambda - 1\}$ (potentially up to a non-zero change of rings; see Convention~\ref{conv:transformation_group_summand}). Therefore, the functor $\LB^{\star}_{(\lambda,\ell)}$ is both very strong and weak polynomial, of both strong and weak degrees $k$. This polynomiality result also holds after applying any non-zero change of rings operation by using Corollary~\ref{coro:change_of_ring} and repeating verbatim the above reasoning, which ends the induction.
\end{proof}

\begin{rmk}\label{rmk:recover_LM1}
The functors $\LB_{((1),2)}\otimes\bC[\bZ]$ and $\LB_{((2),2)}\otimes\bC[\bZ^{2}]$ correspond to the \emph{reduced Burau functor} $\overline{\mathfrak{Bur}}$ and the \emph{Lawrence-Krammer functor} $\mathfrak{LK}$ defined in \cite[\S 1.2]{soulieLM1}. By Corollary~\ref{coro:change_of_ring}, the polynomiality results of Corollary~\ref{coro:polynomiality} recover those of \cite[Props.~$3.25$ and $3.33$]{soulieLM1} for the functors $\LB_{((1),2)}\otimes\bC[\bZ]$ and $\LB_{((2),2)}\otimes\bC[\bZ^{2}]$. Also, the corresponding key short exact sequences using \cite[\S 1.2]{soulieLM1} (i.e.~\eqref{eq:key_SES_classical_braids_classical} for $\lambda\in\{1,2\}$ and $\ell=2$) are proven via an alternative method, which is more algebraic than that of Theorem~\ref{thm:key_SES_classical_braids}.
\end{rmk}

\begin{proof}[Proof of Corollary~\ref{coro:polynomiality} for surface braid groups]
Following \S\ref{ss:SES_surface_braids}, we use here the notation $\fL^{\star}_{(\lambda,\ell)}(S)$ where $\star$ either stands for the blank space or $\star=\unt$ and $S$ is either $\Sigma_{g,1}$ or $\N_{h,1}$ with $g,h\geq1$.

First, we consider the case of $k=1$ (with the trivial partition $\lambda=(1)$).
We recall that the preliminary study of \S\ref{sss:preliminary_SES_surface_braid_groups} (except \eqref{eq:definition-of-pkln} and \eqref{eq:isom-of-modules_braids}) holds for the functor $\fL^{\star}_{((1),\ell)}(S)$. Following Notation~\ref{nota:genus_basis}, for each $\tn\in\obj(\Beta^{S})$, the $\bZ[Q^{\star}_{((1),\ell)}(S)]$-module $\fL^{\star}_{((1),\ell)}(S)(\tn)$ is free of rank $g_{S}+n-1$, with generators denoted by tuples $(w_{1},\ldots,w_{g_{S}+n-1})$, where the $w_{i}\in\{0,1\}$ are integers such that $\sum_{1\leq i\leq g_{S}+n-1}w_{i}=1$. Also, the map $i_{\one}(\fL^{\star}_{((1),\ell)}(S))_{\tn}=\fL^{\star}_{((1),\ell)}(S)([\one,\id_{\one\natural\tn}])$ is the injection defined by $(w_{1},\ldots,w_{g_{S}+n-1})\mapsto (0,w_{1},\ldots,w_{g_{S}+n-1})$. Hence the cokernel $\delta_{\one}\fL^{\star}_{((1),\ell)}(S)(\tn)$ is the free $\bZ[Q^{\star}_{((1),\ell)}(S)]$-module of rank one generated by the element $(1,0,\ldots,0)$ if $\tn \geq \one$, while $\delta_{\one}\fL^{\star}_{((1),\ell)}(S)(\zero)$ is a free $\bZ[Q^{\star}_{((1),\ell)}(S)]$-module of rank $g_{S}$. By expressing $\delta_{\one}\fL^{\star}_{((1),\ell)}(S)([\one,\id_{\one\natural\tn}])$ as a quotient map of $\fL^{\star}_{((1),\ell)}(S)(\sigma_{1}^{-1})\circ\fL^{\star}_{((1),\ell)}(S)([\one,\id_{\one\natural\tn}])$, it follows from Lemma~\ref{lem:cloud-splitting} that $i_{\one}(\delta_{\one}\fL^{\star}_{((1),\ell)}(S))_{\tn}$ is an isomorphism for $\tn\geq \one$, and a fortiori $\delta^{2}_{\one}(\fL^{\star}_{((1),\ell)}(S))(\tn)=0$ when $\tn\geq1$. Therefore, the functor $\delta^{2}_{\one}(\fL^{\star}_{((1),\ell)}(S))$ is either null or atomic (with unique non-null value being the image of $\zero$), and so $\fL^{\star}_{((1),\ell)}(S)$ is weak polynomial of weak degree $1$. Moreover, the target of the map $i_{\one}(\delta^{2}_{\one}\fL^{\star}_{((1),\ell)}(S))_{\tn}$ is $0$ for all $\tn\geq\zero$, so $\delta^{3}_{\one}\fL^{\star}_{((1),\ell)}(S)=0$ and thus $\fL^{\star}_{((1),\ell)}(S)$ is strong polynomial, of strong degree at most $2$.

Now, we prove by induction on $k\geq1$ that, for each ordered partition $\lambda\vdash k$, the functor $L_{\lambda}:=(\tau_{\one}\fL^{\star}_{(\lambda,\ell)}(S))_{\mid\geq2}$ is weak polynomial with $\wdeg(L_{\lambda})=k$ and strong polynomial with $\sdeg(L_{\lambda})$ equal to $k$ or $k+1$, and that these properties still hold after any non-zero change of rings operation. The base case corresponds to studying the polynomiality of $L_{(1)}:=(\tau_{\one}\fL^{\star}_{((1),\ell)}(S))_{\mid\geq2}$. Repeating mutatis mutandis the first arguments above, we deduce that $\kappa_{\one}L_{(1)} = 0$ and $\delta_{\one}(L_{(1)})(\tn)$ is the free $\bZ[Q^{\star}_{((1),\ell)}(S)]$-module of rank one generated by the element $(1,0,\ldots,0)$ if $\tn \geq \one$, while $\delta_{\one}(L_{(1)})(\zero)$ is the free $\bZ[Q^{\star}_{((1),\ell)}(S)]$-module of rank $g_{S}+1$ on the generators $(w_{1},\ldots,w_{g_{S}+1})$ with $w_{i}\in\{0,1\}$. Viewing $\delta_{\one}L_{(1)}([\one,\id_{\one\natural\tn}])$ as a quotient map of $L_{(1)}(\sigma_{1}^{-1})\circ L_{(1)}([\one,\id_{\one\natural\tn}])$, we deduce from Lemma~\ref{lem:cloud-splitting} that $i_{\one}(\delta_{\one}L_{(1)})_{\tn}$ is an isomorphism for $\tn\geq \one$.
Hence we have $\delta^{2}_{\one}(L_{(1)})(\tn)=0$ when $\tn\geq1$, and the target of the map $i_{\one}(\delta^{2}_{\one}L_{(1)})$ is $0$ for all $\tn\geq\zero$. Therefore, as above, the functor $L_{(1)}$ is weak polynomial with $\wdeg(L_{(1)})=1$, and strong polynomial with $\sdeg(L_{(1)})$ equal to $1$ or $2$ by Lemma~\ref{lem:strong_weak_polynomiality_comparisons}. Furthermore, since $\delta_{\one}L_{(1)}(\tn)$ is a free $\bZ[Q^{\star}_{(\lambda,\ell)}(D)]$-module for each $\tn\in \obj (\Beta^{S})$ and $\kappa_{\one}L_{(1)} = 0$, we deduce that this polynomiality result for $L_{(1)}$ still holds after any non-zero change of rings operation by using Lemma~\ref{lem:change_of_rings_SES} and then repeating verbatim the above argument.

We do the inductive step on a fixed ordered partition $\lambda\vdash k$ with $k\geq1$. Namely, we assume that for each $\lambda'\in\{\lambda - 1\}$ (see Notation~\ref{not:sets_partitions}), the functor $L_{\lambda'}:=(\tau_{\one}\fL^{\star}_{(\lambda',\ell)}(S))_{\mid\geq2}$ is weak polynomial of weak degree $k-1$ and strong polynomial of strong degree $k-1$ or $k$, and that these properties still hold after any non-zero change of rings operation. By the inductive assumption on $L_{\lambda'}$, the commutation properties of $\delta_{\one}$ with colimits and with $\pi_{\langle \Beta,\Beta^{S}\rangle}$ and the right-exactness of $\pi_{\langle \Beta,\Beta^{S}\rangle}$, we deduce that the functor $\bigoplus_{1\leq j\leq r} L_{\lambda[j]}$ (as well as its versions after applying any non-zero change of rings operations to each one of the $L_{\lambda[j]}$ factors) is weak polynomial of weak degree $k-1$ and strong polynomial of strong degree $k-1$ or $k$.
By Theorem~\ref{thm:key_SES_surface_braid_groups}, the difference functor $\delta_{\one}(\fL^{\star}_{(\lambda,\ell)}(S)_{\mid\geq2})$ is an extension of the atomic functor $\tau_{\one}(\fL^{\star}_{(\lambda,\ell)}(S)_{\mid\geq2})(\one)$ by $\bigoplus_{1\leq j\leq r} L_{\lambda[j]}$ (potentially up to a non-zero change of rings; see Convention~\ref{conv:transformation_group_summand}).
It thus follows from the properties on extensions of Proposition~\ref{prop:properties_polynomiality} that the functor $\delta_{\one}(\fL^{\star}_{(\lambda,\ell)}(S)_{\mid\geq2})$ is weak polynomial of weak degree at most $k-1$, and strong polynomial of strong degree at most $k$.
In the stable category $\mathbf{St}(\langle \Beta,\Beta^{S}\rangle , {\bZ[Q^{\star}_{(\lambda,\ell)}(S)]}\lmod^{\bullet})$, we note that $\pi_{\langle \Beta,\Beta^{S}\rangle}(\tau_{\one}(\fL^{\star}_{(\lambda,\ell)}(S)_{\mid\geq2})(\one))=0$ because $\kappa(\tau_{\one}(\fL^{\star}_{(\lambda,\ell)}(S)_{\mid\geq2})(\one))=\tau_{\one}(\fL^{\star}_{(\lambda,\ell)}(S)_{\mid\geq2})(\one)$. Hence $\pi_{\langle \Beta,\Beta^{S}\rangle}(\delta_{\one}(\fL^{\star}_{(\lambda,\ell)}(S)_{\mid\geq2}))\cong \bigoplus_{1\leq j\leq r}\pi_{\langle \Beta,\Beta^{S}\rangle}(L_{\lambda[j]})$ and thus the weak degree of $\delta_{\one}(\fL^{\star}_{(\lambda,\ell)}(S)_{\mid\geq2})$ is equal to $k-1$. Also, by Lemma~\ref{lem:strong_weak_polynomiality_comparisons}, the strong degree of $\delta_{\one}(\fL^{\star}_{(\lambda,\ell)}(S)_{\mid\geq2})$ is thus $k-1$ or $k$. Therefore, the functor $\fL^{\star}_{(\lambda,\ell)}(S)_{\mid\geq2}$ is weak polynomial of weak degree $k$, and strong polynomial of strong degree $k$ or $k+1$. Then, by the properties on extensions of Proposition~\ref{prop:properties_polynomiality}, we deduce from Corollary~\ref{coro:key_SES_surface_braid_groups} that the functor $L_{\lambda}$ is weak polynomial of weak degree at most $k$, and strong polynomial of strong degree at most $k+1$. Since $\delta_{\one}$ commutes with colimits, the inductive assumption on the functors $L_{\lambda[j]}$ implies that $\delta_{\one}^{k}\pi_{\langle \Beta,\Beta^{S}\rangle}(\bigoplus_{1\leq j\leq r} L_{\lambda[j]})=0$. Then, by iterating the short exact sequence \eqref{eq:LESdelta_stable} on \eqref{eq:key_SES_braid_coro}, we deduce that the functor $\delta_{\one}^{k}\pi_{\langle \Beta,\Beta^{S}\rangle}(L_{\lambda})$ is isomorphic to $\delta_{\one}^{k}\pi_{\langle \Beta,\Beta^{S}\rangle}(\fL^{\star}_{(\lambda,\ell)}(S)_{\mid\geq2})$, which is non-null since $\wdeg(\fL^{\star}_{(\lambda,\ell)}(S)_{\mid\geq2})=k$. Hence $\wdeg(L_{\lambda})=k$, and thus $\sdeg(L_{\lambda})$ is equal to $k$ or $k+1$ by Lemma~\ref{lem:strong_weak_polynomiality_comparisons}. These polynomiality results for $\fL^{\star}_{(\lambda,\ell)}(S)_{\mid\geq2}$ and $L_{\lambda}$ also hold after applying any non-zero change of rings operation, by using Corollary~\ref{coro:change_of_ring} and repeating verbatim the above reasoning, which ends the induction.

In particular, we prove in the process of the induction that, for each $k\geq1$ and each ordered partition $\lambda\vdash k$, the functor $\fL^{\star}_{(\lambda,\ell)}(S)_{\mid\geq2}$ is weak polynomial of weak degree $k$, and strong polynomial of strong degree $k$ or $k+1$. Now, recall from \S\ref{ss:SES_surface_braids} that $\fL^{\star}_{(\lambda,\ell)}(S)$ is an extension of the atomic functor $\fL^{\star}_{(\lambda,\ell)}(S)(\one)$ by $\fL^{\star}_{(\lambda,\ell)}(S)_{\mid\geq2}$, i.e~there is a short exact sequence $\fL^{\star}_{(\lambda,\ell)}(S)_{\mid\geq2}\hookrightarrow\fL^{\star}_{(\lambda,\ell)}(S)\twoheadrightarrow \fL^{\star}_{(\lambda,\ell)}(S)(\one)$.
By the properties on extensions of Proposition~\ref{prop:properties_polynomiality}, the functor $\fL^{\star}_{(\lambda,\ell)}(S)$ is weak polynomial of weak degree with $\wdeg(\fL^{\star}_{(\lambda,\ell)}(S))\leq k$, and strong polynomial with $\sdeg(\fL^{\star}_{(\lambda,\ell)}(S))\leq k+1$.
Then we have $\pi_{\langle \Beta,\Beta^{S}\rangle}(\fL^{\star}_{(\lambda,\ell)}(S)(\one))=0$ because $\kappa(\fL^{\star}_{(\lambda,\ell)}(S)(\one))=\fL^{\star}_{(\lambda,\ell)}(S)(\one)$, so $\pi_{\langle \Beta,\Beta^{S}\rangle}(\fL^{\star}_{(\lambda,\ell)}(S))\cong \pi_{\langle \Beta,\Beta^{S}\rangle}(\fL^{\star}_{(\lambda,\ell)}(S)_{\mid\geq2})$. We thus deduce that $\wdeg(\fL^{\star}_{(\lambda,\ell)}(S))= k$, and so $\sdeg(\fL^{\star}_{(\lambda,\ell)}(S))= k$ or $k+1$ by Lemma~\ref{lem:strong_weak_polynomiality_comparisons}.
\end{proof}

\begin{coro}\label{coro:HS_surface_braids}
Twisted homological stability holds for the classical braid groups and surface braid groups with coefficients in the homological representation functors of Theorems~\ref{thm:key_SES_classical_braids} and \ref{thm:key_SES_surface_braid_groups}.
\end{coro}
\begin{proof}
Following \S\ref{sss:preliminary_SES_surface_braid_groups}, we use the generic notation $\fL^{\star}_{(\lambda,\ell)}(S)$ to study the functors $\LB^{\star}_{(\lambda,\ell)}$ and $\fL^{\star}_{(\lambda,\ell)}(S)$ of Theorems~\ref{thm:key_SES_classical_braids} and \ref{thm:key_SES_surface_braid_groups}. Recall from the description \eqref{eq:image_injection_hom_rep_functor_surface} of the map $i_{\one}(\fL^{\star}_{(\lambda,\ell)}(S))_{\tn}$ that $\kappa_{\one}\fL^{\star}_{(\lambda,\ell)}(S)(\tn)=0$ for all $\tn\geq\two$. Then, using the commutation property of $\delta_{\one}$ with $\tau_{\one}$ (and noting that $\delta_{\one}(\fL^{\star}_{(\lambda,\ell)}(S)_{\mid\geq2})(\tn)\cong \delta_{\one}(\fL^{\star}_{(\lambda,\ell)}(S))(\tn)$ by \eqref{eq:LESkappadelta} for $\tn\geq\two$ when $S\neq \bD$), it follows from a clear iteration of the short exact sequences \eqref{eq:key_SES_classical_braids_classical} and \eqref{eq:key_SES_braid_surface} that $\kappa_{\one}(\delta^{r}_{\one}\fL^{\star}_{(\lambda,\ell)}(S))(\tn)=0$ for all $\tn\geq\two$ and all $r\geq0$. By this last property along with Corollary~\ref{coro:polynomiality}, each functor $\fL^{\star}_{(\lambda,\ell)}(S)$ thus satisfies the condition of Theorem~\ref{thm:homological_stability} (as long as we choose $\mathtt{N} \geq \two + d$), whence the result.
\end{proof}

\begin{rmk}\label{rmk:polynomiality_L_one_functor}
In the above proof of Corollary~\ref{coro:polynomiality} for surface braid groups, we have not determined if the strong degree of $\fL^{\star}_{(\lambda,\ell)}(S)$ is $k$ or $k+1$, nor addressed the question of whether or not this functor is \emph{very} strong polynomial. This is actually an aftereffect of the difficulty of computing the first entries $\delta_{\one}^{m}(\fL^{\star}_{(\lambda,\ell)}(S))(\zero)$ and $\kappa_{\one}(\delta_{\one}^{m-1}(\fL^{\star}_{(\lambda,\ell)}(S)))(\zero)$ for $m\geq 2$ with the techniques of the present paper. Indeed, although it is not difficult to compute $\delta^{2}_{\one}(\fL^{\star}_{(\lambda,\ell)}(S))(\tn)$ and check that $\kappa_{\one}(\delta_{\one}\fL^{\star}_{(\lambda,\ell)}(S))(\tn)=0$ for $\tn \geq \one$ via the methods of \S\ref{sss:preliminary_SES_surface_braid_groups} and \S\ref{ss:SES_surface_braids}, the map
\[
i_{\one}(\delta_{\one}\fL^{\star}_{(\lambda,\ell)}(S))_{\zero}\colon (\delta_{\one}\fL^{\star}_{(\lambda,\ell)}(S))(\zero)\cong \fL^{\star}_{(\lambda,\ell)}(S)(\one) \to (\delta_{\one}\fL^{\star}_{(\lambda,\ell)}(S))(\one)
\]
is however much trickier to study. Let us illustrate this with the case of $k=1$, $\lambda = (1)$, $\ell=2$ and $S=\Sigma_{1,1}$, for which $Q_{((1)),2)}(S) \cong \bZ^{2}\cong \langle A\rangle \oplus \langle B\rangle$ (see Remark~\ref{rmk:properties_surface_braid_groups_functors}). We recall that the generators $\{[1,0], [0,1]\}$ form a basis of $\fL_{((1),2)}(\Sigma_{1,1})(\one)\cong\bZ[\bZ^{2}]^{\oplus 2}$. Using the further techniques of \cite{PSIIi}, one may then prove that the kernel of $i_{\one}(\delta_{\one}\fL_{((1),2)}(S))_{\zero}\colon \bZ[\bZ^{2}]^{\oplus 2} \to \bZ[\bZ^{2}]$ is isomorphic to the free $\bZ[\bZ^{2}]$-submodule of $\bZ[\bZ^{2}]^{\oplus 2}$ generated by $(1-A)[0,1]-(1-B)[1,0]$, while $\delta^{2}_{\one}(\fL_{((1),2)}(S))(\zero)\cong \bZ[\bZ^{2}]/(1-A,1-B)$.
This along with the above proof of Corollary~\ref{coro:polynomiality} proves that $\fL_{((1),2)}(\Sigma_{1,1})$ is strong polynomial of degree $2$ but not \emph{very} strong polynomial. Nevertheless, note that if we apply the change of rings functor for the homomorphism $q\colon\bZ[\bZ^{2}]\to \bQ(\bZ^{2})$ (where $\bQ(\bZ^{2})$ is the field of fractions of $\bZ[\bZ^{2}]$), we deduce that the functor $q_{!}\fL_{((1),2)}(\Sigma_{1,1})$ is strong polynomial of degree $1$. The strong degree thus decisively depends on the ground ring of the functor in this case. Furthermore, analogous results may be proved for more general $\fL^{\star}_{(\lambda,\ell)}(S)$.

In contrast, by similar work to \S\ref{ss:SES_surface_braids} and the proof of Corollary~\ref{coro:polynomiality}, it is routine to check that the shifted-by-$\one$ functor $\tau_{\one}\fL^{\star}_{(\lambda,\ell)}(S)$ is very strong and weak polynomial of both strong and weak degrees $k$. This exemplifies how the strong polynomial degree may be heavily affected by the low values of a given functor, thus not being optimal to describe its global behaviour, in particular for homological stability. Moreover, this shows the interest of the notion of weak polynomiality since it reflects more accurately the stable behaviour of functors.
\end{rmk}

\subsubsection{Vertical-type alternatives}\label{ss:SES_surface_braids_alternative}

We now deal with the vertical-type alternatives of the homological representation functors of the surface braid groups.
Following the framework of \S\ref{ss:SES_surface_braid_groups}, we consider the generic vertical homological representation functor $\fL^{\star,\vrtcl}_{(\lambda,\ell)}(S)$ where $\star$ either stands for the blank space or $\star=\unt$, $S\in\{\bD,\Sigma_{g,1},\N_{h,1}\}$ with $g\geq1$ and $h\geq1$ and the associated transformation group is denoted by $Q^{\star}_{(\lambda,\ell)}(S)$.
We recall from Proposition~\ref{lem:module_structure_BM_homology-check} that, for each $\tn\in\obj(\Beta^{S})$, the $\bZ[Q^{\star}_{(\lambda,\ell)}(S)]$-module $\fL^{\star,\vrtcl}_{(\lambda,\ell)}(S)(\tn)$ is free with basis indexed by the set of (``vertical'') tuples $\bw^{\vrtcl}$ as pictured in Figures~\ref{fig:model-or-braids-dual} and \ref{fig:model-nor-braids-dual}, which has the same dimension as the (``classical'') functor $\fL^{\star}_{(\lambda,\ell)}(S)(\tn)$.

First of all, we focus on a significant fact about the behaviour of the functor $\fL^{\star,\vrtcl}_{(\lambda,\ell)}(S)$ after applying the operation $\delta_{\one}$.

\begin{lem}
\label{rel:trivial_morphism_vertical_alternatives_surface_braid_groups}
The functor $\delta_{\one} \fL^{\star,\vrtcl}_{(\lambda,\ell)}(S)$ sends every morphism that is not an endomorphism to zero.
\end{lem}
\begin{proof}
Recall that, by construction (see Lemma~\ref{lem:extension_Quillen_source_homological_rep_functors}), the morphism $\tn \to \one \natural \tn$ of the domain category $\langle \Beta,\Beta^{S}\rangle$ is sent, under each of our functors $\fL^{\star,\vrtcl}_{(\lambda,\ell)}(S)$, to the map on Borel-Moore homology induced by the evident inclusion of configuration spaces. Since every morphism of the domain category that is not an endomorphism factors through one of these canonical morphisms, it suffices to show that all of these are sent to zero under $\delta_{\one} \fL^{\star,\vrtcl}_{(\lambda,\ell)}(S)$. In other words, we wish to show that the map labelled by $(*)$ in the following diagram is zero, where the rows are exact:
\begin{equation}
\label{eq:action-on-delta1}
\begin{tikzcd}
\fL^{\star,\vrtcl}_{(\lambda,\ell)}(S)(\tn) \ar[r] \ar[d] & \tau_{\one} \fL^{\star,\vrtcl}_{(\lambda,\ell)}(S)(\tn) \ar[r] \ar[d,"(\dagger)"] \ar[dl,dashed] & \delta_{\one} \fL^{\star,\vrtcl}_{(\lambda,\ell)}(S)(\tn) \ar[d,"(*)"] \ar[r] & 0 \\
\fL^{\star,\vrtcl}_{(\lambda,\ell)}(S)(\one \natural \tn) \ar[r,"(\ddagger)",swap] & \tau_{\one} \fL^{\star,\vrtcl}_{(\lambda,\ell)}(S)(\one \natural \tn) \ar[r] & \delta_{\one} \fL^{\star,\vrtcl}_{(\lambda,\ell)}(S)(\one \natural \tn) \ar[r] & 0
\end{tikzcd}
\end{equation}
To do this, it suffices to show that there is a diagonal morphism making the two triangles commute. Recalling that $\tau_{\one} F(\tn) = F(\one \natural \tn)$ in general, we will be able to take the diagonal morphism to be the identity as long as the two maps labelled $(\dagger)$ and $(\ddagger)$ are equal (we note that the top-left horizontal map $\fL^{\star,\vrtcl}_{(\lambda,\ell)}(S)(\tn) \to \tau_{\one} \fL^{\star,\vrtcl}_{(\lambda,\ell)}(S)(\tn)$ and the left-most vertical map $\fL^{\star,\vrtcl}_{(\lambda,\ell)}(S)(\tn) \to \fL^{\star,\vrtcl}_{(\lambda,\ell)}(S)(\one \natural \tn)$ in \eqref{eq:action-on-delta1} are always equal by definition of the natural transformation $\Identity \to \tau_{\one}$).

By definition of $\tau_{\one} \fL^{\star,\vrtcl}_{(\lambda,\ell)}(S)$, its action on the canonical morphism $[\one,\id_{\one \natural \tn}]\colon \tn \to \one \natural \tn$ is given by the action of $\fL^{\star,\vrtcl}_{(\lambda,\ell)}(S)$ on the canonical morphism $[\one,\id_{\two \natural \tn}] \colon \one \natural \tn \to \two \natural \tn$ composed with $(b_{\one,\one}^{\Beta})^{-1} \natural \id_{\tn}$, where $b_{\one,\one}^{\Beta}$ is the braiding $\one\natural \one\cong \one\natural \one$ of the groupoid $\Beta$; see \eqref{eq:formula_morphism_{i}d_plus_morphism}. This describes the map $(\dagger)$; on the other hand, the map $(\ddagger)$ is given simply by the action of $\fL^{\star,\vrtcl}_{(\lambda,\ell)}(S)$ on the canonical morphism $\one \natural \tn \to \two \natural \tn$. It is therefore enough to prove that the automorphism $b_{\one,\one}^{\Beta} \natural \id_{\tn}$, which canonically identifies with the Artin generator $\sigma_{1}$, acts by the identity on the image of $(\ddagger)$. This is immediate from Figure~\ref{fig:action-of-braiding}, where the image of an arbitrary basis element $\bw^{\vrtcl}=(w_{1},\ldots,w_{n})$ under $(\ddagger)$ is depicted in green (supported on the vertical arcs) and the support of a diffeomorphism representing the mapping class $\sigma_{1}$ is shaded in grey. Since these supports are disjoint, the action of $b_{\one,\one}^{\Beta} \natural \id_{\tn}$ on the image of $(\ddagger)$ is trivial.
\end{proof}

\begin{figure}[tb]
    \centering
    \includegraphics[scale=0.65]{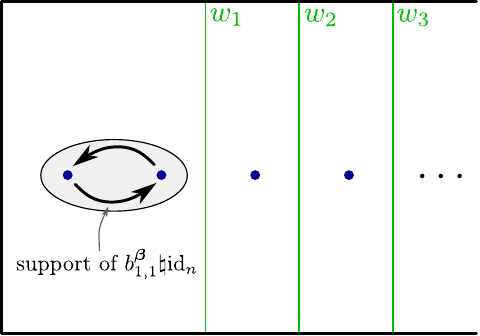}
    \caption{The support of a diffeomorphism representing the mapping class $\sigma_{1}=b_{\one,\one}^{\Beta} \natural \id_{\tn}$ and the image of an arbitrary basis element $\bw^{\vrtcl}=(w_{1},\ldots,w_{n})$, shown in green, under the map $(\ddagger)$ of \eqref{eq:action-on-delta1}.}
    \label{fig:action-of-braiding}
\end{figure}

\begin{rmk}
It is instructive to consider why the same argument does not also show that the functor $\delta_{\one} \fL^{\star}_{(\lambda,\ell)}(S)$ sends every canonical morphism $\tn \to \one \natural \tn$ to the zero morphism. This boils down to the fact that, in the analogue of Figure~\ref{fig:action-of-braiding} for the \emph{non-vertical} version $\fL^{\star}_{(\lambda,\ell)}(S)$ of the functor, the supports are not disjoint.
\end{rmk}

We are now ready to prove the (non-)polynomiality results of Theorem~\ref{thm:polynomiality_vs_vertical}, which actually hold for any one of these vertical-type alternatives $\fL^{\star,\vrtcl}_{(\lambda,\ell)}(S)$ (i.e.~also for the \emph{untwisted} versions). In particular, this shows that these vertical-type alternatives exhibit unexpected interesting behaviour with respect to polynomiality, which thoroughly differs from their ``classical'' (non-vertical) counterparts studied in \S\ref{ss:SES_classical_braids}--\S\ref{ss:SES_surface_braids}. Indeed, they are not strong polynomial, which is a counterintuitive property since the dimensions of the representations encoded by each $\fL^{\star,\vrtcl}_{(\lambda,\ell)}(S)$ grow in the same polynomial way as those of $\fL^{\star}_{(\lambda,\ell)}(S)$.

\begin{proof}[Proof of Theorem~\ref{thm:polynomiality_vs_vertical} for surface braid groups]
By similar reasoning to that of \S\ref{sss:preliminary_SES_surface_braid_groups}, we see that the map $\fL^{\star,\vrtcl}_{(\lambda,\ell)}(S)([\one,\id_{\one\natural\tn}])$ is the injection defined on basis elements by $(w_{1},\ldots,w_{g_{S}+n-1})^{v}\mapsto (\varnothing,w_{1},\ldots,w_{g_{S}+n-1})^{v}$, so it follows that $\delta_{\one}\fL^{\star,\vrtcl}_{(\lambda,\ell)}(S)(\tn)$ is a non-trivial free $\bZ[Q^{\star}_{(\lambda,\ell)}(S)]$-module with the same dimension as $\delta_{\one}\fL^{\star}_{(\lambda,\ell)}(S)(\tn)$ for all $\tn\geq\one$. Meanwhile it follows from Lemma~\ref{rel:trivial_morphism_vertical_alternatives_surface_braid_groups} that $\delta_{\one}\fL^{\star,\vrtcl}_{(\lambda,\ell)}(S)$ assigns the zero map to all morphisms of $\langle \Beta, \Beta^{S}\rangle(\tn,\tm)$ with $\tn\neq \tm$. So $\delta_{\one}\fL^{\star,\vrtcl}_{(\lambda,\ell)}(S)$ is isomorphic to a direct sum of infinitely many atomic functors. It follows that $\delta_{\one}^{m}\fL^{\star,\vrtcl}_{(\lambda,\ell)}(S)\neq 0$ for any $m\in \bN$ while $\pi_{\langle \Beta, \Beta^{S}\rangle}(\delta_{\one}\fL^{\star,\vrtcl}_{(\lambda,\ell)}(S))=0$ in the stable category $\mathbf{St}(\langle \Beta, \Beta^{S}\rangle ,{\bZ[Q^{\star}_{(\lambda,\ell)}(S)]}\lmod^{\bullet})$, whence the result.
\end{proof}

\begin{rmk}\label{rmk:SES_vertical_surface_braid_groups}
The first steps of the proofs of Theorems~\ref{thm:key_SES_classical_braids} and \ref{thm:key_SES_surface_braid_groups} do go through in the vertical setting, inducing a short exact sequence of functors defined on the groupoid $\Beta^{S}$ for each $\fL^{\star,\vrtcl}_{(\lambda,\ell)}(S)$, analogous to \eqref{eq:key_SES_classical_braids_classical} and \eqref{eq:key_SES_braid_surface} but only at the level of automorphism groups.
\end{rmk}

Finally, we briefly deal with the duals of the homological representations of Theorems~\ref{thm:key_SES_classical_braids} and \ref{thm:key_SES_surface_braid_groups}. Let us consider any one of the above homological representation functors $\fL^{\star}_{(\lambda,\ell)}(S)$.
By Corollary~\ref{coro:perfect-pairing}, the $\B_{n}(S)$-representation $H_{k}^{\partial}(C_{\lambda}(\bD_{n}\natural S); \bZ[Q^{\star}_{(\lambda,\ell)}(S)]\otimes \cO)$ of \S\ref{ss:dual-bases} is the dual representation of $\fL^{\star}_{(\lambda,\ell)}(S)(\tn)$.
Gathering these representations and assigning for each $[\tm,\id_{\tm\natural\tn}]\in \langle \Beta,\Beta^{S}\rangle$ the evident analogue of the map $\iota_{\tm,\tn}$ of \S\ref{sss:def_homological_rep_functors} for homology relative to the boundary, one may easily prove the analogue of Lemma~\ref{lem:extension_Quillen_source_homological_rep_functors} so that we define a functor $\fL^{\star,\vee}_{(\lambda,\ell)}(S)\colon \langle\Beta,\Beta^{S}\rangle \to \bZ[Q^{\star}_{(\lambda,\ell)}(S)]\lmod^{\bullet}$. Then the reasoning of the proof of Theorem~\ref{thm:polynomiality_vs_vertical} repeats verbatim:
\begin{thm}
\label{thm:dual_representation_functors_surface_braid_groups}
The functor $\fL^{\star,\vee}_{(\lambda,\ell)}(S)$ is not strong polynomial, but is weak polynomial of weak degree $0$.
\end{thm}

\subsection{For mapping class group functors}\label{ss:poly_mcg}

In this section, we prove the polynomiality results of Corollary~\ref{coro:polynomiality} and Theorem~\ref{thm:polynomiality_vs_vertical} for the homological representation functors for mapping class groups defined in \S\ref{sss:representations_mapping_class_groups}.
Following \S\ref{ss:SES_mcg}, we use the generic notation $\fL^{\star}_{(\lambda,\ell)}$ for any one of the functors \eqref{eq:hom_rep_functor_rep_mcg_o} and \eqref{eq:hom_rep_functor_rep_mcg_no_ell} indexed by an ordered partition $\lambda\vdash k$ of an integer $k\geq1$ and by the depth $\ell\geq1$ of a lower central series, $\fL^{\star,\vrtcl}_{(\lambda,\ell)}$ for the vertical-type alternative functor, $Q^{\star}_{(\lambda,\ell)}(\sS)$ with $\sS\in\{\bT,\bM\}$ for the associated transformation group and $\cM$ for either $\M^{+}$ or $\M^{-}$.

\begin{proof}[Proof of Corollary~\ref{coro:polynomiality} and Theorem~\ref{thm:polynomiality_vs_vertical} for mapping class groups]
We proceed by induction on $k\geq1$, reasoning on each ordered partition $\lambda\vdash k$ and considering the functor $\fL^{\star}_{(\lambda,\ell)}$.  First, we consider the case of $k=1$ with $\lambda=(1)$.
We recall that the preliminary study of \S\ref{ss:mcg_hom_rep_poly_preliminary} up to the paragraph ``Difference functor decomposition'' holds for $\fL^{\star}_{((1)),\ell)}$. Hence $\kappa_{\one}\fL^{\star}_{((1),\ell)}=0$, while the $\MCG(\sS^{\natural n})$-representation $\delta_{\one}\fL^{\star}_{((1),\ell)}(\tn)$ is a free $\bZ[Q^{\star}_{(\lambda,\ell)}(\sS)]$-module of rank $1$ for each $\tn\in\obj(\cM)$, with generating set given by the tuples $([w_{0},w'_{0}],[w_{1},w_{2}],\ldots,[w_{2n-1},w_{2n}])$ such that $\vert w_{0} \vert + \vert w'_{0}\vert = 1$ if $\sS=\bT$, and the tuples $([w_{0}],[w_{1}],\ldots,[w_{n}])$ such that $\vert w_{0}\vert= 1$ if $\sS=\bM$. By expressing $\delta_{\one}\fL^{\star}_{((1),\ell)}([\one,\id_{\one\natural\tn}])$ as a quotient map of $\fL^{\star}_{((1),\ell)}(\sigma_{1}^{-1})\circ\fL^{\star}_{((1),\ell)}([\one,\id_{\one\natural\tn}])$, it follows from equations \eqref{eq:identity-braiding-or} and \eqref{eq:identity-braiding-nor} of Lemma~\ref{lem:computations_MCG_braiding} and from the description \eqref{eq:image_injection_hom_rep_functor_mcg} of $\fL^{\star}_{((1),\ell)}([\one,\id_{\one\natural\tn}])$ that $i_{\one}(\delta_{\one}\fL^{\star}_{((1),\ell)})_{\tn}$ is an isomorphism for $\tn\geq\zero$. So $\delta_{\one}^{2}\fL^{\star}_{((1),\ell)}=0$ and $\kappa_{\one}\delta_{\one}\fL^{\star}_{((1),\ell)}=0$, and thus $\fL^{\star}_{((1),\ell)}$ is both very strong and weak polynomial, of both strong and weak degrees $1$.
Furthermore, the part of the proof of Theorem~\ref{thm:key_SES_mapping_class_groups} showing that the $\bZ[Q^{\star}_{(\lambda,\ell)}(\sS)]$-module injections $\{(\Delta'_{\one}\fL^{\star}_{(\lambda,\ell)})_{\tn}\}_{\tn\in \obj(\cM)}$ (see \eqref{eq:Delta'_def}) assemble into a natural transformation in $\Fct(\langle\cM,\cM\rangle,{\bZ[Q^{\star}_{(\lambda,\ell)}(\sS)]}\lmod^{\bullet})$ repeats verbatim for $\lambda=(1)$, because Lemmas~\ref{lem:invariance_translation_block_MCG_action} and \ref{lem:computations_MCG_braiding} also hold in this case. Then, we deduce from the definition of these injections that $(\Delta_{\one}\fL^{\star}_{((1),\ell)})_{\tn}\circ (\Delta'_{\one}\fL^{\star}_{((1),\ell)})_{\tn}=\id_{\delta_{\one}\fL^{\star}_{((1),\ell)}(\tn)}$ for each $\tn\in\obj(\cM)$, and so $\Delta'_{\one}\fL^{\star}_{((1),\ell)}$ is a section of $\Delta_{\one}\fL^{\star}_{((1),\ell)}$ in $\Fct(\langle\cM,\cM\rangle,{\bZ[Q^{\star}_{((1),\ell)}(\sS)]}\lmod^{\bullet})$. Since $\kappa_{\one}\fL^{\star}_{((1),\ell)}=0$ (because $i_{\one}(\fL^{\star}_{((1),\ell)})_{\tn}$ is clearly injective for each $\tn\in \cM$), the exact sequence \eqref{eq:ESCaract} for $\fL^{\star}_{((1),\ell)}$ is a split short exact sequence, and so $\fL^{\star}_{((1),\ell)}$ is split polynomial.
Furthermore, since $\delta_{\one}\fL^{\star}_{((1),\ell)}(\tn)$ is a free $\bZ[Q^{\star}_{(\lambda,\ell)}(\sS)]$-module for each $\tn\in \obj (\cM)$ and $\kappa_{\one}\fL^{\star}_{((1),\ell)} = 0$, we deduce that this polynomiality result for $\fL^{\star}_{((1),\ell)}$ still holds after any non-zero change of rings operation by using Lemma~\ref{lem:change_of_rings_SES} and then repeating verbatim the above argument for the splitting.

We do the inductive step on any fixed $\lambda\vdash k\geq2$. Namely, we assume that for each $k'\geq1$ and each $\lambda'\in\{\lambda - k'\}$ (see Notation~\ref{not:sets_partitions}), the functor $\fL^{\star}_{(\lambda',\ell)}$ is both split and weak polynomial, of both strong and weak degrees $k-k'$, and that this property still holds after any non-zero change of rings operation.
Since $\mathcal{SP}ol_{d}(\langle \cM,\cM \rangle ,{\bZ[Q^{\star}_{(\lambda,\ell)}(\sS)]}\lmod^{\bullet})$ is closed under $\tau_{\one}$ and under normal subobjects by Proposition~\ref{prop:properties_polynomiality}, this inductive assumption implies that each functor $\tau_{\one}\fL^{\star}_{(\lambda',\ell)}$, and also $(\tau_{\one}\fL^{\star}_{(\lambda',\ell)}(\MCGo))_{[\varnothing]}$ in the orientable setting, are split and weak polynomial of both strong and weak degrees $k-k'$, as are the versions of these functors after any non-zero change of rings operation.
Now, we deduce from Theorem~\ref{thm:key_SES_mapping_class_groups} that the translation functor $\tau_{\one}\fL^{\star}_{(\lambda,\ell)}$ is isomorphic to $\fL^{\star}_{(\lambda,\ell)}\oplus \delta_{\one}\fL^{\star}_{(\lambda,\ell)}$, where $\delta_{\one}\fL^{\star}_{(\lambda,\ell)}$ is determined by a direct sum of functors (potentially up to a non-zero change of rings, see Convention~\ref{conv:transformation_group_summand}) of the form $\tau_{\one}\fL^{\star}_{(\lambda',\ell)}$ with $\lambda'\in\{\lambda - 1,\lambda - 2 \}$, and also $(\tau_{\one}\fL^{\star}_{(\lambda'',\ell)}(\MCGo))_{[\varnothing]}$ with $\lambda''\in\{\lambda - 1,\lambda - 2 \}$ in the orientable setting. Therefore, the functor $\fL^{\star}_{(\lambda',\ell)}$ is both very strong and weak polynomial, of both strong and weak degrees $k$. This polynomiality result also holds after applying any non-zero change of rings operation, by using Corollary~\ref{coro:change_of_ring} and repeating verbatim the above reasoning, which ends the induction.

Fixing $\ell\in\{1,2\}$, the same polynomiality results follow for $\fL^{\star,\vrtcl}_{(\lambda,\ell)}$ by repeating mutatis mutandis the same arguments, using Theorem~\ref{thm:SES_MCG_alternatives} instead of Theorem~\ref{thm:key_SES_mapping_class_groups}.
\end{proof}

Furthermore, we briefly deal here with the duals of the homological representations of Theorem~\ref{thm:key_SES_mapping_class_groups}.
By Corollary~\ref{coro:perfect-pairing}, the $\MCG(\sS^{\natural n})$-representation $H_{k}^{\partial}(C_{\lambda}(\sS^{\natural n} \smallsetminus I); \bZ[Q^{\star}_{(\lambda,\ell)}(\sS)]\otimes \cO)$ of \S\ref{ss:dual-bases}, for $\sS \in \{\bT,\bM\}$, is the dual of the $\MCG(\sS^{\natural n})$-representation $H_{k}^{\BM}(C_{\lambda}(\sS^{\natural n} \smallsetminus I); \bZ[Q^{\star}_{(\lambda,\ell)}(\sS)])$. Assigning for each morphism $[\tm,\id_{\tm \natural \tn}]$ the obvious analogue of the map $\iota_{\tm,\tn}$ of \S\ref{sss:def_homological_rep_functors} for homology relative to the boundary, these collections of representations extend to functors $\fL^{\star,\vee}_{(\lambda,\ell)}(\MCGo)$ and $\fL^{\star,\vee}_{(\lambda,\ell)}(\MCGno)$ of the form $\langle\cM\dv,\cM\dv\rangle \to \bZ[Q^{\star}_{(\lambda,\ell)}(\sS)]\lmod$. We may then deduce analogous short exact sequences to those of Theorem~\ref{thm:SES_MCG_alternatives}, and Theorem~\ref{thm:polynomiality_vs_vertical} repeats verbatim for these functors:

\begin{thm}
\label{thm:dual_representation_functors_mcg}
For $\ell \in \{1,2\}$, the functors $\fL^{\star,\vee}_{(\lambda,\ell)}(\MCGo)$ and $\fL^{\star,\vee}_{(\lambda,\ell)}(\MCGno)$ are split polynomial and weak polynomial, of both strong and weak degrees $k$.
\end{thm}

Finally, as a direct consequence Theorem~\ref{thm:homological_stability}, we deduce the following result.
\begin{coro}\label{coro:HS_MCG}
Twisted homological stability holds for the mapping class groups of surfaces with coefficients in the homological representation functors of Theorems~\ref{thm:key_SES_mapping_class_groups} and \ref{thm:SES_MCG_alternatives}.
\end{coro}

\subsection{Analyticity of a quantum representation}\label{ss:analyticity}

Jackson and Kerler \cite{Jackson_Kerler} introduce a representation $\bV$ over the group ring $\bL:=\bZ[\mathfrak{s}^{\pm 1},\mathfrak{q}^{\pm 1}]$, called the \emph{generic Verma module}, of $\bU_{q}(\mathfrak{sl}_{2})$, the quantum enveloping algebra of the Lie algebra $\mathfrak{sl}_{2}$. Since $\bU_{q}(\mathfrak{sl}_{2})$ is a quasitriangular Hopf algebra, the representation $\bV$ comes equipped with an automorphism $S \in \mathrm{Aut}_{\bL}(\bV \otimes \bV)$.
This induces a $\B_{n}$-representation on $\bV^{\otimes n}$ given by sending $\sigma_{i} \in \B_{n}$ to $\id_{i-1} \otimes S \otimes \id_{n-i-1}$, which we call the \emph{Verma module representation}; see \cite[\S 1]{Jackson_Kerler}. For $k\geq0$, the \emph{weight space} $V_{n,k} \subseteq \bV^{\otimes n}$ is the eigenspace of the action of a certain generator $K \in \bU_q(\mathfrak{sl}_{2})$ corresponding to the eigenvalue $\mathfrak{s}^n \mathfrak{q}^{-2k}$. The $\B_{n}$-action on $\bV^{\otimes n}$ restricts to a sub-$\B_{n}$-representation on $V_{n,k}$ for each $k\geq0$, called the \emph{quantum representation} of $\B_{n}$ of \emph{weight} $k$. This provides a decomposition of the Verma module representation via the $\B_{n}$-equivariant isomorphism:
\begin{equation}\label{eq:Verma_decomposition}
\bV^{\otimes n}\;\cong\;\bigoplus_{k\geq0}V_{n,k}.
\end{equation}
The relation between the variables $\mathfrak{s}$ and $\mathfrak{q}$ and the generators $q$ and $t$ of $Q_{((2),2)}(\bD) = \bZ^{2} = \bZ\langle q,t \rangle$ (defining the representation $\LB_{((2),2)}(\tn)$) is given by the ring homomorphism $\Theta\colon \bK:=\bZ[q^{\pm 1},t^{\pm 1}] \to \bL$ defined by $(q,t)\mapsto (\mathfrak{s}^{2},-\mathfrak{q}^{-2})$.
(We note as a warning to the reader that the notation in the literature is not consistent; in particular \cite{Jackson_Kerler} and \cite{martel} use different notation from each other and from the notation used in this section, which is instead consistent with the notation of \cite{bigelow2001braid}.)
In particular, $\bL$ is a left $\bK$-module via $\Theta$; the change of rings operation $- \otimes_\bK \bL$ corresponds to adjoining square roots of $q$ and $t$.

A key relationship between these quantum $\B_{n}$-representations and the homological representation functors studied in this paper is the following lemma.

\begin{lem}\label{lem:Verma_tau_LB}
For $n,k\geq1$ there is an isomorphism of $\B_{n}$-representations $V_{n,k} \cong \tau_{\one} \LB_{k}(\tn)\otimes_{\bK} \bL$.
\end{lem}

\begin{proof}
Let $\bD'_{n}$ denote the closed disc minus $n$ interior points and minus one point on its boundary. An alternative description of $\tau_{\one} \LB_{k}(\tn)$ is given by the twisted Borel-Moore homology of the space of configurations of $k$ unordered points in $\bD'_{n}$, namely the $\B_{n}$-representation $H_{k}^{\BM}(C_{k}(\bD'_{n});\bZ[\bZ^{2}])$ over $\bZ[\bZ^{2}]$. Gluing $\bD'_{0}$ to $\bD'_{n}$ so that the two boundary punctures coincide induces an embedding $\bD'_{n} \hookrightarrow \bD_{1+n}$, which in turn induces an embedding $C_{k}(\bD'_{n})\hookrightarrow C_{k}(\bD_{1+n})$. This latter embedding defines a (covariant) map on Borel-Moore homology since its image is closed, thus it is a proper map, and Borel-Moore homology is covariantly functorial with respect to proper maps; see \cite[Proposition~V.4.5]{bredonsheaf}. We also note that the local coefficient system that we consider on $C_{k}(\bD'_{n})$ is the restriction of the one that we consider on $C_{k}(\bD_{1+n})$. There is therefore a well-defined map
\begin{equation}
\label{eq:quantum-isom-1}
H_{k}^{\BM}(C_{k}(\bD'_{n});\bZ[\bZ^{2}]) \longrightarrow \tau_{\one} \LB_{k}(\tn)
\end{equation}
of $\B_{n}$-representations over $\bZ[\bZ^2] = \bK$. The fact that this map is an isomorphism follows from the evident bijection that it induces on the free bases as $\bK$-modules obtained from Theorem~\ref{thm:free_generating_sets_rep}.

We now consider the subspace $C^-_{k}(\bD_{n}) \subset C_{k}(\bD_{n})$ of all configurations that intersect a particular fixed point on the boundary. In particular, we consider the $\B_{n}$-representation given by the $\bK$-module $H_{k}^{\BM}(C_{k}(\bD_{n}),C^-_{k}(\bD_{n});\bK)$, which is introduced in \cite[\S 2]{martel}.
Since $C_{k}(\bD'_{n})$ is an open subspace of $C_{k}(\bD_{n})$ with closed complement $C^-_{k}(\bD_{n})$, the inclusion $(C_{k}(\bD'_{n}) , \varnothing) \hookrightarrow (C_{k}(\bD_{n}) , C^-_{k}(\bD_{n}))$ is an open embedding. Relative Borel-Moore homology is contravariantly functorial with respect to open embeddings (since it is the composition of reduced homology with the contravariant functor from locally-compact, Hausdorff spaces and open embeddings to based spaces given by one-point compactification), so we have a map
\begin{equation}
\label{eq:quantum-isom-2}
H_{k}^{\BM}(C_{k}(\bD_{n}) , C^-_{k}(\bD_{n}) ; \bK) \longrightarrow H_{k}^{\BM}(C_{k}(\bD'_{n}) ; \bK).
\end{equation}
This is a map of $\B_{n}$-representations over $\bK$ since the $\B_{n}$-action (up to homotopy) on $C_{k}(\bD_{n})$ preserves its partition into $C^-_{k}(\bD_{n})$ and $C_{k}(\bD'_{n})$. The fact that this map is an isomorphism follows from the evident bijection that it induces on the free bases as $\bK$-modules obtained from Theorem~\ref{thm:free_generating_sets_rep} for the right-hand side and \cite[Prop.~$3.6$]{martel} for the left-hand side (see also \cite[Cor.~$3.9$]{martel}).

Now, \cite[Th.~$1.5$]{martel} provides an isomorphism
\begin{equation}
\label{eq:quantum-isom-3}
V_{n,k} \cong H_{k}^{\BM}(C_{k}(\bD_{n}),C^-_{k}(\bD_{n});\bK)\otimes_{\bK} \bL
\end{equation}
of $\B_{n}$-representations over $\bL$. The claimed isomorphism of the lemma is then the composition of \eqref{eq:quantum-isom-1}, \eqref{eq:quantum-isom-2} and \eqref{eq:quantum-isom-3} (tensoring the first two isomorphisms over $\bK$ with $\bL$).
\end{proof}

\begin{coro}\label{coro:Verma_tauLB}
For each $n\geq 2$, there is an isomorphism of $\B_{n}$-representations over $\bL$
\begin{equation}\label{eq:Verma}
\bV^{\otimes n}\;\cong\;\bigoplus_{k\geq0}\tau_{\one} \LB_{k}(\tn) \otimes_{\bK} \bL .
\end{equation}
We may therefore define the \textbf{Verma module representation functor} $\mathfrak{Ver}\colon\langle \Beta, \Beta\rangle\to{\bL}\lmod$ to be the colimit $\bigoplus_{k\geq0} \tau_{\one}\LB_{k} \otimes_{\bK} \bL$.
This functor $\mathfrak{Ver}$ is \textbf{analytic}, i.e.~it is a colimit of polynomial functors, and \textbf{exponential}, i.e.~it is a strong monoidal functor $(\langle \Beta, \Beta\rangle, \natural,\zero)\to ({\bL}\lmod, \otimes,\bL)$. However, the functor $\mathfrak{Ver}$ is \textbf{not polynomial}.
\end{coro}

\begin{proof}
The isomorphisms \eqref{eq:Verma} follow directly from Lemma~\ref{lem:Verma_tau_LB} and the decomposition of the Verma module representation \eqref{eq:Verma_decomposition}.
The analyticity of the functor $\mathfrak{Ver}$ follows from its definition and Corollary~\ref{coro:polynomiality}. We deduce from Theorem~\ref{thm:key_SES_classical_braids} (using Corollary~\ref{coro:change_of_ring}) that $\delta_{\one}^{m} \mathfrak{Ver} \cong \bigoplus_{k\geq0} \tau_{\one}^{m+1}\LB_{k} \otimes_{\bK} \bL$ for all $m\geq 1$. Hence there is a natural embedding $\mathfrak{Ver}\hookrightarrow\delta_{\one}^{m}\mathfrak{Ver}$ for all $m\geq 1$, which proves that the functor $\mathfrak{Ver}$ is not polynomial. That it is an exponential functor straightforwardly follows from the isomorphism $\mathfrak{Ver}(\tn)\cong \bV^{\otimes n}$.
\end{proof}

\begin{rmk}\label{rmk:Magnus_Heisenberg}
Analogous arguments to those of Corollary~\ref{coro:Verma_tauLB} may be repeated verbatim for functors for the mapping class groups of surfaces extending the Magnus representations (see for instance \cite[\S 4]{SakasaiMagnus} or \cite{SuzukiII} for the definition of these representations) or the representations constructed from actions on discrete Heisenberg groups introduced by \cite{BlanchetPalmerShaukat}.
\end{rmk}

\phantomsection
\addcontentsline{toc}{section}{References}
\renewcommand{\bibfont}{\normalfont\small}
\setlength{\bibitemsep}{0pt}
\printbibliography

\clearpage

\noindent Martin Palmer \\
{\itshape School of Mathematics, University of Leeds, Leeds, LS2 9JT, UK},\\ \texttt{m.d.palmer-anghel@leeds.ac.uk}\\
{\itshape Institutul de Matematic\u{a} Simion Stoilow al Academiei Rom{\^a}ne, 21 Calea Griviței, 010702 București, Romania},\\
\texttt{mpanghel@imar.ro}
\vspace{1ex}

\noindent Arthur Souli{\'e} \\
{\itshape Normandie Univ., UNICAEN, CNRS, LMNO, 14000 Caen, France},\\
\texttt{artsou@hotmail.fr}, \texttt{arthur.soulie@unicaen.fr}

\end{document}